\documentclass[11pt]{preprint}
\usepackage[full]{textcomp}
\usepackage[osf]{newtxtext} 
\usepackage{colortbl}

\usepackage{upgreek}
\usepackage{comment}
\usepackage{amssymb}
\usepackage{lmodern}
\usepackage{mathtools}

\usepackage{hyperref}
\usepackage{breakurl}
\usepackage{aligned-overset}
\usepackage{mhenvs}
\usepackage{mhequ} 
\newcommand{\be}{\begin{equation*}}
\newcommand{\ee}{\end{equation*}}
\usepackage{mhsymb}
\usepackage{booktabs}
\usepackage{tcolorbox}
\usepackage[utf8]{inputenc}
\usepackage{longtable}
\usepackage{wrapfig}
\usepackage{subcaption}
\usepackage{epsfig}
\usepackage{microtype}
\usepackage[dvipsnames]{xcolor}
\usepackage{mathrsfs}

\usepackage{centernot}
\usepackage{enumitem}
\usepackage{bm}
\usepackage{stackrel}
\usepackage{graphicx}
\usepackage{ulem}

\usepackage[nameinlink,noabbrev]{cleveref} 

\makeatletter
\newcommand{\globalcolor}[1]{%
  \color{#1}\global\let\default@color\current@color
}
\makeatother

\usetikzlibrary{calc}
\usetikzlibrary{decorations}
\usetikzlibrary{positioning}
\usetikzlibrary{shapes}
\usetikzlibrary{shapes.misc}

\tikzset{cross/.style={cross out, draw=black, fill=none, minimum size=2*(#1-\pgflinewidth), inner sep=0pt, outer sep=0pt}, cross/.default={2pt}}

\definecolor{blush}{rgb}{0.87, 0.36, 0.51}
	\definecolor{brightcerulean}{rgb}{0.11, 0.67, 0.84}
	\definecolor{greenryb}{rgb}{0.4, 0.69, 0.2}

\newif\ifdark
\darkfalse

\ifdark
\definecolor{darkred}{rgb}{0.9,0.2,0.2}
\definecolor{darkblue}{rgb}{0.7,0.3,1}
\definecolor{darkgreen}{rgb}{0.1,0.9,0.1}
\definecolor{franck}{rgb}{0,0.8,1}
\definecolor{pagebackground}{rgb}{.15,.21,.18}
\definecolor{pageforeground}{rgb}{.84,.84,.85}
\pagecolor{pagebackground}
\AtBeginDocument{\globalcolor{pageforeground}}
\definecolor{symbols}{rgb}{0,0.7,1}
\colorlet{connection}{red!80!black}
\colorlet{boxcolor}{blue!50}

\else

\definecolor{darkred}{rgb}{0.7,0.1,0.1}
\definecolor{darkblue}{rgb}{0.4,0.1,0.8}
\definecolor{darkgreen}{rgb}{0.1,0.7,0.1}
\definecolor{franck}{rgb}{0,0,1}
\definecolor{pagebackground}{rgb}{1,1,1}
\definecolor{pageforeground}{rgb}{0,0,0}
\colorlet{symbols}{blue!90!black}
\colorlet{connection}{red!30!black}
\colorlet{boxcolor}{blue!50!black}

\fi

\def\slash{\leavevmode\unskip\kern0.18em/\penalty\exhyphenpenalty\kern0.18em}
\def\dash{\leavevmode\unskip\kern0.18em--\penalty\exhyphenpenalty\kern0.18em}

\DeclareMathAlphabet{\mathbbm}{U}{bbm}{m}{n}

\DeclareFontFamily{U}{BOONDOX-calo}{\skewchar\font=45 }
\DeclareFontShape{U}{BOONDOX-calo}{m}{n}{
  <-> s*[1.05] BOONDOX-r-calo}{}
\DeclareFontShape{U}{BOONDOX-calo}{b}{n}{
  <-> s*[1.05] BOONDOX-b-calo}{}
\DeclareMathAlphabet{\mcb}{U}{BOONDOX-calo}{m}{n}
\SetMathAlphabet{\mcb}{bold}{U}{BOONDOX-calo}{b}{n}

\setlist{noitemsep,topsep=4pt,leftmargin=1.5em}

\DeclareMathAlphabet{\BF}{OT1}{ptm}{bx}{n}
\DeclareMathAlphabet{\mathbbm}{U}{bbm}{m}{n}

\DeclareMathAlphabet{\mcb}{U}{BOONDOX-calo}{m}{n}
\SetMathAlphabet{\mcb}{bold}{U}{BOONDOX-calo}{b}{n}
\DeclareFontFamily{U}{mathx}{\hyphenchar\font45}
\DeclareFontShape{U}{mathx}{m}{n}{
      <5> <6> <7> <8> <9> <10>
      <10.95> <12> <14.4> <17.28> <20.74> <24.88>
      mathx10
      }{}
\DeclareSymbolFont{mathx}{U}{mathx}{m}{n}
\DeclareMathSymbol{\bigtimes}{1}{mathx}{"91}

\def\i{\mathfrak{i}}

\newcommand{\talloblong}{
\,\begin{tikzpicture}[baseline=-0.7ex]
    \draw[line width=.5pt] (0,-0.2) rectangle (0.08,0.20);
\end{tikzpicture}\,
}

\providecommand{\figures}{false}
{ \ifthenelse{\equal{\figures}{false}} {#1}{\[ {\rm Figure \ missing !} \]} }{}

\def\Id{\mathrm{Id}}

\def\mcT{\mathcal{T}}

\tikzstyle{tinydots}=[dash pattern=on \pgflinewidth off \pgflinewidth]
\tikzstyle{superdense}=[dash pattern=on 4pt off 1pt]

\newcommand{\mcE}{\mathcal{E}}
\newcommand{\mcH}{\mathcal{H}}

\newcommand{\mcV}{\mathcal{V}}
\newcommand{\mcR}{\mathcal{R}}

\newcommand{\mcB}{\mathcal{B}}

\newcommand{\mcS}{\CS}

\newcommand{\mcL}{\mathcal{L}}

\newcommand{\mcK}{\mathcal{K}}
\newcommand{\mcD}{\mathcal{D}}

\newcommand{\mcP}{\mathcal{P}}

\newcommand{\rmd}{{\mathrm{d}}}
\newcommand{\e}{{\mathrm{e}}}

\newcommand{\beq}{\begin{equation}}
\newcommand{\eeq}{\end{equation}}

\usepackage{empheq}

\newcommand{\mfc}{\mathfrak{c}}

\newcommand{\mfC}{\mathfrak{C}}

\newcommand{\mfl}{\mathfrak{l}}

\newcommand{\mfs}{\mathfrak{s}}

\newcommand{\msB}{\mathscr{B}}

\newcommand{\msQ}{\mathscr{Q}}
\newcommand{\msP}{\mathscr{P}}
\newcommand{\msR}{\mathscr{R}}
\newcommand{\msH}{\mathscr{H}}

\newcommand{\epmu}{{\eps,\mu}}
\newcommand{\epnu}{{\eps,\nu}}
\def\i{\mathrm{i}}
\def\x{\boldsymbol{x}}

\def\z{\boldsymbol{z}}
\def\w{\boldsymbol{w}}

\def\v{\boldsymbol{v}}
\def\n{\boldsymbol{n}}

\newcommand{\p}{\boldsymbol{p}}

\newcommand{\rmD}{\mathrm{D}}

\DeclareSymbolFont{rmlargesymbols}{OMX}{mdbch}{m}{n}
\DeclareMathSymbol{\rmintop}{\mathop}{rmlargesymbols}{82}

\newcommand{\D}{\partial}
\newcommand{\dint}{\,\mathrm{d}}
\newcommand\Item[1][]{%
  \ifx\relax#1\relax  \item \else \item[#1] \fi
  \abovedisplayskip=0pt\abovedisplayshortskip=0pt~\vspace*{-\baselineskip}}

\newcounter{theorem}

\newenvironment{DIFnomarkup}{}{} 

\newtheorem{assumption}{Assumption}
 
\theorembodyfont{\rmfamily}

\newfont{\indic}{bbmss12}

\definecolor{indigo}{rgb}{0.4,0,1}

\makeatletter
\pgfdeclareshape{crosscircle}
{
  \inheritsavedanchors[from=circle] 
  \inheritanchorborder[from=circle]
  \inheritanchor[from=circle]{north}
  \inheritanchor[from=circle]{north west}
  \inheritanchor[from=circle]{north east}
  \inheritanchor[from=circle]{center}
  \inheritanchor[from=circle]{west}
  \inheritanchor[from=circle]{east}
  \inheritanchor[from=circle]{mid}
  \inheritanchor[from=circle]{mid west}
  \inheritanchor[from=circle]{mid east}
  \inheritanchor[from=circle]{base}
  \inheritanchor[from=circle]{base west}
  \inheritanchor[from=circle]{base east}
  \inheritanchor[from=circle]{south}
  \inheritanchor[from=circle]{south west}
  \inheritanchor[from=circle]{south east}
  \inheritbackgroundpath[from=circle]
  \foregroundpath{
    \centerpoint%
    \pgf@xc=\pgf@x%
    \pgf@yc=\pgf@y%
    \pgfutil@tempdima=\radius%
    \pgfmathsetlength{\pgf@xb}{\pgfkeysvalueof{/pgf/outer xsep}}%
    \pgfmathsetlength{\pgf@yb}{\pgfkeysvalueof{/pgf/outer ysep}}%
    \ifdim\pgf@xb<\pgf@yb%
      \advance\pgfutil@tempdima by-\pgf@yb%
    \else%
      \advance\pgfutil@tempdima by-\pgf@xb%
    \fi%
    \pgfpathmoveto{\pgfpointadd{\pgfqpoint{\pgf@xc}{\pgf@yc}}{\pgfqpoint{-0.707107\pgfutil@tempdima}{-0.707107\pgfutil@tempdima}}}
    \pgfpathlineto{\pgfpointadd{\pgfqpoint{\pgf@xc}{\pgf@yc}}{\pgfqpoint{0.707107\pgfutil@tempdima}{0.707107\pgfutil@tempdima}}}
    \pgfpathmoveto{\pgfpointadd{\pgfqpoint{\pgf@xc}{\pgf@yc}}{\pgfqpoint{-0.707107\pgfutil@tempdima}{0.707107\pgfutil@tempdima}}}
    \pgfpathlineto{\pgfpointadd{\pgfqpoint{\pgf@xc}{\pgf@yc}}{\pgfqpoint{0.707107\pgfutil@tempdima}{-0.707107\pgfutil@tempdima}}}
  }
}
\makeatother

\def\decorate#1#2{
        \ifnum#2>0
    		\foreach \count in {1,...,#2}{
	       	let
				\p1 = (sourcenode.center),
                \p2 = (sourcenode.east),
				\n1 = {\x2-\x1},
				\n2 = {1mm},
				\n3 = {(1.3+0.6*(\count-1))*\n1},
				\n4 = {0.7*\n1}
			in 
        		node[rectangle,fill=symbols,rotate=30,inner sep=0pt,minimum width=0.2*\n2,minimum height=\n2] at ($(sourcenode.center) + (\n3,\n4)$) {}
				}
		\fi
        \ifnum#1>0
    		\foreach \count in {1,...,#1}{
	       	let
				\p1 = (sourcenode.center),
                \p2 = (sourcenode.east),
				\n1 = {\x2-\x1},
				\n2 = {1mm},
				\n3 = {(1.3+0.6*(\count-1))*\n1},
				\n4 = {0.7*\n1}
			in 
        		node[rectangle,fill=symbols,rotate=-30,inner sep=0pt,minimum width=0.2*\n2,minimum height=\n2] at ($(sourcenode.center) + (-\n3,\n4)$) {}
				}
		\fi
}

\tikzset{
    dectriangle/.style 2 args={
        triangle,
        alias=sourcenode,
        append after command={\decorate{#1}{#2}}
    },
    dectriangle/.default={0}{0},
}

\tikzset{
	dot/.style={circle,fill=black,inner sep=0pt, minimum size=1mm},
	graydot/.style={circle,draw=gray,inner sep=0pt, minimum size=1mm},
	loopnode/.style={circle,draw=black,inner sep=0pt, minimum size=1.5mm},
	treenode/.style={circle,fill=black,inner sep=0pt, minimum size=1.5mm},
	charge/.style={circle,draw=black,inner sep=0pt, minimum size=3mm},
	loopline/.style={->,semithick,shorten >=1pt,shorten <=1pt},
	treeline/.style={semithick, densely dashed, shorten >=1pt,shorten <=1pt},
	homoge/.style={font=\scriptsize,draw=black,fill=white,inner sep=2pt},
	homo/.style={font=\scriptsize},
	KK/.style={thick,shorten >=1pt,shorten <=1pt, densely dashed, >=latex},
kerAlg/.style={thick},
kerAlg2/.style={semithick,dashed},
kerP/.style={very thick},
arrow/.style={->,thick,shorten >=3pt,shorten <=3pt},
renorm/.style={ultra thick},
	kernel/.style={semithick,shorten >=1pt,shorten <=1pt},
	}

\tikzset{
	cross/.style={path picture={ 
  		\draw[symbols]
			(path picture bounding box.south east) -- (path picture bounding box.north west) (path picture bounding box.south west) -- (path picture bounding box.north east);
		}},
root/.style={circle,fill=green!50!black,inner sep=0pt, minimum size=1.2mm},
        dot/.style={circle,fill=pageforeground,inner sep=0pt, minimum size=1mm},
        dotred/.style={circle,fill=pageforeground!50!pagebackground,inner sep=0pt, minimum size=2mm},
        var/.style={circle,fill=pageforeground!10!pagebackground,draw=pageforeground,inner sep=0pt, minimum size=3mm},
        kernel/.style={semithick,shorten >=2pt,shorten <=2pt},
        kernels/.style={snake=zigzag,shorten >=2pt,shorten <=2pt,segment amplitude=1pt,segment length=4pt,line before snake=2pt,line after snake=5pt,},
        rho/.style={densely dashed,semithick,shorten >=2pt,shorten <=2pt},
           testfcn/.style={dotted,semithick,shorten >=2pt,shorten <=2pt},
        renorm/.style={shape=circle,fill=pagebackground,inner sep=1pt},
        labl/.style={shape=rectangle,fill=pagebackground,inner sep=1pt},
        xic/.style={very thin,circle,draw=symbols,fill=symbols,inner sep=0pt,minimum size=1.2mm},
        g/.style={very thin,rectangle,draw=symbols,fill=symbols!10!pagebackground,inner sep=0pt,minimum width=2.5mm,minimum height=1.2mm},
        xi/.style={very thin,circle,draw=symbols,fill=symbols!10!pagebackground,inner sep=0pt,minimum size=1.2mm},
	xies/.style={very thin,rectangle,fill=green!50!black!25,draw=symbols,inner sep=0pt,minimum size=1.1mm},
	xiesf/.style={very thin,rectangle,fill=green!50!black,draw=symbols,inner sep=0pt,minimum size=1.1mm},
        xix/.style={very thin,crosscircle,fill=symbols!10!pagebackground,draw=symbols,inner sep=0pt,minimum size=1.2mm},
        X/.style={very thin,cross,rectangle,fill=pagebackground,draw=symbols,inner sep=0pt,minimum size=1.2mm},
	xib/.style={thin,circle,fill=symbols!10!pagebackground,draw=symbols,inner sep=0pt,minimum size=1.6mm},
	xie/.style={thin,circle,fill=green!50!black,draw=symbols,inner sep=0pt,minimum size=1.6mm},
	xid/.style={thin,circle,fill=symbols,draw=symbols,inner sep=0pt,minimum size=1.6mm},
	xibx/.style={thin,crosscircle,fill=symbols!10!pagebackground,draw=symbols,inner sep=0pt,minimum size=1.6mm},
	kernels2/.style={very thick,draw=connection,segment length=12pt},
	keps/.style={thin,draw=symbols,->},
	kepspr/.style={thick,draw=connection,->},
	krho/.style={thin,draw=symbols,superdense,->},
	krhopr/.style={thick,draw=connection,superdense},
	triangle/.style = { regular polygon, regular polygon sides=3},
	not/.style={thin,circle,draw=connection,fill=connection,inner sep=0pt,minimum size=0.5mm},
	diff/.style = {very thin,draw=symbols,triangle,fill=red!50!black,inner sep=0pt,minimum size=1.6mm},
	diff1/.style = {very thin,dectriangle={1}{0},fill=red!50!black,draw=symbols,inner sep=0pt,minimum size=1.6mm},
	diff2/.style = {very thin,dectriangle={1}{1},fill=red!50!black,draw=symbols,inner sep=0pt,minimum size=1.6mm},
		diffmini/.style = {very thin,rectangle,fill=black,draw=black,inner sep=0pt,minimum size=0.75mm},
	 kernelsmod/.style={very thick,draw=connection,segment length=12pt},
	 rec/.style = {very thin,rectangle,fill=black,draw=black,inner sep=0pt,minimum size=2mm},
	cerc/.style={very thin,circle,draw=black,fill=symbols,inner sep=0pt,minimum size=2mm},
	stars/.style={very thin,star,star points=6,star point ratio=0.5, draw=black,fill=red,inner sep=0pt,minimum size=0.7mm},
	>=stealth,
        }
        \tikzset{
root/.style={circle,fill=black!50,inner sep=0pt, minimum size=3mm},
        circ/.style={circle,fill=white,draw=black,very thin,inner sep=.5pt, minimum size=1.2mm},
        round1/.style={fill=white,outer sep = 0,inner sep=2pt,rounded corners=1mm,draw,text=black,thin,minimum size=1.2mm},
          circ1/.style={circle,fill=red!10,draw=red,very thin,inner sep=.5pt, minimum size=1.2mm},
        rect/.style={fill=white,outer sep = 0,inner sep=2pt,rectangle,draw,text=black,thin,minimum size=1.2mm},
        rect1/.style={fill=white,outer sep = 0,inner sep=2pt,rectangle,draw,text=black,thin,minimum size=1.2mm},
        round2/.style={fill=red!10,outer sep = 0,inner sep=2pt,rounded corners=1mm,draw,text=black,thin,minimum size=1.2mm},
       round3/.style={fill=blue!10,outer sep = 0,inner sep=2pt,rounded corners=1mm,draw,text=black,thin,minimum size=1.2mm}, 
        rect2/.style={fill=black!10,outer sep = 0,inner sep=2pt,rectangle,draw,text=black,thin,minimum size=1.2mm},
        dot/.style={circle,fill=black,inner sep=0pt, minimum size=1.2mm},
        dotred/.style={circle,fill=black!50,inner sep=0pt, minimum size=2mm},
        var/.style={circle,fill=black!10,draw=black,inner sep=0pt, minimum size=3mm},
        kernel/.style={semithick,shorten >=2pt,shorten <=2pt},
         diag/.style={thin,shorten >=4pt,shorten <=4pt},
        kernel1/.style={thick},
        kernels/.style={snake=zigzag,shorten >=2pt,shorten <=2pt,segment amplitude=1pt,segment length=4pt,line before snake=2pt,line after snake=5pt,},
		kernels1/.style={snake=zigzag,segment amplitude=0.5pt,segment length=2pt},
		rho1/.style={densely dotted,semithick},
        rho/.style={densely dashed,semithick,shorten >=2pt,shorten <=2pt},
           testfcn/.style={dotted,semithick,shorten >=2pt,shorten <=2pt},
           visible/.style={draw, circle, fill, inner sep=0.25ex},
        renorm/.style={shape=circle,fill=white,inner sep=1pt},
        labl/.style={shape=rectangle,fill=white,inner sep=1pt},
        xic/.style={very thin,circle,fill=symbols,draw=black,inner sep=0pt,minimum size=1.2mm},
        xi/.style={very thin,circle,fill=blue!10,draw=black,inner sep=0pt,minimum size=1.2mm},
	xib/.style={very thin,circle,fill=blue!10,draw=black,inner sep=0pt,minimum size=1.6mm},
	xie/.style={very thin,circle,fill=green!50!black,draw=black,inner sep=0pt,minimum size=1mm},
	xid/.style={very thin,circle,fill=symbols,draw=black,inner sep=0pt,minimum size=1.6mm},
	edgetype/.style={very thin,circle,draw=black,inner sep=0pt,minimum size=5mm},
	nodetype/.style={very thick,circle,draw=black,inner sep=0pt,minimum size=5mm},
	kernels2/.style={very thick,draw=connection,segment length=12pt},
clean/.style={thin,circle,fill=black,inner sep=0pt,minimum size=1mm},	not/.style={thin,circle,fill=symbols,draw=connection,fill=connection,inner sep=0pt,minimum size=0.8mm},
	>=stealth,
        }

\def\one{\BF{1}}
\def\eps{\varepsilon}
\newcommand{\snnorm}[1]{{\big\vert\kern-0.25ex\big\vert\kern-0.25ex\big\vert #1 
    \big\vert\kern-0.25ex\big\vert\kern-0.25ex\big\vert}}
\newcommand{\nnorm}[1]{{\Big\vert\kern-0.25ex\Big\vert\kern-0.25ex\Big\vert #1 
    \Big\vert\kern-0.25ex\Big\vert\kern-0.25ex\Big\vert}}
\newcommand{\bnnorm}[1]{{\Biggl\vert\kern-0.25ex\Biggl\vert\kern-0.25ex\Biggl\vert #1 
    \Biggl\vert\kern-0.25ex\Biggl\vert\kern-0.25ex\Biggl\vert}}

\DeclareMathAlphabet{\mathpzc}{OT1}{pzc}{m}{it}

\let\eps\varepsilon

\def\eqref#1{(\ref{#1})}

\makeatletter 
\newcommand*{\bigcdot}{}
\DeclareRobustCommand*{\bigcdot}{%
  \mathbin{\mathpalette\bigcdot@{}}%
}
\newcommand*{\bigcdot@scalefactor}{.5}
\newcommand*{\bigcdot@widthfactor}{1.15}
\newcommand*{\bigcdot@}[2]{%
  \sbox0{$#1\vcenter{}$}
  \sbox2{$#1\cdot\m@th$}%
  \hbox to \bigcdot@widthfactor\wd2{%
    \hfil
    \raise\ht0\hbox{%
      \scalebox{\bigcdot@scalefactor}{%
        \lower\ht0\hbox{$#1\bullet\m@th$}%
      }%
    }%
    \hfil
  }%
}
\makeatother

\tcbset
{colframe=boxcolor,colback=symbols!7!pagebackground,coltext=pageforeground,
fonttitle=\bfseries,nobeforeafter,center title,size=fbox,boxsep=1.5pt,
top=0mm,bottom=0mm,boxsep=0mm,tcbox raise base}

\def\two{{\<generic>\kern0.05em\<genericb>}}
\def\twoI{{\<Ito>\kern0.05em\<Itob>}}

\def\mail#1{\burlalt{#1}{mailto:#1}}

\usepackage{thmtools}

\usepackage[margin=6em]{geometry}

\renewcommand{\leq}{\leqslant}
\renewcommand{\geq}{\geqslant}

\begin{document}

\title{Stochastic PDE construction of some non-local or non-translation-invariant Anderson operators in the full subcritical regime}

\author{L. Ferdinand$^1$, S. Mackowiak$^{2,3}$ and A. Minguella$^3$}
\institute{ 
 Max Planck Institute for Mathematics in the Sciences, Leipzig, Germany \and Univ Brest,  CNRS UMR 6205,  Laboratoire de Mathématiques de Bretagne Atlantique,  F-29200,  Brest,  France \and Université de Lorraine, CNRS, IECL, F-54000 Nancy, France
  \\
Email:\ \begin{minipage}[t]{\linewidth}
\mail{leonard.ferdinand@mis.mpg.de}\\
\mail{samael.mackowiak@univ-lorraine.fr}\\
\mail{aurelien.minguella@univ-lorraine.fr}
\end{minipage}}

\maketitle 

\begin{abstract}
    We construct Anderson-type operators in the full subcritical regime via a SPDE argument. Our work encompasses perturbations by a Gaussian noise of the fractional Laplacian on the torus or on the full space, and of the harmonic oscillator on the full space. In every case, we are able to derive an optimal estimate on the difference between the solution and the deterministic heat flow, which allows us to get a Weyl law for the fractional case on the torus and for the harmonic oscillator case, and so in the full subcritical regime. For the non-translation-invariant operator, the main novelty is that we develop a general argument in order to prove how to renormalise the operator with some constants. The SPDE argument uses the flow approach introduced by Duch \cite{Duc21}. To construct the operators, we use the Hille-Yosida theorem for the fractional Laplacian on the torus and the harmonic oscillator, and the Klein-Landau theorem for the fractional Laplacian in the full space.
\end{abstract}
\setcounter{tocdepth}{2}
\setcounter{secnumdepth}{4}
\tableofcontents

\section{Introduction and main results}



The aim of the present work is to construct an analyse some non-translation-invariant and non-local Anderson operators on $\T^d$ or $\R^d$, in the full subcritical regime. These Anderson operators are random Schr\"odinger operators that formally read
\begin{equs}
    \mathrm{A}^{\!\alpha,m}_\xi=\mathrm{D}^{\alpha,m}-\xi+\infty\,,
\end{equs}
where $\mathrm{D}^{\alpha,m}$ denotes a suitable pseudo-differential operator of order $\alpha\in(0,2]$ and $\xi$ is a Gaussian noise. The problem is scaling subcritical if the almost sure regularity of $\xi$ is greater than $-\alpha$. In the above display, the term $\infty$ is here to recall that such operators must be constructed as limits of a sequence of approximations reading 
\begin{equs}
    \mathrm{A}^{\!\alpha,m}_{\xi,\eps}=\mathrm{D}^{\alpha,m}-\xi_\eps+\mfC^{\alpha,m}_\eps\,,
\end{equs}
where the noise $\xi_\eps$ is smooth and where the counterterm $\mfC^{\alpha,m}_\eps$ is allowed to diverge as $\eps\searrow0$.

We will construct the fractional Anderson operator by studying its semi-group $\e^{-t\mathrm A^{\!\alpha,m}_\xi}$, whose action on some function $u_0$ in obtained solving the parabolic linear stochastic PDE
\begin{equation}\label{mainSPDE}
(\D_t+\mathrm{D}^{\alpha,m})u^{\alpha,m}=(\xi-\infty) u^{\alpha,m}\,,\qquad u^{\alpha,m}\!\!\upharpoonright_{t=0}\,=u^{\alpha,m}_0\,.
\end{equation}
In fact, \eqref{mainSPDE} is a formal way of writing a singular equation and, in order to solve it, it is necessary to approximate the solution with a sequence $(u_\eps)_\eps$ of regularised solutions solving the equation
\begin{equation}\label{Eq:RenomalizedParabolicRegularized}
    (\D_t+\mathrm D^{\alpha,m})u^{\alpha,m}_\eps=(\xi_\eps-\mfC^{\alpha,m}_\eps) u^{\alpha,m}_\eps\,,\qquad u^{\alpha,m}_\eps\!\!\upharpoonright_{t=0}\,=u^{\alpha,m}_{0,\eps}\,,
\end{equation}
where $\xi_\eps$ denotes a smooth approximation of $\xi$ and $\mfC^{\alpha,m}_\eps$ is a suitably chosen deterministic constant depending on the law of $\xi_\eps$ only.

To make things concrete, let us give our assumptions on $\mathrm{D}^\alpha$ and $\xi$, starting with the noise.

\begin{assumption}\label{condalphabeta}
 Both on $\T^d$ and $\R^d$, we assume that $\xi$ is a translation-invariant Gaussian noise with covariance given by 
 \begin{equs}
\mathrm{Cov}(x)\eqdef\E[\xi(x)\xi(0)]=\langle\nabla\rangle^{2\alpha-2\beta-d}(x)\,,
 \end{equs}
 where
\begin{equation}\label{condbeta}
    0\vee (\alpha/2-d/{4})<\beta\leq\alpha/2\,.
\end{equation}
\end{assumption}

\begin{remark}
The choice of covariance kernel is for concreteness. In fact, we could handle any Gaussian noise with covariance kernel satisfying the estimate 
    \begin{equs}
        |\partial^\mfl_x\mathrm{Cov}(x)|\lesssim |x|^{-2(\alpha-\beta)-|\mfl|}\e^{-c|x|}\,.
    \end{equs}
Regarding the restrictions on $\beta$, they come from the following:
\begin{itemize}
    \item For $\beta>\alpha/2$, the equation is not singular. Indeed, $\xi$ being of almost sure regularity $-\alpha+\beta-\kappa$, the solution $u$ will be of regularity $\beta-\kappa$, so the product $\xi u$ is well-defined since the regularities sum to $\beta-2\kappa>0$.
    \item The equation is scaling subcritical for $\beta>0$. Indeed, with the above heuristic, one can expect the solution to be of regularity $\beta-2\kappa$, so multiplying the noise by the solution should improve its behaviour only for $\beta>0$ (one can hope that $\xi u$ behaves like the product of the regularities, so $-\alpha+2\beta-2\kappa>-\alpha$).
    \item If $d< 2\alpha$, then for $\beta\leqslant \alpha/2-d/4$ the equation could display a covariance blow up. Indeed, with the above heuristic, the product $\xi u$ would indeed behaves like $-\alpha+2\beta-2\kappa<-d/2$, so its covariance might not be of finite total variation.
\end{itemize}
\end{remark}

Let us now turn to our assumptions on the pseudo-differential operator $\mathrm{D}^{\alpha,m}$. In practice, we have $\mathrm{D}^{\alpha,m}=(-\Delta+m^2)^{\alpha/2}$ where either $\alpha\in(0,2]$ and $m\in[0,\infty)$ or $\alpha=2$ and $m=q$, where $q$ is the operator $\mcS'(\R^d)\rightarrow\mcS'(\R^d)$ acting by multiplication by $|x|$, so $q\cdot f(x)\eqdef |x|f(x)$. In other words, we will be interested by the cases of the possibly massive \textit{fractional Laplacian} $(-\Delta+m^2)^{\alpha/2}$, and of the \textit{harmonic oscillator} $\mathrm{H}\eqdef \mathrm{ D}^{2,q}\equiv-\Delta+q^2$. Note that we only defined $q$ on $\R^d$, since the harmonic oscillator is mostly interesting in infinite volume. On the other hand, while on $\T^d$ we can cover the case of the massless fractional Laplacian, on $\R^d$ we must restrict ourselves to the massive fractional Laplacian. For these reasons, we formulate below two different sets of assumptions, one for the case of the torus and one for the infinite volume.
\begin{assumption}\label{assum:torus}
 On the torus $\T^d$, we consider the following three situations:
    \begin{enumerate}
     \item The Laplacian: $\alpha=2$ and $m=0$, so $\rmD^{2,0}=-\Delta$.
      \item The massive fractional Laplacian: $\alpha\in(0,2)$ and $m>0$, so $\rmD^{\alpha,m}=(-\Delta+m^2)^{\alpha/2}$. 
        \item The massless fractional Laplacian: $\alpha\in(0,2)$ and $m=0$, so $\rmD^{\alpha,0}=|\nabla|^{\alpha}$.  
    \end{enumerate}
\end{assumption}
\begin{assumption}\label{assum:infinitevol}
 In infinite volume $\R^d$, we consider the following three situations:
  \begin{enumerate}
  \item The Laplacian: $\alpha=2$ and $m=0$, so   
      $\rmD^{2,0}=-\Delta$.
        \item The massive fractional Laplacian: $\alpha\in(0,2)$ and $m>0$, so $\rmD^{\alpha,m}=(-\Delta+m^2)^{\alpha/2}$. 

 \item The harmonic oscillator: $\alpha=2$ and $m=q$, so  
      $\mathrm{H}=-\Delta+q^2$. 
    \end{enumerate}
\end{assumption}

\subsection{Main results}

Our main results are organised as follows. In our main two theorems, we state a global well-posedness result for the associated parabolic Anderson models (see Theorem~\ref{maintheoremspde}), which is the main input in order to prove Theorem~\ref{Th:ExistenceAndersonOp} giving a construction and characterisation of the Anderson operator.
\begin{theorem}\label{maintheoremspde}
Suppose that Assumption~ \ref{condalphabeta} is satisfied. Let $r\in[1,\infty]$ and $\kappa>0$ a small parameter.
\begin{enumerate}
    \item Consider the equation~\eqref{Eq:RenomalizedParabolicRegularized} on $\R_{\geqslant0}\times\T^d$, in the setting of either items of Assumption~\ref{assum:torus} (\ie with the Laplacian or the possibly massless fractional Laplacian) and with initial condition $u_0\in\mcB^{-\sigma}_{r,\infty}(\T^d)$ for some $\sigma\in[0,\beta)$. Then, there exists a diverging deterministic constant $\mfC^{\alpha,m}_\eps$ unique up to shift by a real number, such that the solution $u_\eps$ of (\ref{Eq:RenomalizedParabolicRegularized}) converges in probability to some non-trivial function $u$ in $L_t^\infty\mcB_{p,\infty}^{\gamma}(\R_{\geqslant0}\times\T^d)$. This limit is independent from the choice of mollification. Moreover, the solution $u$ satisfies the estimate uniformly in $t>0$ and for $(\sigma,\gamma,p,\eta)$ such that $\eta>0$, $p\geq r$, $0<\sigma<\beta$ and $0\leq \gamma< \beta-\sigma-d(1/r-1/p)$,
    \begin{equation}\label{Eq:EstimateS-G}
        \big\|u(t)-\e^{-t\rmD^{\alpha,m}}u_0\big\|_{\mcB_{p,\infty}^\gamma(\T^d)}\leq \e^{ct}(t\wedge1)^{(\beta-\sigma-\gamma-d(1/r-1/p)-\eta)/\alpha}  \|u_0\|_{\mcB_{r,\infty}^{-\sigma}(\T^d)}\,,
    \end{equation}
    for some positive random constant $c$.
    \item Consider the equation~\eqref{Eq:RenomalizedParabolicRegularized} on $\R_{\geqslant0}\times\R^d$, in the setting of one of the first two items of Assumption~\ref{assum:infinitevol} (i.e with the heat operator or the massive fractional Laplacian). Then, there exists a diverging constant $\mfC^{\alpha,m}_\eps$ unique up to shift by a real number such that for any $u_0\in\mcB^{-\sigma}_{r,\infty}(\R^d)$ and $T>0$, the solution $u_\eps$ converges in probability to some function $u$ in the weighted space $\mcE^{\gamma,-\gamma-\sigma-d(1/r-1/p)}_{p,\ell,T}$ for $(\sigma,\gamma,p,\eta)$ such that $\eta>0$, $p\geq r$, $q\in[1,\infty]$, $0<\sigma<\beta$, and $0\leq \gamma< \beta-\sigma-d(1/r-1/p)$ (see \eqref{weightedspace} for a definition of this space), and for every $\ell>0$. This limit is independent of the choice of mollification. Moreover, the solution satisfies for $t\in(0,T]$,
    \begin{equation}\label{eq:estimateS-Ginf}
        \|\mcb e_{-\ell-t} (u(t)-\e^{-t\rmD^\alpha}u_0)\|_{\mcB_{p,q}^\gamma(\R^d)}\lesssim_T(t\wedge 1)^{(\beta-\sigma-\gamma-d(1/r-1/p)-\eta)/\alpha}  \|\mcb e_{-\ell}u_0\|_{\mcB_{r,q}^{-\sigma}(\R^d)}\,.
    \end{equation}
    \item Consider the equation~\eqref{Eq:RenomalizedParabolicRegularized} on $\R_{\geqslant0}\times\R^d$, in the setting of the third item of Assumption~\ref{assum:infinitevol} (\ie with harmonic oscillator) and with initial condition $u_0\in\msB^{-\sigma}_{r,\infty}(\R^d)$ for some $\sigma\in[0,\beta)$. Then, taking the same renormalisation constant as in the above point for the Laplacian, i.e. setting $\mfC^{2,q}_\eps=\mfC^{2,0}_\eps$, the solution $u_\eps$ of (\ref{Eq:RenomalizedParabolicRegularized}) converges in probability to some non-trivial function $u$ in $L_t^\infty \msB_{p,\infty}^{\gamma}(\R_{\geqslant0}\times\R^d)$. This limit is independent from the choice of mollification. Moreover, the solution $u$ satisfies the estimate uniformly in $t>0$ and for $(\sigma,\gamma,p,\eta)$ such that $\eta>0$, $p\geq r$, $0<\sigma<\beta$ and $0\leq \gamma< \beta-\sigma-d(1/r-1/p)$,
    \begin{equation}\label{Eq:EstimateS-GOH}
        \big\|u(t)-\e^{-t\mathrm{H}}u_0\big\|_{\msB_{p,\infty}^\gamma(\R^d)}\leq \e^{ct}(t\wedge1)^{(\beta-\sigma-\gamma-d(1/r-1/p)-\eta)/2}  \|u_0\|_{\msB_{r,\infty}^{-\sigma}(\R^d)}\,,
    \end{equation}
    for some positive random constant $c$.
\end{enumerate}
\end{theorem}
\begin{proof}
    The three items of the theorem are proved in the second section, see respectively Corollary~\ref{Cor:GlobalisationSeps}, the end of
Section~\ref{sec:infinitevol}, and Corollary~\ref{Cor:GlobalisationSepsOH}
\end{proof}

\begin{remark}\label{Rem:RemarksTheorem1}
    A few remarks are in order.
    \begin{itemize}
        \item The sharp estimate \eqref{Eq:EstimateS-G} was obtained for the first time in \cite{BDM22} in the case of the Anderson operator on compact Riemannian surfaces via paracontrolled calculus. Our result generalises it to an arbitrarily long ansatz. It is, in addition, extended to the more general contexts with \eqref{eq:estimateS-Ginf} and \eqref{Eq:EstimateS-GOH}.
        \item These estimates are refined versions of the Schauder estimate that reads in the case of the torus for example
        \begin{equation}
            \|  u(t)\|_{\mcB_{p,\infty}^\gamma(\T^d)}\leq \e^{ct}(t\wedge 1)^{(-\sigma-\gamma-d(1/r-1/p))/\alpha}  \|u_0\|_{\mcB_{r,\infty}^{-\sigma}(\T^d)}\,.
        \end{equation}
        \item In the third estimate \eqref{Eq:EstimateS-GOH}, we have used Besov spaces based on the harmonic oscillator $\mathrm{H}$. Using these spaces is more demanding than their classical counterparts. Indeed, we have the inclusions $\mcB_{p,\infty}^{-\sigma}(\R^d)\subset\msB_{p,\infty}^{-\sigma}(\R^d)$, and $\msB_{p,\infty}^{\gamma}(\R^d)\subset\mcB_{p,\infty}^{\gamma}(\R^d)$. Intuitively, the operator takes as input functions that have an authorised additional growth at infinity and yields an output that has additional decay at infinity.
        \item In \eqref{eq:estimateS-Ginf}, we do not keep track of the growth as in the other items for simplicity. We believe though that it is possible with more efforts and that one should expect super-exponential growth.
    \end{itemize}
\end{remark}

\begin{theorem}\label{Th:ExistenceAndersonOp}
    Suppose that Assumption \ref{condalphabeta} is satisfied.
    \begin{enumerate}
        \item On $\T^d$, under Assumption \ref{assum:torus}, there exists an event of probability one, on which the following holds.
        \begin{enumerate}
            \item For any $p\in(1,\infty)$, there exists a quasi-sectorial operator $\mathrm{A}^{\!\alpha,m}_\xi$ on $L^p_x$ such that the solution $u$ of \eqref{mainSPDE}, constructed in Theorem \ref{maintheoremspde}, solves
            $$\partial_t u + \mathrm{A}^{\!\alpha,m}_\xi u= 0.$$
            Moreover, $\mathrm{A}^{\!\alpha,m}_\xi$ is the limit in norm-resolvent of
            $$\mathrm{A}^{\!\alpha,m}_{\xi,\eps_k} = \rmD^{\alpha,m}+\xi_{\eps_k}-\mfC_{\eps_k}\,,$$
            where $\eps_k$ is a deterministic sequence going to $0$, and for any $0<\gamma<\beta$ and $z>0$ large enough, $(\mathrm{A}^{\!\alpha,m}_\xi+z)^{-\frac{\gamma}{\alpha}}\in\mathcal{L}(L^p_x,H^{\gamma,p})$.
            \item On $L^2_x$, $\mathrm{A}^{\!\alpha,m}_\xi$ is self-adjoint and has compact resolvent. Thus, there exists a complete orthonormal basis $(\phi_n)_{n\in\N}$ and a non-decreasing sequence
            $$\lambda_0 \leqslant \lambda_1 \leqslant \cdots \leqslant \lambda_n \to \infty$$
            such that $\mathrm{A}^{\!\alpha,m}_\xi\phi_n=\lambda_n\phi_n$ and 
            \begin{equation}\label{Eq:WeylLawAndersonTheorem}
                \#\{n\in\N,\; \lambda_n\leqslant\lambda\}\sim \frac{\pi^{\frac{d}{2}}\lambda^{\frac{d}{\alpha}}}{\Gamma\left(\frac{d}{2}+1\right)}
            \end{equation}
            as $\lambda$ goes to $\infty$. Moreover, for any $p\in(1,\infty)$, the spectrum of $\mathrm{A}^{\!\alpha,m}_\xi$ on $L^p_x$ is given exactly by the previous sequence of eigenvalues, with same multiplicity and eigenfunctions.
            \item Finally, for any $\lambda>-\lambda_0$, $\mathrm{A}^{\!\alpha,m}_\xi+\lambda$ is sectorial with bounded holomorphic calculus of any angle on any $L^p_x$.
        \end{enumerate}
        \item On $\R^d$, under Assumption \ref{assum:infinitevol} 1. or 2., almost surely, there exists an unbounded self-adjoint operator $\mathrm{A}^{\!\alpha,m}_\xi$ on $L^2_x$ such that the solution $u$ of \eqref{mainSPDE}, constructed in Theorem \ref{maintheoremspde}, solves
        $$\partial_t u + \mathrm{A}^{\!\alpha,m}_\xi u= 0.$$ 
        \item On $\R^d$, under Assumption \ref{assum:infinitevol} 3. (i.e. with the harmonic oscillator), there exists an event of probability one, on which the following holds.
        \begin{enumerate}
            \item For any $p\in(1,\infty)$, there exists a quasi-sectorial operator $\mathrm A^{\!2,q}_\xi$ on $L^p_x$ such that the solution $u$ of \eqref{mainSPDE}, constructed in Theorem \ref{maintheoremspde}, solves
            $$\partial_t u + \mathrm A^{\!2,q}_\xi u= 0.$$
            Moreover, $\mathrm A^{\!2,q}_\xi$ is the limit in norm-resolvent of
            $$\mathrm A^{\!2,q}_{\xi,\eps_k} = \rmD^{2,q}+\xi_{\eps_k}-\mfC_{\eps_k}\,,$$
            where $\eps_k$ is a deterministic sequence going to $0$, and for any $0<\gamma<\beta$ and $z>0$ large enough, $(\mathrm A^{\!2,q}_\xi+z)^{-\frac{\gamma}{\alpha}}\in\mathcal{L}(L^p_x,\msH^{\gamma,p})$.
            \item On $L^2_x$, $\mathrm A^{\!2,q}_\xi$ is self-adjoint and has compact resolvent. Thus, there exists a complete orthonormal basis $(\phi_n)_{n\in\N}$ and a non-decreasing sequence
            $$\lambda_0 \leqslant \lambda_1 \leqslant \cdots \leqslant \lambda_n \to \infty$$
            such that $\mathrm A^{\!2,q}_\xi\phi_n=\lambda_n\phi_n$ and 
            \begin{equation}\label{Eq:WeylLawAndersonTheoremOH}
                \#\{n\in\N,\; \lambda_n\leqslant\lambda\}\sim \frac{\lambda^{d}}{2^d\Gamma\left(d+1\right)}
            \end{equation}
            as $\lambda$ goes to $\infty$. Moreover, for any $p\in(1,\infty)$, the spectrum of $\mathrm A^{\!2,q}_\xi$ on $L^p_x$ is given exactly by the previous sequence of eigenvalues, with same multiplicity and eigenfunctions.
            \item Finally, for any $\lambda>-\lambda_0$, $\mathrm A^{\!2,q}_\xi+\lambda$ is sectorial with bounded holomorphic calculus of any angle on any $L^p_x$.
        \end{enumerate}
    \end{enumerate}
\end{theorem}
\begin{proof}
    The three items of the theorem are proved in the fourth section, respectively at the end of Sections~\ref{sec:4-2}, \ref{sec:4-3} and \ref{sec:4-4}. 
\end{proof}
\subsection{Strategy of proof}

Our proof relies on the renormalisation group flow approach to singular SDPEs introduced by Kupiainen~\cite{Kup16}, and more precisely on its continuous version, which was set up by Duch~\cite{Duc21}. The idea of the flow approach is to obtain a scale-by-scale description of the solution $u_\eps$ by looking at the equation solved by the restriction of the solution to a certain scale $\mu>0$, $u_\epmu$. The right-hand sides of the equations for $u_{\epmu}$ and $u_{\eps,\mu'}$ are in fact related by a dynamical system. As $\mu$ increases, only a few directions are amplified under the linear evolution, the relevant directions: they must be propagated backwards in scale, thereby leading to some potentially infinite renormalisation counterterms in the original equation. The strategy of the flow approach is to show that, once the flow of the relevant directions is solved backwards, the other directions, called irrelevant, can be propagated forward, and that their growth is under control.

While the case of the Laplacian is by now classical \cite{HL15,HL18,HsL26}, non-local and non-translation-invariant pseudo-differential operators are more challenging. Indeed, for non-local operators such as the fractional Laplacian, the kernel of the associated semi-group $\e^{-t\rmD^{\alpha,m}}(x)$ only decays at large $x$ like
\begin{equs}
    (1+|x|/{t}^{1/\alpha})^{-d-\alpha} \e^{-m|x|}\,.
\end{equs}
Compared with the Gaussian decay in $|x|/{t}^{1/\alpha}$ of the heat kernel, this is a very slow decay, which causes some technical difficulties, even on the torus, where on the other hand the exponential decay due to the mass is not needed. In infinite volume, these difficulties do not worsen, but the mass term is required in order to construct the solution via a fixed-point argument. Extending a general solution theory (in finite volume) and a priori estimates on the stationary solution for $\Phi^4$ from the case of the Laplacian to that of the fractional Laplacian was one of the major achievements of \cite{Duc21} and \cite{DGR23}. In the present work, we face a new difficulty, since we need to close a fixed point argument in infinite volume. In general, solving 
\begin{equs}
    Lu = u\xi
\end{equs}
is problematic when the driver $\xi$ has a growth $\langle x\rangle^\kappa$ at infinity, which is the case of our stationary Gaussian noises. Indeed, without additional assumption on $L$, it is impossible to close a fixed point argument with $u$ bounded by $\langle x\rangle^a$ at large distance, because $u\xi$ will be bounded by $\langle x\rangle^{a+\kappa}$. It turns out that as utilised by Hairer and Labbé, when $L$ is a parabolic operator, things simplify using exponential weights in space-time, since we can convert a blow-up at large scales into a loss of local regularity (see Lemma~\ref{HaiLabestim}). The price of working with exponential weights, is that the kernels we use must have exponential decay. While this excludes the massless fractional heat kernel, one of our contributions is to show that the argument of Hairer and Labbé extends to the massive fractional Laplacian (which despite not having a fact decay at rate $t^{1/\alpha}$ does have a fact decay at rate $1/m$).

The difficulty raised by non-translation-invariant operators such as the harmonic oscillator is different. Indeed, in a non-translation-invariant setting, because random fields built from the noise are genuinely not translation invariant, it might a priori seem unclear how to correctly renormalise the operator, since the counterterm $\mfC_\eps$ could depend on the position. This appears to be an important difficulty, since letting $\mfC_\eps$ vary is not recommanded. Indeed, because $\mfC_\eps$ diverges, it is always obtained up to a finite shift which, when $\mfC_\eps$ is constant, corresponds to choosing the mass to the operator (which does not affect its properties). If however the counterterm were to be taken to be a function $\mfC_\eps^{2,q}(x)$, then its finite shift could be pretty much any function, including for instance $-|x|^2$, which in the expression $\mathrm{A}^{\!2,q}_{\xi,\eps}=-\Delta+|x|^2-\xi_\eps+\mfC^{2,q}_\eps$ would destroy the harmonic oscillator and turn it into a Laplacian, thereby leading to a major inconsistent. An important possibility offered by the flow approach is to leverage the description of the right-hand side of the equation as a dynamical system in order to clarify how to renormalise non-translation-invariant equations. Indeed, because at very short distances the harmonic oscillator is close to the Laplacian, one can expect that directions breaking translation-invariance are irrelevant, while relevant directions are the same as for the translation-invariant case. This strongly suggests that, in the case of the harmonic oscillator, one should definitely not try to define new random fields (a new model in the jargon of regularity structure), but instead use the random fields from the equation with the Laplacian as an input, and construct the needed random fields associated to the harmonic oscillator pathwise in those of the Laplacian. A major advantage of the flow approach is that the flow equations solved by the random fields allow us to relate different random fields built from different operators. More precisely, denoting by $\upxi^{\mathrm{H}}_\mu$ the random fields for the harmonic oscillator and by $\upxi^{\Delta}_\mu$ the random fields for the Laplacian, we will make the assumption that 
\begin{equs}\label{eq:ansatz}
    \upxi^{\mathrm{H}}_\mu(t,x;s,y)= \e^{\mu^2\frac{|x|^2+|y|^2}{8}}\upxi^{\Delta}_\mu(t,x;s,y)+\upiota_\mu(t,x;s,y)\,,
\end{equs}
and derive a flow equation for  $\upiota_\mu$ (see \eqref{eq:floweqiota}). It turns out that $\upiota_\mu$ is irrelevant, and can therefore be constructed deterministically as a function of $\upxi_\mu^\Delta$. Proving this, and in fact formulating the ansatz~\eqref{eq:ansatz}, relies on a local description of the difference between the heat kernels of the harmonic oscillator and of the Laplacian. However we cannot be to naive about this estimate and we must be careful about the fact that we are working in infinite volume. The first estimate on the difference we could think about is
\begin{equation}
    \big|G^{2,q}(t,x,y)-G^{2,0}(t,x-y)\big|\lesssim t\big(1+|x|^2+|y|^2\big)\,.
\end{equation}
There is a gain in $t$ but the price we need to pay for that is a quadratic growth that is obviously not integrable in large volume. The way out of this conundrum is to use a refined estimate
\begin{equation}
    \Big|G^{2,q}(t,x,y)-\e^{-t\frac{|x|^2+|y|^2}{4}}G^{2,0}(t,x-y)\Big|\lesssim t\big(1+|x|^2+|y|^2\big)\e^{-ct\frac{|x|^2+|y|^2}{4}}\,,
\end{equation}
that holds for any $0<c<1$ and for small $t$. See \eqref{eq:estimatediffkernels} for a more precise statement. We believe that such an estimate is of independent interest. This added Gaussian term transfers in the ansatz \eqref{eq:ansatz} in the form of an additional growing weight. We refer to the discussion at the beginning of Section~\ref{sec:OH} for a more detailed justification of the ansatz~\eqref{eq:ansatz}. The idea to relate the solution to a flow equation is a non-translation-invariant setting to that in a translation-invariant setting originates in some works on perturbative solutions to the Polchinski equation on manifolds and domains with boundaries \cite{Borji1,Borji2,Borji3}. Recently, the first author has used a similar technique in order to study the renormalisation of singular SPDEs on domains with boundaries \cite{BBF}.

\subsection{Literature}

Singular SPDEs have known a spectacular development in the last decade. This progress relies mainly on two major breakthroughs: regularity structures \cite{reg} and paracontrolled calculus \cite{GIP15}. These theories both allowed to solve a very broad class of parabolic equations. See \cite{BHZ,CH16,BCCH} for the complete framework on the side of regularity structures. They encompass equations for which the well-posedness has been an open problem for decades, such as the KPZ equation or the $\Phi^4_3$ stochastic quantisation equation. One of the main challenges of these theories is the implementation of complex renormalisation procedures that rely on a heavy combinatorial machinery. They are inspired by the BPHZ renormalisation in quantum field theory \cite{BP57,KH69,WZ69}. In the flavour of the present article, Hairer and Labbé used regularity structures in \cite{HL18} to give a pathwise solution theory to the PAM equation on the full space, opening the path to the construction of the Anderson operator on $\R^3$.

The construction and study of singular random operators has been an active fields in parallel of the study of singular SPDEs over the last decade. The construction of Schrödinger operators with white noise potential goes back to Fukushima and Nakano \cite{FN77} on compact intervals. The first construction of a singular Anderson operator was obtained by Allez and Chouk \cite{AK15} using paracontrolled calculus on $\T^2$. It was extended on compact surfaces by Mouzard \cite{Mo20} who obtained the Weyl law, and the properties of the operator were further investigated by Bailleul, Dang and Mouzard \cite{BDM22}. See also \cite{MZ24} where Mouzard and Zachhuber used the paracontrolled approach to obtain Strichartz' estimates on compact surfaces. The same paracontrolled approach was used by Morin and Mouzard \cite{MM22} to construct the magnetic Anderson operator on $\T^2$. On $\T^3$, the Anderson operator was constructed by Gubinelli, Ugurcan and Zachhuber \cite{GUZ20} and De Vecchi, Ji and Zachhuber \cite{DJZ25}, in which Strichartz' estimates are obtained. In parallel, the Anderson operator was constructed on $\T^d$ for $d\leqslant 3$ by Labbé in \cite{Lab19}, using regularity structures. Let us also mention the cosntruction of Mouzard and Ouhabaz \cite{MO25} using quadratic forms, which allows to recover easily the spectral gap on $\T^d$ for $d\leqslant 3$.

On $\R^d$, the construction is more complex due to the growth at infinity of the noise. In particular, the quadratic form of the operator is no more bounded from below. On the full line, the constuction and study of the Anderson operator was conducted in a series of works of Dumaz and Labbé \cite{DL20,DL23,DL24a,DL24c} using probabilistic methods and a twisted Sturm-Liouville approach. On $\R^2$, Ueki \cite{Ue25} constructed the operator by combining a localization argument with a paracontrolled approach in the fashion of Mouzard. In parallel, Hsu and Labbé \cite{HsL26} used the Klein-Landau theorem to construct the Anderson operator on $\R^d$ for $d\leqslant 3$ from the flow of the PAM equation. More recently, Mouzard and Zachhuber \cite{MZ26} constructed the solution of (nonlinear) Anderson-Schrödinger equation on $\R^d$ for $d\leqslant 3$ using paracontrolled calculus, which can be seen has a construction of the operator by Stone's theorem.

When the base operator is a confining Schrödinger operator, the associated Anderson operator has been constructed by Minami \cite{Mi15} in dimension 1, using properties of fractional Gaussian noises. In particular, the stochastic Airy operator (i.e. $-\partial_x^2+x+\xi$ on $L^2_x(\R_{\geq0})$ with Dirichlet condition at $0$) has been studied by Dumaz and Labbé \cite{DL22} for its link with eigenvalues of large random matrices \cite{RRV11}. In dimension 2, the known results are due to the second author in the case of a white noise potential. They prove constructed the solution of (nonlinear) Anderson-Schrödinger equation with harmonic potential on $\R^2$ \cite{Ma25a}. The paracontrolled analysis of the Hermite-Anderson operator was done in \cite{Ma25b}, in which Strichartz' estimates where obtained. Upper bound on the Sobolev regularity of eigenfunctions was obtained in \cite{Ma25c}, as well as the complete characterisation of the domain in dimension 1. These results were extended to more general potentials $V(x)\approx \langle x\rangle^\beta$ in their PhD thesis \cite{MaPhD}. So far all the renormalisation counterterms were non-constant. In a parallel work to the current paper \cite{Ma26+}, the second author obtained constant counterterms for a large class of perturbations of confining Schrödinger operators, including some pseudo-differential operators. In the case of the Hermite-Anderson operator, the constructions of \cite{Ma26+} and of the current paper gives the same operator up to a constant shift. The concomitant achievement of renormalisation by constants of some non-translation invariant Schrödinger operators in two different direction (pseudo-differentials perturbations in \cite{Ma26+} and the full subcritical regime in the current paper) are a good indication of a possible universality of this phenomenon. In the context of Schrödinger equations, adding a constant to the Hamiltonian can be equivalently seen as a global change in the phase of the solution. When studying quantum mean-fields limits, it is convenient to go from the Schrödinger viewpoint to the Heisenberg viewpoint, i.e. to consider density matrices rather than wave functions. As density matrices are invariant by a global change in phase of the wave function, the renormalisation constant is somehow irrelevant. This would not be the case for non-constant counterterms, thus giving another reason to seek for constant counterterms.

The flow approach to singular SPDEs has recently been introduced by Duch \cite{Duc21, Duc22}, taking its inspiration from Wilsonian renormalisation group techniques. More precisely, it adapts ideas from the work of Polchinski \cite{Pol84} that allows to implement the renormalisation of Feynman diagrams inductively. We refer to the monograph \cite{M03} for a better grasp of the physics literature, as well as the lecture notes \cite{Duc23} for a smooth introduction to the subject in an elliptic context. The main idea of the method is to define a continuous flow of scales. On the opposite, a few years before, Kupiainen and Marcozzi used techniques from discrete renormalisation group theory in \cite{Kup16,KM17} to solve the $\Phi^4_3$ and the gKPZ equations. They were however limited by the fact that their parameter was discrete, where Duch's approach is able to treat the full subcritical regime easily. The approach has been developed by Chandra and the first author in \cite{CF24a,CF24b}, where they solve the gKPZ equation in the full subcritical regime, as well as rough ODEs. Another landmark is the work \cite{DGR23}, where the authors construct the $\Phi^4$ measure in the full subcritical regime on the full space. Some of the techniques developed therein are reused in the present work. On the side of renormalisation, Bruned and the third author exhibited in \cite{BM25} that the renormalisation procedure of the flow approach matched in the BPHZ renormalisation scheme of regularity structures. Outside of the field of SPDEs, continuous renormalisation group techniques have known a gain of interest recently. Bauerschmidt and Bodineau proved a log-Sobolev inequality for the sine-Gordon model in \cite{BB21}, which led to a series of works. See the lecture notes \cite{BBD24} and the references therein. Another reference of interest is the works of Gubinelli and Meyer \cite{GM24} and Meyer \cite{Mey26} in which the authors construct the sine-Gordon measure with a particular type of stochastic quantisation relying on a forward-backward SDE. They are able to derive relevant properties on the constructed measure, such as the Osterwalder-Schrader axioms. Finally, we would like to mention a recent work of Duch \cite{Duc24} in which he constructs the Gross-Neveu model using the Polchinski flow.

\subsection{Notations}

We will denote in a general manner that $a\lesssim b$ if there is a positive constant $C$, that does not depend on the parameter at stake, depending on the context, such that $a\leq Cb$. We denote $\SS=(-\infty,1]\times\T^d$ or $\Lambda=(-\infty,1]\times\R^d$, depending on the section, which should be clear in the context. We will specify it otherwise. Similarly, we denote $\Lambda_{T}=[0,T]\times\R^d$. We choose to use fractional parabolic scaling for the space-time in $\SS$. We set $\mfs=(\alpha,1,\dots,1)$ and, for $z\in\SS$,
\begin{equation}
|z|_\mfs=|z_0|^\frac{1}{\alpha}+\sum_{i=1}^d|z_i|.
\end{equation}
In particular, the effective dimension is $|\mfs|=d+\alpha$. In addition, we set for $k\in\N^{d+1}$, $|k|_\mfs=\alpha k_0+\sum_{i=1}^dk_i$. For $x\in\R^d$, we define the Japanese bracket as $\langle x\rangle=(1+|x|^2)^\frac{1}{2}$, as well as $\langle\nabla\rangle$ the Fourier multiplier with symbol $\langle\bigcdot\rangle$. In particular, $\langle\nabla\rangle^2=1-\Delta$.\\

\textit{Weights.} We set for $\mu\in(0,1]$ two $\mu$-dependent Japanese brackets on $x\in\R^d$ and $z\in\SS$ by
\begin{equation}\label{mubracket}
\langle x\rangle_{\mu}=\left(1+\frac{|x|^2}{\mu^2}\right)^\frac{1}{2}~~\mathrm{and}~~\langle z\rangle_{\mu}=\left(1+\frac{|z|_\mfs^2}{\mu^2}\right)^\frac{1}{2}.
\end{equation}
Moreover, we define the exponential weight for $\ell\in\R$
\begin{equation}\label{expweight}
    \mcb e_\ell(x)=\exp\big(\ell\langle x\rangle^{1/2}\big)\,.
\end{equation}
We also set for conciseness $\mcb e_{\ell}(x;y)\eqdef\mcb e_{\ell}(x-y)$. We point out that this kernel is of sub-exponential growth, so any kernel with exponential decay can absorb it. This is the case of the fractional heat kernel and the massive Green's function. Finally we define some exponential weights that will be crucial in the study of the harmonic oscillator. For $x,y\in\R^d$ and $\mu\in(0,1]$, we set
\begin{equation}\label{eq:wieghtW}
    \mcb W_\mu(x,y)=\e^{\mu^2\frac{|x|^2+|y|^2}{8}}\,,~~\mathrm{and}~~\mcb W_\mu^{-1}(x,y)=\e^{-\mu^2\frac{|x|^2+|y|^2}{8}}\,.
\end{equation}

\textit{Power counting.}
We will require a notion of degree, that follows the rules of power counting. We let 
\begin{equation}
\deg(k,n)\eqdef-\alpha-\sigma+n\beta +k(\beta+\sigma)\,,
\end{equation}
and for conciseness,
\begin{equation}
\deg(n)\eqdef\deg(1,n)\equiv-\alpha+(n+1)\beta \,.
\end{equation}
We also set
\begin{equs}\label{eq:N0N1}
    N_0\eqdef\lfloor(\alpha+\sigma)/\beta\rfloor\quad\text{and}\quad N_1\eqdef\lfloor (\alpha-\beta)/\beta\rfloor\,,
\end{equs}
which are the smallest integers such that 
\begin{equs}
    n>N_k\Rightarrow \deg(k,n)>0\,.
\end{equs}

\textit{Kernels.} We define, for $(t,x)\in\SS$, the heat kernel
\begin{equation}\label{greenfunction}
G^{\alpha,m}(t,x)=\one_{(0,\infty)}(t)\e^{-t\rmD^{\alpha,m}}(x)\,,
\end{equation}
the fundamental solution of the differential operator $\D_t+\rmD^{\alpha,m}$. We also define, for a scale parameter $\mu\in[0,1]$, and $(t,x)\in \R\times\T^d$,
\begin{equation}
    Q_\mu(x)=(1-\mu^2\Delta)^{-1}(x)\,,
\end{equation}
and
\begin{equation}
    K_{1,\mu}(t,x)=\mu^{-\alpha} \e^{-\frac{t}{\mu^\alpha}}\one_{(0,\infty)}(t)(1-\mu^2\Delta)^{-1}(x)\,,
\end{equation}
which are respectively the fundamental solutions of $(1-\mu^2\Delta)$ and $\mcP_\mu=(1+\mu^\alpha\D_t)\otimes(1-\mu^2\Delta)$. We define for a function $\psi\colon\R^{d+1}\backslash\{0\}\to\R$, $\mcS_\mu\psi(t,x)=\mu^{-d-\alpha}\psi\left(\frac{t}{\mu^\alpha},\frac{x_1}{\mu},\dots,\frac{x_d}{\mu}\right)$. We have then $K_{1,\mu}=\mcS_\mu K_{1,1}$.
For a given integer $N\in\N$, we set
\begin{equation}
K_{N,\mu}=K_{1,\mu}^{*N}.
\end{equation}
In practice, it will be sufficient to take $K_\mu=K_{N,\mu}$ with $N=\lceil\alpha-\beta\rceil+2$. We also set $\mcR_\mu=\big(\mcP_\mu^{N}\big)^*$, where $*$ denotes the adjunction with respect to the $L^2(\Lambda)$ inner product. We also introduce some kernels associated with the harmonic oscillator. First, we define the Mehler kernel
\begin{equation}\label{eq:Mehlerkernel}
    M(t,x,y)=\one_{t>0}\frac{1}{\big(2\pi\sinh(2t)\big)^{d/2}}\exp\left(-\frac{|x-y|^2}{4\tanh(t)}-\tanh(t)\frac{|x+y|^2}{4}\right)\,,
\end{equation}
which is the fundamental solution of $\D_t-\Delta+|x|^2$. Moreover, we define a family of kernels that are going to serve the same role as $K_\mu$ in the solution theory for the SPDE with the harmonic oscillator. First let a family of positive numbers $a_n=1+n/(2(N_0+1))$, for $0\leq n\leq N_0+1$. $N_0$ is defined in \eqref{eq:N0N1}. In particular $1\leq a_n\leq 3/2<2$. We let
\begin{equation}\label{eq:defJmu}
    J_{\mu,n}(t,x,y)=K_{\mu}(t,x-y)\e^{-\frac{a_n\mu^2|y|^2}{8}}\,,
\end{equation}
and $J_\mu=J_{\mu,N_0+1}$. We also define $\mcT_{\mu,n}=(J_{\mu,n}^{-1})^*$. More precisely, it reads, for $f\colon\Lambda\to\R$,
\begin{equation}\label{eq:defTmu}
    (\mcT_{\mu,n}f)=\mcR_\mu\Big(\e^{\frac{a_n\mu^2|\bigcdot|^2}{8}}f\Big)\,.
\end{equation}
We set $\mcT_\mu=\mcT_{\mu,N_0+1}$. Finally, we define $\msQ_\nu$, $\msP_\mu$, and $\msR_\mu$ the same way as $Q_\mu$, $\mcP_\mu$, and $\mcR_\mu$ but with the Laplacian $-\Delta$ replaced by the harmonic oscillator $-\Delta+|x|^2$.\\

\textit{Function spaces.} For $f\colon\R^d\to\R$ we define the norm for $p\in[1,\infty]$,
\begin{equation}\label{eq:defLp}
    \|f\|_{L^p_x}=\left(\int_{\R^d}|f(x)|^p\dint x\right)^\frac{1}{p},
\end{equation}
and $L^p_x(\R^d)$ the usual Lebesgue-based space associated to it. We also denote $L_{\mcb e_\ell}^2(\R^d)$ the space of functions such that $\mcb e_\ell f\in L^2_x(\R^d)$. Similarly, for $f\colon\SS\to\R$, we define, for $p,q\in[1,\infty]$,
\begin{equation}
    \|f\|_{L_{t,x}^{p,q}}=\left(\int_\R\left(\int_{\R^d}|f(t,x)|^q\dint x\right)^\frac{p}{q}\dint t\right)^\frac{1}{p},
\end{equation}
and consequently the associated space. When $p=q$, we simply denote it $L_{t,x}^p(\SS)$ or $L_z^p(\SS)$. We will try to drop as much as possible in the notations the dependence on the underlying spaces. It should be clear in each section whether it is $\T^d$ or $\R^d$.

For $E$ and $E'$ two Banach spaces, we denote $\mcL(E,E')$ the space of bounded linear maps from $E$ to $E'$. It is classically endowed with the operator norm $\|A\|_{\mcL(E,E')}=\sup_{\|x\|_E\leq1}\|Ax\|_{E'}$. We set $\mcL(E)=\mcL(E,E)$. We make this norm explicit for kernel operators from $L_{t,x}^{\infty,r}$ to $L_{t,x}^{\infty,p}$, whose action reads
\begin{equation}
    f\in L_{t,x}^{\infty,r}\mapsto\left((t,x)\mapsto\int_0^\infty\int_{\R^d}k(t-s,x,y)f(s,y)\dint y\dint s\right)\,,
\end{equation}
for some kernel $k\colon\R\times\R^d\times\R^d$. We will make the confusion between the operator and its kernel. In this context, we will denote $\|k\|_{\mcL(L_{t,x}^{\infty,r},L_{t,x}^{\infty,p})}$ the norm of the kernel operator. We define an operator norm on functions from $\Lambda^2$ to $\R$.
\begin{equs}\label{eq:normOp}
\|\bigcdot\|_{\mathrm{Op}}=\|\bigcdot\|_{L_{w_1}^\infty L_{w_2}^1}+\|\bigcdot\|_{L_{w_2}^\infty L_{w_1}^1}\,.
\end{equs}
For $\eta<0$, $p\in[1,\infty]$, we define the Besov space $\mcB_{p,q}^{\eta}(\R^d)$ as the completion of compactly supported smooth functions on $\R^d$ with respect to the norm
\begin{equation}
\|\bigcdot\|_{\mcB_{p,q}^{\eta}}=\Big\|\mu^{-\eta}\|Q_{\mu}^{*\lceil-\eta\rceil}\bigcdot\|_{L_{x}^{p}}\Big\|_{L^q(\mathrm{d\mu/\mu})}\,.
\end{equation}
Moreover, for $\eta\in(0,1)$, we define the space Besov spaces $\mcB_{p,\infty}^\eta(\R^d)$ as the completion of compactly supported smooth functions on $\R^d$ with respect to the norm
\begin{equation}
    \|\bigcdot\|_{\mcB_{p,\infty}^{\eta}}=\|\bigcdot\|_{L^p_{x}}+\Big\|\mu^{-\eta}\|(Q_\mu-\mathrm{Id})\bigcdot\|_{L^p_{x}}\Big\|_{L^q(\mathrm{d\mu/\mu})}\,.
\end{equation}
We also define $W^{\eta,p}=\mcB_{p,p}^\eta$ the usual $L^p_x$-based Sobolev-Slobodeckij spaces and we denote by $H^{s,p}$ the Bessel-Sobolev spaces with norm
\begin{equation}
     \|\bigcdot\|_{\mathrm{H}^{\eta,p}}=\|(1-\Delta)^{\frac{\eta}{2}}\bigcdot\|_{L^p_x}\,.
\end{equation} For $\eta\in(0,1)$, $p\in[1,\infty]$, and an interval $I$, we define the Besov space $L_t^\infty\mcB_{p,\infty}^{\eta}(I\times\R^d)$ as the completion of compactly supported smooth functions on $I\times\R^d$ with respect to the norm
\begin{equation}
\|f\|_{L_t^\infty\mcB_{p,\infty}^{\eta}}=\|f\|_{L_{t,x}^{\infty,p}}+\sup_{t\in I}\sup_{\mu\in (0,1]}\mu^{-\eta}\|(Q_\mu-\mathrm{Id})f(t,\bigcdot)\|_{L_{x}^{p}}\,.
\end{equation}
In infinite volume, we will require a specific type of weighted Besov space. We set
\begin{equation}\label{weightedspace}
    \|f\|_{\mcE^{\eta,\delta}_{p,q,\ell,T}}=\sup_{t\leq T}t^{-\delta/\alpha}\|\mcb e_{-\ell-t}f(t,\bigcdot)\|_{\mcB_{p,q}^\eta(\R^d)}\,,
\end{equation}
and we set $\mcE^{\eta}_{p,\ell,T}=\mcE^{\eta,0}_{p,\ell,T}$.
We finally define Sobolev and Besov spaces based on the harmonic oscillator. For $\eta\in\R$ and $1<p<\infty$, Sobolev spaces are given by the completion of $\mcS(\R^d)$ under the norm
\begin{equation}
     \|\bigcdot\|_{\msH^{\eta,p}}=\|\mathrm H^{\frac{\eta}{2}}\bigcdot\|_{L^p_x}\,.
\end{equation}
For $\eta<0$, Besov spaces are given by the completion of $\mcS(\R^d)$ under the norm
\begin{equation}
    \|\bigcdot\|_{\msB_{p,q}^{\eta}}=\Big\|\mu^{-\eta}\|\msQ^{\lceil-\eta\rceil}_{\mu} \bigcdot\|_{L^p_x}\Big\|_{L^q(\mathrm{d\mu}/\mu)}\,,
\end{equation}
and for $0<\eta<1$, under the norm
\begin{equation}
    \|\bigcdot\|_{\msB_{p,q}^{\eta}}=\|\bigcdot\|_{L^p_x}+\Big\|\mu^{-\eta}\|(\msQ_\mu-\mathrm{Id}) \bigcdot\|_{L^p_x}\Big\|_{L^q(\mathrm{d\mu}/\mu)}\,.
\end{equation}
We define $L_t^\infty\msB_{p,\infty}^{\eta}(I\times\R^d)$ as $L_t^\infty\mcB_{p,\infty}^{\eta}(I\times\R^d)$, but with $Q_\mu$ replaced by $\msQ_\mu$.

\subsection*{Acknowledgments}

{\small
    S.M. was supported by the ANR project Smooth ANR-22-CE40-0017. A.M. gratefully acknowledges funding support from the European Research Council (ERC) through the ERC Starting Grant Low Regularity Dynamics via Decorated Trees (LoRDeT), grant agreement No.\ 101075208. Views and opinions expressed are however those of the author(s) only and do not necessarily reflect those of the European Union or the European Research Council. Neither the European Union nor the granting authority can be held responsible for them. 
}

\section{Deterministic construction of the solution}\label{nonstatio}\label{SecSPDE}

In this section, we construct the semigroup associated to the fractional Anderson operator by solving the SPDE \eqref{mainSPDE}. For that purpose, we use the flow approach recently introduced by Duch \cite{Duc21}.

\subsection{General framework of the flow approach}\label{Framework}

Since \eqref{mainSPDE} is not a good starting point, we need to consider a sequence of approximations. More precisely, let $\eps>0$ and $\xi_\eps\eqdef\mcS_\eps\rho*\xi$, where $\rho$ is a smooth function compactly supported in the unit ball and that integrates to $1$. We solve the bare equation obtained when the noise is replaced by $\xi_\eps$
\begin{equation}\label{renormSPDE}
(\D_t+\rmD^{\alpha,m})u^{\alpha,m}_\eps=\big(\xi_\eps-\mfC^{\alpha,m}_\eps\big) u^{\alpha,m}_\eps\,,
\end{equation}
where $\mfC^{\alpha,m}_\eps=\sum_{n=0}^{N_1}\mfc_\eps^{\alpha,m;n}$ and with regularised initial data $u^{\alpha,m}_{0,\eps}$. Naive upper bounds on the constants $\mfc_\eps^{\alpha,m;n}$ typically diverge as $\eps\to0$ and the divergence speed decreases as $n$ increases. 

We let the linear functional
\begin{equation}\label{eq:defS0}
    S^{\alpha,m}_\eps[u]\eqdef\big(\xi_\eps-\mfC^{\alpha,m}_\eps\big) u\,.
\end{equation}
Letting $\psi^{\alpha,m}_\eps(t,x)\eqdef\delta(t)u^{\alpha,m}_{0,\eps}(x)$, we can rewrite \eqref{renormSPDE} as 
\begin{equation}\label{mildform}
    u^{\alpha,m}_\eps=G^{\alpha,m}(\one_{(0,\infty)}S^{\alpha,m}_\eps[u^{\alpha,m}_\eps]+\psi^{\alpha,m}_\eps)\,.
\end{equation}
We recall that $G^{\alpha,m}$ has been defined in \eqref{greenfunction}. Let $v^{\alpha,m}_\eps=u^{\alpha,m}_\eps-G^{\alpha,m}\psi^{\alpha,m}_\eps$. It satisfies the equation
\begin{equation}
    v^{\alpha,m}_\eps=G^{\alpha,m}(\one_{(0,\infty)}S^{\alpha,m}_\eps[v^{\alpha,m}_\eps+G^{\alpha,m}\psi^{\alpha,m}_\eps])\,.
\end{equation}
The indicator function above is taken in the time variable. We define
\begin{equation}\label{eq:defT0}
    T^{\alpha,m}_{\eps}[u]\eqdef S^{\alpha,m}_\eps[u+G^{\alpha,m}\psi^{\alpha,m}_\eps]-\psi^{\alpha,m}_\eps\,.
\end{equation}
Letting 
\begin{equation}
    F^{\alpha,m}_\eps[v]\eqdef \one_{(0,\infty)}T^{\alpha,m}_\eps[v]=\one_{(0,\infty)}\big(S^{\alpha,m}_\eps[v+G^{\alpha,m}\psi^{\alpha,m}_\eps]-\psi^{\alpha,m}_\eps\big)\,,
\end{equation}
the SPDE has been rewritten as
\begin{equation}\label{mildformF}
    v^{\alpha,m}_\eps=G^{\alpha,m}F^{\alpha,m}_\eps[v^{\alpha,m}_\eps]\,.
\end{equation}
Above, we used that $\one_{(0,\infty)}\psi_\eps=0$, so $ F^{\alpha,m}_\eps[v]=\one_{(0,\infty)}S^{\alpha,m}_\eps[v+G^{\alpha,m}\psi^{\alpha,m}_\eps]$.

We now implement the effective scale and the associated cut-off Green function. Namely, we let, for $\mu\in[0,1]$, 
\begin{equation}
G^{\alpha,m}_\mu(t,x) \eqdef  \chi_{\mu}(t) G^{\alpha,m}(t,x)\,,
\end{equation} 
where $\chi_{\mu}(t) = \chi(t/\mu^\alpha)$ with the convention $ \chi_{0}(t) = 1 $. Here, we have supposed that the time $t$ belongs to $[0,1]$. The map $\chi : \R_+ \rightarrow \R_+$ is a fixed smooth and non-decreasing function  such that $\chi |_{[0,1]} = 0$ and $\chi |_{[2,\infty)} = 1$. In particular, $G^{\alpha,m}_0=G^{\alpha,m}$. We also denote $\dot G^{\alpha,m}_\mu=\D_\mu G^{\alpha,m}_\mu$. 
One has the following key properties on the support of $G_{\mu}$
\begin{equation}\label{cutscale}
	\mathrm{Supp}\,G^{\alpha,m}_{\mu}\subset[\mu^\alpha,1]\times\R^d,\quad\mathrm{Supp}\,\dot{G}^{\alpha,m}_{\mu} \subset [\mu^\alpha,2\mu^\alpha] \times \R^d\,.
\end{equation}

We define a regularised version of the solution supported after time $\mu^2$, the effective solution:
\begin{equation}
    v^{\alpha,m}_{\eps,\mu}\eqdef G^{\alpha,m}_\mu F^{\alpha,m}_\eps[v^{\alpha,m}_\eps]\,.
\end{equation}
The strategy of the flow approach is to look for the right-hand side of the equation solved by the effective solution. We will construct this right-hand side in terms of another functional $F_\epmu^{\alpha,m}$ called the effective force, which is an explicit function of the force $F^{\alpha,m}_\eps$. The effective force is related to the force $F^{\alpha,m}_\eps$ as follows. Denote by
\begin{equs}
    H_\epmu^{\alpha,m}[v]\eqdef   \D_\mu F^{\alpha,m}_{\eps,\mu}[v]+\mathrm{D}F^{\alpha,m}_{\eps,\mu}\dot G^{\alpha,m}_\mu F^{\alpha,m}_{\eps,\mu}[v]\,.
\end{equs}
D is a Fréchet derivative. We require that $F_\epmu^{\alpha,m}$ satisfies
\begin{equs}\label{eq:definitionofH}
   \mu   H_\epmu^{\alpha,m}[v]\longrightarrow 0 \;\text{as}\;\mu\searrow0\quad\text{and}\quad F_{\eps,0}^{\alpha,m}[v]=F_{\eps}^{\alpha,m}[v]\,.
\end{equs}
In other words, the effective force is an approximate solution to the inviscid Polchinski flow equation
\begin{equs}\label{eq:flotFam}
 \D_\mu F^{\alpha,m}_{\eps,\mu}+\mathrm{D}F^{\alpha,m}_{\eps,\mu}\dot G^{\alpha,m}_\mu F^{\alpha,m}_{\eps,\mu}=0\,,\quad
F_{\eps,0}^{\alpha,m}=F_{\eps}^{\alpha,m}\,.
\end{equs}
Assuming \eqref{eq:definitionofH} implies 
\begin{equs}\label{eq:defR}
    R^{\alpha,m}_\epmu \eqdef F^{\alpha,m}_\eps[v^{\alpha,m}_\eps]-F^{\alpha,m}_\epmu[v^{\alpha,m}_\epmu]\longrightarrow 0 \;\text{as}\;\mu\searrow0\,.
\end{equs}
Indeed, $R^{\alpha,m}_\epmu$ is related to the effective force, and therefore to the force, through the system 
\begin{subequations}
\begin{empheq}[left=\empheqlbrace]{alignat=2}\label{eq:flowsys}
  \partial_\mu  v^{\alpha,m}_\epmu&=\dot G^{\alpha,m}_\mu\big(F^{\alpha,m}_{\eps,\mu}[v^{\alpha,m}_{\eps,\mu}]+R^{\alpha,m}_{\eps,\mu}\big)\,,\\
 \partial_\mu   R^{\alpha,m}_\epmu&=- \mathrm{D}F^{\alpha,m}_\epmu \dot G^{\alpha,m}_\mu R^{\alpha,m}_\epmu-H^{\alpha,m}_{\eps,\mu}[v^{\alpha,m}_{\eps,\mu}]\,.
\end{empheq}
\end{subequations}
Note that if the effective force solves \eqref{eq:flotFam}, then $R_\epmu^{\alpha,m}=0$.

Constructing the solution now boils down to solving \eqref{eq:flotFam}, possibly modulo an error. We will solve this equation using Picard iterations, expanding the effective force over the \textit{effective force coefficients} $(\upzeta_{{\eps,\mu}}^{\alpha,m;n,k})$, writing for ${N_{\max}}\in\N\cup\{\infty\}$
\begin{equs}\label{eq:ansatzwithNmax}
F^{\alpha,m}_{\eps,\mu}[v](z)&=F^{\alpha,m;0}_{\eps,\mu}(z)+F^{\alpha,m;1}_{\eps,\mu}[v](z)=\sum_{n=0}^{N_{\max}}\upzeta_{{\eps,\mu}}^{\alpha,m;0,n}(z)+\sum_{n=0}^{N_{\max}}\upzeta_{{\eps,\mu}}^{\alpha,m;1,n}[v](z)\,.
\end{equs}
Here, 
\begin{equs}
    \upzeta_{\eps,\mu}^{{\alpha,m};1,n}[v](z)=\int_{\SS}\upzeta_{\eps,\mu}^{{\alpha,m};1,n}(z,w)v(w)\dint w\,.
\end{equs}
The inviscid Polchinski flow equation \eqref{eq:flotFam} implies that the force coefficient must solve the hierarchy of equations 
\begin{equs}\label{eq:hierarchyzeta}
\D_\mu \upzeta^{{\alpha,m};k,n}_{\eps,\mu}(z,\bigcdot)=-\sum_{i=0}^{n-1}\upzeta_{\eps,\mu}^{{\alpha,m};1,i}[\dot G^{\alpha,m}_\mu\upzeta_{\eps,\mu}^{{\alpha,m};k,n-1-i}(\bigcdot)](z)\,.
\end{equs}

While solving \eqref{eq:hierarchyzeta} is our ultimate goal, the effective force coefficients are not well-suited to study the renormalisation problem because, in view of the initial condition $ F^{\alpha,m}_\eps[v]=\one_{(0,\infty)}S^{\alpha,m}_\eps[v+G^{\alpha,m}\psi^{\alpha,m}_\eps]$, 
\begin{enumerate}
    \item The effective force coefficients are not stationary in time;
    \item The effective force coefficients depend on the initial condition;
    \item In the case of the harmonic oscillator, the effective force coefficients are not stationary in space.
\end{enumerate}
We solve these three difficulties one by one. Since item 3. is peculiar to the case of the harmonic oscillator, we defer the discussion to a paragraph below dedicated to this setting. Item 1. and 2. being general difficulties arising both on the torus and in infinite volume, for both the fractional Laplacian and the harmonic oscillator, we discuss now how to solve them. 

\begin{definition}
  We denote by $S_\epmu^{\alpha,m}$ and $T_\epmu^{\alpha,m}$ the possibly approximate solutions to the inviscid Polchinski flow equation \eqref{eq:flotFam} with initial condition respectively $S_{\eps}^{\alpha,m}$ and $T_\eps^{\alpha,m}$, these quantifies being introduced in \eqref{eq:defS0} and \eqref{eq:defT0}. There are respectively called the \textit{stationary effective force} and the \textit{translated effective force}. We expand these forces respectively over the \textit{stationary force coefficients} $(\upxi_{{\eps,\mu}}^{\alpha,m;n})$ and the \textit{translated force coefficients} $(\uptau_{{\eps,\mu}}^{\alpha,m;k,n})$, writing 
 \begin{equs}
S^{\alpha,m}_{\eps,\mu}[v](z)&=\sum_{n=0}^{N_{0}}\upxi_{{\eps,\mu}}^{\alpha,m;n}[v](z)\,,\\
T^{\alpha,m}_{\eps,\mu}[v](z)&=T^{\alpha,m;0}_{\eps,\mu}(z)+T^{\alpha,m;1}_{\eps,\mu}[v](z)=\sum_{n=0}^{N_{0}}\uptau_{{\eps,\mu}}^{\alpha,m;0,n}(z)+\sum_{n=0}^{N_{0}}\uptau_{{\eps,\mu}}^{\alpha,m;1,n}[v](z)\,.
\end{equs}
The stationary force coefficients satisfy the hierarchy
\begin{equs}\label{eq:hierarchyxi}
\D_\mu \upxi^{{\alpha,m};n}_{\eps,\mu}(z,\bigcdot)=-\sum_{i=0}^{n-1}\upxi_{\eps,\mu}^{{\alpha,m};i}[\dot G^{\alpha,m}_\mu\upxi_{\eps,\mu}^{{\alpha,m};n-1-i}(\bigcdot)](z)
\end{equs}
and the translated force coefficients satisfy the hierarchy
\begin{equs}\label{eq:hierarchytau}
\D_\mu \uptau^{{\alpha,m};k,n}_{\eps,\mu}(z,\bigcdot)=-\sum_{i=0}^{n-1}\uptau_{\eps,\mu}^{{\alpha,m};1,i}[\dot G^{\alpha,m}_\mu\uptau_{\eps,\mu}^{{\alpha,m};k,n-1-i}(\bigcdot)](z)\,.
\end{equs}

Moreover, we assume that the stationary force coefficients satisfy the following initial condition, which is compatible with $S_{\eps,0}^{\alpha,m}=S_{\eps}^{\alpha,m}$:
\begin{equation}
\upxi^{\alpha,m;0}_{\eps,0}(z_1,z_2)=\xi_\eps(z_1)\delta(z_1-z_2)\,,
\end{equation}
as well as
\begin{equation}
    \upxi_{0,\eps}^{{\alpha,m};n}(z_1,z_2)=\begin{cases}\mfc_\eps^{{\alpha,m};n}\delta(z_1-z_2)&\mathrm{if}~~1\leq n\leq N_1\,,\\
    0&\mathrm{otherwise}.
\end{cases}
\end{equation}
  \end{definition}

\begin{example}\label{exampleforce}
We would like to point out that the stochastic terms $\upxi_{\eps,\mu}^n$ can be seen, as in regularity structures, as trees. In particular, as the equation is linear, they are ladder trees. $n$ would stand for the number of edges. We give below an example for the \textnormal{unrenormalised} version of $\upxi_{\eps,\mu}^2$, for $m\neq q$.
\begin{equs}
    ~&\int_{\Lambda}\upxi_{\eps,\mu}^2(z_1,z_2)\phi(z_2)\dint z_2=\begin{tikzpicture}[scale=0.4,baseline=0.15cm]
            \draw[line width=0.25mm, color=blue] (0,2.25) -- (0,3.25);
            \draw[line width=0.25mm] (0,0.25) node[dot] {} -- (0.5,1.25) node[dot] {} -- (0,2.25) node[dot] {};
            \draw (0,-0.25) node[scale=0.7] {$z_1$};
\end{tikzpicture}\\
&=\int_{\Lambda^2}{\xi}_\eps(z_1)(G^{\alpha,m}-G^{\alpha,m}_\mu)(z_1-z){\xi}_\eps(z)(G^{\alpha,m}-G^{\alpha,m}_\mu)(z-z_2){\xi}_\eps(z_2)\phi(z_2)\dint z\dint z_2\,.
\end{equs}
The black dots represent the noise, the black edges represent convolutions with the kernel $G-G_\mu$, and the loose blue edge represents $\phi$. The hierarchy of flow equations \eqref{eq:hierarchyxi} also has a graphical representation. For $n=2$, it would read
    \begin{equs}
        \D_\mu\begin{tikzpicture}[scale=0.4,baseline=0.15cm]
            \draw[line width=0.25mm, color=blue] (0,2.25) -- (0,3.25);
            \draw[line width=0.25mm] (0,0.25) node[dot] {} -- (0.5,1.25) node[dot] {} -- (0,2.25) node[dot] {};
            \draw (0,-0.25) node[scale=0.7] {$z_1$};
\end{tikzpicture}=\begin{tikzpicture}[scale=0.4,baseline=0.15cm]
            \draw[line width=0.25mm, color=blue] (0,2.25) -- (0,3.25);
            \draw[line width=0.25mm, color=red] (0,0.25) node[dot] {} -- (0.5,1.25) node[dot] {};
            \draw[line width=0.25mm] (0.5,1.25) node[dot] {} -- (0,2.25) node[dot] {};
            \draw (0,-0.25) node[scale=0.7] {$z_1$};
\end{tikzpicture}+\begin{tikzpicture}[scale=0.4,baseline=0.15cm]
            \draw[line width=0.25mm, color=blue] (0,2.25) -- (0,3.25);
            \draw[line width=0.25mm] (0,0.25) node[dot] {} -- (0.5,1.25) node[dot] {};
            \draw[line width=0.25mm, color=red] (0.5,1.25) node[dot] {} -- (0,2.25) node[dot] {};
            \draw (0,-0.25) node[scale=0.7] {$z_1$};
\end{tikzpicture}\,,
    \end{equs}
where this time the red edges represent a convolution by the kernel $-\dot G_\mu$. The inductive part enters when splitting the trees on the red edges. For instance
\begin{equs}
    \begin{tikzpicture}[scale=0.4,baseline=0.15cm]
            \draw[line width=0.25mm, color=blue] (0,2.25) -- (0,3.25);
            \draw[line width=0.25mm, color=red] (0,0.25) node[dot] {} -- (0.5,1.25) node[dot] {};
            \draw[line width=0.25mm] (0.5,1.25) node[dot] {} -- (0,2.25) node[dot] {};
            \draw (0,-0.25) node[scale=0.7] {$z_1$};
\end{tikzpicture}\to \begin{tikzpicture}[scale=0.4,baseline=0.15cm]
            \draw[line width=0.25mm, color=blue] (0,0.25) node[dot] {} -- (0,1.25);
            \draw (0,-0.25) node[scale=0.7] {$z_1$};
\end{tikzpicture},~~\begin{tikzpicture}[scale=0.4,baseline=0.15cm]
            \draw[line width=0.25mm, color=blue] (0.5,1.25) -- (0.5,2.25);
            \draw[line width=0.25mm] (0,0.25) node[dot] {} -- (0.5,1.25) node[dot] {};
            \draw (0,-0.25) node[scale=0.7] {$v_2$};
\end{tikzpicture}\,.
\end{equs}
\end{example}

The interest of the stationary and translated effective forces is that
\begin{enumerate}
    \item   The stationary effective force is independent from the initial condition and stationary in time (and in space outside the case of the harmonic oscillator);
    \item  The translated effective force involves the stationary in time kernel $G^{\alpha,m}$.
\end{enumerate}
Therefore, apart from the case of the harmonic oscillator, we will renormalise $S^{\alpha,m}_\epmu$ only, thereby enforcing that $\mfC^{\alpha,m}_\epmu$ does not depend on time and on the initial condition, then build $T^{\alpha,m}_\epmu$ pathwise in $S^{\alpha,m}_\epmu$, thereby dealing with the initial condition, and then build $F^{\alpha,m}_\epmu$ pathwise in $T^{\alpha,m}_\epmu$ thereby dealing with the breaking of time-translation-invariance. 

While going from the translated effective force to the effective is an analytic step which requires some work, going from the stationary force to the translated force is essentially algebraic. We give below the expressions of the translated force coefficients.
\begin{lemma}\label{Lem:expressiontau}
Fix some stationary force coefficients satisfying the hierarchy \eqref{eq:hierarchyxi} and take
    \begin{equs}
        \uptau^{\alpha,m;0,0}_{\eps,\mu}=-\psi_\eps\quad \text{and}\quad\uptau^{\alpha,m;0,n}_{\eps,\mu}=\upxi^{\alpha,m;n-1}_{\eps,\mu}[G^{\alpha,m}_\mu\psi^{\alpha,m}_\eps]\;\;\text{for}\;\;n\geqslant1\,,
    \end{equs}
alongside
\begin{equs}
  \uptau^{\alpha,m;1,n}_{\eps,\mu}[v]=  \upxi_{\eps,\mu}^{\alpha,1;n}[v] \;\;\text{for}\;\;n\geqslant0\,.
\end{equs}
Then, we have that the above translated force coefficients satisfy the hierarchy \eqref{eq:hierarchytau}.
\end{lemma}
\begin{proof}
    This follows from the fact that, if $S_\epmu^{\alpha,m}$ is an exact solution to the inviscid Polchinski flow equation with initial condition $S_\eps^{\alpha,m}$, then $S^{\alpha,m}_\epmu[\bigcdot+G_\mu^{\alpha,m}\psi^{\alpha,m}_\eps]-\psi^{\alpha,m}_\eps$ solves this equation with initial condition $T_\eps^{\alpha,m}$.
\end{proof}
Going from the translated force to the effective force relies on the fact that, after a certain time, translated force coefficients and effective force coefficient coincide. This crucial property is stated in the following lemma.
\begin{lemma}\label{Lem:suppprop}
    In view of the support properties of the scale decomposition of the heat kernel and of the flow equation \eqref{eq:hierarchyxi}, we have that $\upxi^{\alpha,m;n}_{\eps,\mu}(z_1,z_2)$ is supported on $t_1-2n\mu^\alpha\leq t_2\leq t_1$, which implies that 
\begin{equs}
    \upzeta^{\alpha,m;k,n}_{\eps,\mu}(z,\bigcdot)=\uptau^{\alpha,m;k,n}_{\eps,\mu}(z,\bigcdot) \qquad \text{for all} \quad t\geqslant 2n\mu^\alpha\,.
\end{equs}
\end{lemma}
\begin{proof}
    Using the induction hypothesis, we can replace $ \upzeta^{\alpha,m;1,i}_{\eps,\mu}$ by  $\uptau^{\alpha,m;1,i}_{\eps,\mu}=\upxi^{\alpha,m;i}_{\eps,\mu}$ on the right-hand side of \eqref{eq:hierarchyzeta}, and then using the support property of the stationary force coefficients and again the induction hypothesis, we can then replace $ \upzeta^{\alpha,m;k,n-1-i}_{\eps,\mu}$ by  $\uptau^{\alpha,m;k,n-1-i}_{\eps,\mu}$.
\end{proof}
\begin{remark}
For the reader's convenience, let us verify that this property is indeed satisfied by the first three effective force coefficients, whose expression is explicit:
\begin{enumerate}
    \item We have $\upzeta^{\alpha,m;0,0}_\epmu=0$ and $\uptau^{\alpha,m;0,0}_\epmu=-\psi^{\alpha,m}_\eps$ which indeed is zero after the null-time hyperplane.
    \item We have $\upzeta^{\alpha,m;1,0}_\epmu[v]=\one_{(0,\infty)}\upxi_{\eps,0}^{\alpha,m;0}[v]$ and $\uptau^{\alpha,m;1,0}_\epmu[v]=\upxi_{\eps,0}^{\alpha,m;0}[v]$, so both agree after the null-time hyperplane.
     \item We have $\upzeta^{\alpha,m;0,1}_\epmu=\upxi^{\alpha,m;0}_{\eps,0}[G^{\alpha,m}\psi^{\alpha,m}_\eps]$ and $\uptau^{\alpha,m;0,1}_{\eps,\mu}=\upxi^{\alpha,m;0}_{\eps,0}[G^{\alpha,m}_\mu\psi^{\alpha,m}_\eps]$, so indeed their difference $\upxi^{\alpha,m;0}_{\eps,0}[(G^{\alpha,m}-G^{\alpha,m}_\mu)\psi^{\alpha,m}_\eps]$ is supported on times in $[0,2\mu^\alpha]$.
\end{enumerate} 
\end{remark}

Constructing the effective force now reduces to the construction of the stationary force, of the translated force, and then finally of the effective force. This steps are similar on the torus and in infinite volume.

We now turn to discussing how to construct the solution given by the effective force. From now on, the method will slightly differ whether we are in finite volume of not. 

\textit{Finite volume.} On the torus, we partially follow the original strategy of \cite{Duc21}, assuming that the effective force is an exact solution to the inviscid Polchinski equation \eqref{eq:flotFam} so we use the ansatz \eqref{eq:ansatzwithNmax} with $N_{\max}=\infty$. This entails that $H^{\alpha,m}_\epmu[v]=0$ and therefore $R^{\alpha,m}_\epmu=0$.

\textit{Infinite volume.} In infinite volume, the presence of weights in the norms gives little hope for the summability of the series of the force coefficients. We need to adopt the strategy introduced in \cite{DGR23}, that is the introduction of a remainder. We thus assume that the effective force takes the form \eqref{eq:ansatzwithNmax} with $N_{\max}=N_0$, which amounts to truncating \eqref{eq:flotFam} at order $N_0$ in the Picard iterations. For more details see Section 2.2.1 in \cite{CF24a}. 

Given a function $\upxi_\mu$ of $2$ variables, depending on $\mu\in(0,1]$, we set the notation
\begin{equation}
    \big(K_\mu^{\otimes 2}\upxi_\mu\big)(w_1,w_2)=\int_{\SS^2} K_\mu(w_1-z_1)K_\mu(w_2-z_2)\upxi_\mu(z_1,z_2)\dint z_1\dint z_2\,.
\end{equation}
This notation will also be used with other kernels.

\begin{definition}
    We define for such a function the following weighted $L^\infty_{w_1} L^1_{w_2}$ norm,
    \begin{equs}
    \|\upxi_\mu\|_{\mu,\kappa,\ell,\varsigma}&\eqdef\sup_{w_1\in\SS}\langle w_1\rangle^{-\kappa}\int_{\SS}\big\langle w_1;w_2\rangle_\mu^{\varsigma-\kappa}\mcb e_\ell(x_1;x_2)~\big|\big(K_{\mu}^{\otimes 2}\upxi_\mu\big)(w_1,w_2)\big|\dint w_2\\
        &\quad+\sup_{w_2\in\SS}\int_{\SS}\langle w_1\rangle^{-\kappa}\big\langle w_1;w_2\rangle_\mu^{\varsigma-\kappa}\mcb e_\ell(x_1;x_2)~\big|\big(K_{\mu}^{\otimes 2}\upxi_\mu\big)(w_1,w_2)\big|\dint w_1
        \,,
    \end{equs}
    where have denoted for compactness, for $w_1,w_2\in\SS$, $\langle w_1;w_2\rangle_\mu=\langle w_1-w_2\rangle_\mu$. We recall that we defined $\langle\bigcdot\rangle_\mu$ in \eqref{mubracket}. When $\varsigma=\kappa$ or $\ell=0$ or $\kappa=\varsigma=0$, we omit them writing 
    $$\|\bigcdot\|_{\mu,\kappa}\eqdef\|\bigcdot\|_{\mu,\kappa,0,\kappa}\,,\quad\|\bigcdot\|_{\mu,\kappa,\ell}\eqdef\|\bigcdot\|_{\mu,\kappa,\ell,\kappa}\quad\text{and}\quad\|\bigcdot\|_{\mu}\eqdef\|\bigcdot\|_{\mu,0,0,0}\,.$$
\end{definition}

\begin{example}
    Keeping with the same example, we give a graphical representation of this norm.
    \begin{equs}
        \|\upxi_{\mu}^2\|_\mu=\begin{tikzpicture}[scale=0.4,baseline=0.15cm]
            \draw[line width=0.3mm, color=indigo] (0,2.25) -- (0,3.25);
            \draw[line width=0.3mm, color=indigo] (0,-0.75) -- (0,0.25);
            \draw[line width=0.25mm] (0,0.25) node[dot] {} -- (0.5,1.25) node[dot] {} -- (0,2.25) node[dot] {};
            \draw (0,-1.25) node[scale=0.7] {$L^\infty_{w_1}$};
            \draw (0,3.75) node[scale=0.7] {$L^1_{w_2}$};
        \end{tikzpicture}\,.
    \end{equs}
This time, the purple edges represent a convolution with the kernels $K_\mu$. Note that we have not included the weights in this representation.
\end{example}
Then, we keep track of the size of all the force coefficients all at once, introducing 
\begin{equation}
  \talloblong\upxi^{\alpha,m}\talloblong_{\kappa,\ell,\varsigma}\eqdef\max_{n\leq N_0}\Big(\sup_{\eps,\mu\in(0,1]}\mu^{-\deg(n)+\kappa}\|\upxi^{\alpha,m;n}_{\eps,\mu}\|_{\mu,\kappa,\ell,\varsigma}\Big)^{1/(n+1)}\,.
\end{equation}

We can now state the main result about the construction of the translation-invariant force coefficients and there associated stochastic estimates.
\begin{proposition}\label{prop:stoestimate}
In this proposition, we cover $\alpha\in(0,2)$ and $m\geqslant0$ in finite volume, $\alpha\in(0,2)$ and $m>0$ in infinite volume, and $(\alpha,m)=(2,0)$ both in finite and infinite volume. Fix $\ell\in\R$ ($\ell=0$ in finite volume), $\kappa>0$ small enough, and $\varsigma>0$ small enough for $\alpha<2$ or $\varsigma=1$ for $(\alpha,m)=(2,0)$. Let $\kappa_n\eqdef 4^{n-N_0-1}\kappa$. There is a choice of counterterms $(\mfc_\eps^{n})_{n\leq N_1}$ such that the following holds.

    For every $0\leq n\leq N_0$, uniform in $\epmu\in(0,1]$, we have
    \begin{equation}
        \|\upxi_\epmu^{\alpha,m;n}\|_{\mu,\kappa_n,\ell,\varsigma}\lesssim{\talloblong\upxi^{\alpha,m}\talloblong_{\kappa_{0},\ell,\varsigma}}\mu^{\deg(n)-\kappa_n}\,,
    \end{equation}
alongside the moment bound
    \begin{equation}
        \E\Big[\talloblong\upxi^{\alpha,m}\talloblong^k_{\kappa_0,\ell,\varsigma}\Big]\lesssim \big(k(N_0+1)\big)^{k/2}\,,
    \end{equation}
for $k\geq1$. Moreover, we have convergence of the force coefficients in the above topology
    \begin{equation}
        \|\upxi_\epmu^{\alpha,m;n}-\upxi^{\alpha,m;n}_{\eps',\mu}\|_{\mu,\kappa_n,\ell,\varsigma}\lesssim(\eps\vee\eps')^{\eta}\mu^{\deg(n)-\kappa_n-\eta}\,.
    \end{equation}
\end{proposition}
\begin{proof}
This is a consequence of Proposition \ref{estimatemoment} and of a Kolmogorov-type argument developed in Section \ref{Kolmogorov}.
\end{proof}


 With the estimates we derived in the previous section on the stationary coefficients, we can construct the non-stationary force coefficients of $F_\epmu$, and with more work the solution to the mollified SPDE \eqref{renormSPDE}. For this second point, the strategy differs whether we are in finite volume or not. In the first case, we are able to construct the effective force as a series, which simplifies some of the arguments. However, in infinite volume, more care must be taken, since we need to deal with exponential weights. Our strategy is to use a remainder as it was introduced in \cite{DGR23}.

\subsection{From the stationary force to the effective force}\label{sec:fromstatiotoeff}

In this section, we show how to construct the effective force coefficients starting from the stationary force coefficients. This construction is pathwise. We perform in both in finite and infinite volume, for the massive and massless fractional Laplacian in finite volume, for the massive fractional Laplacian in infinite volume, and for the Laplacian both in finite and infinite volume. From now on, we thus fix $(\alpha,m)$ corresponding to the Laplacian, the massive fractional Laplacian, or in finite volume the massless fractional Laplacian, and leave the dependence of all the force coefficients on $(\alpha,m)$ implicit. Moreover, we assume that $\ell\in\R$ is fixed and $\ell=0$ in finite volume, and that $\kappa,\varsigma>0$ are fixed small enough (of course of the Laplacian the smallness condition on $\varsigma$ is not necessary, but we would anyway not use a large $\varsigma$). To lighten the notation, since the parameters are fixed, we write
$ \talloblong\upxi\talloblong$ instead of $\talloblong\upxi^{\alpha,m}\talloblong_{\kappa_0,\ell,\varsigma}$. Here, as usual $\kappa_0=4^{-N_0-1}\kappa$.

We start by constructing the translated force coefficients, which take the presence of the initial condition into account. This relies on the explicit expressions given in Lemma~\ref{Lem:expressiontau}.
\begin{lemma}\label{Lem:constructiontau}
Recall that $\kappa_n=4^{n-N_0-1}\kappa$. For every $0\leq n\leq N_0$, uniform in $\epmu\in(0,1]$, we have
\begin{equs}\label{eq:esti_tau}
    \|\langle x\rangle^{-\kappa_n} \mcb e_{-\ell}K_\mu\uptau^{0,n}_{\eps,\mu} \|_{L_{t,x}^{\infty,r}}\lesssim {\talloblong\upxi\talloblong^{n}\|  \mcb e_{-\ell}u_0^\eps\|_{\mcB^{-\sigma}_{r,q}} }\mu^{\deg(0,n)-\kappa_n }\,.
\end{equs}
Here, notice that we do not need to put a weight in time, since $G_\mu\psi_\eps$ and hence $\uptau^{n,0}_{\eps,\mu}$ are supported after time $\mu^\alpha$. Note also that on the torus the spacial weight is spurious.
\end{lemma}
\begin{proof}
    The proof boils down to estimating $\mcR^2_\mu G_\mu\psi_\epsilon=\mcR^2_\mu G_\mu(\delta\otimes u_0^\eps)$. Recall that $G_\mu$ is the fractional heat kernel localised in time after $\mu^\alpha$ and that $\uptau_\epmu^{0,n}=\upxi_\epmu^{n-1}[G_\mu\psi_\eps]$. Using Lemma \ref{ITexp}, we have
\begin{equs}
    ~&\|\langle\bigcdot \rangle^{-\kappa_n}\mcb e_{-\ell}K_\mu\uptau_\epmu^{0,n}\|_{L^{\infty,r}_{t,x}}^r\\
    &=\sup_{t_1}\int_{\R^d} \left|\langle x_1\rangle^{-\kappa_n}\mcb e_{-\ell}(x_1)\int_{\Lambda} K_\mu(w_1-z_1)\upxi_\epmu^{n-1}(z_1,z_2)(G_\mu\psi_\eps)(z_2)\dint z_2\right|^r\dint x_1\\
    &\leq \sup_{t_1}\int_{\R^d} \bigg|\langle x_1\rangle^{-\kappa_n}\int K_\mu(w_1-z_1)\upxi_\epmu^{n-1}(z_1,z_2)K_\mu(z_2-w_2)\mcb e_{-\ell}(x_1)\\
    &\qquad\qquad\qquad\qquad\qquad\qquad\qquad\qquad\qquad\qquad\qquad\times\mcb e_{\ell}(x_2)(\mcb e_{-\ell}\mcR_\mu^2G_\mu\psi_\eps)(w_2)\dint z_2\dint w_2\bigg|^r\dint x_1\\
    &\leq \sup_{t_1}\int_{\R^d} \bigg|\langle x_1\rangle^{-\kappa_n}\int_{\Lambda^2} \mcb e_{\ell}(x_1;x_2)K_\mu(w_1-z_1)\upxi_\epmu^{n-1}(z_1,z_2)K_\mu(z_2-w_2)\\
    &\qquad\qquad\qquad\qquad\qquad\qquad\qquad\qquad\qquad\qquad\qquad\qquad~~\times(\mcb e_{-\ell}\mcR_\mu^2G_\mu\psi_\eps)(w_2)\dint z_2\dint w_2\bigg|^r\dint x_1\\
    &\leq \|\upxi_\epmu^{n-1}\|_{\mu,\kappa_n,\ell,\kappa_n}^r\|\mcb e_{-\ell}\mcR_\mu^2G_\mu\psi_\eps\|_{L^{\infty,r}_{t,x}}^r
\end{equs}
Using the same argument we can show that, with the heat kernel estimate in Lemma \ref{heatkernelestimate},
\begin{equs}
    \|\mcb e_{-\ell}\mcR_\mu^2G_\mu\psi_\eps\|_{L^{\infty,r}_{t,x}}\lesssim\mu^{-\sigma}\|\mcb e_{-\ell}u_0^\eps\|_{\mcB_{r,q}^{-\sigma}}\,.
\end{equs}
    
\end{proof}

We now turn from the step in which we go from the translated force coefficients to the effective force coefficients. It turns out that, as shown in \cite[Section 10.4]{Duc21}, it is possible to solve the hierarchy of flow equations \eqref{eq:hierarchyzeta} via a deterministic construction. Indeed, in view of Lemma~\ref{Lem:suppprop}, we only need to construct $\upzeta^{k,n}_{\eps,\mu}(z,\bigcdot)$ on the remaining interval $[0,2n\mu^\alpha]\times\R^d$. There, we can rely on the fact that the remaining interval being small, using an $L^{1}$ norm in time improves the power counting by two, thus making all the effective force coefficients with $(k,n)\neq(0,0)$ irrelevant. The $L^1$ norm is defined below.
\begin{equs}
\|\upzeta^1_\mu\|_{1,\mu,\kappa,\ell,\varsigma}&\eqdef\sup_{x_1\in\R^d}\langle x_1\rangle^{-\kappa}\int_{[0,1]\times\SS}\big\langle w_1;w_2\rangle_\mu^{\varsigma-\kappa}\mcb e_\ell(x_1;x_2)~\big|\big(K_{\mu}^{\otimes 2}\upxi_\mu\big)(w_1,w_2)\big|\dint t_1\dint w_2\\
    &\quad+\sup_{x_2\in\R^d}\int_{[0,1]\times\SS}\langle x_1\rangle^{-\kappa}\big\langle w_1;w_2\rangle_\mu^{\varsigma-\kappa}\mcb e_\ell(x_1;x_2)~\big|\big(K_{\mu}^{\otimes 2}\upxi_\mu\big)(w_1,w_2)\big|\dint t_2\dint w_1
    \,.
\end{equs}
On the torus, we use the same notation for the same norm with $\R^d$ being replaced by $\T^d$. We have used in a general way the notation $w=(t,x)$. We present in the following the estimates on the coefficients of the effective force.

\begin{lemma}\label{Lem:constructionzeta}
Recall that $\kappa_n=4^{n-N_0-1}\kappa$. For every $0\leq n\leq N_0$, uniform in $\epmu\in(0,1]$, we have 
\begin{equation}
    \label{eq:esti_zeta}\|\upzeta_\epmu^{1,n}\|_{\mu,\kappa_n,\ell,\varsigma}\lesssim
    {\talloblong\upxi\talloblong^{n+1}}\mu^{\deg(1,n)- \kappa_n}\,,
\end{equation}
and
\begin{equation}
    \| \langle x \rangle^{-\kappa_n}  \mcb e_{-\ell}K_\mu\upzeta^{0,n}_{\eps,\mu} \|_{L^{\infty,r}_{t,x}}\lesssim {\talloblong\upxi\talloblong^{n}\| \mcb e_{-\ell} u_0^\eps\|_{\mcB^{-\sigma}_{r,q}} }\mu^{\deg(0,n)-\kappa_n}\,.
\end{equation}
\end{lemma}
\begin{proof}

We denote by $\mcb v_\mu^n$ and $\mcb w_\mu^n$ two smooth functions localising respectively after and before $2n\mu^\alpha$ in time. In view of Lemma~\ref{Lem:suppprop}, we thus have
\begin{equs} \label{eq:zetamu}   \upzeta^{k,n}_\epmu= \mcb v^n_\mu\uptau_\epmu^{k,n}+\mcb w^n_\mu\upzeta^{k,n}_\epmu\,.
\end{equs}
The contribution of the first term is easy to control, and stems from \eqref{eq:esti_tau}. We thus turn to the second term in \eqref{eq:zetamu}. The proof is by induction on $n$, so we start by discussing the initialisation: the base case is $\upzeta_\epmu^{1,0}[\varphi]=\one_{(0,\infty)}\upxi^0_{\eps,0}[\varphi]$. The noise being of H\"older-Besov regularity $-\alpha+\beta-\kappa$, it is more regular than minus the parabolic dimension $\alpha$ of the null-time hyperplane, so the multiplication by the indicator is well-defined via a deterministic estimate. Regarding the induction step, it relies on the following flow equation:
\begin{equs}
    \partial_\nu\big( \mcb w_\nu\upzeta^{k,n}_\epnu\big)&= 
(1-K_\nu)
\big(\partial_\nu(\mcb  w^n_\nu\upzeta^{k,n}_\epnu)\big)
    +K_\nu
   \big( (\partial_\nu  \mcb w^n_\nu)\upzeta^{k,n}_\epnu\big)+ K_\nu\big(\mcb w^n_\nu\partial_\nu\upzeta^{k,n}_\epnu\big)\\&=
   (1-K_\nu)
\big(\partial_\nu( \mcb w^n_\nu\upzeta^{k,n}_\epnu)\big)
    -K_\nu
   \big( (\partial_\nu \mcb  v^n_\nu)\uptau^{k,n}_\epnu\big)+ K_\nu\big(\mcb w^n_\nu\partial_\nu\upzeta^{k,n}_\epnu\big)
\,.\label{eq:dmuwmuzeta}
\end{equs}
For the first term, \eqref{eq:comKmuKnu1} gives $\mu^{-2}\nu^2$, which is sufficient for having an irrelevant power counting. For the second and third terms, we can sacrifice the $K_\mu$ to perform a Sobolev embedding from $L^\infty$ to $L^1$ in time, which results in a loss of $\mu^{-\alpha}$. For the second term, one can use the good support of $\partial_\nu v_\nu^n$ to turn the $L^1$ norm in time into an $L^\infty$ norm again, gaining a factor $\nu^{\alpha}$, thereby making the power counting irrelevant. For the third term, we add to the induction hypothesis the improved estimate
\begin{equs}
    \|\mcb w^n_\nu\upzeta_\epmu^{1,n}\|_{1,\mu,\kappa_n,\ell,\varsigma}\lesssim
    {\talloblong\upxi\talloblong^{n+1}} \mu^{\deg(1,n)+\alpha-\kappa_n}\,,
\end{equs}
so that, using the flow equations \eqref{eq:hierarchyzeta}, again we end up with an irrelevant power counting.

For the reader's convenience, we detail the argument for the third term, considering for instance the case $k=0$. We start from 
\begin{equs}
   ~&\bigg\|\langle x\rangle^{-\kappa_n} \mcb e_{-\ell} K_\mu\int_0^\mu K_\nu\big(\mcb w^n_\nu\partial_\nu\upzeta^{0,n}_\epnu\big)\dint\nu\bigg\|_{L^{\infty,r}_{t,x}}\\
   &\lesssim\mu^{-\alpha}\int_0^\mu \|\langle x\rangle^{-\kappa_n}  \mcb e_{-\ell} K_\nu\big(\mcb w^n_\nu\partial_\nu\upzeta^{k,n}_\epnu\big)\|_{L^{1,r}_{t,x}}\dint\nu\\&\lesssim
  \mu^{-\alpha}\int_0^\mu \sum_{i=0}^{n-1}\|\langle x\rangle^{-\kappa_n} \mcb e_{-\ell} K_\nu\big(\mcb w^n_\nu\upzeta_{\eps,\nu}^{1,i}[\dot G_\nu\upzeta_{\eps,\nu}^{0,n-1-i}(\bigcdot)]\big)\|_{L^{1,r}_{t,x}}\dint\nu\\
  &\lesssim
   \mu^{-\alpha}\int_0^\mu\nu^{-1+\alpha} \sum_{i=0}^{n-1}\|\mcb w^i_\nu\upzeta_{\eps,\nu}^{1,i}\|_{1,\nu,\kappa_i,\ell,\kappa_i+\kappa_{n-1-i}}\|\langle x\rangle^{-\kappa_{n-1-i}} \mcb e_{-\ell}{K_\nu}\upzeta_{\eps,\nu}^{0,n-1-i}\|_{L^{\infty,r}_{t,x}}\dint\nu
   \\
  &\lesssim\talloblong\upxi\talloblong^n\| \mcb e_{-\ell}u_0^\eps\|_{\mcB_{r,q}^{-\sigma}}
   \mu^{-\alpha}\int_0^\mu\nu^{-1+\alpha+\deg(0,n)-2\kappa_{n-1}}\dint\nu \\
   &\lesssim \talloblong\upxi\talloblong^{n}\| \mcb e_{-\ell} u_0^\eps\|_{\mcB^{-\sigma}_{r,q}} \mu^{\deg(0,n)-\kappa_n }\,.
\end{equs}
Here, we using the fact that putting a weight $\langle x\rangle^{-\kappa_{n-1-i}}$ in front of $\upzeta_\epnu^{0,n-1-i}$ forces $\upzeta_\epnu^{1,i}$ to absorb a weight $\langle w_1,w_2\rangle_\nu^{\kappa_{n-1-i}}$, and results in a loss of the original $\langle x\rangle^{-\kappa_{n}}$ which becomes $\langle x\rangle^{-\kappa_{n}+\kappa_{n-1-i}}$. This is safe since ${-\kappa_{n}+\kappa_{n-1-i}}\leq -\kappa_{n}+\kappa_{n}/4\leq  -\kappa_{i}$. In the last step, we relied on the fact that the gain of $\alpha$ is sufficient for the integral over $\nu$ to converge.

The proof is similar in the case where $k=1$, but now we must also distribute the weight $\langle w_1,w_2\rangle_\nu^{\varsigma-\kappa_{n}}$ on the two force coefficients. Therefore, $\upzeta_\epnu^{1,i}$ must now absorb a bigger weight $\langle w_1,w_2\rangle_\nu^{\varsigma-\kappa_n+\kappa_{n-1-i}}$. Again, this is safe since $-\kappa_n+\kappa_{n-1-i}\leq-\kappa_i$.
\end{proof}

\subsection{The fractional Laplacian in finite volume}
In this section, we carry on with the construction of the solution in finite volume. Here $(\alpha,m)$ correspond to the Laplacian and the massive and massless fractional Laplacian, $\ell\equiv0$, $\kappa,\varsigma>0$ are fixed small enough, and we just write $\talloblong\upxi\talloblong$ instead of $\talloblong\upxi^{\alpha,m}\talloblong_{\kappa_0,0,\varsigma}$.

On the torus, Lemma~\ref{Lem:constructionzeta} is sufficient to construct all effective force coefficients, the remaining force coefficients being irrelevant. 
\begin{corollary}\label{Lem:constructionzetairrelevant}
On the torus $\Lambda=[0,1]\times\T^d$, for every $n\geq 0$, uniform in $\epmu\in(0,1]$ we have the estimates
\begin{equs}
\|\upzeta_\epmu^{1,n}\|_{\mu}\lesssim
    \frac{\talloblong\upxi\talloblong^{n+1}}{(n+1)^2} \mu^{\deg(1,n)- (n+1)\kappa}\,,
\quad\mathrm{and}\quad
    \| K_\mu\upzeta^{0,n}_{\eps,\mu} \|_{L^{\infty,r}_{t,x}}\lesssim \frac{\talloblong\upxi\talloblong^{n}\| u_0^\eps\|_{\mcB^{-\sigma}_{r,\infty}} }{( n+1)^2}\mu^{\deg(0,n)-(n+1)\kappa }\,.\qquad
\end{equs}
\end{corollary}
\begin{proof}
    This is an immediate consequence of the fact that the topology we use is compatible with the hierarchy of equations \eqref{eq:hierarchyzeta}. Indeed, on the compact space $[0,1]\times \T^d$, we can eliminate the weight $\langle w_1\rangle^{-\kappa}$. The proof then follows by induction over $n$ starting from $N_0+1$, accumulating the losses of the factors $\mu^{-\kappa}$ at each step of the induction.
    
    \begin{equs}
        \|\upzeta_\epmu^{1,n}\|_{\mu}&\lesssim\int_0^\mu \sup_{w_1\in\Lambda}\int_{[0,1]\times\T^d} \big|\big(K_\nu^{\otimes2}\D_\nu\upzeta_\epnu^{1,n}\big)(w_1,w_2)|\dint w_2\dint \nu\\
        &\lesssim \int_0^\mu \sum_{i=0}^{n-1}\|\upzeta_\epnu^{1,i}\|_{\nu}\|\dot G_\nu\|_{L^1_{t,x}}\|\upzeta_\epnu^{1,n-1-i}\|_{\nu}\dint\nu\\
        &\lesssim \int_0^\mu \sum_{i=0}^{n-1}\frac{\talloblong\upxi\talloblong^{i+1}}{(i+1)^2} \nu^{\deg(i)- (i+1)\kappa}\nu^{-1+\alpha}\frac{\talloblong\upxi\talloblong^{n-i}}{(n-i)^2} \nu^{\deg(n-1-i)- (n-i)\kappa}\dint\nu\\
        &\lesssim \talloblong\upxi\talloblong^{n+1}\int_0^\mu \nu^{\deg(n)-1-(n+1)\kappa}\dint\nu\sum_{i=0}^{n-1}\frac{1}{(i+1)^2(n-i)^2}\\
        &\lesssim\frac{\talloblong\upxi\talloblong^{n+1}}{(n+1)^2} \mu^{\deg(n)- (n+1)\kappa}\,.
    \end{equs}
    We have used \eqref{eq:Kmunu} and Lemma \ref{estimdotGmu}.

\end{proof}

{In this section, we denote by $\| \bigcdot\|_{\mu}$ the usual norm but with all the times restricted to $[0,T]$. Equivalently, we view the force coefficients $\upzeta_\epmu^{k,n}$ as supported on $[0,T]\times\T^d$ only. 
}
\begin{lemma}\label{Lem:constructionforce}
Fix a deterministic combinatorial constant $0<T_0\leq1$ (independent from the initial condition) to be taken small enough, and define the random time horizon
\begin{equs}
    T\eqdef\big(\talloblong\upxi\talloblong+T_0^{-\kappa/\alpha}\big)^{-\alpha/\kappa}\,.
\end{equs}
Then, uniform in $\eps\in(0,1]$ and $\mu\in(0,T^{1/\alpha}]$, it holds
\begin{equs}
\label{eq:esti_force}\|F^1_\epmu\|_{\mu}\lesssim\talloblong\upxi\talloblong
    \mu^{-\alpha+3\kappa}\,,
\quad\text{and}\quad
    \| K_\mu F_\epmu^0\|_{L_{t,x}^{\infty,r}(\Lambda_T)}\lesssim \talloblong\upxi\talloblong\|u_0^\eps\|_{\mcB^{-\sigma}_{r,\infty}} \mu^{-\alpha-\sigma+\beta-2\kappa}\,.
\end{equs}
\end{lemma}
\begin{proof}
    This is an immediate consequence of the estimate \eqref{eq:esti_zeta} and of 
\begin{equs}
        \deg(1,n)-(n+1)\kappa&=-\alpha+(n+1)(\beta-\kappa)\geq
        -\alpha+\beta+n\kappa  = -\alpha+\beta-\kappa+n\kappa \\
        &\geqslant
         -\alpha+3\kappa+n\kappa\,,
    \end{equs}
    and for $n\geq1$   
    \begin{equs}
        \deg(0,n)-(n+1)\kappa&=-\alpha-\sigma+n\beta -(n+1)\kappa=-\alpha-\sigma+\beta-2\kappa+ (n-1)(\beta-\kappa)\\
        &\geq-\alpha-\sigma+\beta-2\kappa+(n-1)\kappa\,.
    \end{equs}
\end{proof}

We introduce a novel way to construct the solution in finite volume. It relies on a fixed-point argument and bypasses some tedious steps that we would have to go through otherwise. This method is proper to the use of a non-truncated effective force.

\begin{definition}
For $\mu\in(0,T^{1/\alpha}]$, we define the \textit{planted effective force}
\begin{equs}
   P_\epmu[\varphi]\eqdef (G-G_{\mu})F_{\epmu}[\varphi]\,.
\end{equs}
Right as we did for the force, we split it as $ P_\epmu[\varphi]=  P_\epmu^{{ 0}}+ P^1_\epmu[\varphi]$, where $P^1_\epmu$ is a linear operator.
\end{definition}
The name stems from the fact that $P_\epmu[\varphi]$ is the sum of all the planted trees. Note that $P_\epmu$ is related to the difference between the solution and the effective solution at scale $\mu$ via
\begin{equs}
    v_\eps-v_\epmu=P_\epmu[v_\epmu]\,.
\end{equs}
In particular, using the fact that $v_{\eps,T^{1/\alpha}}=0$ on $\Lambda_T$, we have
\begin{equs}
    v_\eps=P^0_{\eps,T^{1/\alpha}}\,.
\end{equs}

\begin{lemma}
    The planted effective force satisfies the flow equation
\begin{equs}
    \D_\mu   P_\epmu[\varphi]= (-1+\mathrm D P_\epmu)\dot G_\mu F_{\epmu}[\varphi]\,,
\end{equs}
with null initial condition.
\end{lemma}
\begin{definition}
    Define the map $\Phi_\eps$ acting on trajectories $P=(P_\mu)_{\mu\in(0,T^{1/\alpha}]}$ of functionals of the form $P^0_\mu+P^1_\mu[\varphi]$ where $P_\mu^1$ is linear by
    \begin{equs}
        \Phi_\epmu(P)[\varphi]=\int_0^\mu (-1+\mathrm D P_\nu)\dot G_\nu F_{\epnu}[\varphi]\rmd\nu\,.
    \end{equs}
    Equivalently, projecting on the two components, we have
      \begin{equs}
        \Phi^k_\epmu(P)=\int_0^\mu (-1+ P^1_\nu)\dot G_\nu F^k_{\epnu}\rmd\nu\,.
    \end{equs}
We aim to construct the solution by showing that $\Phi$ admits a fixed point. We will prove that indeed it is a contraction under the norms designed as follows. First, introduce the operator norm
\begin{equs}
    \|P_\mu^{1}\|_{P,\mu}\eqdef\|\big(\Id\otimes K_{\mu} \big)P_\mu^{1}\|_{L_{w_1}^\infty L_{w_2}^1(\Lambda_T^2)}+\|\big(\Id\otimes K_{\mu} \big)P_\mu^{1}\|_{L_{w_2}^\infty L_{w_1}^1(\Lambda_T^2)}\,.
\end{equs} 

Then, for $\gamma\in[0,1]$ and $p\in[1,\infty]$, we define the following norm for the trajectory of planted effective force: 
\begin{multline}
    \snnorm{P_\eps}_{\gamma,p,T}\eqdef \Big(\sup_{\mu\in(0,T^{1/\alpha}]}\mu^{-\kappa}  \|P^1_\epmu\|_{P,\mu} \Big)\\\vee\Big(\|u_0^\eps\|_{\mcB_{r,\infty}^{-\sigma}}^{-1}\sup_{\mu\in(0,T^{1/\alpha}]} \sup_{t\leqslant T} t^{-(\beta-\sigma-\gamma-5\kappa-d(1/r-1/p))/\alpha}\|P^0_\epmu(t)\|_{\mcB_{p,\infty}^\gamma}\Big) \,.
\end{multline}
\end{definition}
\begin{lemma}\label{lemmafixedpointplanted}
Recall the dependence of $T$ on $T_0$. Assuming that $T_0$ is chosen large enough (independently on the initial condition), then the map $\Phi_\eps$ is a contraction for the norm $\snnorm{\bigcdot}_{\gamma,p,T}$ on the unit ball, and is continuous in the data $(F_\epmu)_{\mu\in(0,T^{1/\alpha}]}$.

Therefore, the planted effective force exists, its norm $\snnorm{\bigcdot}_{\gamma,p,T}$ is bounded by one and, in particular, the solution $v_\eps$ is constructed on the (random) interval $[0,T]$ and satisfies the following estimate: for every pair $(\gamma,p)$ such that $p\geq r$ and $0\leq \gamma\leqslant\beta-\sigma-5\kappa-d(1/r-1/p)$ and for every $t\in(0, T]$, 
\begin{equs}
 \|v_\eps(t)\|_{\mcB_{p,\infty}^\gamma}\leq t^{(\beta-\sigma-\gamma-5\kappa-d(1/r-1/p))/\alpha}  \|u_0^\eps\|_{\mcB_{r,\infty}^{-\sigma}}\,.
\end{equs}
\end{lemma}
\begin{proof}
We prove that $\Phi$ maps the unit ball into itself, the proof of the difference bound being similar. Assume $\snnorm{P_\eps}_{\gamma,p,T}\leqslant1$. We have 
\begin{equs}
    \|\Phi^1_\epmu (P)\|_{P,\mu}&\lesssim\int_0^\mu
   \big(1+\|P^1_\epnu \|_{P,\mu}\big)\| \mcR_\nu^2\dot G_\nu\|_{L_{t,x}^1}\|F^1_\epnu\|_{\nu}\dint\nu\\&\lesssim\talloblong\upxi\talloblong\int_0^\mu
   (1+\nu^{\kappa}) \nu^{-1+3\kappa}\dint\nu\lesssim T_0^{\kappa/\alpha} T^{\kappa/\alpha}
   \talloblong\upxi\talloblong \mu^{\kappa}\leqslant \mu^\beta\,.
\end{equs}
Here, we have absorbed $\talloblong\upxi\talloblong$ with $T^{\kappa/\alpha}$ and taken $T_0$ small enough to absorb the combinatorial constant. Next, we have
\begin{equs}
   \lambda^{-\gamma}   \|(Q_\lambda-1)\Phi^0_\epmu(P)(t) \|_{L^p_x}\lesssim 
    \lambda^{-\gamma} \int_0^\mu  \|(Q_\lambda-1)(-1+ P^1_\epnu)\dot G_\nu F^0_{\epnu}(t) \|_{L^p_x}\dint\nu\,.
\end{equs}
Here, we need to consider several cases. 

Firstly, if $\lambda\geq\mu$, then we do not need to rely on the kernel $(Q_\lambda-1)$. Bounding the integral up to $\mu$ by an integral up to $\lambda$, we have for $\lambda\geq\mu$
\begin{equs}
   \lambda^{-\gamma}   \|(Q_\lambda-1)\Phi^0_\epmu(P)(t) \|_{L^p_x}\lesssim 
    \lambda^{-\gamma} \int_0^\lambda \big(1+\|P^1_\epnu \|_{P,\nu}\big)  \|\mcR_\nu\dot G_\nu F^0_{\epnu}(t) \|_{L^p_x}\dint\nu\,.
\end{equs}
Here, we can use the fact that $\mcR^2_\nu\dot G_\nu$ is localised in time at scale $\nu^\alpha$, which implies that $t\geqslant\nu^\alpha$ and that the time component of the argument of $\mcR_\nu^2 \dot G_\nu$ is bounded by $\nu^\alpha\leqslant\lambda^\alpha$, we obtain
\begin{equs}
   \lambda^{-\gamma}   \|(Q_\lambda-1)\Phi^0_\epmu(P)(t) \|_{L^p_x}&\lesssim 
     \int_0^\lambda   \|\big(\BF{t}^{-\gamma/\alpha}\mcR^2_\nu\dot G_\nu\big) K_\nu F^0_{\epnu}(t) \|_{L^p_x}\dint\nu\\
     &\lesssim
     \int_0^{\lambda\wedge t^{1/\alpha}}   \|\BF{t}^{-\gamma/\alpha}\mcR^2_\nu\dot G_\nu\|_{\mcL(L^{\infty,r}_{t,x},L^{\infty,p}_{t,x})}\|K_\nu F^0_{\epnu} \|_{L^{\infty,r}_{t,x}}\dint\nu\\
     &\lesssim\talloblong\upxi\talloblong\| u_0^\eps\|_{\mcB_{r,\infty}^{-\sigma}}\int_0^
{ t^{1/\alpha}}   \|\mcR^2_\nu\dot G_\nu\|_{\mcL(L^{\alpha/(\gamma+\kappa),r}_{t,x},L^{\infty,p}_{t,x})}\nu^{-\alpha-\sigma+\beta-2\kappa}\dint\nu\\
      &\lesssim\talloblong\upxi\talloblong\| u_0^\eps\|_{\mcB_{r,\infty}^{-\sigma}}
     \int_0^{t^{1/\alpha}}  \nu^{-1-\sigma+\beta-\gamma-3\kappa-d(1/r-1/p)}\dint\nu
\\
      &\lesssim T_0^{\kappa/\alpha}T^{\kappa/\alpha}\talloblong\upxi\talloblong\| u_0^\eps\|_{\mcB_{r,\infty}^{-\sigma}}
     t^{(\beta-\sigma-\gamma-5\kappa-d(1/r-1/p))/\alpha}
     \\
      &\leq \| u_0^\eps\|_{\mcB_{r,\infty}^{-\sigma}}
     t^{(\beta-\sigma-\gamma-5\kappa-d(1/r-1/p))/\alpha}
      \,.
\end{equs}

In the case where $\lambda\leqslant\mu$, then we split the integral over $\nu$ as an integral below $\lambda$ plus an integral above $\lambda$. The integral below $\lambda$ being already treated in the previous case, we are left with handling the integral above $\lambda$:
\begin{equs}
     \lambda^{-\gamma} &\int_\lambda^\mu  \|(Q_\lambda-1)(-1+ P^1_\epnu)\dot G_\nu F^0_{\epnu}(t) \|_{L^p_x}\dint\nu\\&\lesssim
       \lambda^{-\gamma} \int_\lambda^\mu  \|Q_\nu(Q_\lambda-1)(-1+ P^1_\epnu)\mcR_\nu^2\dot G_\nu K_\nu F^0_{\epnu}(t) \|_{L^p_x}\dint\nu
       \\&\lesssim
       \lambda^{\alpha-\gamma} \int_\lambda^\mu\nu^{-\alpha}  \|(-1+ P^1_\epnu)\mcR_\nu^2\dot G_\nu K_\nu F^0_{\epnu}(t) \|_{L^p_x}\dint\nu
  \\&\lesssim\one_{\{\lambda^\alpha\leq t\}}
       \lambda^{\alpha-\gamma} \int_\lambda^1\nu^{-\alpha} 
       \|\mcR^2_\nu\dot G_\nu\|_{\mcL(L^{\infty,r}_{t,x},L^{\infty,p}_{t,x})}\|K_\nu F^0_{\epnu} \|_{L^{\infty,r}_{t,x}}\dint\nu\\
       &\lesssim\one_{\{\lambda^\alpha\leq t\}}\talloblong\upxi\talloblong\| u_0^\eps\|_{\mcB_{r,\infty}^{-\sigma}}
       \lambda^{\alpha-\gamma} \int_\lambda^1\nu^{-1-\alpha+\beta-\sigma-2\kappa-d(1/r-1/p)} 
      \dint\nu
      \\
       &\lesssim\one_{\{\lambda^\alpha\leq t\}}T_0^{\kappa/\alpha}T^{\kappa/\alpha}\talloblong\upxi\talloblong\| u_0^\eps\|_{\mcB_{r,\infty}^{-\sigma}}
       \lambda^{\beta-\sigma-\gamma-5\kappa-d(1/r-1/p)} 
     \\
      &\leqslant\tfrac{1}{2}\| u_0^\eps\|_{\mcB_r^{-\sigma}}
       t^{(\beta-\sigma-\gamma-5\kappa-d(1/r-1/p))/\alpha} \,.
       \end{equs}
       Here, note that in the first inequality we used the fact that $P^1_\epnu$ is still translation-invariant in space (translation invariance is broken in time only) so we commuted its action with $Q_\nu$. To conclude the argument, we also need to treat the inhomogeneous term. We have
       \begin{equs}
        \|\Phi^0_\epmu(P)(t)\|_{L^p_x}&\lesssim\int_0^\mu\|\mcR^2_\nu\dot G_\nu K_\nu F^0_{\epnu}(t) \|_{L^p_x}\dint\nu\\
        &\lesssim\int_0^{\mu}\|\mcR^2_\nu\dot G_\nu\|_{\mcL(L^{\infty,r}_{t,x},L^{\infty,p}_{t,x})}\|K_\nu F^0_{\epnu} \|_{L^{\infty,r}_{t,x}}\dint\nu\\
        &\lesssim\talloblong\upxi\talloblong\| u_0^\eps\|_{\mcB_{r,\infty}^{-\sigma}}\int_0^{ \mu}   \|\mcR^2_\nu\dot G_\nu\|_{\mcL(L^{\alpha/\kappa,r}_{t,x},L^{\infty,p}_{t,x})}\nu^{-\alpha-\sigma+\beta-2\kappa}\dint\nu\\
        &\lesssim\talloblong\upxi\talloblong\| u_0^\eps\|_{\mcB_{r,\infty}^{-\sigma}}
        \int_0^{T^{1/\alpha}}  \nu^{-1-\sigma+\beta-3\kappa-d(1/r-1/p)}\dint\nu\\
&\lesssim\talloblong\upxi\talloblong\| u_0^\eps\|_{\mcB_{r,\infty}^{-\sigma}}
        \int_0^{T^{1/\alpha}}  \nu^{-1+2\kappa} \dint\nu t^{\beta-\sigma-5\kappa-d(1/r-1/p)}
        \\
        &\lesssim T_0^{\kappa/\alpha}T^{\kappa/\alpha}\talloblong\upxi\talloblong\| u_0^\eps\|_{\mcB_{r,\infty}^{-\sigma}}t^{(\beta-\sigma-5\kappa-d(1/r-1/p))/\alpha}\\
        &\leq \| u_0^\eps\|_{\mcB_{r,\infty}^{-\sigma}}t^{(\beta-\sigma-5\kappa-d(1/r-1/p))/\alpha}\,,
        \end{equs}
        which actually behaves better than the previous terms. Again, we rely on the fact that the argument $t-s$ of $\mcR^2_\mu\dot G_\mu$ satisfies $t-s\geq \nu^{\alpha}$ which combined with $s\geq0$ implies that $\nu\leqslant t^{1/\alpha}$.
\end{proof}

To conclude this section, we need to use a globalisation argument to eventually obtain the bound \eqref{Eq:EstimateS-Geps} which will lead to \eqref{Eq:EstimateS-G} in Theorem \ref{maintheoremspde}, point 1.

\begin{corollary}\label{Cor:GlobalisationSeps}
    For $\eps,\eta>0$ and $(\sigma,\gamma,p)$ such that $p\geq r$, $0<\sigma<\beta$ and $0\leq \gamma< \beta-\sigma-d(1/r-1/p)$, $u_\eps$ satisfies
    \begin{equation}\label{Eq:EstimateS-Geps}
        \big\|u_\eps(t)-\e^{-t\rmD^\alpha}u_0^\eps\big\|_{\mcB_{p,\infty}^\gamma(\T^d)}\leq \e^{ct}t^{(\beta-\sigma-\gamma-d(1/r-1/p)-\eta)/\alpha}  \|u_0^\eps\|_{\mcB_{r,\infty}^{-\sigma}(\T^d)}\,,
    \end{equation}
    for some positive random constant $c$.
\end{corollary}

\begin{proof}
    From Lemma \ref{lemmafixedpointplanted}, we know that, on the interval $[0,T]$, we have the desired estimate, by choosing $\kappa=\eta/5$. Since the equation is linear, we can restart the dynamics on successive intervals $[T,2T]$, $[2T,3T]$, and so on. Note that it is crucial that the noise does not depend on time to use the same stopping time every time we restart the dynamics. To get a bound after the time $T$, we perform the crude operation
    \begin{equs}
        \|v_\eps(t)\|_{\mcB_p^\gamma(\T^d)}\leq\|u_\eps(t)\|_{\mcB_p^\gamma(\T^d)}+\|\e^{t\rmD^{\alpha,m}}u_0^\eps\|_{\mcB_p^\gamma(\T^d)}\,.
    \end{equs}
    On $[0,T]$, the above is bounded by $2\|u_0^\eps\|_{\mcB_r^{-\sigma}(\T^d)}$. Iterating this on the following intervals yields the result with $c=\log(2)/T$. Remark that $c$ blows up as $\kappa$ (equivalently $\eta$) goes to $0$.
\end{proof}

\subsection{The fractional Laplacian in infinite volume}\label{sec:infinitevol}

We turn to the construction of the solution to the SPDE \eqref{renormSPDE} on the full space $\R^d$. The main difference with the finite volume is the presence of weights. Indeed, the lift of the noise can only be controlled with an arbitrarily small polynomial weight. Its presence forces us to control the solution with exponential weights defined in \eqref{expweight}. Because of this, we cannot construct the solution as a series as we did in the previous section, and therefore we are not able to derive bounds as subtle as the one in Lemma \ref{lemmafixedpointplanted}. Instead, we use a fixed point with a remainder. We assume that $(\alpha,m)$ correspond to the Laplacian or the massive fractional Laplacian, $\ell\in\R$ is fixed, $\kappa,\varsigma>0$ are fixed small enough, and we just write $\talloblong\upxi\talloblong$ instead of $\talloblong\upxi^{\alpha,m}\talloblong_{\kappa_0,\ell,\varsigma}$.

Recall the definition of the truncated force
\begin{equs}
F_{\eps,\mu}[v](z)&=
F_{\eps,\mu}^0(z)+F_{\eps,\mu}^1[v](z)=
\sum_{n=0}^{N_0}\upzeta_{\eps,\mu}^{0,n}(z)+\sum_{n=0}^{N_0}\upzeta_{\eps,\mu}^{1,n}[v](z)
\,,
\end{equs}
and of the cross-term
\begin{equs}
H_{\eps,\mu}[v](z)&=
H_{\eps,\mu}^0(z)+H_{\eps,\mu}^1[v](z)\\&=
\sum_{n=N_0+1}^{2N_0+1}\sum_{i=0}^{n-1}\upzeta_{\eps,\mu}^{1,i}[\dot G_\mu \upzeta_{\eps,\mu}^{0,n-1-i}(\bigcdot)](z)+\sum_{n=N_0+1}^{2N_0+1}\sum_{i=0}^{n-1}\upzeta_{\eps,\mu}^{1,i}[\dot G_\mu \upzeta_{\eps,\mu}^{1,n-1-i}[v](\bigcdot)](z)
\,.
\end{equs}
{Finally, in this section, we denote by $\| \bigcdot\|_{\mu,\kappa,\ell}$ the usual norm but with all the times restricted to $[0,T]$. Equivalently, we view the force coefficients $\upzeta_\epmu^{k,n}$ as supported on $[0,T]\times\R^d$ only. 
}


\begin{lemma}
Fix a deterministic combinatorial constant $0<T_0\leq1$ (independent from the initial condition) to be taken small enough, and define the random time horizon
\begin{equs}
    T\eqdef\big(\talloblong\upxi\talloblong+T_0^{-\kappa/\alpha}\big)^{-\alpha/\kappa}\,.
\end{equs}
Then, uniformly in $\eps\in(0,1]$ and $\mu\in(0,T^{1/\alpha}]$, it holds
\begin{equs}
\label{eq:esti_force_infvol}\|F^1_\epmu\|_{\mu,\kappa,\ell}\lesssim\talloblong\upxi\talloblong
    \mu^{-\alpha+3\kappa}
\;\text{and}\;
    \| \big\langle x\rangle^{-\kappa}\mcb e_{-\ell}K_\mu F_\epmu^0\|_{L^{\infty,r}_{t,x}(\Lambda_T)}\lesssim \talloblong\upxi\talloblong\|\mcb e_{-\ell} u_0^\eps\|_{\mcB^{-\sigma}_{r,q}} \mu^{-\alpha-\sigma+\beta-2\kappa}\,,\qquad
\end{equs}
alongside
\begin{equs}
\label{eq:H1}\|H^1_\epmu\|_{\mu,\kappa,\ell}\lesssim\talloblong\upxi\talloblong^{(N_0+1)}
    \mu^{-1+(N_0+2)\kappa}\,,
\end{equs}
and
    \begin{equs}\label{eq:H0}
          \| \big\langle {x}\rangle^{-\kappa}\mcb e_{-\ell}K_\mu H_\epmu^0\|_{L^{\infty,r}_{t,x}(\Lambda_T)}\lesssim \talloblong\upxi\talloblong^{N_0}\|\mcb e_{-\ell}u_0^\eps\|_{\mcB^{-\sigma}_{r,q}} \mu^{-1+(N_0+2)\kappa }\,.
\end{equs}
\end{lemma}

\begin{proof}
The estimates on the truncated force and the cross-term are direct from the estimates on $\upzeta_\epmu^{k,n}$ and Lemma \ref{estimdotGmu}. It remains to justify the power counting. For the truncated force, as in \eqref{eq:esti_force}, $F^1_\epmu$ and $F^0_\epmu$ respectively have power counting $-\alpha+\beta-\kappa$ and $-\alpha-\sigma+\beta-2\kappa$ and taking $\kappa$ small enough both quantities can be made greater than $-\alpha+2\kappa$ (recall that $\sigma<\beta$ by assumption). Regarding the cross-term, by definition, $H^k_\epmu$ has power counting
    \begin{equs}
        -1+\deg(k,N_0+1)-\kappa
    \end{equs}
    and, in view of our choice of $N_0$, for both $k=0,1$ it holds $\deg(k,N_0+1)>0$ so again taking $\kappa$ small enough one has 
      \begin{equs}
        -1+\deg(k,N_0+1)-\kappa\geq 1+(N_0+2)\kappa\,.
    \end{equs}
    The statement follows observing that for both the truncated force and the cross-term, higher order terms come with higher powers of $\talloblong\upxi\talloblong$ but, by subcriticality, they also come with the same higher power of $\mu^\kappa\leqslant T^{\kappa/\alpha}$, which is sufficient to absorb the powers of $\talloblong\upxi\talloblong$.
\end{proof}

We recall that we introduced in \eqref{eq:defR} a trajectory of \textit{remainders} $(R_\epmu)_{\mu\in(0,T^{1/\alpha}]}$ probing the error resulting from the truncation of the effective force:
\begin{equs}
    R_\epmu=F_{\eps}[v_\eps]-F_\epmu[v_\epmu]\,.
\end{equs}
We rewrite the dynamical system \eqref{eq:flowsys} with the convenient notations that we introduced in this section. Further simplifications also appear from the fact that our SPDE of interest is linear. Since $\D_\mu v_\epmu=\dot G_\mu F_{\eps,0}[v_\eps]=\dot G_\mu \big(F_{\eps,\mu}[v_\epmu]+R_\epmu\big)$, the remainder solves the flow equation 
    \begin{equs}
      \D_\mu  R_\epmu=-H^0_\epmu 
      -H^1_\epmu [v_\epmu]
     -F_\epmu^1[\dot G_\mu R_\epmu]\,,
    \end{equs}
with null initial condition. This motivates the introduction of the map $\Psi_\eps$ acting on pairs $(v,R)=(v_\mu,R_\mu)_{\mu\in(0,T^{1/\alpha}]}$ and given by   
\begin{subequations}
\begin{empheq}[left=\empheqlbrace]{alignat=2}\label{sysflow1}
  \Psi^v_\epmu(v,R)&=-\int_\mu^{T^{1/\alpha}}\dot G_\nu\big(
  F^0_\epnu 
      +F^1_\epnu [v_\nu]
     + R_\nu
  \big)\dint\nu\,,
    \\ \Psi^R_\epmu(v,R)&=-\int_0^\mu\big(H^0_\epnu 
      +H^1_\epnu [v_\nu]
     +F_\epnu^1[\dot G_\nu R_\nu]\big)\dint\nu\,.
\end{empheq}
\end{subequations}
We show that we can close a fixed point argument for this map using the following norm
\begin{equs}
    \snnorm{v,R}\eqdef \|\mcb e_{-\ell}u_0^\eps\|_{\mcB_{r,q}^{-\sigma}}^{-1}\sup_{\mu\in(0,T^{1/\alpha}]} \sup_{t\leq T}\Big( \|\mcb  e_{-\ell-t}\mcR_\mu v_\mu(t)\|_{{L^r_x}}\vee\mu^{-\kappa}  \|\mcb  e_{-\ell-t} \big\langle\bigcdot \rangle^{-\kappa}K_\mu R_\mu(t)\|_{L^r_x} \Big) \,.
\end{equs}

\begin{lemma}\label{Lemmafixedpoint}
    Choosing $T=\big(\talloblong\upxi\talloblong+T_0^{-\kappa/\alpha}\big)^{-\alpha/\kappa}$, the map $\Psi_\eps$ is a contraction for the norm defined above, and thus admits a fixed point $(v,R)$ lying in the unit ball. The fixed point is continuous in the data, namely the random force coefficients and the initial data.
\end{lemma}

\begin{proof}
    Consider a couple $(v,R)$ such that $\snnorm{v,R}\leq1$. Let $\mu\in(0,T^{1/\alpha}]$ and $t\leq T$. Let us start with the first component. We get
    \begin{equs}
        \|\mcb  e_{-\ell-t}\mcR_\mu v_\mu(t)\|_{{L^r_x}}&\leq \int_\mu^{T^{1/\alpha}}\|\mcb e_{-\ell-t}\mcR_\nu^2\dot G_\nu\big( K_\nu F_{\eps,\nu}^0+K_\nu F_{\eps,\nu}^1[K_\nu\mcR_\nu v_\nu]+K_\nu R_\nu\big)(t)\|_{{L^r_x}}\dint\nu\,.
    \end{equs}
    We have three terms to estimate, all of them requiring to use the estimates of Lemmas \ref{HaiLabestim} and \ref{estimdotGmu}. For the first one we have,
    \begin{equs}
        ~&(\mcb e_{-\ell-t}\mcR^2_\nu\dot G_\nu K_\nu F_{\eps,\nu}^0)(t,x)\\
        &=\int_{\R^d}\int_0^t\mcb e_{-\ell-t}(x)\big(\mcR^2_\nu\dot G_\nu\big)(t-s,x-y)\big(K_\nu F_{\eps,\nu}^0\big)(s,y)\dint s\dint y\\
        &\lesssim\int_{\R^d}\int_0^t\mcb e_{-\ell-t}(x)\big(\mcR^2_\nu\dot G_\nu\big)(t-s,x-y)\langle y\rangle^{\kappa}\mcb e_{\ell+s}(y)\langle y\rangle^{-\kappa}\mcb e_{-\ell}(y)\big(K_\nu F_{\eps,\nu}^0\big)(s,y)\dint s\dint y\\
        &\lesssim\int_{\R^d}\int_0^t\mcb e_{-\ell-t}(x)(t-s)^{-\kappa}\big(\mcR^2_\nu\dot G_\nu\big)(t-s,x-y)\mcb e_{\ell+t}(y)\langle y\rangle^{-\kappa}\mcb e_{-\ell}(y)\big(K_\nu F_{\eps,\nu}^0\big)(s,y)\dint s\dint y\\
        &\lesssim\int_{\R^d}\int_0^t(t-s)^{-\kappa}\big(\mcR^2_\nu\dot G_\nu\big)(t-s,x-y)\mcb e_{\ell+t}(x;y)\langle y\rangle^{-\kappa}\mcb e_{-\ell}(y)\big(K_\nu F_{\eps,\nu}^0\big)(s,y)\dint s\dint y\,.
    \end{equs}
    It gives the bound
    \begin{equs}
        \big\|\mcb e_{-\ell-t}(\mcR_\nu^2\dot G_\nu K_\nu F_{\eps,\nu}^0)(t)\big\|_{L_x^r}\lesssim\nu^{-1+\alpha-\kappa}\big\|\langle \bigcdot\rangle^{-\kappa}\mcb e_{-\ell}(K_\nu F_\epnu^0)(t)\big\|_{L_x^{r}}\,.
    \end{equs}
    We now turn to the second one. We write
    \begin{equs}
        ~&\big(\mcb e_{-\ell-t}\mcR_\nu^2\dot G_\nu K_\nu F_{\eps,\nu}^1[K_\nu\mcR_\nu v_\nu]\big)(t,x)\\
        &=\int_{\R^d}\int_0^t\mcb e_{-\ell-t}(x)\big(\mcR^2_\nu\dot G_\nu\big)(t-s,x-y)\int_{\R^d}\int_0^s\big(K_\nu^{\otimes2}F_{\eps,\nu}^1\big)(s,y;r,v)\big(\mcR_\nu v_\nu\big)(r,v)\dint r\dint v\dint s\dint y\\
        &\lesssim\int_{\R^d}\int_0^t\mcb e_{-\ell-t}(x)\big(\mcR^2_\nu\dot G_\nu\big)(t-s,x-y)\\
        &\qquad\qquad\qquad\times\int_{\R^d}\int_0^s\langle y\rangle^{\kappa}\langle y\rangle^{-\kappa}\big(K_\nu^{\otimes2}F_{\eps,\nu}^1\big)(s,y;r,v)\mcb e_{\ell+s}(v)\mcb e_{-\ell-r}(v)\big(\mcR_\nu v_\nu\big)(r,v)\dint r\dint v\dint s\dint y\\
        &\lesssim\int_{\R^d}\int_0^t\mcb e_{-\ell-t}(x)\big(\mcR^2_\nu\dot G_\nu\big)(t-s,x-y)\int_{\R^d}\int_0^s e_{\ell+s}(y)\langle y\rangle^{\kappa}\\
        &\qquad\qquad\qquad\qquad\qquad\times\langle y\rangle^{-\kappa}\mcb e_{\ell+s}(y;v)\big(K_\nu^{\otimes2}F_{\eps,\nu}^1\big)(s,y;r,v)\mcb \mcb e_{-\ell-r}(v)\big(\mcR_\nu v_\nu\big)(r,v)\dint r\dint v\dint s\dint y\\
        &\lesssim\int_{\R^d}\int_0^t\mcb e_{\ell+t}(x;y)(t-s)^{-\kappa}\big(\mcR^2_\nu\dot G_\nu\big)(t-s,x-y) \\
        &\qquad\qquad\qquad\quad\times\int_{\R^d}\int_0^s\mcb e_{\ell+t}(y;v)\langle y\rangle^{-\kappa}\big(K_\nu^{\otimes2}F_{\eps,\nu}^1\big)(s,y;r,v)\mcb \mcb e_{-\ell-r}(v)\big(\mcR_\nu v_\nu\big)(r,v)\dint r\dint v\dint s\dint y\\
    \end{equs}
    This is enough to get
    \begin{equs}
        \big\|\mcb e_{-\ell-t}\mcR_\nu^2\dot G_\nu K_\nu F_{\eps,\nu}^1[K_\nu\mcR_\nu v_\nu](t)\big\|_{L^r_x}\lesssim\nu^{-1+\alpha}\|F_{\eps,\nu}^1\|_{\nu,\kappa,\ell+t}\|\mcb e_{-\ell}u_0^\eps\|_{\mcB_{r,q}^{-\sigma}}\,.
    \end{equs}
    We treat the third term.
    \begin{equs}
        ~&(\mcb e_{-\ell-t}\mcR_\nu^2\dot G_\nu K_\nu R_\nu)(t,x)\\
        &=\int_{\R^d}\int_0^t\mcb e_{-\ell-t}(x)\mcR_\nu^2\dot G_\nu(t-s,x-y)(K_\nu R_\nu)(s,y)\dint s\dint y\\
        &=\int_{\R^d}\int_0^t\mcb e_{-\ell-t}(x)\mcR_\nu^2\dot G_\nu(t-s,x-y)\mcb e_{\ell+s}(y)\langle y\rangle^{\kappa}\mcb e_{-\ell-s}(y)\langle y\rangle^{-\kappa}(K_\nu R_\nu)(s,y)\dint s\dint y\\
        &\lesssim\int_{\R^d}\int_0^t(t-s)^{-\kappa}\mcR_\nu^2\dot G_\nu(t-s,x-y)\mcb e_{\ell+t}(x-y)\mcb e_{-\ell-s}(y)\langle y\rangle^{-\kappa}(K_\nu R_\nu)(s,y)\dint s\dint y\,.
    \end{equs}
    Using Young's convolution inequality as well as the definition of $\snnorm{\bigcdot,\bigcdot}$, we get
    \begin{equs}
        \big\|\mcb e_{-\ell-t}\mcR_\nu^2\dot G_\nu K_\nu R_\nu(t)\big\|_{{L^r_x}}&\leq t^{-\kappa}\big\|\mcb e_{\ell+t}\mcR_\nu^2 \dot G_\nu\big\|_{L^1_x}\sup_{t\leq T}\|\mcb  e_{-\ell-t} \big\langle\bigcdot \rangle^{-\kappa}K_\nu R_\nu(t)\|_{L^r_x}\\
        &\leq \nu^{-1+\alpha-\kappa}\nu^\kappa\|\mcb e_{-\ell}u_0^\eps\|_{\mcB_{r,q}^{-\sigma}}\\
        &=\nu^{-1+\alpha}\|\mcb e_{-\ell}u_0^\eps\|_{\mcB_{r,q}^{-\sigma}}\,.
    \end{equs}
    Using the estimates \eqref{eq:esti_force_infvol}, we finally get
    \begin{equs}
        \|\mcb  e_{-\ell-t}\mcR_\mu v_\mu(t)\|_{L^r_x}&\lesssim\int_\mu^{T^{1/\alpha}}\big(\nu^{-1+2\kappa}+\nu^{-1+\alpha}\big)\dint\nu\talloblong\upxi\talloblong\snnorm{v,R} \|\mcb e_{-\ell}u_0^\eps\|_{\mcB_{r,q}^{-\sigma}}\\
        &\lesssim\int_0^{T^{1/\alpha}} \nu^{-1+2\kappa}\dint\nu\talloblong\upxi\talloblong\|\mcb e_{-\ell}u_0^\eps\|_{\mcB_{r,q}^{-\sigma}}\\
        &\lesssim T^{2\kappa/\alpha}\talloblong\upxi\talloblong \|\mcb e_{-\ell}u_0^\eps\|_{\mcB_{r,q}^{-\sigma}}\leqslant \|\mcb e_{-\ell}u_0^\eps\|_{\mcB_{r,q}^{-\sigma}}\,.
    \end{equs}
    In the last equality, we used one power of $T^{\kappa/\alpha}$ to absorb the growth of the enhancement of the noise and one to absorb the implicit constant.

    We now turn to the second component of the map $\Psi_\eps$. We have
    \begin{multline*}
        \big\|\mcb e_{-\ell-t}\langle\bigcdot\rangle^{-\kappa}K_\mu R_\mu(t)\big\|_{L_x^r}\\\leq\int_0^\mu\big\|\mcb e_{-\ell-t}\langle\bigcdot\rangle^{-\kappa}K_\nu\big( H^0_{\eps,\nu}+ H_{\eps,\nu}^1[K_\nu\mcR_\nu v_\nu]+ F_{\eps,\nu}^1[K_\nu\mcR^2_\nu \dot G_\nu K_\nu R_\nu]\big)(t)\big\|_{L_x^r}\dint\nu\,. 
    \end{multline*}
    We again have three terms to control. The first comes directly from the estimate \eqref{eq:H0}. However, the other two are quite subtle. For the second one, we have,
    \begin{equs}
        ~&\mcb e_{-\ell-t}(x)\langle x\rangle^{-\kappa} K_\nu H_{\eps,\nu}^1[K_\nu\mcR_\nu v_\nu](t,x)\\
        &=\int_{\R^d}\int_0^t\mcb e_{-\ell-t}(x)\langle x\rangle^{-\kappa}(K_\nu^{\otimes2}H_{\eps,\nu}^1)(t,x;s,y)\mcb e_{\ell+s}(y)\mcb e_{-\ell-s}(y)(\mcR_\nu v_\nu)(s,y)\dint s\dint y\\
        &\lesssim\int_{\R^d}\int_0^t\mcb \langle x\rangle^{-\kappa}\mcb e_{\ell+t}(x;y)(K_\nu^{\otimes2}H_{\eps,\nu}^1)(t,x;s,y)\mcb e_{-\ell-s}(y)(\mcR_\nu v_\nu)(s,y)\dint s\dint y\\
        &\lesssim\int_{\R^d}\int_0^t\mcb \langle x\rangle^{-\kappa}\mcb e_{\ell+T}(x;y)(K_\nu^{\otimes2}H_{\eps,\nu}^1)(t,x;s,y)\mcb e_{-\ell-s}(y)(\mcR_\nu v_\nu)(s,y)\dint s\dint y\,.
    \end{equs}
    It allows to get
    \begin{equs}
        \big\|\mcb e_{-\ell-t}\langle\bigcdot\rangle^{-\kappa} K_\nu H_{\eps,\nu}^1[K_\nu\mcR_\nu v_\nu](t)\big\|_{L_x^r}\lesssim\nu^\kappa\|H_{\eps,\nu}^1\|_{\nu,\kappa,\ell+T}\|\mcb e_{-\ell}u_0^\eps\|_{\mcB_{r,q}^{-\sigma}}\,.
    \end{equs}
    Finally, for the last term, we crucially use again Lemma \ref{HaiLabestim}. It yields
    \begin{equs}
        ~&\big(\mcb e_{-\ell-t}(x)\langle\bigcdot\rangle^{-\kappa}K_\nu F_{\eps,\nu}^1[K_\nu\mcR^2_\nu \dot G_\nu K_\nu R_\nu]\big)(t,x)\\
        &=\int_{\R^d}\int_0^t \mcb e_{-\ell-t}(x)\langle x\rangle^{-\kappa}\big(K_\nu^{\otimes2}F_{\eps,\nu}^1\big)(t,x;s,y)\\
        &\qquad\qquad\qquad\qquad\qquad\qquad\qquad\qquad\times\int_{\R^d}\int_0^s\big(\mcR_\nu^2\dot G_\nu\big)(s-r,y-v)\big(K_\nu R_\nu\big)(r,v)\dint r\dint v\dint s\dint y\\
        &\lesssim\int_{\R^d}\int_0^t \mcb e_{-\ell-t}(x)\langle x\rangle^{-\kappa}\big(K_\nu^{\otimes2}F_{\eps,\nu}^1\big)(t,x;s,y)\\
        &\qquad\quad\times\int_{\R^d}\int_0^s(s-r)^{-\kappa}\big(\mcR_\nu^2\dot G_\nu\big)(s-r,y-v)\mcb e_{\ell+s}(v)\mcb e_{-\ell-r}(v)\langle v\rangle^{-\kappa}\big(K_\nu R_\nu\big)(r,v)\dint r\dint v\dint s\dint y\\
        &\lesssim\int_{\R^d}\int_0^t \mcb e_{-\ell-t}(x)\mcb e_{\ell+s}(y)\langle x\rangle^{-\kappa}\big(K_\nu^{\otimes2}F_{\eps,\nu}^1\big)(t,x;s,y)\\
        &\qquad\times\int_{\R^d}\int_0^s(s-r)^{-\kappa}\big(\mcR_\nu^2\dot G_\nu\big)(s-r,y-v)\mcb e_{\ell+s}(y;v)\mcb e_{-\ell-r}(v)\langle v\rangle^{-\kappa}\big(K_\nu R_\nu\big)(r,v)\dint r\dint v\dint s\dint y\\
        &\lesssim\int_{\R^d}\int_0^t \langle x\rangle^{-\kappa}\mcb e_{\ell+T}(x;y)\big(K_\nu^{\otimes2}F_{\eps,\nu}^1\big)(t,x;s,y)\\
        &\qquad\times\int_{\R^d}\int_0^s(s-r)^{-\kappa}\big(\mcR_\nu^2\dot G_\nu\big)(s-r,y-v)\mcb e_{\ell+s}(y;v)\mcb e_{-\ell-r}(v)\langle v\rangle^{-\kappa}\big(K_\nu R_\nu\big)(r,v)\dint r\dint v\dint s\dint y\,.
    \end{equs}
    It allows to write
    \begin{equs}
        \big\|\mcb e_{-\ell-t}\langle\bigcdot\rangle^{-\kappa}K_\nu F_{\eps,\nu}^1[K_\nu\mcR^2_\nu \dot G_\nu K_\nu R_\nu](t)\big\|_{L^r_x}\lesssim\nu^{-1+\alpha-\kappa}\nu^\kappa\|F_{\eps,\nu}^1\|_{\nu,\kappa,\ell+T}\|\mcb e_{-\ell}u_0^\eps\|_{\mcB_{r,q}^{-\sigma}}\,.
    \end{equs}
    Putting everything together, we get, using the estimates \eqref{eq:esti_force_infvol}, \eqref{eq:H1}, and \eqref{eq:H0}
    \begin{equs}
        ~&\big\|\mcb e_{-\ell-t}\langle\bigcdot\rangle^{-\kappa}(K_\mu R_\mu)(t)\big\|_{L_x^r}\\&\lesssim\int_0^\mu\big(\nu^{-1+(N_0+2)\kappa}\talloblong\upxi\talloblong^{N_0}+\nu^{-1+(N_0+3)\kappa}\talloblong\upxi\talloblong^{N_0+1}+\nu^{-1+3\kappa}\talloblong\upxi\talloblong\big)\dint\nu\|\mcb e_{-\ell}u_0^\eps\|_{\mcB_{r,q}^{-\sigma}}\,.
        \end{equs}
        Absorbing all the powers of $\talloblong\upxi\talloblong$ with $\nu^\kappa\leqslant T^{\kappa/\alpha}$ thus gives
        \begin{equs}
             ~&\big\|\mcb e_{-\ell-t}\langle\bigcdot\rangle^{-\kappa}(K_\mu R_\mu)(t)\big\|_{L_x^r}\lesssim\mu^{2\kappa}\|\mcb e_{-\ell}u_0^\eps\|_{\mcB_{r,q}^{-\sigma}}
\leqslant\mu^{\kappa}\|\mcb e_{-\ell}u_0^\eps\|_{\mcB_{r,q}^{-\sigma}}\,,
    \end{equs}
    where we used a last power of $\mu^\kappa\leqslant T^{\kappa/\alpha}$ to absorb the implicit constant.
\end{proof}

We now turn to the actual construction of the solution. For this, we need first the crucial lemma.
\begin{lemma}
    Let $\eps>0$ and $T\in(0,1]$. If $u_\eps$ solves \eqref{mildform} on $\Lambda_{0;T}$, then we have
    \begin{equation}
        u_\eps=G\psi_\eps+(G-G_{T^{1/\alpha}})(F_{\eps,T^{1/\alpha}}[0]+R_{\eps,T^{1/\alpha}})\,.
    \end{equation}
\end{lemma}

\begin{proof}
    Writing \eqref{mildformF} at $\mu=T^{1/\alpha}$, we get
    \begin{equs}
        u_\eps=G\psi_\eps+(G-G_{T^{1/\alpha}})(F_{\eps,T^{1/\alpha}}[v_{\eps,T^{1/\alpha}}]+R_{\eps,T^{1/\alpha}})\,.
    \end{equs}
    Using the support properties of $G_{T^{1/\alpha}}$ given in \eqref{cutscale}, we get that $v_{\eps,T^{1/\alpha}}$ is supported outside of $\Lambda_{0;T}$ and whence the result.
\end{proof}

\begin{lemma}\label{lemma:constructsolinf}
    There is a random $T\in(0,1]$ such that, for $\eta\in\big(0,\beta-\sigma-d(1/r-1/p)\big)$,
    \begin{equation}
        \|u_\eps-G\psi_\eps\|_{\mcE_{p,q,\ell,T}^{\gamma,\beta-\sigma-\gamma-d(1/r-1/p)-\eta}}\leq\|\mcb e_{-\ell}u_0^\eps\|_{\mcB_{r,q}^{-\sigma}}\,,
    \end{equation}
    for some small $\eta$.
\end{lemma}

\begin{proof}
    We follow quite closely the proof of Lemma 3.10 in \cite{CF24a}. Let us write
    \begin{equs}
        u_\eps-G\psi_\eps=-\int_0^{T^{1/\alpha}}\dot G_\nu (F_{\eps,T^{1/\alpha}}[0]+R_{\eps,T^{1/\alpha}})\dint\nu\,.
    \end{equs}
    Restating the definition of the considered norm
    \begin{equs}
        \|v_\eps\|_{\mcE^{\gamma,\delta}_{p,q,\ell,T}}=\sup_{t\leq T}t^{-\delta/\alpha}\left(\big\|\mcb e_{-\ell-t}v_\eps(t)\big\|_{L_x^p}+\left(\int_0^1\Big(\mu^{-\gamma}\big\| \big(\big(Q_\mu-\mathrm{Id}\big)\mcb e_{-\ell-t} v_\eps\big)(t)\big\|_{L^p_x}\Big)^q\frac{\dint\mu}{\mu}\right)^\frac{1}{q}\right)
    \end{equs}
    we get, in the case $\mu\leq T^{1/\alpha}$,
    \begin{equs}
        ~&\big\|\big(\big(Q_\mu-\mathrm{Id}\big)\mcb e_{-\ell-t} v_\eps\big)(t)\big\|_{L^p_{x}}\\
        &\leq\int_0^{T^{1/\alpha}}\big\|(Q_\mu-\mathrm{Id})\mcb e_{-\ell-t} \dot G_\nu (F_{\eps,T^{1/\alpha}}[0]+R_{\eps,T^{1/\alpha}})(t)\big\|_{L^{p}_{x}}\dint\nu\\
        &\leq\int_0^\mu \big\|(Q_\mu-\mathrm{Id})\mcb e_{-\ell-t}\dot G_\nu (F_{\eps,T^{1/\alpha}}[0]+R_{\eps,T^{1/\alpha}})(t)\big\|_{L^{p}_{x}}\dint\nu\\
        &\quad+\int_\mu^{T^{1/\alpha}}\big\|(Q_\mu-\mathrm{Id})\mcb e_{-\ell-t}Q_\nu \mcP_\nu\dot G_\nu (F_{\eps,T^{1/\alpha}}[0]+R_{\eps,T^{1/\alpha}})(t)\big\|_{L^{p}_{x}}\dint\nu\,.
    \end{equs}
    Once here, we can distribute the exponential weights as in proof of Lemma \ref{Lemmafixedpoint}. We perform this operation in details for the first term and leave the second one to the reader. We have
    \begin{equs}
        ~&((Q_\mu-\mathrm{Id})\mcb e_{-\ell-t}\dot G_\nu (F_{\eps,T^{1/\alpha}}[0]+R_{\eps,T^{1/\alpha}}))(t,x)\\
        &=\int_{\R^d} \mcb (Q_\mu-\delta)(x-y)e_{-\ell-t}(y)\int_0^t(\mcR_\nu^2\dot G_\nu)(t-s,y-v)(K_\nu(F_{\eps,T^{1/\alpha}}[0]+R_{\eps,T^{1/\alpha}}))(s,v)\dint s\dint v\dint y\\
        &=\int_{\R^d} (Q_\mu-\delta)(x-y)\mcb e_{-\ell-t}(y)\times\\
        &\qquad\int_0^t(\mcR_\nu^2\dot G_\nu)(t-s,y-v)\langle v\rangle^{\kappa}\mcb e_{\ell+s}(v)\langle v\rangle^{-\kappa}\mcb e_{-\ell-s}(v)(K_\nu(F_{\eps,T^{1/\alpha}}[0]+R_{\eps,T^{1/\alpha}}))(s,v)\dint s\dint v\dint y\\
        &\lesssim\int_{\R^d} (Q_\mu-\delta)(x-y)\mcb e_{-\ell-t}(y)\int_0^t(t-s)^{-\kappa}(\mcR_\nu^2\dot G_\nu)(t-s,y-v)\times\\
        &\qquad\mcb e_{\ell+s}(v)\langle v\rangle^{-\kappa}\mcb e_{-\ell-s}(v)(K_\nu(F_{\eps,T^{1/\alpha}}[0]+R_{\eps,T^{1/\alpha}}))(s,v)\dint s\dint v\dint y\\
        &\lesssim\int_{\R^d}(Q_\mu-\delta)(x-y)\int_0^t(t-s)^{-\kappa}(\mcR_\nu^2\dot G_\nu)(t-s,y-v)\mcb e_{\ell+t}(y;v)\times\\
        &\qquad\langle v\rangle^{-\kappa}\mcb e_{-\ell-s}(v)(K_\nu(F_{\eps,T^{1/\alpha}}[0]+R_{\eps,T^{1/\alpha}}))(s,v)\dint s\dint v\dint y\,.
    \end{equs}
    Using \eqref{Kmuexp}, \eqref{GmuL1weight}, as well as Lemma \ref{Lemmarecurrence} and \ref{HaiLabestim}, it yields
    \begin{equs}
        \big\|(Q_\mu-\mathrm{Id})\mcb e_{-\ell-t}\dot G_\nu (F_{\eps,T^{1/\alpha}}[0]+R_{\eps,T^{1/\alpha}})(t)\big\|_{L^p_x}\lesssim\mathbf{1}_{\nu\leq t^{1/\alpha}}\nu^{-1-\sigma-\gamma-d(1/r-1/p)-2\kappa}\|\mcb e_{-\ell}u_0^\eps\|_{\mcB_{r,q}^{-\sigma}(\R^d)}\,.
    \end{equs}
    We then have
    \begin{equs}
        ~&\mu^{-\gamma}\int_0^\mu \big\|(Q_\mu-\mathrm{Id})\mcb e_{-\ell-t}\dot G_\nu (F_{\eps,T^{1/\alpha}}[0]+R_{\eps,T^{1/\alpha}})(t)\big\|_{L^{p}_{x}}\dint\nu\\
        &\lesssim\int_0^\mu \nu^{-\gamma}\big\|(Q_\mu-\mathrm{Id})\mcb e_{-\ell-t}\dot G_\nu (F_{\eps,T^{1/\alpha}}[0]+R_{\eps,T^{1/\alpha}})(t)\big\|_{L^{p}_{x}}\dint\nu\\
        &\lesssim\int_0^{\mu\wedge t^{1/\alpha}} \nu^{-1+\beta-\sigma-\gamma-d(1/r-1/p)-2\kappa}\dint\nu\|\mcb e_{-\ell}u_0^\eps\|_{\mcB_{r,q}^{-\sigma}}\\
        &\lesssim t^{(\beta-\sigma-\gamma-d(1/r-1/p)-3\kappa)/\alpha}\int_0^{\mu\wedge t^{1/\alpha}}\nu^{-1+\kappa}\dint \nu\|\mcb e_{-\ell}u_0^\eps\|_{\mcB_{r,q}^{-\sigma}}\\
        &\lesssim t^{(\beta-\sigma-\gamma-d(1/r-1/p)-3\kappa)/\alpha}\mu^\kappa\|\mcb e_{-\ell}u_0^\eps\|_{\mcB_{r,q}^{-\sigma}}
    \end{equs}
    We have used the support properties of $\dot G_\nu$ in the convolution to assert that $\nu^\alpha\leq t$. The inhomogeneous term is easier to treat. The proof can be carried out as in the end of the proof of Lemma \ref{lemmafixedpointplanted}, with the usual integration of the exponential weights. The treatment of the other term requires to use \eqref{eq:comKmuKnu}, and the rest of the argument follows as in \cite[Lemma 3.10]{CF24a}. The case $\mu\geq T^{1/\alpha}$ is easier and dealt with similarly.
    \end{proof}

\begin{lemma}\label{Lemmarecurrence}
    Let $\tau\in(0,1)$ such that
    \begin{equation}\label{2/tau2}
        (2/\tau^\alpha-1)^{2M}\leq2\,.    
    \end{equation}
    Then, there is a random $T\in(0,1]$ such that, for every $p\in\N$,
    \begin{equation}\label{HR}
        \sup_{t\leq T}\big\|\langle\bigcdot\rangle^{-\kappa}\mcb e_{-\ell-t}K_{\tau^pT^{1/\alpha}} (F_{\eps,T^{1/\alpha}}[0]+R_{\eps,T^{1/\alpha}})(t)\big\|_{L^{r}_{x}}\leq C(\tau^pT^{1/\alpha})^{-\alpha-\sigma+\beta-\kappa}\|\mcb e_{-\ell}u_0^\eps\|_{\mcB_{r,q}^{-\sigma}}\,.
    \end{equation}
    Consequently, we get
    \begin{equation}
        \sup_{t\leq T}\big\|\langle\bigcdot\rangle^{-\kappa}\mcb e_{-\ell-t}K_\mu (F_{\eps,T^{1/\alpha}}[0]+R_{\eps,T^{1/\alpha}})(t)\big\|_{L^{r}_{x}}\lesssim\mu^{-\alpha-\sigma+\beta-\kappa}\|\mcb e_{-\ell}u_0^\eps\|_{\mcB_{r,q}^{-\sigma}}\,.
    \end{equation}
\end{lemma}
\begin{proof}
    The proof is very similar to the proof of Lemma 3.12 in \cite{CF24a}, with the difference that we need to integrate the weights. We proceed by induction on $p$. The base case follows directly from \eqref{eq:esti_force_infvol} and Lemma \ref{Lemmafixedpoint}, and reads
    \begin{equs}
        \sup_{t\leq T}\big\|\langle\bigcdot\rangle^{-\kappa}\mcb e_{-\ell-t}K_{T^{1/\alpha}} (F_{\eps,T^{1/\alpha}}[0]+R_{\eps,T^{1/\alpha}})(t)\big\|_{L^{r}_{x}}\lesssim T^{(-\alpha+\beta-\kappa)/\alpha}\|\mcb e_{-\ell}u_0^\eps\|_{\mcB_{r,q}^{-\sigma}}\,.
    \end{equs}
    Now suppose that \eqref{HR} holds for some $p\in\N$. To go from $p$ to $p+1$, some care has to be taken with the weights. Recall that we have, $K_{\tau^{p+1}T^{1/\alpha}}=\mcR_{\tau^{p}T^{1/\alpha}}K_{\tau^{p+1}T^{1/\alpha}}K_{\tau^{p}T^{1/\alpha}}$. Thus we get
    \begin{equs}
        ~&\langle x\rangle^{-\kappa}\mcb e_{-\ell-t}(x)K_\mu (F_{\eps,T^{1/\alpha}}[0]+R_{\eps,T^{1/\alpha}})(t,x)\\
        &=\int_{\R^d}\int_0^t\langle x\rangle^{-\kappa}\mcb e_{-\ell-t}(x)(\mcR_{\tau^{p}T^{1/\alpha}}K_{\tau^{p+1}T^{1/\alpha}})(t-s,x-y)K_{\tau^{p}T^{1/\alpha}}(F_{\eps,T^{1/\alpha}}[0]+R_{\eps,T^{1/\alpha}})(s,y)\dint s\dint y\\
        &\leq\int_{\R^d}\int_0^t\langle x\rangle^{-\kappa}\mcb e_{\ell+t}(x-y)(\mcR_{\tau^{p}T^{1/\alpha}}K_{\tau^{p+1}T^{1/\alpha}})(t-s,x-y)\\
        &\qquad\qquad\qquad\qquad\qquad\qquad\qquad\quad\times\langle y\rangle^{\kappa}\langle y\rangle^{-\kappa}\mcb e_{-\ell-s}(y)K_{\tau^{p}T^{1/\alpha}}(F_{\eps,T^{1/\alpha}}[0]+R_{\eps,T^{1/\alpha}})(s,y)\dint s\dint y\\
        &\leq\int_{\R^d}\int_0^t\langle x;y\rangle^{\kappa}\mcb e_{\ell+t}(x-y)(\mcR_{\tau^{p}T^{1/\alpha}}K_{\tau^{p+1}T^{1/\alpha}})(t-s,x-y)\\
        &\qquad\qquad\qquad\qquad\qquad\qquad\qquad\qquad~~~\times\langle y\rangle^{-\kappa}\mcb e_{-\ell-s}(y)K_{\tau^{p}T^{1/\alpha}}(F_{\eps,T^{1/\alpha}}[0]+R_{\eps,T^{1/\alpha}})(s,y)\dint s\dint y\,.
    \end{equs}
    We have used that $\langle y\rangle^\kappa\lesssim\langle y;x\rangle^\kappa\langle x\rangle^\kappa$. Then, using the condition \eqref{2/tau2} and \eqref{eq:Kmunu}, we get
    \begin{equs}
        \sup_{t\leq T}\big\|\langle\bigcdot\rangle^{-\kappa}\mcb e_{-\ell-t}K_{\tau^{p+1}T^{1/\alpha}} (F_{\eps,T^{1/\alpha}}[0]+R_{\eps,T^{1/\alpha}})(t)\big\|_{L^{r}_{x}}&\leq2C(\tau^pT^{1/\alpha})^{-\alpha+\beta-\kappa}\|\mcb e_{-\ell}u_0^\eps\|_{\mcB_{r,q}^{-\sigma}}\\
        &\leq2C(\tau^{p+1}T^{1/\alpha})^{-\alpha+\beta-\kappa}\|\mcb e_{-\ell}u_0^\eps\|_{\mcB_{r,q}^{-\sigma}}\,.
    \end{equs}
    We then write
    \begin{equs}\label{from0toT}
        F_{\eps,T^{1/\alpha}}[0]+R_{\eps,T^{1/\alpha}}=F_{\eps,\tau^{p+1}T^{1/\alpha}}[v_{\eps,\tau^{p+1}T^{1/\alpha}}]+R_{\eps,\tau^{p+1}T^{1/\alpha}}\,.
    \end{equs}
    Using \eqref{sysflow1}, we have, proceeding as in the proof of Lemma \ref{Lemmafixedpoint},
    \begin{equs}
        \sup_{t\leq T}\|\mcb e_{-\ell-t} v_{\eps,\tau^{p+1}T^{1/\alpha}}(t)\|_{L_x^r}&\leq \int_{\tau^{p+1}T^{1/\alpha}}^{T^{1/\alpha}}\sup_{t\leq T}\big\|\mcb e_{-\ell-t}\dot G_\nu(F_{\eps,T^{1/\alpha}}[0]+R_{\eps,T^{1/\alpha}})(t)\big\|_{L_x^r}\dint\nu\\
        &\leq 2CC_G\int_{\tau^{p+1}T^{1/\alpha}}^{T^{1/\alpha}}\nu^{-1+\beta-2\kappa}\dint\nu\\
        &\leq 2CC_G(\beta-\kappa)^{-1}T^{(\beta-2\kappa)/\alpha}\,,
    \end{equs}
    where $C_G$ is the implicit constant in Lemma \ref{estimdotGmu}. By taking $T$ small enough, we can impose that this latter quantity is bounded by $1$. Taking \eqref{from0toT}, we have, using Lemma \ref{Lemmafixedpoint},
    \begin{equs}
        ~&\sup_{t\leq T}\big\|\langle\bigcdot\rangle^{-\kappa}\mcb e_{-\ell-t}K_{\tau^{p+1}T^{1/\alpha}}(F_{\eps,T^{1/\alpha}}[0]+R_{\eps,T^{1/\alpha}})(t)\big\|_{L_x^r}\\
        &\leq \sup_{t\leq T}\big\|\langle\bigcdot\rangle^{-\kappa}\mcb e_{-\ell-t}K_{\tau^{p+1}T^{1/\alpha}}F_{\eps,\tau^{p+1}T^{1/\alpha}}[v_{\eps,\tau^{p+1}T^{1/\alpha}}](t)\big\|_{L_x^r}\\
        &\quad+\sup_{t\leq T}\big\|\langle\bigcdot\rangle^{-\kappa}\mcb e_{-\ell-t}K_{\tau^{p+1}T^{1/\alpha}}R_{\eps,\tau^{p+1}T^{1/\alpha}}(t)\big\|_{L_x^r}\\
        &\leq C_1 (\tau^{p+1}T^{1/\alpha})^{-1+\beta-\kappa}+C_2(\tau^{p+1}T^{1/\alpha})^{\kappa}\,,
    \end{equs}
    where $C_1$ and $C_2$ are implicit constants given by Lemma \ref{Lemmafixedpoint}, and the optimal constant in \eqref{eq:esti_force_infvol}. Finally, we set $C=C_1+C_2$ to conclude the proof.
\end{proof}

We can now conclude the proof of the second item of Theorem \ref{maintheoremspde}.
\begin{proof}[of Theorem \ref{maintheoremspde}, 2]
    The proof follows as the one of Corollary \ref{Cor:GlobalisationSeps} with the integration of the exponential weights. The short time estimate yields, the following uniform bound on $[0,T]$ ($T$ is here the random stopping time)
    \begin{equs}
        \|\mcb e_{-t}(u(t)-\e^{-t\mathrm{D}^{\alpha,m}}u_0\|_{\mcB_{p,\infty}^\gamma(\R^d)}\leq t^{(\beta-\sigma-\gamma-d(1/r-1/p))/\alpha}\|u_0\|_{\mcB_{r,\infty}^{-\sigma}(\R^d)}\,.
    \end{equs}
    The restarted dynamics on $[NT,(N+1)T]$ has initial condition $\mcb e_{-NT} u(T)$. Contrary to the finite volume case, we are not able to track precisely the growth ($\e^{ct}$ in this latter case) as we might get additional losses when we absorb the exponential weights with de decay of heat kernel of the massive fractional Laplacian. Nonetheless, we believe that this would be possible with more work, and that one should expect super-exponential growth.
\end{proof}

\subsection{The harmonic oscillator}\label{sec:OH}

In this section, we deal with the case of the harmonic oscillator $\mathrm H$, $(\alpha,m)=(2,q)$. We denote by $M\equiv G^{2,q}$ the kernel of the semi-group $e^{-t\mathrm{H}}$ of the harmonic oscillator on $\R^d$, also known in the literature as \textit{Mehler kernel}. This kernel has an explicit expression that we recalled in \eqref{eq:Mehlerkernel}. As usual, we define
\begin{equation} M_\mu(t,x,y)=\chi_\mu(t)M(t,x,y)\,.
\end{equation}
The purpose of the present section is to construct the force coefficient $\upxi^{2,q;n}_\epmu$ of the stationary force $S^{2,q}_\epmu$. Note that while $S^{2,q}_\epmu$ is still stationary in time, working with the Mehler kernel $M$ which is not translation-invariant in space breaks stationarity in space. However, we claim that we will be able to construct $S^{2,q}_\epmu$ with the same initial condition as in the space-translation-invariant case of the Laplacian, namely with $S^{2,q}_{\eps,0}=S^{2,0}_{\eps,0}$, and using a purely deterministic argument taking the space-translation-invariant effective force $S^{2,0}_\epmu$ as an input. In other words, we will construct the force coefficients $\upxi^{2,q;n}_\epmu$ pathwise, as continuous functions of the force coefficients $\upxi^{2,0;n}_\epmu$ which we will obtain in Section~\ref{sec:stocest} in the case of the Laplacian. 

To avoid any clash and to clarify the notation, we denote by $O_\epmu\equiv S^{2,q}_\epmu$ the effective force satisfying the truncation at order $N_0$ of the flow equation
\begin{equation}\label{eq:flowmehler}
    \D_\mu O_\epmu+\mathrm{D}O_\epmu \dot M_\mu O_\epmu=0\,,
\end{equation}
started with initial condition $O_{\eps,0}=S^{2,0}_\eps$, and by $\omega_\epmu^n\equiv\upxi^{2,q;n}_\epmu$ is force coefficients, so
\begin{equation}
    O_\epmu[\phi](z)=\sum_{n=0}^{N_0}\int_\Lambda \upomega_\epmu^n(z,w)\phi(w)\dint w\,.
\end{equation}
Now that the notation is clarified, throughout the present section, we will write $S_\epmu$ and $\upxi^n_\epmu$ to denote $S^{2,0}_\epmu$ and $\upxi^{2,0;n}_\epmu$ the effective force and effective force coefficients obtained in the case of the Laplacian, namely solving the truncation at order $N_0$ of 
\begin{equation}\label{eq:flowsrecalled}
    \D_\mu S_\epmu+\mathrm{D}S_\epmu \dot G_\mu S_\epmu=0\,,
\end{equation}
where $G\equiv G^{2,0}$ denotes the heat kernel of the Laplacian on $\R^d$. Recall the definition \eqref{eq:wieghtW} of the weight $\mcb W_\mu$. We make the assumption that $O_\epmu$ and $S_\epmu$ are related by the following ansatz
\begin{equs}\label{eq:ansatzHO}
    O_{\epmu}=\mcb W_\mu S_{\epmu}+I_\epmu\,,
\end{equs}
where $I_\epmu$ is a remainder term which we call the \textit{irrelevant effective force} with force coefficients denoted $\upiota^n_\epmu$, so 
\begin{equation}
    I_\epmu[\phi](z)=\sum_{n=0}^{N_0}\int_\Lambda\upiota_\epmu^n(z,w)\phi(w)\dint w\,.
\end{equation}
The idea behind this ansatz is that, since in view of \eqref{eq:Mehlerkernel} the Mehler kernel is close to the heat kernel at short time, $O_\epmu$ should also look like $S_\epmu$ at short scales. This sole observation justifies looking for $O_\epmu$ near $S_\epmu$, but why look for it near $\mcb W_\mu S_{\epmu}$ or, in other words, why deviate from $S_\epmu$ to blow up at large scales? This is natural, considering the Mehler kernel at time $\sim\mu^2$ has such a strong decay that it can absorb a Gaussian growth at rate $\sim\mu^2$ and return the same Gaussian decay in space. Therefore, in order to construct the effective force which will feed inside the Mehler kernel, it is natural to take the effective force which feeds inside the heat kernel and to amplify it in order, eventually, to obtain two quantities that are comparable. Indeed, trying to set $O_\epmu=S_\epmu+\tilde{I}_\epmu$ would replace the kernel $\dot M_\nu-\mcb W_\nu^{-1}\dot G_\nu$ in the flow equation \eqref{eq:floweqiota} by $\dot M_\nu-\dot G_\nu$, which would considerably weaken the estimate on the difference, since the heat kernel only has Gaussian decay off the diagonal and not on the full space, as mentioned above, while the first different has rapid decay in all directions (see Lemma~\ref{lemma:diffMG}). Moreover, observe that, further splitting $ O_{\epmu}=S_{\epmu}+\tilde{I}_\epmu=\mcb W_\mu S_{\epmu}+(1-\mcb W_\mu )S_{\epmu}+\tilde{I}_\epmu$, we pull inside the irrelevant part a term proportional to $\big|1-\e^{\mu^2|y_1|^2/4}\big|\lesssim\mu^2|y_1|^2\e^{\mu^2|y_1|^2/4}$. Note that modulo a growth like $|y_1|^2\e^{\mu^2|y_1|^2/4}$ at large spatial scales, it is indeed reasonable to expect this term to be irrelevant, since its power counting in $\mu$ is increased by two. We will thus control $I_\epmu$ in a topology allowing a growth like $|y_1|^2\e^{\mu^2|y_1|^2/4}$, see \eqref{eq:HR+2} and the definition \eqref{eq:defJmu} of $J_{\mu,n}$.

At the level of the force coefficients, the ansatz \eqref{eq:ansatzHO} reads
\begin{equation}\label{eq:ansatzomega}
    \upomega^{n}_{\epmu}(z_1,z_2)=\mcb W_\mu \upxi^{n}_{\epmu}(z_1,z_2)+\upiota^n_\epmu(z_1,z_2)=\e^{\mu^2 \frac{|y_1|^2}{8}}\upxi^n_\mu(z_1,z_2)\e^{\mu^2 \frac{|y_2|^2}{8}}+\upiota^n_\mu(z_1,z_2)\,.
\end{equation}
We have denoted $y_i$ the spatial coordinate of $z_i$. Therefore, the flow equations \eqref{eq:flowsrecalled} and \eqref{eq:flowmehler} for the force $S_\epmu$ and $O_\epmu$ imply that the irrelevant force coefficients satisfy the hierarchy of flow equations
\begin{equs}\begin{aligned}\label{eq:floweqiota}
  \upiota_\epmu^n&=-\int_0^\mu\bigg(\frac{\nu}{4}(|y_1|^2+|y_2|^2)\mcb W_\nu\upxi_\nu^n+\sum_{i=0}^{n-1} \Big(\mcb W_\nu\upxi_\epnu^i[(\dot M_\nu-\mcb W_\nu^{-1}\dot G_\nu)\mcb W_\nu\upxi_\epnu^{n-1-i}]\\
    &\qquad\qquad\qquad+\mcb W_\nu\upxi_\epnu^i[\dot M_\mu\upiota_\epnu^{n-1-i}]+\upiota_\epnu^i[\dot M_\nu\mcb W_\nu\upxi_\epnu^{n-1-i}]+\upiota_\epnu^i[\dot M_\nu\upiota_\epnu^{n-1-i}]\Big)\bigg)\dint\nu\,.
    \end{aligned}
\end{equs}
The induction is started from $\upomega_\epmu^0(z_1,z_2)=\xi_\eps(z_1)\delta(z_1-z_2)$, which implies
\begin{equation}
  \upiota_\epmu^0(z_1,z_2)=\Big(1-\e^{\mu^2\frac{|y_1|^2}{4}}\Big)\xi_\eps(z_1)\delta(z_1-z_2)\,.
\end{equation}

\subsubsection{Construction of the non-translation-invariant coefficients}

For the rest of Section \ref{sec:OH}, we assume that we have the following control on the stationary force coefficients
\begin{equation}\label{eq:estimstatio}
    \sup_{w_1\in\Lambda}\int_\Lambda \langle w_1\rangle^{-\kappa_0}\langle w_1;w_2\rangle_\mu^\varsigma\big|\big(K_\mu^{\otimes2}\upxi_\epmu^n\big)(w_1,w_2)\big|\dint w_2\lesssim\talloblong\upxi\talloblong^{n+1}\mu^{\deg(n)-\kappa_0}\,,
\end{equation}
for some $\varsigma>0$ large enough and uniformly in $\epmu$. Above and until the end of the section, $\talloblong\upxi\talloblong$ is a short-hand notation for $\talloblong\upxi^{2,0}\talloblong_{\kappa_0,0,\varsigma}$. We can allow ourselves to absorb any polynomial growth of parameter $\varsigma$ thanks to the Gaussian decay of heat kernel of the Laplacian (see \eqref{eq:estimate_heatkernel_alphais2}). More rigorously, one would need to put some decreasing power of $\kappa_0$ on the diagonal weight, but since we could start from a higher power and go back to $\varsigma$ we do not develop the matter any further. We may now state the main object of this section, namely the estimate on the force coefficients of $O_\epmu$. We would like to emphasize the fact that we manage, thanks to the decay properties of the Mehler kernel to transform the logarithmic growth in space of the noise into some minor loss in $\mu$. This will allow to construct the solution \textit{without} exponential weights.

\begin{proposition}\label{prop:boundomega}
Recall that the kernel $J_{\mu,n}$ was defined in \eqref{eq:defJmu} and that $J_\mu=J_{\mu,N_0+1}$. For every $\kappa>0$ and $0\leq n\leq N_0$,
    \begin{equation}
        \label{eq:boundomega}\sup_{w_1\in\Lambda}\int_{\Lambda} \langle t_1\rangle^{-\kappa}\langle w_{1};w_2\rangle^{1-\kappa}_\mu\big|\big(J_{\mu}^{\otimes2}\upomega^n_\epmu\big)(w_1,w_2)\big|\dint w_2\lesssim
        \talloblong\upxi\talloblong^{n+1}
        \mu^{\deg(n)-\kappa}\,,
    \end{equation}
    uniformly in $\eps$ and $\mu$.
\end{proposition}

\begin{remark}
    Before turning to the proof of this proposition, we would like to point out two facts about the force coefficients $\upomega^n_\epmu$. Firstly, we control them in a topology where the weight in space is absent. In other words, the complications due to the growth of the stochastic terms in large volume that we had to tackle for the fractional Laplacian are not present here. Secondly, the right-hand side of \eqref{eq:boundomega} shows that these coefficients are built from the stationary force coefficients $\upxi^{2,0;n}_\epmu$ pathwise, which is an interesting feature of our construction. In particular, with a little more effort, we can prove that they are continuous in the stationary force so that convergence of the stationary force implies convergence of these force coefficients.
\end{remark}

\begin{proof}
    We first prove that
    \begin{equs}
        \sup_{w_1\in\Lambda}\int_{\Lambda} \langle t_1\rangle^{-\kappa}\langle w_{1};w_2\rangle^{1-\kappa}_\mu\big|\big(J_{\mu,n+1}^{\otimes2}\upomega^n_\epmu\big)(w_1,w_2)\big|\dint w_2\lesssim\talloblong\upxi\talloblong^{n+1}\mu^{\deg(n)-\kappa}\,,
    \end{equs}
    and conclude using \eqref{eq:lessJn}. For that purpose, we decompose
    \begin{equs}
        ~&\sup_{w_1\in\Lambda}\int_{\Lambda} \langle t_1\rangle^{-\kappa}\langle w_{1};w_2\rangle^{1-\kappa}_\mu\big|\big(J_{\mu,n+1}^{\otimes2}\upomega^n_\epmu\big)(w_1,w_2)\big|\dint w_2\\
        &\leq\sup_{w_1\in\Lambda}\int_{\Lambda}\langle t_1\rangle^{-\kappa} \langle w_{1};w_2\rangle^{1-\kappa}_\mu\big|\big(J_{\mu,n+1}^{\otimes2}\mcb W_\nu\upxi_\epmu^n\big)(w_1,w_2)\big|\dint w_2\\
        &\quad+\sup_{w_1\in\Lambda}\int_{\Lambda} \langle t_1\rangle^{-\kappa}\langle w_{1};w_2\rangle^{1-\kappa}_\mu\big|\big(J_{\mu,n+1}^{\otimes2}\upiota^n_\epmu\big)(w_1,w_2)\big|\dint w_2\,.
    \end{equs}
    The control of the second term is the object of Lemma \ref{lemma:induction+2}, which is the heart of the construction of the coefficients $\upiota_\epmu^n$ (and the difficult part). We are left with the first term. We have, denoting in a general way $w=(t,x)$ and $z=(s,y)$,
    \begin{equs}
        ~&\sup_{w_1\in\Lambda}\int_{\Lambda} \langle t_1\rangle^{-\kappa}\langle w_{1};w_2\rangle^{1-\kappa}_\mu\big|\big(J_{\mu,n+1}^{\otimes2}\mcb W_\nu\upxi_\epmu^n\big)(w_1,w_2)\big|\dint w_2\\
        &=\sup_{w_1\in\Lambda}\int_{\Lambda} \langle t_1\rangle^{-\kappa}\langle w_{1};w_2\rangle^{1-\kappa}_\mu\\
        &\qquad\times\bigg|\int_{\Lambda^2}K_\mu(w_1-z_1)K_\mu(w_2-z_2)\e^{-(a_{n+1}-1)\mu^2\frac{|y_1|^2+|y_2|^2}{8}}\upxi_\epmu^n(z_1,z_2)\dint z_1\dint z_2\bigg|\dint w_2\\
        &\lesssim\sup_{w_1\in\Lambda}\int_{\Lambda} \langle t_1\rangle^{-\kappa}\langle w_{1};w_2\rangle^{1-\kappa}_\mu\\
        &\qquad\times\bigg|\int_{\Lambda^2}\mcR_\mu^{\otimes2}\bigg(K_\mu(w_1-z_1)K_\mu(w_2-z_2)\e^{-(a_{n+1}-1)\mu^2\frac{|y_1|^2+|y_2|^2}{8}}\bigg)\big(K_\mu^{\otimes2}\upxi_\epmu^n\big)(z_1,z_2)\dint z_1\dint z_2\bigg|\dint w_2\,.
    \end{equs}
    In the above, the operators $\mcR_\mu$ apply on $z_1$ and $z_2$ respectively. Using \eqref{triangleineq}, we have
    \begin{equs}
        \langle w_1;w_2\rangle^{1-\kappa}_\nu\lesssim\langle w_1;z_1\rangle^{1-\kappa}_\nu\langle z_1;z_2\rangle^{1-\kappa}_\nu\langle z_2;w_2\rangle^{1-\kappa}_\nu\,.
    \end{equs}
    Using Hölder's inequality, we get that the above is bounded by
    \begin{equs}
        ~&\sup_{w_1\in\Lambda}\int_{\Lambda}\langle t_1\rangle^{-\kappa}\bigg|\int_{\Lambda^2}\langle w_{1};z_1\rangle^{1-\kappa}_\mu\langle w_2;z_2\rangle^{1-\kappa}_\mu\mcR_\mu^{\otimes2}\bigg(K_\mu(w_1-z_1)K_\mu(w_2-z_2)\e^{-(a_{n+1}-1)\mu^2\frac{|y_1|^2+|y_2|^2}{8}}\bigg) \\
        &\qquad\qquad\qquad\qquad\qquad\qquad\qquad\qquad\times\langle z_1\rangle^{\kappa_0}\langle z_1\rangle^{-\kappa_0}\langle z_1;z_2\rangle^{1-\kappa}_\mu\big(K_\mu^{\otimes2}\upxi_\epmu^n\big)(z_1,z_2)\dint z_1\dint z_2\bigg|\dint w_2\\
        &\leq\sup_{w_1,z_2\in\Lambda}\int_{\Lambda}\langle t_1\rangle^{-\kappa}\bigg|\int_{\Lambda}\langle w_{1};z_1\rangle^{1-\kappa}_\mu\langle w_2;z_2\rangle^{1-\kappa}_\mu\\
        &\qquad\qquad\times\mcR_\mu^{\otimes2}\bigg(K_\mu(w_1-z_1)K_\mu(w_2-z_2)\e^{-(a_{n+1}-1)\mu^2\frac{|y_1|^2+|y_2|^2}{8}}\bigg)\langle z_1\rangle^{\kappa_0}\dint z_1\bigg|\dint w_2\times\|\upxi_\epmu^n\|_{\mu,\kappa_0,0,1}\,.
    \end{equs}
    $\kappa_0<\kappa$ is defined in Lemma \ref{lemma:induction+2}. Thanks to \eqref{triangleineq}, we decompose 
    \begin{equs}
        \langle z_1\rangle^{\kappa_0}\lesssim\langle s_1\rangle^{\kappa_0}\langle y_1\rangle^{\kappa_0}\lesssim\langle s_1;t_1\rangle^{\kappa_0}\langle t_1\rangle^{\kappa_0}\langle y_1\rangle^{\kappa_0}\lesssim\langle z_1;w_1\rangle_\mu\langle t_1\rangle^{\kappa_0}\langle y_1\rangle^{\kappa_0}\,.
    \end{equs}
    We can now use \eqref{eq:estimreloue} to get
    \begin{equs}
        ~&\sup_{w_1,z_2\in\Lambda}\int_{\Lambda^2}\langle w_{1};z_1\rangle^{2}_\mu\langle w_2;z_2\rangle_\mu\\
        &\qquad\qquad\qquad\qquad\qquad~~\times\bigg|\mcR_\mu^{\otimes2}\bigg(K_\mu(w_1-z_1)K_\mu(w_2-z_2)\e^{-(a_{n+1}-1)\mu^2\frac{|y_1|^2+|y_2|^2}{8}}\bigg)\bigg|\langle y_1\rangle^{\kappa_0}\dint z_1\dint w_2\\
        &\lesssim\mu^{-\kappa_0}\sup_{w_1,z_2\in\Lambda}\int_{\Lambda^2}\langle w_{1};z_1\rangle^{2}_\mu\langle w_2;z_2\rangle_\mu\big(\mcS_\mu\hat K\big)(w_1-z_1)\\
        &\qquad\qquad\qquad\qquad\times\big(\mcS_\mu\hat K\big)(w_2-z_2)\big(1+\mu^2|y_1|^2+\mu^2|y_2|^2\big)\e^{-(a_{n+1}-1)\mu^2\frac{|y_1|^2+|y_2|^2}{8}}\langle \mu y_1\rangle^{\kappa_0}\dint z_1\dint w_2\\
        &\lesssim\mu^{-\kappa_0}\sup_{w_1,z_2\in\Lambda}\int_{\Lambda^2}\langle w_{1};z_1\rangle^{2}_\mu\langle w_2;z_2\rangle_\mu\big(\mcS_\mu\hat K\big)(w_1-z_1)\big(\mcS_\mu\hat K\big)(w_2-z_2)\dint z_1\dint w_2\\
        &\lesssim\mu^{-\kappa_0}\|\langle\bigcdot\rangle^2\hat K\|_{L^1_{t,x}}^2\lesssim\mu^{-\kappa_0}\,.
    \end{equs}
    We recall that the kernels $\hat K$ are integrable with exponential decay. Eventually, we have
    \begin{equs}
        \sup_{w_1\in\Lambda}\int_{\Lambda} \langle t_1\rangle^{-\kappa}\langle w_{1};w_2\rangle^{1-\kappa}_\mu\big|\big(J_{\mu,n+1}^{\otimes2}\mcb W_\nu\upxi_\epmu^n\big)(w_1,w_2)\big|\dint w_2\lesssim\talloblong\upxi\talloblong^{n+1}\mu^{\deg(n)-2\kappa_0}\leq\talloblong\upxi\talloblong^{n+1}\mu^{\deg(n)-\kappa}\,,
    \end{equs}
    which in turn yields the result.
\end{proof}

\subsubsection{Construction of the non-stationary coefficients}

We recall from Section \ref{sec:fromstatiotoeff} that, once the stationary coefficients were constructed, one had to perform two subsequent constructions $\upxi_\epmu^n\to\uptau_\epmu^n\to\upzeta^n_\epmu$. The first one handles the initial condition and the second one tackles non-stationarity. In the present case, the scheme is somewhat similar. We have the steps
\begin{equs}
    \upxi_\epmu^n\to\upomega_\epmu^n\to\uptau_\epmu^n\to\upzeta^n_\epmu\,.
\end{equs}
Step 1 was performed in the previous section, step 2 is the object of the present section, and we let the reader convince him/herself that the last step is identical to what is done in section \ref{sec:fromstatiotoeff}.

\begin{lemma}\label{lemma:zetaOH}
    The following holds uniformly in $\mu\in(0,1]$ and $0\leq n\leq N_1$
    \begin{equation}
        \|J_\mu\upzeta_\epmu^{0,n}\|_{L_{t,x}^{\infty,r}}\lesssim  \talloblong\upxi\talloblong^{n}\|u_0^\eps\|_{\msB_{r,\infty}^{-\sigma}}\mu^{\deg(0,n)-\kappa}\,,
    \end{equation}
    and
    \begin{equation}
        \big\|J_\mu^{\otimes2}\upzeta_{\epmu}^{1,n}\big\|_{\mathrm{Op}}\lesssim  \talloblong\upxi\talloblong^{n+1}\mu^{\deg(1,n)-\kappa}\,.
    \end{equation}
    We recall that the norm $\|\bigcdot\|_{\mathrm{Op}}$ was defined in \eqref{eq:normOp}.
\end{lemma}
\begin{proof}
    This is a consequence of Lemma \ref{lemma:boundtauOH} below and the rest of the proof follows as in Lemma \ref{Lem:constructionzeta}.
\end{proof}

\begin{lemma}\label{lemma:boundtauOH}
    The following holds uniformly in $\mu\in(0,1]$ and $0\leq n\leq N_1$
    \begin{equation}
        \|J_\mu\uptau_\epmu^{0,n}\|_{L_{t,x}^{\infty,r}}\lesssim  \talloblong\upxi\talloblong^{n}\|u_0^\eps\|_{\msB_{r,\infty}^{-\sigma}}\mu^{\deg(0,n)-\kappa}\,.
    \end{equation}
\end{lemma}
\begin{proof}
    We first use Proposition \ref{prop:boundomega} to write
    \begin{equs}
        \|\langle t\rangle^{-\kappa}J_\mu\uptau_\epmu^{0,n}\|_{L_{t,x}^{\infty,r}}&\leq \sup_{w_1\in\Lambda}\int_{\Lambda} \langle t_1\rangle^{-\kappa}\langle w_{1};w_2\rangle^{1-\kappa}_\mu\big|\big(J_{\mu}^{\otimes2}\upomega^{n-1}_\epmu\big)(w_1,w_2)\big|\dint w_2 \times \|\mcT_\mu M_\mu\psi_\eps\|_{L_{t,x}^{\infty,r}}\\
        &\lesssim\mu^{\deg(1,n-1)-\kappa}\|\mcT_\mu M_\mu\psi_\eps\|_{L_{t,x}^{\infty,r}}\,.
    \end{equs}
    We then have $\|\mcT_\mu M_\mu\psi_\eps\|_{L_{t,x}^{\infty,r}}=\|\mcT_\mu M_\mu u_0^\eps\|_{L_{x}^{r}}$. Lemma \ref{lemma:heatflowmehler} in the appendix allows to conclude.
\end{proof}

\subsubsection{Fixed-point argument and construction of the solution}

We finish in the present section the construction of the solution of the SPDE associated to the harmonic oscillator. As we explained at the beginning of the section \ref{sec:OH}, we cannot construct the solution as a series. This is due to the inductive argument we used to construct the coefficients $\upomega_\epmu^n$ in Lemma \ref{lemma:induction+2}, and more precisely because we loose some $\kappa_0$ at every step if the induction. The solution is to use a fixed-point argument, as in infinite volume. We recall that
\begin{equs}
    F_\epmu[\phi](z)=F^0_\epnu(z)+F_\epmu^1[\phi](z)=\sum_{n=0}^{N_0}\upzeta_{\eps,\mu}^{0,n}(z)+\sum_{n=0}^{N_0}\upzeta_{\eps,\mu}^{1,n}[\varphi](z)\,,
\end{equs}
and
\begin{equs}
H_{\eps,\mu}[\phi](z)&=H_{\eps,\mu}^0(z)+H_{\eps,\mu}^1[\varphi](z)\\
&=\sum_{n=N_0+1}^{2N_0+1}\sum_{i=0}^{n-1}\upzeta_{\eps,\mu}^{1,i}[\dot M_\mu \upzeta_{\eps,\mu}^{0,n-1-i}(\bigcdot)](z)+\sum_{n=N_0+1}^{2N_0+1}\sum_{i=0}^{n-1}\upzeta_{\eps,\mu}^{1,i}[\dot M_\mu \upzeta_{\eps,\mu}^{1,n-1-i}[\varphi](\bigcdot)](z)\,.
\end{equs}
Note that, compared to the previous section, the kernel here is the Mehler kernel. Our estimates on $F_\epmu$ and $H_\epmu$ are analogous to \eqref{eq:esti_force_infvol}, \eqref{eq:H1} and \eqref{eq:H0} in the previous section, namely 
 \begin{equs}\label{eq:esti_force_OH}
    \|J_\mu^{\otimes2}F^1_\epmu\|_{\mathrm{Op}}\lesssim\talloblong\upxi\talloblong\mu^{-2+3\kappa}\,,
    \qquad
        \|J_\mu F_\epmu^0\|_{L_{t,x}^{\infty,r}(\Lambda_T)}\lesssim \talloblong\upxi\talloblong\|u_0^\eps\|_{\msB^{-\sigma}_{r,\infty}} \mu^{-2-\sigma+\beta-2\kappa}\,,
    \end{equs}
    and
\begin{equs}
\|J_\mu^{\otimes2}H^1_\epmu\|_{\mathrm{Op}}\lesssim\talloblong\upxi\talloblong^{(N_0+1)}
    \mu^{-1+(N_0+2)\kappa}
\,,\qquad
          \| J_\mu H_\epmu^0\|_{L_{t,x}^{\infty,r}(\Lambda_T)}\lesssim \talloblong\upxi\talloblong^{N_0}\|u_0^\eps\|_{\msB^{-\sigma}_{r,\infty}} \mu^{-1+(N_0+2)\kappa }\,.
\end{equs}
As in Section \ref{sec:infinitevol}, we introduce a map $\Psi_\eps$ on trajectories $(v_\mu,R_\mu)_{\mu\in(0,T^{1/2}]}$

\begin{subequations}
\begin{empheq}[left=\empheqlbrace]{alignat=2}\label{sysflowOH}
  \Psi^v_\epmu(v,R)&=-\int_\mu^{T^{1/2}}\dot M_\nu\big(F^0_\epnu+F^1_\epnu [v_\nu]+R_\nu\big)\dint\nu\,,\\
  \Psi^R_\epmu(v,R)&=-\int_0^\mu\big(H^0_\epnu +H^1_\epnu [v_\nu]+F_\epnu^1[\dot M_\nu R_\nu]\big)\dint\nu\,.
\end{empheq}
\end{subequations}
We show that we can close a fixed point argument for this map using the following norm
\begin{equation}
    \snnorm{v,R}\eqdef \|u_0^\eps\|_{\msB_{r,\infty}^{-\sigma}}^{-1}\sup_{\mu\in(0,T^{1/2}]} \sup_{t\leq T}\Big( \|\mcb \mcT_\mu v_\mu(t)\|_{{L^r_x}}\vee\|J_\mu R_\mu(t)\|_{L^r_x} \Big) \,.
\end{equation}

\begin{lemma}\label{lemma:fixedpointOH}
    Choosing $T=\big(\talloblong\upxi\talloblong+T_0^{-\kappa/2}\big)^{-2/\kappa}$, the map $\Psi_\eps$ is a contraction for the norm defined above. 
   
\end{lemma}

\begin{proof}
    We let the reader convince him/herself that the proof is the same as the proof of Lemma \ref{Lemmafixedpoint}, but that it is much easier since we do not require to use exponential weights. More precisely, one rewrites
    \begin{empheq}[left=\empheqlbrace]{align*}
      \mcT_\mu\Psi^v_\eps(v,R)&=-\int_\mu^{T^{1/2}}(\mcT_\mu\otimes\mcT_\nu)\dot M_\nu\big(J_\nu F^0_\epnu+(J_\mu\otimes J_\nu) F^1_\epnu [\mcT_\nu v_\nu]+J_\mu R_\nu\big)\dint\nu\,,\\
      J_\mu \Psi^R_\eps(v,R)&=-\int_0^\mu\Big(J_\mu H^0_\epnu+(J_\mu\otimes J_\nu)\big(H^1_\epnu [\mcT_\nu v_\nu]+ F_\epnu^1[\mcT_\nu^{\otimes2}\dot M_\nu J_\nu R_\nu]\big)\Big)\dint\nu\,.
    \end{empheq}
    One can then conclude the fixed point thanks to \eqref{eq:lessJmu}, Lemma \ref{lemma:estimMmu}, and Lemma \ref{lemma:zetaOH}.
\end{proof}

\begin{lemma}\label{lemma:constructsolOH}
    There is a random $T\in(0,1]$ such that, for $t\leq T$ and $\gamma\in\big(0,\beta-\sigma-d(1/r-1/p)\big)$,
    \begin{equation}
        \|u_\eps(t)-M\psi_\eps(t)\|_{\msB_{p,\infty}^\gamma}\leq t^{(\beta-\sigma-\gamma-d(1/r-1/p)-\kappa)/2}\|u_0^\eps\|_{\msB_{r,\infty}^{-\sigma}}\,.
    \end{equation}
\end{lemma}

\begin{proof}
    The proof follows the same lines as Lemma \ref{lemma:constructsolinf}, although the absence of exponential weights makes it substantially simpler. We write
    \begin{equs}
        u_\eps-M\psi_\eps=-\int_0^{T^{1/2}}\dot M_\nu (F_{\eps,T^{1/2}}[0]+R_{\eps,T^{1/2}})\dint\nu\,.
    \end{equs}
    We only treat the case $\mu\leq T^{1/2}$, the other one being easier. We have
    \begin{equs}
        \|((\msQ_\mu-\mathrm{Id})v_\eps)(t)\|_{L^p_x}&\leq \int_0^{T^{1/2}}\|(\msQ_\mu-\mathrm{Id})\dot M_\nu(F_{\eps,T^{1/2}}[0]+R_{\eps,T^{1/2}})(t)\|_{L_x^p}\dint\nu\\
        &\leq\int_0^{\mu}\|\dot M_\nu(F_{\eps,T^{1/2}}[0]+R_{\eps,T^{1/2}})(t)\|_{L_x^p}\dint\nu\\
        &\quad+\int_\mu^{T^{1/2}}\|(\msQ_\mu-\mathrm{Id})\msQ_\nu(\msP_\nu\otimes\mcT_\nu)\dot M_\nu J_\nu(F_{\eps,T^{1/2}}[0]+R_{\eps,T^{1/2}})(t)\|_{L_x^p}\dint\nu\,.
    \end{equs}
    We have used that $\|\msQ_\mu\|_{\mcL(L_{x}^\infty)}\leq1$ (see \eqref{eq:msQleq1}). The first term is handled using the support properties in time of the kernel $\dot M_\nu$.
    \begin{equs}
        ~&\int_0^{\mu}\|\dot M_\nu(F_{\eps,T^{1/2}}[0]+R_{\eps,T^{1/2}})(t)\|_{L_x^p}\dint\nu\\
        &\lesssim\int_0^{\mu\wedge t^{1/2}} \|(\mathrm{Id}\otimes \mcT_\nu)\dot M_\nu\|_{\mcL(L^p_{t,x},L^r_{t,x})}\|J_\nu (F_{\eps,T^{1/2}}[0]+R_{\eps,T^{1/2}})(t)\|_{L_x^r}\dint\nu\,.
    \end{equs}
    By Lemma \ref{lemma:estimMmu}, $\|(\mathrm{Id}\otimes \mcT_\mu)\dot M_\nu\|_{\mcL(L^p_{t,x},L^r_{t,x})}\lesssim\nu^{1-d(1/r-1/p)}$. Then,
    \begin{equs}
        \mu^{-\gamma}\int_0^{\mu}\|\dot M_\nu(F_{\eps,T^{1/2}}[0]+R_{\eps,T^{1/2}})(t)\|_{L_x^p}\dint\nu &\lesssim\int_0^{\mu\wedge t^{1/2}}\nu^{-1+\beta-\sigma-\gamma-d(1/r-1/p)-2\kappa}\dint\nu\\
        &\lesssim t^{(\beta-\sigma-\gamma-d(1/r-1/p)-2\kappa)/2}\int_0^{t^{1/2}}\nu^{-1+\kappa}\dint\nu \|u_0^\eps\|_{\msB_{r,\infty}^{-\sigma}}\\
        &\lesssim t^{(\beta-\sigma-\gamma-d(1/r-1/p)-\kappa)/2}\|u_0^\eps\|_{\msB_{r,\infty}^{-\sigma}}\,.
    \end{equs}
    We need to control the second term. We have
    \begin{equs}
        ~&\mu^{-\gamma}\int_\mu^{T^{1/2}}\|(\msQ_\mu-\mathrm{Id})\msQ_\nu(\msP_\nu\otimes\mcT_\nu)\dot M_\nu J_\nu(F_{\eps,T^{1/2}}[0]+R_{\eps,T^{1/2}})(t)\|_{L_x^p}\dint\nu\\
        &\lesssim\one_{\mu^2\leq t} \mu^{-\gamma}\int_\mu^{T^{1/2}} \|(\msQ_\mu-\mathrm{Id})\msQ_\nu\|_{\mcL(L_{x}^{p})}\|(\msP_\nu\otimes\mcT_\nu)\dot M_\nu\|_{\mcL(L_{t,x}^{\infty,r},L_{t,x}^{\infty,p})}\\
        &\qquad\qquad\qquad\qquad\qquad\qquad\qquad\qquad\qquad\qquad\qquad\times\|J_\nu(F_{\eps,T^{1/2}}[0]+R_{\eps,T^{1/2}})(t)\|_{L_x^r}\dint\nu
    \end{equs}
    Using \eqref{eq:msQmumsQnu}, \eqref{eq:PmuTmuMmu}, and Lemma \ref{LemmarecurrenceOH}, we get that the above is then bounded by
    \begin{equs}
        \one_{\mu^2\leq t}\mu^{2-\gamma}\int_\mu^1\nu^{-2}\nu^{1-d(1/r-1/p)}\nu^{\beta-\sigma-\kappa}\dint\nu\|u_0^\eps\|_{\msB_{r,\infty}^{-\sigma}}&\lesssim \one_{\mu^2\leq t}\mu^{\beta-\sigma-\gamma-d(1/r-1/p)-\kappa}\|u_0^\eps\|_{\msB_{r,\infty}^{-\sigma}}\\
        &\lesssim t^{(\beta-\sigma-\gamma-d(1/r-1/p)-\kappa)/2}\|u_0^\eps\|_{\msB_{r,\infty}^{-\sigma}}\,.
    \end{equs}
\end{proof}

We can now conclude the section with a globalisation argument to prove point 3 of Theorem \ref{maintheoremspde}.

\begin{corollary}\label{Cor:GlobalisationSepsOH}
    For $\eps,\eta>0$ and $(\sigma,\gamma,p)$ such that $p\geq r$, $0<\sigma<\beta$ and $0\leq \gamma< \beta-\sigma-d(1/r-1/p)$, $u_\eps$ satisfies
    \begin{equation}\label{Eq:EstimateS-GepsOH}
        \big\|u_\eps(t)-\e^{-t\mathrm{H}}u_0^\eps\big\|_{\msB_{p,\infty}^\gamma(\R^d)}\leq \e^{ct}t^{(\beta-\sigma-\gamma-d(1/r-1/p)-\eta)/2}  \|u_0^\eps\|_{\msB_{r,\infty}^{-\sigma}(\R^d)}\,,
    \end{equation}
    for some positive random constant $c$.
\end{corollary}

\begin{proof}
    The proof follows like the one of Corollary \ref{Cor:GlobalisationSeps}.
\end{proof}

\begin{lemma}\label{LemmarecurrenceOH}
    Let $\tau\in(0,1)$ such that
    \begin{equation}
        (2/\tau^2-1)^{2M}\leq2\,.    
    \end{equation}
    Then, there is a random $T\in(0,1]$ such that, for every $p\in\N$,
    \begin{equation}
        \sup_{t\leq T}\big\|J_{\tau^pT^{1/2}} (F_{\eps,T^{1/2}}[0]+R_{\eps,T^{1/2}})(t)\big\|_{L^{r}_{x}}\leq C(\tau^pT^{1/2})^{-2-\sigma+\beta-\kappa}\|u_0^\eps\|_{\msB_{r,\infty}^{-\sigma}}\,.
    \end{equation}
    Consequently, we get
    \begin{equation}
        \sup_{t\leq T}\big\|J_\mu (F_{\eps,T^{1/2}}[0]+R_{\eps,T^{1/2}})(t)\big\|_{L^{r}_{x}}\lesssim\mu^{-2-\sigma+\beta-\kappa}\|u_0^\eps\|_{\msB_{r,\infty}^{-\sigma}}\,.
    \end{equation}
\end{lemma}

\begin{proof}
    The proof is the same as the one of Lemma \ref{Lemmarecurrence}, but without the exponential weights. See also Lemma 3.12 in \cite{CF24a}. The main ingredient of the proof is \eqref{eq:esti_force_OH}.
\end{proof}

\subsubsection{Inductive construction of the irrelevant coefficients}

We carry out in this section the construction of the irrelevant coefficients $\upiota_\epmu^n$. The proof relies on an inductive argument and make use of the essential estimates of Lemmas \ref{lemma:estimMmu} and \ref{lemma:diffMG}.

\begin{lemma}\label{lemma:induction+2}
   Recall that $\kappa_n=4^{n-N_0-1} \kappa$. The following two statements hold
    \begin{equation}\label{eq:HR+2}
        \sup_{w_1\in\Lambda}\int_{\Lambda}\langle t_1\rangle^{-\kappa_n}\langle x_{1}\rangle^{-2-\kappa_n}\langle w_{1};w_2\rangle^{1-\kappa_n}_\mu\big|\big(J_{\mu,n}^{\otimes2}\upiota^n_\epmu\big)(w_1,w_2)\big|\dint w_2\lesssim  \talloblong\upxi\talloblong^{n+1}\mu^{\deg(n)+2-\kappa_n}\,,
    \end{equation}
    and
    \begin{equation}\label{eq:HR+0}
        \sup_{w_1\in\Lambda}\int_{\Lambda} \langle t_1\rangle^{-\kappa_n}\langle w_{1};w_2\rangle^{1-\kappa_n}_\mu\big|\big(J_{\mu,n+1}^{\otimes2}\upiota^n_\epmu\big)(w_1,w_2)\big|\dint w_2\lesssim  \talloblong\upxi\talloblong^{n+1}\mu^{\deg(n)-2\kappa_n}\,.
    \end{equation}
    We have used the notation $w=(t,x)$.
\end{lemma}

\begin{remark}
    Before turning to the proof, we would like to mention that the decreasing $\kappa_0$ exponents come from the inductive structure of the proof. Because of them, we do not have the choice but to set a truncated (up to $N_0$) ansatz. It forces us to use a fixed point argument and not a construction via a series as it is done in \cite{Duc21}.
\end{remark}

\begin{proof}
    We prove the statement by induction. The base case follows from term 1 below so we postpone the proof thereafter. We first prove the second statement from the first one, and then we can carry out the induction with the induction hypothesis on the second statement at hand. Let us start with the first task. The idea is to proceed in the same lines as in the proof of Proposition \ref{prop:boundomega}. We have, writing $a_{n+1}=a_n+b$,
    \begin{equs}
        ~&\sup_{w_1\in\Lambda}\int_{\Lambda} \langle t_1\rangle^{-\kappa_n}\langle w_{1};w_2\rangle^{1-\kappa_n}_\mu\big|\big(J_{\mu,n+1}^{\otimes2}\upiota^n_\epmu\big)(w_1,w_2)\big|\dint w_2\\
        &=\sup_{w_1\in\Lambda}\int_{\Lambda} \langle t_1\rangle^{-\kappa_n}\langle w_{1};w_2\rangle^{1-\kappa_n}_\mu\\
        &\qquad\times\left|\int_{\Lambda^2}K_{\mu}(w_1-z_1)K_{\mu}(w_2-z_2)\e^{-a_{n+1}\mu^2\frac{|y_1|^2+|y_2|^2}{8}}\upiota^n_\epmu(z_1,z_2)\dint z_1\dint z_2\right|\dint w_2\\
        &=\sup_{w_1\in\Lambda}\int_{\Lambda}\langle t_1\rangle^{-\kappa_n} \langle w_{1};w_2\rangle^{1-\kappa_n}_\mu\\
        &\qquad\qquad~\times\left|\int_{\Lambda^2}\mcR_\mu^{\otimes2}\bigg(K_{\mu}(w_1-z_1)K_{\mu}(w_2-z_2)\e^{-b\mu^2\frac{|y_1|^2+|y_2|^2}{8}}\bigg)\big(J_{\mu,n}^{\otimes2}\upiota^n_\epmu\big)(z_1,z_2)\dint z_1\dint z_2\right|\dint w_2\\
        &\lesssim\sup_{w_1\in\Lambda}\int_{\Lambda}\langle t_1\rangle^{-\kappa_n} \bigg|\int_{\Lambda^2}\langle w_{1};z_1\rangle_\mu\langle w_2;z_2\rangle_\mu\mcR_\mu^{\otimes2}\bigg(K_{\mu}(w_1-z_1)K_{\mu}(w_2-z_2)\e^{-b\mu^2\frac{|y_1|^2+|y_2|^2}{8}}\bigg)\\
        &\qquad\qquad\qquad\quad~\times\langle s_1\rangle^{\kappa_n}\langle y_1\rangle^{2+\kappa_n}\langle s_1\rangle^{-\kappa_n}\langle y_1\rangle^{-2-\kappa_n}\langle z_{1};z_2\rangle^{1-\kappa_n}_\mu\big(J_{\mu,n}^{\otimes2}\upiota^n_\epmu\big)(z_1,z_2)\dint z_1\dint z_2\bigg|\dint w_2\\
        &\leq \sup_{w_1,z_2\in\Lambda}\int_{\Lambda} \bigg|\int_{\Lambda}\langle w_{1};z_1\rangle^{2}_\mu\langle w_2;z_2\rangle_\mu\langle y_1\rangle^{2+\kappa_n}\\
        &\qquad\qquad\qquad\qquad\qquad\qquad\qquad~~\times\mcR_\mu^{\otimes2}\bigg(K_{\mu}(w_1-z_1)K_{\mu}(w_2-z_2)\e^{-b\mu^2\frac{|y_1|^2+|y_2|^2}{8}}\bigg)\dint z_1\bigg|\dint w_2\\
        &\quad\times \sup_{z_1\in\Lambda}\int_\Lambda\langle s_1\rangle^{-\kappa_n} \langle y_1\rangle^{-2-\kappa_n}\langle z_{1};z_2\rangle^{1-\kappa_n}_\mu\big|\big(J_{\mu,n}^{\otimes2}\upiota^n_\epmu\big)(z_1,z_2)\big|\dint z_2\,.
    \end{equs}
    We have used that $\langle s_1\rangle\lesssim\langle s_1;t_1\rangle\langle t_1\rangle\lesssim\langle s_1;t_1\rangle_\nu\langle t_1\rangle$. Controlling the first term as in the proof of Proposition \ref{prop:boundomega}, and using \eqref{eq:HR+2}, we eventually get
    \begin{equs}
        ~&\sup_{w_1\in\Lambda}\int_{\Lambda} \langle w_{1};w_2\rangle^{1-\kappa_n}_\mu\big|\big(J_{\mu,n+1}^{\otimes2}\upiota^n_\epmu\big)(w_1,w_2)\big|\dint w_2\\
        &\lesssim\mu^{-2-\kappa_n}\sup_{w_1\in\Lambda}\int_{\Lambda}\langle t_1\rangle^{-\kappa_n}\langle x_{1}\rangle^{-2-\kappa_n}\langle w_{1};w_2\rangle^{1-\kappa_n}_\mu\big|\big(J_{\mu,n}^{\otimes2}\upiota^n_\epmu\big)(w_1,w_2)\big|\dint w_2\\
        &\lesssim\talloblong\upxi\talloblong^{n+1}\mu^{-2-\kappa_n}\mu^{\deg(n)+2-\kappa_n}=\talloblong\upxi\talloblong^{n+1}\mu^{\deg(n)-2\kappa_n}\,.
    \end{equs}
    We move to the induction part. We have, using the flow equation on $\upiota_\epmu^n$ \eqref{eq:floweqiota},
    \begin{equs}\begin{aligned}
        \label{eq:fiveterms}
J_{\nu,n}^{\otimes2}\D_\nu\upiota_\epnu^n&=-\frac{\nu}{4} J_{\nu,n}^{\otimes2}\big((|x_1|^2+|x_2|^2)\mcb W_\nu\upxi_\nu^n\big)\\
    &\quad-\sum_{i+j=n-1} \Big((J_{\nu,n}\otimes J_{\nu,0})\mcb W_\nu\upxi_\epnu^i[\mcT_{\nu,0}^{\otimes2}(\dot M_\nu-\mcb W_\nu^{-1}\dot G_\nu)(J_{\mu,n}\otimes J_{\nu,0})\mcb W_\nu\upxi_\epnu^{j}]\\
    &\quad\qquad\qquad+(J_{\nu,n}\otimes J_{\nu,0})\mcb W_\nu\upxi_\epnu^i[(\mcT_{\nu,0}\otimes\mcT_{\nu,j})\dot M_\nu(J_{\nu,j}\otimes J_{\nu,n})\upiota_\epnu^{j}]\\
    &\quad\qquad\qquad+(J_{\nu,n}\otimes J_{\nu,i})\upiota_\epnu^i[(\mcT_{\nu,i}\otimes\mcT_{\nu,n})\dot M_\nu (J_{\nu,0}\otimes J_{\nu,n})\mcb W_\nu\upxi_\epnu^{j}]\\
    &\quad\qquad\qquad+(J_{\nu,n}\otimes J_{\nu,i+1})\upiota_\epnu^i[(\mcT_{\nu,i+1}\otimes\mcT_{\nu,j})\dot M_\nu(J_{\nu,j}\otimes J_{\nu,n})\upiota_\epnu^{j}]\Big)\,.
     \end{aligned}
\end{equs}
In what follows we will denote $J_{\mu,n,m}^{\otimes2}\eqdef J_{\mu,n}\otimes J_{\mu,m}$ for compactness of the notations. We will first show that
\begin{equs}\label{eq:boundderivativeiota}
    \sup_{w_1\in\Lambda}\int_{\Lambda}\langle t_1\rangle^{-\kappa_n}\langle x_{1}\rangle^{-2-\kappa_n}\langle w_{1};w_2\rangle^{1-\kappa_n}_\nu\big|\big(J_{\nu,n}^{\otimes2}\D_\nu\upiota^n_\epnu\big)(w_1,w_2)\big|\dint w_2\lesssim\talloblong\upxi\talloblong^{n+1}\nu^{\deg(n)+1-\kappa_n}\,.
\end{equs}
With this bound at hand, we can conclude the proof as follows
\begin{equs}
    ~&\sup_{w_1\in\Lambda}\int_{\Lambda}\langle t_1\rangle^{-\kappa_n}\langle x_{1}\rangle^{-2-\kappa_n}\langle w_{1};w_2\rangle^{1-\kappa_n}_\mu\big|\big(J_{\mu,n}^{\otimes2}\upiota^n_\epmu\big)(w_1,w_2)\big|\dint w_2\\
    &=\sup_{w_1\in\Lambda}\int_{\Lambda}\langle t_1\rangle^{-\kappa_n}\langle x_{1}\rangle^{-2-\kappa_n}\langle w_{1};w_2\rangle^{1-\kappa_n}_\mu\bigg|\bigg(J_{\mu,n}^{\otimes2}\int_0^\mu\D_\nu\upiota^n_\epnu\dint\nu\bigg)(w_1,w_2)\bigg|\dint w_2\\
    &\lesssim\int_0^\mu\sup_{w_1\in\Lambda}\int_{\Lambda}\langle t_1\rangle^{-\kappa_n}\langle x_{1}\rangle^{-2-\kappa_n}\langle w_{1};w_2\rangle^{1-\kappa_n}_\nu\big|\big(J_{\nu,n}^{\otimes2}\D_\nu\upiota^n_\epnu\big)(w_1,w_2)\big|\dint w_2\dint\nu\\
    &\lesssim\talloblong\upxi\talloblong^{n+1}\int_0^\mu \nu^{\deg(n)+1-\kappa_n}\dint\nu\lesssim\talloblong\upxi\talloblong^{n+1}\mu^{\deg(n)+2-\kappa_n}\,.
\end{equs}
We have used \eqref{eq:lessJmu} and the fact that, for $\nu\leq\mu$, $\langle\bigcdot\rangle_\mu\leq\langle\bigcdot\rangle_\nu$. It remains to show \eqref{eq:boundderivativeiota}. We are left with the five terms of the right-hand side of \eqref{eq:fiveterms} to control. We recall that we will denote in a general way $w=(t,x)$, $z=(s,y)$.\\

\noindent\textit{Term 1: linear in $\upxi_\nu$.} We proceed as in the proof of Proposition \ref{prop:boundomega}. We have, using \eqref{eq:estimrelouex2},
\begin{equs}
    ~&\nu\sup_{w_1\in\Lambda}\int_{\Lambda}\langle t_1\rangle^{-\kappa_n}\langle x_{1}\rangle^{-2-\kappa_n}\langle w_1;w_2\rangle^{1-\kappa_n}_\nu \bigg|\int_{\Lambda^2}K_{\nu}(w_1-z_1)K_{\nu}(w_2-z_2)\\
    &\qquad\qquad\qquad\qquad\qquad\qquad\qquad\quad~~\times\e^{-a_n\nu^2\frac{|y_1|^2+|y_2|^2}{8}}\big(|y_1|^2+|y_2|^2\big)\mcb W_\nu\upxi_\epnu^n(z_1,z_2)\dint z_1\dint z_2\bigg|\dint w_2\\
    &\lesssim\nu\sup_{w_1\in\Lambda}\int_{\Lambda}\langle t_1\rangle^{-\kappa_n}\langle x_{1}\rangle^{-2-\kappa_n}\langle w_1;w_2\rangle_\nu \bigg|\int_{\Lambda^2}\mcR_\nu^{\otimes2}\bigg(K_{\nu}(w_1-z_1)K_{\nu}(w_2-z_2)\\
    &\qquad\qquad\qquad\qquad\qquad\quad\times\e^{-(a_n-1)\nu^2\frac{|y_1|^2+|y_2|^2}{8}}\big(|y_1|^2+|y_2|^2\big)\bigg)\big(J_{\nu,0}^{\otimes2}\mcb W_\nu\upxi_\epnu^n\big)(z_1,z_2)\dint z_1\dint z_2\bigg|\dint w_2\\
    &\lesssim\nu\sup_{w_1\in\Lambda}\int_{\Lambda}\langle t_1\rangle^{-\kappa_n}\langle x_{1}\rangle^{-2-\kappa_n}\int_{\Lambda^2}\langle w_1;z_1\rangle_\nu\langle w_2;z_2\rangle_\nu\bigg|\mcR_\nu^{\otimes2}\bigg(K_{\nu}(w_1-z_1)K_{\nu}(w_2-z_2)\\
    &\quad\quad~~\times\e^{-(a_n-1)\nu^2\frac{|y_1|^2+|y_2|^2}{8}}\big(|y_1|^2+|y_2|^2\big)\bigg)\bigg|\langle z_2\rangle^{\kappa_0}\langle z_2\rangle^{-\kappa_0}\langle z_1;z_2\rangle_\nu\big|\big(K_{\nu}^{\otimes2}\upxi_\epnu^n\big)(z_1,z_2)\big|\dint z_1\dint z_2\dint w_2\\
    &\lesssim\nu\sup_{w_1\in\Lambda}\int_{\Lambda}\langle t_1\rangle^{-\kappa_n+\kappa_0}\langle x_{1}\rangle^{-2-\kappa_n+\kappa_0}\int_{\Lambda^2}\langle w_1;z_1\rangle_\nu^2\langle w_2;z_2\rangle_\nu\big(\mcS_\nu \hat K\big)(w_1-z_1)\big(\mcS_\nu \hat K\big)(w_2-z_2)\\
    &\qquad\qquad\qquad\times\e^{-(a_n-1)\nu^2\frac{|y_1|^2+|y_2|^2}{8}}\big(1+|y_1|^2+|y_2|^2\big)\big(1+\mu^2|y_1|^2+\mu^2|y_2|^2\big)\\
    &\qquad\qquad\qquad\qquad\qquad\qquad\qquad\qquad\qquad\qquad\times\langle z_2\rangle^{-\kappa_0}\langle z_1;z_2\rangle^{2}_\nu\big|\big(K_{\nu}^{\otimes2}\upxi_\epnu^n\big)(z_1,z_2)\big|\dint z_1\dint z_2\dint w_2\\
    &\lesssim\nu\sup_{w_1\in\Lambda}\int_{\Lambda^3}\langle w_1;z_1\rangle^{4}_\nu\langle w_2;z_2\rangle_\nu\big(\mcS_\nu \hat K\big)(w_1-z_1)\big(\mcS_\nu \hat K\big)(w_2-z_2)\e^{-(a_n-1)\nu^2\frac{|y_1|^2+|y_2|^2}{8}}\\
    &\qquad\qquad\qquad\qquad\qquad~~\times\big(1+\mu^2|y_1|^2+\mu^2|y_2|^2\big)\langle z_2\rangle^{-\kappa_0}\langle z_1;z_2\rangle^{4}_\nu\big|\big(K_{\nu}^{\otimes2}\upxi_\epnu^n\big)(z_1,z_2)\big|\dint z_1\dint z_2\dint w_2\\
    &\lesssim\nu\sup_{w_1,z_2\in\Lambda}\int_{\Lambda^2}\langle w_1;z_1\rangle^{4}_\nu\langle w_2;z_2\rangle_\nu\big(\mcS_\nu \hat K\big)(w_1-z_1)\big(\mcS_\nu \hat K\big)(w_2-z_2)\\
    &\qquad\qquad\qquad\qquad\qquad\qquad\qquad\qquad\qquad~~\times\e^{-(a_n-1)\nu^2\frac{|y_1|^2+|y_2|^2}{8}}\big(1+\mu^2|y_1|^2+\mu^2|y_2|^2\big)\dint z_1 \dint w_2\\
    &\quad\times\sup_{z_1\in\Lambda}\int_\Lambda\langle z_2\rangle^{-\kappa_0}\langle z_1;z_2\rangle^{4}_\nu\big|\big(K_{\nu}^{\otimes2}\upxi_\epnu^n\big)(z_1,z_2)\big|\dint z_2\,.
\end{equs}
We have used Hölder's inequality again. Precisely, we have used \eqref{triangleineq} to write
\begin{equs}
    \langle z_2\rangle^{\kappa_0}\lesssim\langle z_2;z_1\rangle_\nu\langle z_1;w_1\rangle_\nu \langle t_1\rangle^{\kappa_0}\langle x_1\rangle^{\kappa_0}\,,
\end{equs}
and
\begin{equs}
    \big(1+|y_1|^2+|y_2|^2\big)\lesssim\langle y_2;y_1\rangle_\nu^2\langle y_1;x_1\rangle_\nu^2\langle x_1\rangle^2\,.
\end{equs}
The first term in the product can be controlled as follows
\begin{equs}
    ~&\sup_{w_1,z_2\in\Lambda}\int_{\Lambda^2}\langle w_1;z_1\rangle^{4}_\nu\langle w_2;z_2\rangle_\nu\big(\mcS_\nu \hat K\big)(w_1-z_1)\big(\mcS_\nu \hat K\big)(w_2-z_2)\\
    &\qquad\qquad\qquad\qquad\qquad\qquad\qquad\qquad\qquad\times\e^{-(a_n-1)\nu^2\frac{|y_1|^2+|y_2|^2}{8}}\big(1+\mu^2|y_1|^2+\mu^2|y_2|^2\big)\dint z_1 \dint w_2\\
    &\lesssim\sup_{w_1,z_2\in\Lambda}\int_{\Lambda^2}\langle w_1;z_1\rangle^{4}_\nu\langle w_2;z_2\rangle_\nu\big(\mcS_\nu \hat K\big)(w_1-z_1)\big(\mcS_\nu \hat K\big)(w_2-z_2)\dint z_1\dint w_2\\
    &\lesssim\|\langle\bigcdot\rangle_\nu^4\mcS_\nu\hat K\|_{L^1_{t,x}}^2\lesssim1\,.
\end{equs}
In the end, we get that our original quantity is bounded by
\begin{equs}
    \nu\|\upxi_\epnu^n\|_{\nu,\kappa_0,0,5}\lesssim\talloblong\upxi\talloblong^{n+1}\nu^{\deg(n)-\kappa_n+1}\,.
\end{equs}
\noindent\textit{Term 2: bilinear in $\upxi_\nu$.} For the second term, the cornerstone of the proof is to use Lemma \ref{lemma:diffMG}. We have
\begin{equs}
    ~&\sup_{w_1\in\Lambda}\int_{\Lambda^3}\langle t_1\rangle^{-\kappa_n}\langle x_{1}\rangle^{-2-\kappa_n}\langle w_1;w_2\rangle^{1-\kappa_n}_\nu\big|(J_{\nu,n,0}^{\otimes2}\mcb W_\nu\upxi_\epnu^i)(w_1,z_1)\big|\\
    &\qquad\qquad\qquad\qquad\qquad\qquad\times\big|\mcT_{\nu,0}^{\otimes2}(\dot M_\nu-\mcb W_\nu^{-1}\dot G_\nu)(z_1,z_2)\big|\times\big|(J_{\nu,0,n}^{\otimes2}\mcb W_\nu\upxi_\epnu^{j})(z_2,w_2)\big|\dint z_1\dint z_2\dint w_2\\
    &\lesssim\nu^2\sup_{w_1\in\Lambda}\int_{\Lambda^3}\langle t_1\rangle^{-\kappa_n}\langle x_{1}\rangle^{-2-\kappa_n}\langle w_1;w_2\rangle_\nu\big|(J_{\nu,n,0}^{\otimes2}\mcb W_\nu\upxi_\epnu^i)(w_1,z_1)\big|\\
    &\qquad\qquad\qquad\qquad\qquad\quad\times\big(1+|y_1|^2+|y_2|^2\big)\big|\dot G_\nu(z_1-z_2)\big|\times\big|(J_{\nu,0,n}^{\otimes2}\mcb W_\nu\upxi_\epnu^{j})(z_2,w_2)\big|\dint z_1\dint z_2\dint w_2\\
    &\lesssim\nu^2\sup_{w_1\in\Lambda}\int_{\Lambda^3}\langle t_1\rangle^{-\kappa_n}\langle x_{1}\rangle^{-2-\kappa_n}\langle w_1\rangle^{\kappa_0}\langle w_1\rangle^{-\kappa_0}\langle w_1;w_2\rangle_\nu\big|(J_{\nu,n,0}^{\otimes2}\mcb W_\nu\upxi_\epnu^i)(w_1,z_1)\big|\\
    &\qquad\qquad\quad~~\times\big(1+|y_1|^2+|y_2|^2\big)\big|\dot G_\nu(z_1-z_2)\big|\times\langle z_2\rangle^{\kappa_0}\langle z_2\rangle^{-\kappa_0}\big|(J_{\nu,0,n}^{\otimes2}\mcb W_\nu\upxi_\epnu^{j})(z_2,w_2)\big|\dint z_1\dint z_2\dint w_2\,.
\end{equs}
Using \eqref{triangleineq}, we have
\begin{equs}
    \langle w_1;w_2\rangle_\nu\lesssim\langle w_1;z_1\rangle_\nu\langle z_1;z_2\rangle_\nu\langle z_2;w_2\rangle_\nu\,.
\end{equs}
Also note that bounding the power $1-\kappa_n$ above by $1$ was crude, but sufficient. We will repeatedly use this kind of crude bound in the rest of the proof, when allowed. Moreover,
\begin{equs}
    \big(1+|y_1|^2+|y_2|^2\big)\lesssim\langle y_2;y_1\rangle^2\langle y_1\rangle^2\lesssim\langle y_2;y_1\rangle^2_\nu \langle y_1;x_1\rangle_\nu^2\langle x_1\rangle^2\,.
\end{equs}
In the last inequality, we used the crude bound $\langle \bigcdot\rangle^2\leq \langle \bigcdot\rangle^2_\nu$. Similarly,
\begin{equs}\label{eq:boundy_2}
    \langle z_2\rangle^{\kappa_0}\lesssim\langle z_2;z_1\rangle^{\kappa_0}_\nu\langle z_1;w_1\rangle_\nu^{\kappa_0}\langle w_1\rangle^{\kappa_0}\lesssim\langle z_2;z_1\rangle_\nu\langle z_1;w_1\rangle_\nu\langle t_1\rangle^{\kappa_0}\langle x_1\rangle^{\kappa_0}\,.
\end{equs}
Then our quantity of interest is bounded by
\begin{equs}
    ~&\nu^2\sup_{w_1\in\Lambda}\int_{\Lambda^3}\langle w_1\rangle^{-\kappa_0}\langle w_1;z_1\rangle^{4}_\nu\big|(J_{\nu,n,0}^{\otimes2}\mcb W_\nu\upxi_\epnu^i)(w_1,z_1)\big|\\
    &\qquad\qquad\qquad\qquad\times\langle z_1;z_2\rangle^{4}_\nu\big|\dot G_\nu(z_1-z_2)\big|\times\langle w_2\rangle^{-\kappa_0}\langle z_2;w_2\rangle_\nu\big|(J_{\nu,n,0}^{\otimes2}\mcb W_\nu\upxi_\epnu^{j})(z_2,w_2)\big|\dint z_1\dint z_2\dint w_2\\
    &\lesssim\nu^2\|\upxi_\epnu^i\|_{\nu,\kappa_0,0,5}\|\langle\bigcdot\rangle^{4}_\nu\dot G_\nu\|_{L^1_{t,x}}\|\upxi_\epnu^{j}\|_{\nu,\kappa_0,0,2}\lesssim\talloblong\upxi\talloblong^{i+1}\talloblong\upxi\talloblong^{j+1} \nu^{\deg(i)-\kappa_0+\deg(j)-\kappa_0 +2+1}\\
    &\leq\talloblong\upxi\talloblong^{n+1}\nu^{\deg(n)+1-\kappa_n}\,.
\end{equs}
We have used \eqref{eq:lessJn} to lower the exponent in the Gaussian weights and the fact that $J_{\mu,0}^{\otimes2}\mcb W_\mu \upxi_\epmu^n=K_\mu^{\otimes2}\upxi_\epmu^n$.\\

\noindent\textit{Term 3: cross-term $\upxi_\nu\dot M_\nu\upiota_\nu$.} To control this term, we crucially need Lemma \ref{lemma:estimMmu}. We need to bound
\begin{equs}
    ~&\sup_{w_1\in\Lambda}\int_{\Lambda^3}\langle t_1\rangle^{-\kappa_n}\langle x_{1}\rangle^{-2-\kappa_n}\langle w_1;w_2\rangle^{1-\kappa_n}_\nu\big|\big(J_{\nu,n,0}^{\otimes2}\mcb W_\nu\upxi_\epnu^i\big)(w_1,z_1)\big|\\
    &\qquad\qquad\qquad\qquad\qquad\qquad\quad~~~\times\big|\big((\mcT_{\nu,0}\otimes\mcT_{\nu,j})\dot M_\nu\big)(z_1,z_2)\big|\times\big|\big(J_{\nu,j,n}^{\otimes2}\upiota_\epnu^{j}\big)(z_2,w_2)\big|\dint z_1\dint z_2\dint w_2\\
    &\lesssim\sup_{w_1\in\Lambda}\int_{\Lambda^3}\langle t_1\rangle^{-\kappa_n}\langle x_{1}\rangle^{-2-\kappa_n}\langle w_{1}\rangle^{\kappa_0}\langle w_{1}\rangle^{-\kappa_0}\langle w_1;w_2\rangle^{1-\kappa_n}_\nu\big|\big(K_{\nu}^{\otimes2}\upxi_\epnu^i\big)(w_1,z_1)\big|\\
    &\qquad\qquad\qquad\times\big|\big((\mcT_{\nu,0}\otimes\mcT_{\nu,n-1-i})\dot M_\nu\big)(z_1,z_2)\big|\\
    &\qquad\qquad\qquad\qquad\qquad\quad\times\langle s_2\rangle^{\kappa_j}\langle y_{2}\rangle^{2+\kappa_j}\langle s_2\rangle^{-\kappa_j}\langle y_2\rangle^{-2-\kappa_j}\big|\big(J_{\nu,j,n}^{\otimes2}\upiota_\epnu^{n-1-i}\big)(z_2,w_2)\big|\dint z_1\dint z_2\dint w_2\,.
\end{equs}
As usual, we decompose
\begin{equs}
    \langle w_1;w_2\rangle^{1-\kappa_n}_\nu\lesssim\langle w_1;z_1\rangle_\nu\langle z_1;z_2\rangle_\nu\langle z_2;w_2\rangle^{1-\kappa_j}_\nu\,.
\end{equs}
Furthermore, we have, proceeding as in \eqref{eq:boundy_2},
\begin{equs}
    \langle y_2\rangle^{2+\kappa_j}\lesssim\langle y_2;y_1\rangle_\nu^3\langle y_1;x_1\rangle_\nu^3\langle x_{1}\rangle^{2+\kappa_j}\,.
\end{equs}
We bound the term in $s_2$ with the same argument. Note that this bound is a bit suboptimal but this is not a problem since we can absorb any polynomial growth on the translation-invariant terms. Our quantity of interest is thus bounded by
\begin{equs}
    ~&\sup_{w_1\in\Lambda}\int_{\Lambda^3}\langle w_{1}\rangle^{-\kappa_0}\big|\langle \bigcdot\rangle^5_\nu\big(J_{\nu,n,0}^{\otimes2}\mcb W_\nu\upxi_\epnu^i\big)(w_1,z_1)\big|\\
    &\qquad\quad\times\big|\langle\bigcdot\rangle^5_\nu\big((\mcT_{\nu,0}\otimes\mcT_{\nu,j})\dot M_\nu\big)(z_1,z_2)\big|\times\langle s_2\rangle^{-\kappa_j}\langle y_2\rangle^{-2-\kappa_j}\big|\langle \bigcdot\rangle^{1-\kappa_j}_\nu\big(J_{\nu,j,n}^{\otimes2}\upiota_\epnu^{j}\big)(z_2,w_2)\big|\dint z_1\dint z_2\dint w_2\\
    &\lesssim\nu\|\upxi^i_\epnu\|_{\nu,\kappa_0,0,6}\times\sup_{z_2}\int_{\Lambda}\langle s_2\rangle^{-\kappa_j}\langle y_2\rangle^{-2-\kappa_j}\langle z_2;w_2\rangle^{1-\kappa_j}_\nu\big|\big(J_{\nu,j,n}^{\otimes2}\upiota^{j}_\epnu\big)(z_2,w_2)\big|\dint w_2\\
    &\lesssim\talloblong\upxi\talloblong^{i+1}\talloblong\upxi\talloblong^{j+1}\nu^{1+\deg(i)-\kappa_i+\deg(j)+2-\kappa_j+2}\leq\talloblong\upxi\talloblong^{n+1}\nu^{\deg(n)+1-\kappa_n}\,.
\end{equs}

\noindent\textit{Term 4: cross-term $\upiota_\nu\dot M_\nu\upxi_\nu$.} We want to bound
\begin{equs}
    ~&\sup_{w_1\in\Lambda}\int_{\Lambda^3}\langle t_1\rangle^{-\kappa_n}\langle x_{1}\rangle^{-2-\kappa_n}\langle w_1;w_2\rangle^{1-\kappa_n}_\nu\big|\big(J_{\nu,n,i}^{\otimes2}\upiota_\epnu^i\big)(w_1,z_1)\big|\\
    &\qquad\qquad\qquad\qquad\qquad\quad\times\big|\big((\mcT_{\nu,i}\otimes\mcT_{\nu,0})\dot M_\nu\big)(z_1,z_2)\big|\times\big|\big(J_{\nu,0,n}^{\otimes2}\mcb W_\nu\upxi_\epnu^j\big)(z_2,w_2)\big|\dint z_1\dint z_2\dint w_2\\
    &\lesssim\sup_{w_1\in\Lambda}\int_{\Lambda^3}\langle t_1\rangle^{-\kappa_0}\langle x_{1}\rangle^{-\kappa_0}\langle t_1\rangle^{-\kappa_i}\langle x_{1}\rangle^{-2-\kappa_i}\langle w_1;w_2\rangle^{1-\kappa_n}_\nu\big|\big(J_{\nu,n,i}^{\otimes2}\upiota_\epnu^i\big)(w_1,z_1)\big|\\
    &\qquad\qquad\quad\times\big|\big((\mcT_{\nu,i}\otimes\mcT_{\nu,0})\dot M_\nu\big)(z_1,z_2)\big|\times\langle z_{2}\rangle^{\kappa_0}\langle z_2\rangle^{-\kappa_0}\big|\big(J_{\nu,0,n}^{\otimes2}\mcb W_\nu\upxi_\epnu^j\big)(z_2,w_2)\big|\dint z_1\dint z_2\dint w_2\,.
\end{equs}
We have
\begin{equs}
    \langle z_{2}\rangle^{\kappa_0}\lesssim\langle z_2;z_1\rangle_\nu \langle z_1;w_1\rangle^{\kappa_0}_\nu\langle t_1\rangle^{\kappa_0}\langle x_1\rangle^{\kappa_0}\,,
\end{equs}
and
\begin{equs}
    \langle w_1;w_2\rangle^{1-\kappa_n}_\nu\lesssim \langle w_1;z_1\rangle^{1-\kappa_{i+1}}_\nu\langle z_1;z_2\rangle_\nu\langle z_2;w_2\rangle_\nu\,,
\end{equs}
since $i\leq n-1$. Using that $\kappa_0-\kappa_{i+1}\leq-\kappa_i$, the above is bounded by
\begin{equs}
    ~&\sup_{w_1\in\Lambda}\int_{\Lambda^3}\langle t_1\rangle^{-\kappa_i}\langle x_{1}\rangle^{-2-\kappa_i}\langle w_1;z_1\rangle^{1-\kappa_i}_\nu\big|\big(J_{\nu,n,i}^{\otimes2}\upiota_\epnu^i\big)(w_1,z_1)\big|\\
    &\qquad\qquad\times\big|\big(\langle\bigcdot\rangle_\nu^2(\mcT_{\nu,i}\otimes\mcT_{\nu,0})\dot M_\nu\big)(z_1,z_2)\big|\times\langle z_2\rangle^{-\kappa_0}\big|\langle \bigcdot\rangle_\nu\big(J_{\nu,0,n}^{\otimes2}\mcb W_\nu\upxi_\epnu^j\big)(z_2,w_2)\big|\dint z_1\dint z_2\dint w_2\,.
\end{equs}
One obtains, using \eqref{lemma:estimMmu}, that we can control the above by
\begin{equs}
    ~&\nu \sup_{w_1\in\Lambda}\int_{\Lambda^3}\langle t_1\rangle^{-\kappa_i}\langle x_{1}\rangle^{-2-\kappa_i}\big|\langle \bigcdot\rangle^{1-\kappa_i}_\nu\big(J_{\nu,n,i}^{\otimes2}\upiota_\epnu^i\big)(w_1,z_1)\big|\dint z_1\\
    &\quad\times\sup_{z_2\in\lambda}\int_\Lambda\langle z_2\rangle^{-\kappa_0}\big|\langle \bigcdot\rangle_\nu\big(J_{\nu,n,0}^{\otimes2}\mcb W_\nu\upxi_\epnu^j\big)(z_2,w_2)\big|\dint w_2\\
    &\lesssim\talloblong\upxi\talloblong^{i+1}\talloblong\upxi\talloblong^{j+1}\nu^{1+\deg(i)+2-\kappa_i+\deg(j)-\kappa_j+2}\leq\talloblong\upxi\talloblong^{n+1}\nu^{\deg(n)+1-\kappa_n}\,.
\end{equs}
Note that the extra $\kappa_0$ that appears here is the reason why we need an increasing number of $\kappa_0$'s on the weights in the induction.\\

\noindent\textit{Term 5: bilinear in $\upiota_\nu$.} For the last term, we need to be a bit careful. Indeed, one cannot naively use the induction hypothesis \eqref{eq:HR+2}, else we would get a $+4$ extra power counting in $\nu$, but we would pay a quartic polynomial growth, which we cannot control properly. The solution out of this conundrum is to use \eqref{eq:HR+0} on one side and \eqref{eq:HR+2} on the other. We need to control the following quantity
\begin{equs}
    ~&\sup_{w_1\in\Lambda}\int_{\Lambda^3}\langle t_1\rangle^{-\kappa_n}\langle x_{1}\rangle^{-2-\kappa_n}\langle w_1;w_2\rangle^{1-\kappa_n}_\nu\big|\big(J_{\mu,n,i}^{\otimes2}\upiota^i_{\epmu}\big)(w_1,z_1)\big|\\
    &\qquad\qquad\qquad\qquad\qquad\qquad\quad~\times\big|\big((\mcT_{\nu,i}\otimes\mcT_{\nu,j+1})\dot M_\nu\big)(z_1,z_2)\big|\times\big|\big(J_{\mu,j+1,n}^{\otimes2}\upiota^{j}_{\epmu}\big)(w_1,z_1)\big|\dint z_1\dint z_2\dint w_2\\
    &\lesssim\sup_{w_1\in\Lambda}\int_{\Lambda^3}\langle t_1\rangle^{-\kappa_n+\kappa_i}\langle t_1\rangle^{-\kappa_i}\langle x_{1}\rangle^{-2-\kappa_i}\langle w_1;z_1\rangle^{1-\kappa_{i+1}}_\nu\big|\big(J_{\mu,n,i}^{\otimes2}\upiota^i_{\epmu}\big)(w_1,z_1)\big|\\
    &\quad\times\big|\big(\langle\bigcdot\rangle_\nu(\mcT_{\nu,i}\otimes\mcT_{\nu,j+1})\dot M_\nu\big)(z_1,z_2)\big|\times\langle s_2\rangle^{\kappa_j}\langle s_2\rangle^{-\kappa_j}\langle z_2;w_2\rangle^{1-\kappa_j}_\nu\big|\big(J_{\mu,j+1,n}^{\otimes2}\upiota^{j}_{\epmu}\big)(z_2,w_2)\big|\dint z_1\dint z_2\dint w_2\\
    &\lesssim\sup_{w_1\in\Lambda}\int_{\Lambda^3}\langle t_1\rangle^{-\kappa_n+\kappa_i+\kappa_j}\langle x_{1}\rangle^{-2-\kappa_i}\langle w_1;z_1\rangle^{1-\kappa_i}_\nu\big|\big(J_{\mu,n,i}^{\otimes2}\upiota^i_{\epmu}\big)(w_1,z_1)\big|\\
    &\qquad\quad~~\times\big|\big(\langle\bigcdot\rangle_\nu^2(\mcT_{\nu,i}\otimes\mcT_{\nu,j+1})\dot M_\nu\big)(z_1,z_2)\big|\times\langle s_2\rangle^{-\kappa_j}\langle z_2;w_2\rangle^{1-\kappa_j}_\nu\big|\big(J_{\mu,j+1,n}^{\otimes2}\upiota^{j}_{\epmu}\big)(z_2,w_2)\big|\dint z_1\dint z_2\dint w_2\\
    &\lesssim\sup_{w_1\in\Lambda}\int_{\Lambda^3}\langle t_1\rangle^{-\kappa_i}\langle x_{1}\rangle^{-2-\kappa_i}\langle w_1;z_1\rangle^{1-\kappa_i}_\nu\big|\big(J_{\mu,n,i}^{\otimes2}\upiota^i_{\epmu}\big)(w_1,z_1)\big|\\
    &\qquad\quad~~\times\big|\big(\langle\bigcdot\rangle_\nu^2(\mcT_{\nu,i}\otimes\mcT_{\nu,j+1})\dot M_\nu\big)(z_1,z_2)\big|\times\langle s_2\rangle^{-\kappa_j}\langle z_2;w_2\rangle^{1-\kappa_j}_\nu\big|\big(J_{\mu,j+1,n}^{\otimes2}\upiota^{j}_{\epmu}\big)(z_2,w_2)\big|\dint z_1\dint z_2\dint w_2\\
    &\lesssim\nu\sup_{w_1\in\Lambda}\int_{\Lambda}\langle t_1\rangle^{-\kappa_i}\langle x_{1}\rangle^{-2-\kappa_i}\langle w_1;z_1\rangle^{1-\kappa_i}_\nu\big|\big(J_{\mu,i}^{\otimes2}\upiota^i_{\epmu}\big)(w_1,z_1)\big|\dint z_1\\
    &\qquad\times\sup_{z_2\in\Lambda}\int_\Lambda\langle s_2\rangle^{-\kappa_j}\langle z_2;w_2\rangle^{1-\kappa_j}_\nu\big|\big(J_{\mu,j+1}^{\otimes2}\upiota^{j}_{\epmu}\big)(z_2,w_2)\big|\dint w_2\\
    &\lesssim\talloblong\upxi\talloblong^{i+1}\talloblong\upxi\talloblong^{j+1}\nu^{1+\deg(i)+2-\kappa_i+\deg(j)-2\kappa_j}\lesssim\talloblong\upxi\talloblong^{n+1}\nu^{\deg(n)+1-\kappa_n}\,.
\end{equs}
Here pay attention to the fact that we relied on $4^i+2\times 4^j\leqslant 3\times 4^{n-1}\leqslant 4^n$. We have also used the inequalities
\begin{equs}
    \langle w_1;w_2\rangle^{1-\kappa_n}_\nu\lesssim\langle w_1;z_1\rangle^{1-\kappa_{i+1}}_\nu\langle z_1;z_2\rangle_\nu\langle z_2;w_2\rangle^{1-\kappa_j}_\nu\,,
\end{equs}
and
\begin{equs}
    \langle s_2\rangle^{\kappa_j}\lesssim\langle z_2;z_1\rangle_\nu\langle z_1;w_1\rangle_\nu^{\kappa_j}\langle t_1\rangle^{\kappa_j}\,,
\end{equs}
thus concluding the proof.
\end{proof}

\section{Probabilistic analysis and renormalisation}\label{sec:stocest}

This section is devoted to the proof of the stochastic estimate stated in Proposition \ref{prop:stoestimate}. The framework and ideas of proof are now quite classical and we follow \cite{DGR23}. Their approach allows us to replace the so-called \textit{generalised force coefficients} by the use of weights. We adapt their approach to our context with choices of norms inspired by \cite{CF24a, CF24b}.

\subsection{Induction on cumulants}

In this section we carry out the induction to get bounds on the cumulants of the stationary force coefficients. We start by recalling the definition of the joint cumulant of random variables.
\begin{definition}
Let $m\in\N$ and $X_1,\dots,X_m$ real-valued random variables. Their joint cumulant is defined as 
\begin{equation}
\mcK(X_1,\dots,X_m)=(-\i)^m\D_{t_1}\dots\D_{t_m}\log\left.\E\big[\exp(\i (t_1X_1+\dots+t_mX_m))\big]\right|_{t_1=\dots=t_m=0}.
\end{equation}
\end{definition}
Note that the latter is symmetric in its arguments, so that in only depends on the collection $\{X_i\colon i\in[m]\}$. We can now introduce the main object of this section.

\begin{definition}
For $\mu\in (0,1]$, $\eps>0$, $m\in\N$, and $n_1,\dots,n_m\in\N$, we let the cumulant
\begin{equation}
\mcK_{\eps,\mu}^{\n^m}(\z_1,\dots,\z_m)=\mcK\left(\upxi_{\eps,\mu}^{n_1}(\z_1),\dots,\upxi_{\eps,\mu}^{n_m}(\z_m)\right),
\end{equation}
where we have denoted $\n^m=(n_1,\dots,n_m)$, and $\z_j=(z_1^j,z^j_2)$ for $1\leq j\leq m$.
\end{definition}
\begin{example}\label{examplecumulant}
    Again, we give more graphical examples. We take $\n^m=(1,1)$.
    \begin{equs}
        \mcK_{\eps,\mu}^{(1,1)}(\z_1,\z_2)=\mcK\Bigg(\begin{tikzpicture}[scale=0.4,baseline=0.15cm]
            \draw[line width=0.25mm] (0,0.25) node[dot] {} -- (0.5,1.25) node[dot] {};
            \draw (0,-0.25) node[scale=0.7] {$z_1^1$};
            \draw (0.5, 1.8) node[scale=0.7] {$z_2^1$};
\end{tikzpicture},\begin{tikzpicture}[scale=0.4,baseline=0.15cm]
            \draw[line width=0.25mm] (0,0.25) node[dot] {} -- (0.5,1.25) node[dot] {};
            \draw (0,-0.25) node[scale=0.7] {$z_1^2$};
            \draw (0.5, 1.8) node[scale=0.7] {$z_2^2$};
\end{tikzpicture}\Bigg)=\begin{tikzpicture}[scale=0.4,baseline=0.15cm]
            \draw[line width=0.25mm] (0,0.25) node[dot] {} -- (0,1.25) node[dot] {};
            \draw (0,-0.25) node[scale=0.7] {$z_1^1$};
            \draw (0,1.8) node[scale=0.7] {$z_2^1$};
            \draw[line width=0.25mm] (1.25,0.25) node[dot] {} -- (1.25,1.25) node[dot] {};
            \draw (1.25,-0.25) node[scale=0.7] {$z_1^2$};
            \draw (1.25,1.8) node[scale=0.7] {$z_2^2$};
            \draw[line width=0.25mm, dashed] (0,0.25) -- (1.25,1.25);
            \draw[line width=0.25mm, dashed] (0,1.25) -- (1.25,0.25);
\end{tikzpicture}+\begin{tikzpicture}[scale=0.4,baseline=0.15cm]
            \draw[line width=0.25mm] (0,0.25) node[dot] {} -- (0,1.25) node[dot] {};
            \draw (0,-0.25) node[scale=0.7] {$z_1^1$};
            \draw (0, 1.8) node[scale=0.7] {$z_2^1$};
            \draw[line width=0.25mm] (1.25,0.25) node[dot] {} -- (1.25,1.25) node[dot] {};
            \draw (1.25,-0.25) node[scale=0.7] {$z_1^2$};
            \draw (1.25,1.8) node[scale=0.7] {$z_2^2$};
            \draw[line width=0.25mm, dashed] (0,0.25) -- (1.25,0.25);
            \draw[line width=0.25mm, dashed] (0,1.25) -- (1.25,1.25);
\end{tikzpicture}
    \end{equs}
    The dashed lines represent the kernel given by the covariance of the noise. We have used Wick's formula to get the pairings. We can also see graphically that taking the cumulants allows to keep only the connected Feynman diagrams.
\end{example}
We set the notation
\begin{equation}
\big(K_\mu^{\otimes 2m}\mcK^{\n^m}\big)(\w_1,\dots,\w_m)=\int_{\SS^{2m}} \prod_{j=1}^mK_\mu(w_1^j-z_1^j)K_\mu(w_2^j-z^j_2) \mcK^{\n^m}(\z)\dint\z.
\end{equation}
Before giving a norm for such objects, let us introduce some weights. For $w_1,w_2\in\SS$, we set $\langle w_1;w_2\rangle_\mu=\langle w_1-w_2\rangle_\mu$. We recall that we defined $\langle\bigcdot\rangle_\mu$ in \eqref{mubracket}. Moreover, for $\w\in\SS^{2m}$, we set
\begin{equation}
\langle\w\rangle_{\mu,m}=\prod_{j=1}^m\langle w_1^j;w_2^j\rangle_\mu.
\end{equation}
\begin{definition}
For $U\in\CS'(\SS^{2m})$, we set a norm, that is $L^\infty_{w_1^1}$ in $w_1^1$, $L^{\infty,p}_{t_i,x_i}$ in the $w_i$'s, where $w_i=(t_i,x_i)$, and $L^1_{w^{1:m}_2}$ in all the others. We have let for the rest of Section \ref{sec:stocest}, $p=1\vee d/(2\alpha-2\beta)$.
\begin{equation}
\|U\|_{\mcV^m}=\sup_{w_1^1\in\SS}\sup_{t_2,\dots, t_m\in \R}\left(\int_{(\R^d)^{m-1}}\left(\int_{\SS^{m}}\big|U(\w)|\dint w^{1:m}_2\right)^{p}\dint x^{2:m}_1\right)^{\frac{1}{p}}\,.
\end{equation}
We further denote by $\mcV^m$ the vector space of distributions of $2m$ variables such that the above norm is finite. We consequently define the following triple norm.
\begin{equation}
\nnorm{\mcK_{\mu}^{\n^m}}_\mu=\Big\|\langle\bigcdot\rangle_{\mu,m}^{\alpha-\beta}\left(K_\mu^{\otimes 2m}\mcK^{\n^m}\right)\Big\|_{\mcV^m}\,,
\end{equation}
where the product with the weight is pointwise.
\begin{remark}
    The power $\alpha-\beta$ is arbitrary. One could have instead chosen any $0<\varsigma<\alpha$ with $\varsigma$ close enough to $\alpha$, but we decide not to add any supplementary notation to avoid to create any confusion with the deterministic analysis.
\end{remark}
\end{definition}
\begin{example}
We give another example for a graphical representation of this norm.
\begin{equs}
    \nnorm{\mcK_{\eps,\mu}^{(1,1)}}_\mu=\mcK\Bigg(\begin{tikzpicture}[scale=0.4,baseline=0.15cm]
            \draw[line width=0.3mm, color=indigo] (0.5,1.25) -- (0.5,2.25);
            \draw[line width=0.3mm, color=indigo] (0,0.25) -- (0,-0.75);
            \draw[line width=0.25mm] (0,0.25) node[dot] {} -- (0.5,1.25) node[dot] {};
            \draw (0,-1.3) node[scale=0.7] {$L^\infty_{w_1^1}$};
            \draw (0,2.75) node[scale=0.7] {$L^1_{w_2^1}$};
\end{tikzpicture},\begin{tikzpicture}[scale=0.4,baseline=0.15cm]
            \draw[line width=0.3mm, color=indigo] (0.5,1.25) -- (0.5,2.25);
            \draw[line width=0.3mm, color=indigo] (0,0.25) -- (0,-0.75);
            \draw[line width=0.25mm] (0,0.25) node[dot] {} -- (0.5,1.25) node[dot] {};
            \draw (0,-1.3) node[scale=0.7] {$L^{\infty,p}_{w_1^2}$};
            \draw (0,2.75) node[scale=0.7] {$L^1_{w_2^2}$};
\end{tikzpicture}\Bigg)\,.
\end{equs}
We have also omitted the weights in this representation.
\end{example}
We let
\begin{equation}
\deg(\n^m)=\sum_{j=1}^m\deg(n_j)+(m-1)\frac{d}{p}\,.
\end{equation}
We can now define two operations that will be used to graft cumulants (seen as Feynman diagrams) respectively on themselves and between each other.
\begin{definition}
We define $A_j:\mcS'(\SS)\times\mcV^{m+1}\to\mcV^m$ and $B_j:\mcS'(\SS)\times\mcV^{m_1+1}\times\mcV^{m_2+1}\to\mcV^m$, with $m_1+m_2=m+1$ by
\begin{equation}
A_j(G,U)(\w_1,\dots,\w_m)=\int_{\SS^2}U\big(\w_1,\dots,(w_1^j,v_1),(v_2,w_2^j),\dots,\w_m\big)G(v_1-v_2)\dint v_1\dint v_2\,.
\end{equation}
\begin{multline}
B_j(G,U,V)(\w_1,\dots,\w_m)=\int_{\SS^2}U\big(\w_1,\dots,(w_1^j,v_1),\dots,\w_{m_1}\big)G(v_1-v_2)\\\times V\big((v_2,w_2^j),\w_{m_1+1},\dots,\w_{m_1+m_2+1}\big)\dint v_1\dint v_2\,.
\end{multline}
\end{definition}
\begin{lemma}\label{estimateAB}
The following estimates hold.
\begin{equation}
    \|A_j(G,U)\|_{\mcV^m}\lesssim\|G\|_{L^{1,q}_{t,x}}\|U\|_{\mcV^{m+1}}\,.
\end{equation}
We have denoted by $q$ the Hölder conjugate of $p$, \ie the only $q>1$ such that $1/p+1/q=1$, with the convention that $q=\infty$ if $p=1$. Moreover, with $m_1+m_2=m+1$,
\begin{equation}
    \|B_j(G,U,V)\|_{\mcV^m}\lesssim \|U\|_{\mcV^{m_1+1}}\|G\|_{L_{t,x}^1}\|V\|_{\mcV^{m_2+1}}\,.
\end{equation}
\end{lemma}

\begin{proof}
Let us begin with the estimate on $A_j$. We will keep in a generic way the notations $w=(t,x)$ and $v=(r,u)$. We have, by doing a $L^\infty-L^1$ Hölder inequality in $v_1$, then a $L^p-L^q$ Hölder inequality in $u_2$, and a $L^\infty-L^1$ Hölder inequality in $r_2$,
\begin{equs}
~&\|A_j(G,U)\|_{\mcV^m}\\&=\sup_{w_1}\sup_{t_{2:m}}\left(\int_{(\R^d)^{m-1}}\left(\int_{\SS^m}\left|\int_{\SS^2}U\big(\w_1,.,(w_1^j,v_1),(v_2,w_2^j),.,\w_m\big)G(v_1-v_2)\dint v_1\dint v_2\right|\dint w_2^{1:m}\right)^p\dint x_1^{2:m}\right)^\frac{1}{p}\\
&\leq\sup_{w_1}\sup_{t_{2:m}}\left(\int_{(\R^d)^{m-1}}\left(\int_{\SS}\int_{\SS^{m+1}}|U\big(\dots\big)|\dint v_1\dint w_2^{1:m}\times\sup_{v_1}|G(v_1-v_2)|\dint v_2\right)^p\dint x_1^{2:m}\right)^\frac{1}{p}\\
&\leq\sup_{w_1}\sup_{t_{2:m}}\Biggl(\int_{(\R^d)^{m-1}}\Biggl(\sup_{r_2}\Biggl(\int_{\R^d}\left(\int_{\SS^{m+1}}|U\big(\dots\big)|\dint v_1\dint w_2^{1:m}\right)^p\dint u_2\Biggl)^\frac{1}{p}\times\\
&~~~~~~~~\int_\R\left(\int_{\SS}\sup_{v_1}|G(v_1-v_2)|^q\dint u_2\right)^\frac{1}{q}\dint r_2\Biggl)^p\dint x_1^{2:m}\Biggl)^\frac{1}{p}\\
&\leq\int_\R\left(\int_{\SS}\sup_{v_1}|G(v_1-v_2)|^q\dint u_2\right)^\frac{1}{q}\dint r_2\sup_{w_1}\sup_{t_{2:m}}\sup_{r_2}\left(\int_{(\R^d)^{m}}\left(\int_{\SS^{m+1}}|U\big(\dots\big)|\dint v_1\dint w_2^{1:m}\right)^p\dint x_1^{2:m}\dint u_2\right)^\frac{1}{p}\\
&=\int_\R\left(\int_{\SS} |G(v_2)|^q\dint u_2\right)^\frac{1}{q}\dint r_2\times \|U\|_{\mcV^{m+1}},
\end{equs}
which concludes the proof of the first estimate. For the second one, the proof is slightly less involved, since we just have to take a supremum in $v_1$ in $G$, and a supremum in $v_2$ in $V$. We get
\begin{equs}
~&\|B_j(G,U,V)\|_{\mcV^m}\\
&=\sup_{w_1}\sup_{t_{2:m}}\Biggl(\int_{(\R^d)^{m-1}}\Biggl(\int_{\SS^m}\Biggl|\int_{\SS^2}U\big(\w_1,\dots,(w_1^j,v_1),\dots,\w_{m_1}\big)G(v_1-v_2)\times\\
&~~~~~~~~ V\big((v_2,w_2^j),\w_{m_1+1},\dots,\v_{m_1+m_2+1}\big)\dint v_1\dint v_2\Biggl|\dint w_2^{1:m}\Biggl)^p\dint x_1^{2:m}\Biggl)^\frac{1}{p}\\
&\leq\sup_{w_1}\sup_{t_{2:m}}\Biggl(\int_{(\R^d)^{m-1}}\Biggl(\int_{\SS^m}\int_{\SS^2}|U\big(\w_1,\dots,(w_1^j,v_1),\dots,\w_{m_1}\big)|\sup_{v_1}|G(v_1-v_2)|\times\\
&~~~~~~~~\sup_{v_2}|V\big((v_2,w_2^j),\w_{m_1+1},\dots,\v_{m_1+m_2+1}\big)|\dint v_1\dint v_2\dint w_2^{1:m}\Biggl)^p\dint x_1^{2:m}\Biggl)^\frac{1}{p}\\
&=\sup_{w_1}\sup_{t_{2:m_1}}\Biggl(\int_{(\R^d)^{m_1}}\Biggl(\int_{\SS^{m_1+1}}|U\big(\dots\big)|\dint v_1\dint w_2^{1:j-1}\dint w_2^{j+1:m_1}\Biggl)^p\dint x_1^{2:m_1}\Biggl)^\frac{1}{p}\times\sup_{v_1}\int_{\SS}|G(v_1-v_2)|\dint v_2\times\\
&~~~~\sup_{v_2}\sup_{t_{m_1+1:m_1+m_2+1}}\Biggl(\int_{(\R^d)^{m_2}}\Biggl(\int_{\SS^{m_2+1}}|V\big(\dots\big)|\dint w_2^j\dint w_2^{m_1+1:m_1+m_2+1}\Biggl)^p\dint x_1^{m_1+1:m_1+m_2+1}\Biggl)^\frac{1}{p}\\
&=\|U\|_{\mcV^{m_1+1}}\|G\|_{L_{t,x}^1}\|V\|_{\mcV^{m_2+1}}.
\end{equs}
\end{proof}
\begin{lemma}\label{estimateABfancy}
The following estimates hold.
\begin{equation}
\nnorm{A_j\big(\dot G_\mu,\mcK_\mu^{\n^m}\big)}_\mu\lesssim\mu^{-1+\alpha-d/p}\nnorm{\mcK_{\mu}^{\n^m}}_\mu.
\end{equation}
We have denoted by $q$ the Hölder conjugate of $p$, \ie the only $q>1$ such that $1/p+1/q=1$, with the convention that $q=\infty$ if $p=1$. Moreover, with $m_1+m_2=m+1$,
\begin{equation}
\nnorm{B_j\big(\dot G_\mu,\mcK_\mu^{\n^{m_1}},\mcK_\mu^{\n^{m_2}}\big)}_\mu\lesssim \mu^{-1+\alpha}\nnorm{\mcK_\mu^{\n^{m_1}}}_\mu\nnorm{\mcK_\mu^{\n^{m_2}}}_\mu.
\end{equation}
\end{lemma}

\begin{proof}
We begin by recalling the estimate \eqref{triangleineq} that gives
\begin{equs}
\langle w_1^j;w_2^j\rangle_{\mu}^{\alpha-\beta}&\lesssim\langle w_1^j;v_1\rangle_{\mu}^{\alpha-\beta}\langle v_1;v_2\rangle_{\mu}^{\alpha-\beta}\langle v_2;w_2^j\rangle_{\mu}^{\alpha-\beta}.
\end{equs}
Using it and the fact that $\mcR^*_\mu K_\mu=\delta_0$, we get that for $U\in\mcV^{m+1}$,
\begin{equs}
\langle\bigcdot\rangle_{\mu,m}^{\alpha-\beta} \Big(K_\mu^{\otimes 2m}A_j(\dot G_\mu,U)\Big)=A_j\Big(\langle\bigcdot\rangle_\mu^{\alpha-\beta}\mcR_\mu^2 \dot G_\mu,\langle\bigcdot\rangle_{\mu,m+1}^{\alpha-\beta}\big(K_\mu^{\otimes2(m+1)}U\big)\Big).
\end{equs}
and for $U\in\mcV^{m_1+1}$ and $V\in\mcV^{m_2+1}$ with $m_1+m_2=m+1$,
\begin{equs}
\langle\bigcdot\rangle_{\mu,m}^{\alpha-\beta}\Big(K_\mu^{\otimes 2m}B_j(\dot G_\mu,U,V)\Big)=B_j\Big(\langle\bigcdot\rangle_\mu^{\alpha-\beta}\mcR_\mu^2\dot G_\mu,\langle\bigcdot\rangle_{\mu,m_1}^{\alpha-\beta}\big(K_\mu^{\otimes 2(m_1+1)}U\big),\langle\bigcdot\rangle_{\mu,m_2}^{\alpha-\beta}\big(K_\mu^{\otimes 2(m_2+1)}V\big)\Big).
\end{equs}
From that, using the fact that $\|\langle\bigcdot\rangle_\mu^{\alpha-\beta}\mcR_\mu^2\dot G_\mu\|_{L_{t,x}^1}\lesssim\mu^{-1+\alpha}$ and $\|\langle\bigcdot\rangle_\mu^{\alpha-\beta}\mcR_\mu^2\dot G_\mu\|_{L_{t,x}^{1,q}}\lesssim\mu^{-1+\alpha-d/p}$ (see Lemma \ref{estimdotGmu}), where we recall that $1/p+1/q=1$, we get, using Lemma \ref{estimateAB},
\begin{equs}
\Big\|\langle\bigcdot\rangle_{\mu,m}^{\alpha-\beta}\Big( K_\mu^{\otimes 2m}A_j(\dot G_\mu,U)\Big)\Big\|_{\mcV^{m}}\lesssim\mu^{-1+\alpha-d/p}\Big\|\langle\bigcdot\rangle_{\mu,m+1}^{\alpha-\beta}\Big(K_\mu^{\otimes 2(m+1)}U\Big)\Big\|_{\mcV^{m+1}},
\end{equs}
and
\begin{multline*}
    \Big\|\langle\bigcdot\rangle_{\mu,m}^{\alpha-\beta}\Big(K_\mu^{\otimes 2m}B_j(G,U,V)\Big)\Big\|_{\mcV^{m}}\lesssim\mu^{-1+\alpha}\Big\|\langle\bigcdot\rangle_{\mu,m_1}^{\alpha-\beta}\Big(K_\mu^{\otimes 2(m_1+1)}U\Big)\Big\|_{\mcV^{m_1+1}}\times\\ \Big\|\langle\bigcdot\rangle_{\mu,m_2}^{\alpha-\beta}\Big(K_\mu^{\otimes 2(m_2+1)}V\Big)\Big\|_{\mcV^{m_2+1}},
\end{multline*}
which allows to conclude.
\end{proof}
We define on give results of Taylor-type expansions. These statements are now classical and we give pointers to existing proofs.
\begin{definition}
Let $U\in\mcV^1$. We define, for $w_1\in\Lambda$,
\begin{equation}
    (\I U)(w_1)=\int_{\SS}U(w_1,w_2)\dint w_2\,,
\end{equation}
as well as, for $k\in\N^{d+1}$ and $w_1,w_2\in\Lambda$,
\begin{equation}
    (\L^k U)(w_1,w_2)=\frac{|k|}{k!}\int_0^1(1-\tau)^{|k|-1}\tau^{-d-\alpha}(X^kU)\big(w_1,w_1+(w_2-w_1)/\tau\big)\dint\tau\,.
\end{equation}
We have used the notation $X^k(w_1,w_2)=(w_1-w_2)^k$.
\end{definition}

\begin{lemma}\label{taylor}
From the integral form of the Taylor formula, we have, for $U\in\mcV^1$ and $w_1,w_2\in\Lambda$,
\begin{equation}
    U(w_1,w_2)=(\I U)(w_1)\delta(w_2-w_1)+\sum_{|k|_\mfs=2}\D^k_{w_2}(\L^kU)(w_1,w_2)\,.
\end{equation}
Moreover, we have the following estimate
\begin{equation}\label{eq:estimI}
    |\I U(w_1)|\leq\|U\|_{\mcV^1}\,.
\end{equation}
\end{lemma}
Before turning to the main result of this section, we recall the crucial property of cumulants of products of random variables. For a proof, see \cite[Proposition 3.2.1]{PT11}.
\begin{lemma}\label{cumuofproduct}
Let $X_1,\dots,X_m$ as above, and $Y$ and $Z$ real random variables, we have
\begin{equation}
\mcK(X_1,\dots,X_m,YZ)=\mcK(X_1,\dots,X_m,Y,Z)+\sum_{I\sqcup J=[m]}\mcK(X_I,Y)\mcK(X_J,Z)\,,
\end{equation}
where we have denoted $X_I=\{X_i\colon i\in I\}$.
\end{lemma}
\begin{proposition}\label{estimatemoment}
There is a choice of constants $\mfc_{\eps}^{n}$ such that for every $m\in\N_{\geq1}$ and $\n^m$,
\begin{equation}
\nnorm{\mcK_{\eps,\mu}^{\n^m}}_\mu \lesssim \mu^{\deg(\n^m)}\,,
\end{equation}
uniformly in $\eps$.
\end{proposition}
The reader can convince him/herself that, from the proof of this proposition, one also gets the following.
\begin{corollary}\label{estimatemomentdmu}
    The estimate above holds true for cumulants of the type
    \begin{equation}
        \mcK\left(\upxi_{\eps,\mu}^{n_1}(\z_1),\dots,\D_\mu\upxi_{\eps,\mu}^{n_j}(\z_j),\dots,\upxi_{\eps,\mu}^{n_m}(\z_m)\right)\,,
    \end{equation}
    with the difference that the bound scales as $\mu^{\deg(\n^m)-1}$.
\end{corollary}

\begin{proof}[of Proposition \ref{estimatemoment}]
\textit{Base case.} We first tackle the term having $\n^2=(0,0)$ (the other cases being trivial as cumulants of order greater than $2$ of Gaussians vanish), that is, denoting $z_i=(s_i,y_i)$ and $w_i=(t_i,x_i)$,
\begin{equs}
\mcK\left(\upxi_{\eps,\mu}^{0}(z_1),\upxi_{\eps,\mu}^{0}(z_2)\right)=\langle\nabla\rangle^{2\alpha-2\beta-d}(y_1-y_2)\,,
\end{equs}
which behaves like $|y_1-y_2|^{-2\alpha+2\beta}$ on small scales. Also note that the latter is independent of $s_1$ and $s_2$. We first need to run some preliminary computations. We denote in a generic way $z=(s,y)$ and $w=(t,x)$. We have
\begin{equs}
\nnorm{\mcK_{\eps,\mu}^{\one^2}}_\mu=\sup_{w_1^1\in\SS}\sup_{t_2\in\R}\Biggl(\int_{\R^d}\Biggl(\int_{\SS^2}\langle w_1^1;w_2^1\rangle_\mu^{\alpha-\beta}\langle w_1^2;w_2^2\rangle_\mu^{\alpha-\beta}\left|\left(K_\mu^{\otimes4}\mcK_{\eps,\mu}^{\one^2}\right)(\w)\right|\dint w_2^1\dint w_2^2\Biggr)^p\dint x_1^2\Biggr)^\frac{1}{p}.
\end{equs}
With the changes of variables $w_2^1\to w_2^1+z_1$ and $w_2^2\to w_2^2+z_2$, the estimate $\langle w_1^i;w_2^i+z_i\rangle_\mu^{\alpha-\beta}\lesssim\langle  w_1^i;z_i\rangle_\mu^{\alpha-\beta}\langle w_2^i\rangle_\mu^{\alpha-\beta}$ for $i\in\{1,2\}$ (see \eqref{estimweight}), and the fact that $\|\langle\bigcdot\rangle_\mu^{\alpha-\beta} K_\mu\|_{L_{t,x}^1}\lesssim1$, we get that
\begin{equs}
~&\int_{\SS^2}\langle w_1^1;w_2^1\rangle_\mu^{\alpha-\beta}\langle w_1^2;w_2^2\rangle_\mu^{\alpha-\beta}\left(K_\mu^{\otimes4}\mcK_{\eps,\mu}^{\one^2}\right)(\w)\dint w_2^1\dint w_2^2\\
&\lesssim\int_{(\R^d)^2}\int_{\R^2}\langle w_1^1;z_1\rangle_\mu^{\alpha-\beta}\langle w_1^2;z_2\rangle_\mu^{\alpha-\beta} K_\mu(w_1^1-z_1)K_\mu(w_1^2-z_2)\dint s_1\dint s_2\mcK_{\eps,\mu}^{\one^2}(y_1,y_2)\dint y_1\dint y_2.
\end{equs}
With the help of Lemma \ref{intKmuweight}, we get that the quantity above is bounded by
\begin{equs}
~&\int_{(\R^d)^2} \langle x_1^1;y_1\rangle_\mu^{\alpha-\beta}\langle x_1^2;y_2\rangle_\mu^{\alpha-\beta}Q_\mu^{-2N}(x_1^1-y_1)Q_\mu^{-2N}(x_1^2-y_2)\mcK_{\eps,\mu}^{\one^2}(y_1,y_2)\dint y_1\dint y_2\\
&\lesssim\int_{\R^d}\langle x_1^1-x_1^2;y\rangle_\mu^{\alpha-\beta}Q_\mu^{-2N}(x_1^1-x_1^2-y)\langle\nabla\rangle^{2\alpha-2\beta-d}(y)\dint y.
\end{equs}
Note that we have used the invariance by translation of $\mcK_{\eps,\mu}^{\one^2}$, and again the fact that $\|\langle\bigcdot\rangle_\mu^{\alpha-\beta} K_\mu\|_{L_{t,x}^1}\lesssim1$. We finally get that 
\begin{equs}
\nnorm{\mcK_{\eps,\mu}^{\one^2}}_\mu&\lesssim\sup_{x_1^1}\left(\int_{\R^d}\left(\int_{\R^d}\langle x_1^1-x_1^2;y\rangle_\mu^{\alpha-\beta}Q_\mu^{-2N}(x_1^1-x_1^2-y)\langle\nabla\rangle^{2\alpha-2\beta-d}(y)\dint y\right)^p\dint x_1^2\right)^\frac{1}{p}\\
&\lesssim\mu^{d/p}\sup_{x_1^1}\left(\int_{\R^d}\left(\int_{\R^d}\langle x_1^1-x_1^2;y\rangle^{\alpha-\beta}\langle\nabla\rangle^{-2N}(x_1^1-x_1^2-y)\langle\nabla\rangle^{2\alpha-2\beta-d}(\mu y)\dint y\right)^p\dint x_1^2\right)^\frac{1}{p}\\
&=\mu^{d/p}\left(\int_{\R^d}\left(\int_{\R^d}\langle x;y\rangle^{\alpha-\beta}\langle\nabla\rangle^{-2N}(x-y)\langle\nabla\rangle^{2\alpha-2\beta-d}(\mu y)\dint y\right)^p\dint x\right)^\frac{1}{p}\\
\end{equs}
One can notice that the integrability of this latter quantity in large volume comes around naturally. To make sense of the convolution above, we require that $N>\beta$. We have chosen $p$ such that $\deg(\one^2)=-2\alpha+2\beta+d/p\leq0$, whether we are in the integrable case or not, which is crucial as we would otherwise get that the norm above is bounded by $1$ only. Finally, we have that
\begin{equs}
\nnorm{\mcK_{\eps,\mu}^{\one^2}}_\mu\lesssim\mu^{-2\alpha+2\beta+d/p},
\end{equs}
which concludes the base case.\\

\noindent\textit{Induction.} For the induction, we write the following flow equation on the cumulants using Lemma \ref{cumuofproduct}.
\begin{equs}\label{inductionmoment}
~&\D_\mu \mcK_{\eps,\mu}^{\n^m}(\z_1,\dots,\z_m)=\sum_{j=1}^m \mcK\left(\upxi_{\eps,\mu}^{n_1}(\z_1),\dots,\D_\mu \upxi_{\eps,\mu}^{n_j}(\z_j),\dots,\upxi_{\eps,\mu}^{n_m}(\z_m)\right)\\
&=-\sum_{j=1}^m\sum_{k=1}^{n_j-1}\int_{\SS^2} \mcK\left(\upxi_{\eps,\mu}^{n_1}(\z_1),\dots, \upxi_{\eps,\mu}^{k}(z_1^j,v_1) \upxi_{\eps,\mu}^{n_j-k}(v_2,z_2^j),\dots,\upxi_{\eps,\mu}^{n_m}(\z_m)\right)\dot G_\mu(v_1-v_2)\dint v_1\dint v_2\\
&=-\sum_{j=1}^m\sum_{k=1}^{n_j-1}\int_{\SS^2} \mcK_{\eps,\mu}^{(n_1,\dots, k,n_j-k,\dots,n_m)}(\z_1,\dots,(z_1^j,v_1),(v_2,z_2^j),\dots,\z_m)\dot G_\mu(v_1-v_2)\\
&~~~~-\sum_{I\sqcup J=\n^m\backslash\{j\}}\mcK_{\eps,\mu}^{(n_{I\cup\{k\}})}(\z_I,(z_1^j,v_1))\dot G_\mu(v_1-v_2)\mcK_{\eps,\mu}^{(n_{J\cup\{n_j-k\}})}((v_2,z_2^j),\z_J)\dint v_1\dint v_2\\
&=-\sum_{j=1}^m\sum_{k=1}^{n_j-1}A_j\bigg(\dot G_\mu,\mcK_{\eps,\mu}^{(n_1,\dots, k,n_j-k,\dots,n_m)}\bigg)-\sum_{I\sqcup J=\n^m\backslash\{j\}}B_j\bigg(\dot G_\mu,\mcK_{\eps,\mu}^{(n_{I\cup\{k\}})},\mcK_{\eps,\mu}^{(n_{J\cup\{n_j-k\}})}\bigg)
\end{equs}
We have used in the first inequality \eqref{eq:Kmunu} and the fact that, for $\nu\leq\mu$, $\langle\bigcdot\rangle_\mu\leq\langle\bigcdot\rangle_\nu$. Note that the cumulants do not depend on the order of their inputs, so that the notation in the penultimate line above is well-defined.

\begin{example}
    We continue Example \ref{examplecumulant} to show how this induction reads on the graphical representation. We have on the first diagram therein
    \begin{equs}
        \D_\mu\begin{tikzpicture}[scale=0.4,baseline=0.15cm]
            \draw[line width=0.25mm] (0,0.25) node[dot] {} -- (0,1.25) node[dot] {};
            \draw (0,-0.25) node[scale=0.7] {$z_1^1$};
            \draw (0, 1.8) node[scale=0.7] {$z_2^1$};
            \draw[line width=0.25mm] (1.25,0.25) node[dot] {} -- (1.25,1.25) node[dot] {};
            \draw (1.25,-0.25) node[scale=0.7] {$z_1^2$};
            \draw (1.25,1.8) node[scale=0.7] {$z_2^2$};
            \draw[line width=0.25mm, dashed] (0,0.25) -- (1.25,0.25);
            \draw[line width=0.25mm, dashed] (0,1.25) -- (1.25,1.25);
\end{tikzpicture}=-\begin{tikzpicture}[scale=0.4,baseline=0.15cm]
            \draw[color=red,line width=0.25mm] (0,0.25) node[dot] {} -- (0,1.25) node[dot] {};
            \draw (0,-0.25) node[scale=0.7] {$z_1^1$};
            \draw (0, 1.8) node[scale=0.7] {$z_2^1$};
            \draw[line width=0.25mm] (1.25,0.25) node[dot] {} -- (1.25,1.25) node[dot] {};
            \draw (1.25,-0.25) node[scale=0.7] {$z_1^2$};
            \draw (1.25,1.8) node[scale=0.7] {$z_2^2$};
            \draw[line width=0.25mm, dashed] (0,0.25) -- (1.25,0.25);
            \draw[line width=0.25mm, dashed] (0,1.25) -- (1.25,1.25);
\end{tikzpicture}-\begin{tikzpicture}[scale=0.4,baseline=0.15cm]
            \draw[line width=0.25mm] (0,0.25) node[dot] {} -- (0,1.25) node[dot] {};
            \draw (0,-0.25) node[scale=0.7] {$z_1^1$};
            \draw (0, 1.8) node[scale=0.7] {$z_2^1$};
            \draw[color=red, line width=0.25mm] (1.25,0.25) node[dot] {} -- (1.25,1.25) node[dot] {};
            \draw (1.25,-0.25) node[scale=0.7] {$z_1^2$};
            \draw (1.25,1.8) node[scale=0.7] {$z_2^2$};
            \draw[line width=0.25mm, dashed] (0,0.25) -- (1.25,0.25);
            \draw[line width=0.25mm, dashed] (0,1.25) -- (1.25,1.25);
\end{tikzpicture}\,.
    \end{equs}
    This one would be cut using a $A$-type operator in the following way
    \begin{equs}
        \begin{tikzpicture}[scale=0.4,baseline=0.15cm]
            \draw[color=red,line width=0.25mm] (0,0.25) node[dot] {} -- (0,1.25) node[dot] {};
            \draw (0,-0.25) node[scale=0.7] {$z_1^1$};
            \draw (0, 1.8) node[scale=0.7] {$z_2^1$};
            \draw[line width=0.25mm] (1.25,0.25) node[dot] {} -- (1.25,1.25) node[dot] {};
            \draw (1.25,-0.25) node[scale=0.7] {$z_1^2$};
            \draw (1.25,1.8) node[scale=0.7] {$z_2^2$};
            \draw[line width=0.25mm, dashed] (0,0.25) -- (1.25,0.25);
            \draw[line width=0.25mm, dashed] (0,1.25) -- (1.25,1.25);
\end{tikzpicture}\rightarrow \begin{tikzpicture}[scale=0.4,baseline=0.15cm]
            \draw[color=white,line width=0.25mm] (0,0.25) node[dot] {} -- (0,1.25) node[dot] {};
            \draw (0,-0.25) node[scale=0.7] {$z_1^1$};
            \draw (0, 1.8) node[scale=0.7] {$z_2^1$};
            \draw[line width=0.25mm] (1.25,0.25) node[dot] {} -- (1.25,1.25) node[dot] {};
            \draw (1.25,-0.25) node[scale=0.7] {$z_1^2$};
            \draw (1.25,1.8) node[scale=0.7] {$z_2^2$};
            \draw[line width=0.25mm, dashed] (0,0.25) -- (1.25,0.25);
            \draw[line width=0.25mm, dashed] (0,1.25) -- (1.25,1.25);
\end{tikzpicture}\,.
    \end{equs}
    This new diagram would represent for instance $\mcK_\epmu^{(0,0,1)}$. Pursuing the induction, this latter diagram would be split using a $B$-type operator, yielding two disconnected diagrams.
    \begin{equs}
        \begin{tikzpicture}[scale=0.4,baseline=0.15cm]
            \draw[color=white,line width=0.25mm] (0,0.25) node[dot] {} -- (0,1.25) node[dot] {};
            \draw (0,-0.25) node[scale=0.7] {$z_1^1$};
            \draw (0, 1.8) node[scale=0.7] {$z_2^1$};
            \draw[color=red, line width=0.25mm] (1.25,0.25) node[dot] {} -- (1.25,1.25) node[dot] {};
            \draw (1.25,-0.25) node[scale=0.7] {$z_1^2$};
            \draw (1.25,1.8) node[scale=0.7] {$z_2^2$};
            \draw[line width=0.25mm, dashed] (0,0.25) -- (1.25,0.25);
            \draw[line width=0.25mm, dashed] (0,1.25) -- (1.25,1.25);\,.
\end{tikzpicture}\rightarrow\begin{tikzpicture}[scale=0.4,baseline=0.15cm]
            \draw[color=white,line width=0.25mm] (0,0.25) node[dot] {} -- (0,1.25) node[dot] {};
            \draw (0,-0.25) node[scale=0.7] {$z_1^1$};
            \draw (0, 1.8) node[scale=0.7] {$z_2^1$};
            \draw[color=white, line width=0.25mm] (1.25,0.25) node[dot] {} -- (1.25,1.25) node[dot] {};
            \draw (1.25,-0.25) node[scale=0.7] {$z_1^2$};
            \draw (1.25,1.8) node[scale=0.7] {$z_2^2$};
            \draw[line width=0.25mm, dashed] (0,0.25) -- (1.25,0.25);
            \draw[line width=0.25mm, dashed] (0,1.25) -- (1.25,1.25);
\end{tikzpicture}\,.
    \end{equs}
    These two small diagrams are actually the base case of the induction since they represent the covariance of the noise $\mcK^{(0,0)}_\epmu$.
\end{example}

For irrelevant cumulants, \ie the ones having $\deg(\n^m)>0$, we have that $\mcK_{0,\eps}^{\n^m}=0$. Indeed, either $m=1$ and it is the expectation of a force term of positive regularity that vanishes when $\mu=0$, or $m\geq 2$ and at least one term in the cumulant is deterministic so it vanishes as well. We get, using Lemma \ref{estimateABfancy}, \eqref{eq:Kmunu} and the fact that $\langle\bigcdot\rangle_\mu\leq\langle\bigcdot\rangle_\nu$ for $\nu\leq\mu$,
\begin{equs}
~&\nnorm{\mcK_{\eps,\mu}^{\n^m}}_\mu\lesssim\int_0^\mu\nnorm{\D_\nu \mcK_{\nu,\eps}^{\n^m}}_\nu\dint\nu\\
&\lesssim\int_0^\mu\sum_{j=1}^m\sum_{k=1}^{n_j-1}\nu^{-1+\alpha-d/p}\nu^{\deg(n_1,\dots, k,n_j-k,\dots,n_m)}+\sum_{I\sqcup J=\n^m\backslash\{j\}}\nu^{-1+\alpha}\nu^{\deg(n_{I\cup\{k\}})}\nu^{\deg(n_{J\cup\{n_j-k\}})}\dint\nu\\
&\lesssim\int_0^\mu\nu^{-1+\alpha-d/p}\nu^{\deg(\n^m)+md/p-\alpha}+\nu^{-1+\alpha}\nu^{\deg(\n^m)-\alpha}\dint\nu\lesssim\int_0^\mu\nu^{\deg(\n^m)-1}\dint\nu\\
&\lesssim\mu^{\deg(\n^m)}.
\end{equs}

We now turn to the relevant cumulants, \ie the ones having $\deg(\n^m)\leq0$. It is paramount to see that this implies $m=1$. Indeed, if $d\geq2\alpha$, recalling the condition \eqref{condbeta}, then $p=d/(2\alpha-2\beta)$ and, if $m\geq2$,
\begin{equs}
\deg(\n^m)\geq m(-\alpha+2\beta)+(m-1)\frac{d}{p}=(m-2)\alpha+2\beta>0.
\end{equs}
If we are in the second case \ie $\alpha\leq2d$, either $p=d/(2\alpha-2\beta)$ and we proceed as above, or $p=1$ and
\begin{equs}
\deg(\n^m)\geq m(-\alpha+2\beta+d)-d> m\frac{d}{2}-d\geq0.
\end{equs}
Our goal is to use the Taylor formula \ref{taylor} to localize these terms, then renormalise the local part, and finally control the remainder. We have,
\begin{equs}
\E[\upxi_{\eps,\mu}^n]&=\mfc_\eps^n\delta+\int_0^\mu\D_\nu\E[\upxi_{\nu,\eps}^n]\dint\nu\\
&=\left(\mfc_\eps^n+\int_0^\mu\I\big(K_\nu^{\otimes2}\D_\nu\E[\upxi_{\nu,\eps}^n]\big)\dint\nu\right)\delta+\int_0^\mu\I\Big(\big(\langle\bigcdot\rangle_\mu^{-2}-1\big)\big(K_\nu^{\otimes2}\D_\nu\E[\upxi_{\nu,\eps}^n]\big)\Big)\dint\nu\times\delta\\
&~~~~+\sum_{|k|_\mfs=2}\int_0^\mu\D^k\L^k\big(\langle\bigcdot\rangle_\mu^{-2}K_\nu^{\otimes2}\D_\nu\E[\upxi_{\nu,\eps}^n]\big)\dint\nu+\int_0^\mu\left(\D_\nu\E[\upxi_{\nu,\eps}^n]-K_\nu^{\otimes2}\D_\nu\E[\upxi_{\nu,\eps}^n]\right)\dint\nu\\
&~~~~+\big(\langle\bigcdot\rangle_\mu^{-2}-1\big)\int_0^\mu K_\nu^{\otimes2}\D_\nu\E[\upxi_{\nu,\eps}^n]\dint\nu.
\end{equs}
Note that we have made a slight abuse of notation by denoting simply $\delta$ the distribution $(z_1,z_2)\mapsto\delta(z_2-z_1)$, and similarly for the weights. We have five terms to control.\\

\noindent\textit{(i)} For the first one, we need to make a choice of renormalising constant $\mfc_\eps^n$. Taking
\begin{equs}
\mfc_\eps^n=-\int_0^1\I(K_\nu^{\otimes2}\D_\nu\E[\upxi_{\nu,\eps}^n])\dint\nu,
\end{equs}
yields
\begin{equs}
\mfc_\eps^n+\int_0^\mu\I(K_\nu^{\otimes2}\D_\nu\E[\upxi_{\nu,\eps}^n])\dint\nu=-\int_\mu^1\I(K_\nu^{\otimes2}\D_\nu\E[\upxi_{\nu,\eps}^n])\dint\nu.
\end{equs}
By \eqref{eq:estimI}, we have the bound
\begin{equs}
\int_\mu^1\big|\I(K_\nu^{\otimes2}\D_\nu\E[\upxi_{\nu,\eps}^n]\big|\dint\nu&\leq\int_\mu^1\|K_\nu^{\otimes2}\D_\nu\E[\upxi_{\nu,\eps}^n]\|_{\mcV^1}\dint\nu\leq\int_\mu^1\big\|\langle\bigcdot\rangle_\mu^{\alpha-\beta}\big(K_\nu^{\otimes2}\D_\nu\E[\upxi_{\nu,\eps}^n]\big)\big\|_{\mcV^1}\dint\nu\\
&\leq\int_\mu^1\snnorm{\D_\nu\E[\upxi_{\nu,\eps}]}_\nu\dint\nu\lesssim\int_\mu^1\nu^{\deg(n)-1}\dint\nu\lesssim\mu^{\deg(n)}.
\end{equs}
Then, following the proof of Lemma 71 in \cite{DGR23},
\begin{equs}
\bnnorm{\left(\mfc_\eps^n+\int_0^\mu\I(K_\nu^{\otimes2}\D_\nu\E[\upxi_{\nu,\eps}^n])\dint\nu\right)\delta}_\mu\lesssim\big\|\langle \bigcdot\rangle_\mu^{\alpha-\beta}K_{2N,\mu}\big\|_{L_{t,x}^1}\int_\mu^1\big|\I(K_\nu^{\otimes2}\D_\nu\E[\upxi_{\nu,\eps}^n])\big|\dint\nu\lesssim\mu^{\deg(n)}.
\end{equs}
\textit{(ii)} Similarly, we get, for $\nu\leq\mu$,
\begin{equs}
~\nnorm{\I\Big(\big(\langle\bigcdot\rangle_\mu^{-2}-1\big)\big(K_\mu^{\otimes2}\D_\nu\E[\upxi_{\nu,\eps}^n]\big)\Big)\times\delta}_\mu&\lesssim\sup_{w_1}\int_{\SS^2}\langle w_1;w_2\rangle_\mu^{\alpha-\beta} K_\mu(w_1-w)K_\mu(w_2-w)\times\\
&~~~~\int_\SS\big(\langle w;v\rangle_\mu^{-2}-1\big)\big(K_\mu^{\otimes2}\D_\nu\E[\upxi_{\nu,\eps}^n]\big)(w,v)\dint v\dint w\dint w_2
\end{equs}
Now we use the estimate \eqref{triangleineq} twice to write
\begin{equs}
\langle w_1;w_2\rangle_\mu^{\alpha-\beta}\lesssim\langle w_1;w\rangle_\mu^{\alpha-\beta}\langle w;v\rangle_\mu^{\alpha-\beta}\langle v;w_2\rangle_\mu^{\alpha-\beta},
\end{equs}
so that the quantity above is bounded by
\begin{equs}
~&\big\|\langle\bigcdot\rangle_\mu^{\alpha-\beta}K_\mu\big\|_{L_{t,x}^1}^2\big\|\langle\bigcdot\rangle_\mu^{\alpha-\beta}\big(\langle \bigcdot\rangle_\mu^{-2}-1\big)\big(K_\mu^{\otimes2}\D_\nu\E[\upxi_{\nu,\eps}^n]\big)\big\|_{\mcV^1}\lesssim\big\|\langle\bigcdot\rangle_\mu^{\alpha-\beta}\big(\langle\bigcdot\rangle_\mu^{-2}-1\big)\big(K_\nu^{\otimes2}\D_\nu\E[\upxi_{\nu,\eps}^n]\big)\big\|_{\mcV^1}\\
&\lesssim\left\|\frac{\langle\bigcdot\rangle_\mu^{\alpha-\beta}}{\langle\bigcdot\rangle_\nu^{\alpha-\beta}}\big(\langle\bigcdot\rangle_\mu^{-2}-1\big)\right\|_{L_{t,x}^\infty}\snnorm{\D_\nu\E[\upxi_{\nu,\eps}^n]}_\nu\lesssim\mu^{\beta-\alpha}\nu^{\alpha-\beta}\nu^{\deg(n)-1}
\end{equs}
For the last inequality, we have used Lemma \ref{estimweight}. Eventually, notice that with our choice of parameters, $\deg(n)+\alpha-\beta>0$. Indeed, following Assumption \ref{condalphabeta}, we have for every $n\geq1$, $\deg(n)\geq-\alpha+2\beta>-\alpha+\beta$. We get the bound
\begin{equs}
\bnnorm{\int_0^\mu\I\Big(\big(\langle\bigcdot\rangle_\mu^{-2}-1\big)\big(K_\nu^{\otimes2}\D_\nu\E[\upxi_{\nu,\eps}^n]\big)\Big)\dint\nu\times\delta}_\mu\lesssim\mu^{\beta-\alpha}\int_0^\mu\nu^{\deg(n)+\alpha-\beta-1}\dint\nu\lesssim\mu^{\deg(n)}\,.
\end{equs}
\textit{(iii)} Fix $k$ such that $|k|_\mfs=2$. We use Lemma \ref{estimateiii} to get the delicate estimate,
\begin{equs}
\nnorm{\D^k\L^k\big(\langle\bigcdot\rangle_\mu^{-2}K_\nu^{\otimes2}\D_\nu\E[\upxi_{\nu,\eps}^n]\big)}_\mu\lesssim\mu^{-2}\nu^2\snnorm{\D_\nu\E[\upxi_{\nu,\eps}^n]}_\nu
\end{equs}
Then, by induction hypothesis, and proceeding as above, we get
\begin{equs}
\bnnorm{\int_0^\mu\D^k\L^k\big(\langle\bigcdot\rangle_\mu^{-2}K_\nu^{\otimes2}\D_\nu\E[\upxi_{\nu,\eps}^n]\big)\dint\nu}_\mu\lesssim\mu^{-2}\int_0^\mu\nu^{\deg(n)+1}\dint\nu\lesssim\mu^{\deg(n)}\,.
\end{equs}
\textit{(iv)} We have, for $\nu\leq\mu$,
\begin{equs}
    \nnorm{\D_\nu\E[\upxi_{\nu,\eps}^n]-K_\nu^{\otimes2}\D_\nu\E[\upxi_{\nu,\eps}^n]}_\mu&=\Big\|\langle\bigcdot\rangle_\mu^{\alpha-\beta}(K_\mu^{\otimes2}- K_\mu^{\otimes2}K_\nu^{\otimes2})\D_\nu\E[\upxi_{\nu,\eps}^n]\Big\|_{\mcV^1}
\end{equs}
Noticing that, for any $\mu,\nu\in(0,1]$, one has $K_\mu-K_\mu K_\nu=\mu^{-\alpha}\nu^\alpha(K_\nu-K_\nu K_\mu)$, we get
\begin{equs}
~&K_\mu^{\otimes2}-K_\mu^{\otimes2}K_\nu^{\otimes2}=K_\mu^{\otimes2}\big(\delta\otimes(\delta-K_\nu)\big)+K_\mu^{\otimes2}\big(K_\nu\otimes(\delta-K_\nu)\big)\\
&\leq\mu^{-\alpha}\nu^\alpha\Big(K_\mu\otimes\big(K_\nu(\delta-K_\mu)\big)\Big)+\mu^{-\alpha}\nu^\alpha\Big(K_\mu K_\nu\otimes\big(K_\nu(\delta-K_\mu)\big)\Big)
\end{equs}
To control these terms, we state that $K_\mu=\mcR_\nu K_\mu*K_\nu$ and recall that $\|\mcR_\nu K_\mu\|_{L_{t,x}^1}=1$ for $\nu\leq\mu$ (see for example \cite{Duc22}, Lemma 4.5). We also use the crude estimate $\|\delta-K_\mu\|_{L_{t,x}^1}\leq2$. Using Young's convolution inequality and the same trick on the weights as in \textit{(ii)}, we eventually get that
\begin{equs}
\nnorm{\D_\nu\E[\upxi_{\nu,\eps}^n]-K_\nu^{\otimes2}\D_\nu\E[\upxi_{\nu,\eps}^n]}_\mu\lesssim\mu^{-\alpha}\nu^\alpha\snnorm{\D_\nu\E[\upxi_{\nu,\eps}^n]}_\nu\lesssim\mu^{-\alpha}\nu^{\deg(n)+\alpha-1}\,,
\end{equs}
and we can thus conclude as in \textit{(ii)}, noticing that we have $\deg(n)+\alpha>0$ for $n\geq1$.\\

\noindent\textit{(v)} We proceed with the same weight decomposition as in \textit{(ii)} and with Lemma \ref{estimweight}. It yields, for $\nu\leq\mu$,
\begin{equs}
\nnorm{\big(\langle\bigcdot\rangle_\mu^{-2}-1\big) K_\nu^{\otimes2}\D_\nu\E[\upxi_{\nu,\eps}^n]}_\nu&\lesssim\big\|\langle\bigcdot\rangle_\mu^{\alpha-\beta} K_\mu\big\|_{L_{t,x}^1}^2\big\|\langle\bigcdot\rangle_\mu^{\alpha-\beta}\big(\langle\bigcdot\rangle_\mu^{-2}-1\big)K_\nu^{\otimes2}\D_\nu\E[\upxi_{\nu,\eps}^n]\big\|_{\mcV^1}\\
&\lesssim\left\|\frac{\langle\bigcdot\rangle_\mu^{\alpha-\beta}}{\langle\bigcdot\rangle_\nu^{\alpha-\beta}}\big(\langle\bigcdot\rangle_\mu^{-2}-1\big)\right\|_{L_{t,x}^\infty}\snnorm{\D_\nu\E[\upxi_{\nu,\eps}^n]}_\nu\lesssim\mu^{\beta-\alpha}\nu^{\deg(n)+\alpha-\beta-1}\,,
\end{equs}
and we conclude as in case \textit{(ii)}.\\
Putting everything together, we have the desired estimate
\begin{equs}
\snnorm{\E[\upxi_{\eps,\mu}^n]}_\mu\lesssim\mu^{\deg(n)}\,.
\end{equs}
This concludes the proof of Proposition \ref{estimatemoment}.
\end{proof}

We state and prove below the lemma used in \textit{(ii)} in the proof above.
\begin{lemma}\label{estimateiii}
Let $\nu\leq\mu$, and $k$ such that $|k|_\mfs=2$. The following estimate holds
\begin{equs}
\nnorm{\D^k\L^k\big(\langle\bigcdot\rangle_\mu^{-2}K_\nu^{\otimes2}\D_\nu\E[\upxi_{\nu,\eps}^n]\big)}_\mu\lesssim\mu^{-2}\nu^2\snnorm{\D_\nu\E[\upxi_{\nu,\eps}^n]}_\nu
\end{equs}
\end{lemma}

\begin{proof}
We first write the left-hand side as
\begin{equs}
\sup_{w_1\in\SS}\int_\SS\langle w_1;w_2\rangle_\mu^2 \left|\int_{\SS^2}K_\mu(w_1-v_1)K_\mu(w_2-v_2)\D^k_{v_2}\L^k\big(\langle\bigcdot\rangle_\mu^{-2}K_\nu^{\otimes2}\D_\nu\E[\upxi_{\nu,\eps}^n]\big)(v_1,v_2)\dint v_1\dint v_2\right|\dint w_2\,.
\end{equs}
Let us carry a change of variable on $\L^k$. We have, for fixed $w_2$ and $v_1$, and $U\in\mcV^1$,
\begin{equs}
~&\int_\SS K_\mu(w_2-v_2)\D^k_{v_2}(\L^kU)(v_1,v_2)\dint v_2\\
&=\int_\SS \D^k K_\mu(w_2-v_2)\frac{|k|}{k!}\int_0^1(1-\tau)^{|k|-1}\tau^{-d-\alpha}(X^kU)\big(v_1,v_1+(v_2-v_1)/\tau\big)\dint\tau\dint v_2\\
&=\frac{|k|}{k!}\int_0^1(1-\tau)^{|k|-1}\int_\SS \D^k K_\mu\big(w_2-v_1-\tau(v_2-v_1)\big)(X^kU)(v_1,v_2)\dint v_2\dint\tau\,.
\end{equs}
Using again the estimate \eqref{estimweight},
\begin{equs}
\langle w_1;w_2\rangle_\mu^2\lesssim\langle w_1;v_1\rangle_\mu^2\langle v_1;v_1+\tau(v_2-v_1)\rangle_\mu^2\langle v_1+\tau(v_2-v_1);w_2\rangle_\mu^2\,,
\end{equs}
and noticing that $\langle v_1;v_1+\tau(v_2-v_1)\rangle_\mu^2\lesssim\langle v_1;v_2\rangle_\mu^2$, we have, using Young's convolution inequality two times, that our quantity of interest is bounded by
\begin{equs}
~&\big\|\langle\bigcdot\rangle_\mu^2 K_\mu\big\|_{L^1_{t,x}}\big\|\langle\bigcdot\rangle_\mu^2 \D^k K_\mu\big\|_{L^1_{t,x}}\big\| X^k\big(K_\nu^{\otimes2}\D_\nu\E[\upxi_{\nu,\eps}^n]\big)\big\|_{\mcV^1}\\
&\lesssim\mu^{-2}\big\|X^k\langle\bigcdot\rangle_\nu^{-2}\big\|_{L_{t,x}^\infty}\snnorm{\D_\nu\E[\upxi_{\nu,\eps}^n]}_\nu.
\end{equs}
Since $\big\|X^k\langle\bigcdot\rangle_\nu^{-2}\big\|_{L_{t,x}^\infty}\leq\nu^2$, we get the desired result.
\end{proof}

\subsection{Kolmogorov argument}\label{Kolmogorov}

We are now ready to construct the stationary force coefficients via a probabilistic argument. We treat the hardest case, which is the infinite volume. 

\begin{lemma}\label{momentkolmogorov}
For every $\kappa>0$, $m\in2\N$, and $\ell\in\R$, it holds
    \begin{equation}
        \E\Big[\talloblong\upxi\talloblong^m_{\kappa,\ell,\varsigma}\Big]\lesssim \big(m(n+1)\big)^{m/2}\,,
    \end{equation}
    with $0<\varsigma<\alpha$ if $\alpha<2$ and $\varsigma>0$ if $\alpha=2$.
\end{lemma}

\begin{proof}
In regard of Lemma \ref{supinside}, it is sufficient to show that, for $s\in\{0,1\}$,
\begin{equs}
    \E\left[\|\D_\mu^s\upxi_{\eps,\mu}^n\|_{\mu,\kappa,\ell,\varsigma}^m\right]\lesssim\mu^{m\deg(n)-s-\kappa m}\,.
\end{equs}
It is easy to see that, thanks to Corollary \ref{estimatemomentdmu}, it suffices to show the above with $s=0$ and that the other case is dealt with similarly. Using Lemmas \ref{lemma:incrmoment} and \ref{lemma:transfercompactweight}, it is sufficient to show that (for some potentially bigger $m$)
\begin{equs}
    \E\left[\left(\sup_{w_1\in\SS}\langle w_1\rangle^{-\kappa}\int_\SS\langle w_1;w_2\rangle_\mu^{\varsigma-\kappa}\mcb e_{\ell}(x_1;x_2)\mcb k_\mu(w_1;w_2)\big|\big(K_\mu^{\otimes2}\upxi_{\eps,\mu}^n\big)(w_1,w_2)\big|\dint w_2\right)^m\right]\lesssim\mu^{m\deg(n)-\kappa}\,,
\end{equs}
where $\mcb k_\mu$ is a smooth function supported in the parabolic ball of radius $\mu$ and we denoted for compactness $\mcb k_\mu(w_1;w_2)=\mcb k_\mu(w_1-w_2)$. This quantity is bounded by
\begin{equs}
~&\left(\sup_{w_1\in\SS}\int_\SS\langle w_1;w_2\rangle_\mu^{\varsigma-\kappa}\mcb e_{\ell}(x_1;x_2)\mcb k_\mu(w_1;w_2)\dint w_2\right)^m\times\E\left[\left(\sup_{w_1,w_2\in\SS}\langle w_1\rangle^{-\kappa}\big|\big(K_\mu^{\otimes2}\upxi_{\eps,\mu}^n\big)(w_1,w_2)\big|\right)^m\right]\\
&\lesssim\mu^{m(d+\alpha)}\E\left[\left(\sup_{w_1,w_2\in\SS}\langle w_1\rangle^{-\kappa}\big|\big(K_\mu^{\otimes2}\upxi_{\eps,\mu}^n\big)(w_1,w_2)\big|\right)^m\right].
\end{equs}
Using Lemma \ref{LinfLm} below, it is bounded by
\begin{equs}
\mu^{(m-2)(d+\alpha)}\E\Big[\big\|\langle w_1\rangle^{-\kappa}\big(K_{N-1,\mu}^{\otimes2}\upxi_{\eps,\mu}^n\big)\big\|_{L^m_{w_1,w_2}}^m\Big].
\end{equs}
Let us take a closer look at the expectation appearing here. Using Fubini's theorem, it is equal to
\begin{equs}
~&\int_{\SS^2}\langle w_1\rangle^{-\kappa m}\E\left[\big|\big(K_{N-1,\mu}^{\otimes2}\upxi_{\eps,\mu}^n\big)(w_1,w_2)\big|^m\right]\dint w_1\dint w_2\\
&\lesssim\int_{\SS^2}\langle w_1\rangle^{-\kappa m}\langle w_1;w_2\rangle_\mu^{-(\alpha-\beta)m}\E\left[\Big|\langle w_1;w_2\rangle_\mu^{\alpha-\beta}\big(K_{N-1,\mu}^{\otimes2}\upxi_{\eps,\mu}^n\big)(w_1,w_2)\Big|^m\right]\dint w_1\dint w_2\\
&\lesssim\int_{\SS^2}\langle w_1\rangle^{-\kappa m}\langle w_1;w_2\rangle_\mu^{-(\alpha-\beta)m}\dint w_1\dint w_2\times\sup_{w_1,w_2\in\SS}\E\left[\Big|\langle w_1;w_2\rangle_\mu^{\alpha-\beta}\big(K_{N-1,\mu}^{\otimes2}\upxi_{\eps,\mu}^n\big)(w_1,w_2)\Big|^m\right]\\
&\lesssim \big(m(n+1)\big)^{m(n+1)/2}\mu^{d+\alpha}\sup_{w_1,w_2\in\SS}\E\left[\Big|\langle w_1;w_2\rangle_\mu^{\alpha-\beta}\big(K_{N-1,\mu}^{\otimes2}\upxi_{\eps,\mu}^n\big)(w_1,w_2)\Big|^2\right]^{m/2}
\end{equs}
We used hypercontractivity in the last inequality. Then we use the moment-cumulant relation \eqref{momcumu} to develop the expectation above in cumulants. Moreover, using Lemma \ref{LinfV} and Proposition \ref{estimatemoment}, we get, putting everything together,
\begin{equs}
    \E\left[\|\upxi_{\eps,\mu}^n\|_{\mu}^m\right]&\lesssim \big(m(n+1)\big)^{m(n+1)/2}\mu^{(m-1)(d+\alpha)}\mu^{-\frac{m}{2}\big(d/p+2(d+\alpha)\big)}\mu^{\frac{m}{2}(2\deg(n)+d/p)}\\
    &=\big(m(n+1)\big)^{m(n+1)/2}\mu^{m\deg(n)-m\kappa},
\end{equs}
where we have set $\kappa=(d+\alpha)/m$. It then yields the desired result.
\end{proof}

\begin{lemma}\label{lemma:incrmoment}
For any $n,m\in \N$ it holds
    \begin{equation}\label{eq:incrmoment}
        \sup_{\mu\in(0,1]}\mu^{-\deg(n)+\kappa_n}\E\big[\|\upxi_\epmu^n\|_{\mu,\kappa_n,\ell,\varsigma}^m\big]^\frac{1}{m}\lesssim1+\sup_{i\leq n}\sup_{\mu\in(0,1]}\Big(\mu^{-\deg(i)+\kappa_i}\E\big[\|\mcb k_\mu\upxi_\epmu^i\|_{\mu,\kappa_i,\ell,\varsigma}^{2^n m}\big]^\frac{1}{2^nm}\Big)^{2^n}\,,
    \end{equation}
    and
    \begin{equation}\label{eq:incrmomentderivative}
        \sup_{\mu\in(0,1]}\mu^{-\deg(n)+1+\kappa_n}\E\big[\|\D_\mu\upxi_\epmu^n\|_{\mu,\kappa_n,\ell,\varsigma}^m\big]^\frac{1}{m}\lesssim1+\sup_{i\leq n}\sup_{\mu\in(0,1]}\Big(\mu^{-\deg(i)+\kappa_i}\E\big[\|\mcb k_\mu\upxi_\epmu^i\|_{\mu,\kappa_i,\ell,\varsigma}^{2^n m}\big]^\frac{1}{2^nm}\Big)^{2^n}\,.
    \end{equation}
\end{lemma}

\begin{proof}
    The proof follows quite closely the one of Lemma 38 in \cite{DGR23}, with the slight difference that we need to take into account the exponential weights. The proof works by induction on $n$. Let us first prove \eqref{eq:incrmomentderivative} for $n$, assuming that \eqref{eq:incrmoment} and \eqref{eq:incrmomentderivative} hold for all $i\leq n$. We have, using \eqref{triangleineq} and Lemma \ref{ITexp},
    \begin{equs}
        \langle w_1;w_2\rangle_\mu^{\varsigma-\kappa_n}\lesssim\langle w_1;z_1\rangle_\mu^{\varsigma-\kappa_n}\langle z_1;z_2\rangle_\mu^{\varsigma}\langle z_2;w_2\rangle_\mu^{\varsigma-\kappa_j}\,,
    \end{equs}
    and
    \begin{equs}
        \mcb e_\ell(x_1;x_2)\lesssim\mcb e_\ell(x_1;y_1)\mcb e_\ell(y_1;y_2)\mcb e_\ell(y_2;x_2)\,.
    \end{equs}
    Using the flow equation on coefficients \eqref{eq:hierarchyxi}, we get
    \begin{equs}
        ~&\sup_{w_1\in\Lambda}\int_\Lambda \langle w_1\rangle^{-2^n \kappa}\langle w_1;w_2\rangle_\mu^{\varsigma-\kappa_n}\mcb e_\ell(x_1;x_2)\big|\big(K_\mu^{\otimes2}\D_\mu\upxi_\epmu^n\big)(w_1,w_2)\big|\dint w_2\\
        &\lesssim\sum_{i+j=n-1}\sup_{w_1\in\Lambda}\int_\Lambda \langle w_1\rangle^{-\kappa_i}\langle w_1;z_1\rangle_\mu^{\varsigma-2^{i}\kappa}\mcb e_\ell(x_1;y_1)\big|\big(K_\mu^{\otimes2}\upxi_\epmu^i\big)(w_1,w_2)\big|\dint z_1\\
        &\quad\times\sup_{z_1\in\Lambda}\int_\Lambda \langle z_1;z_2\rangle^{\kappa_j} \mcb e_\ell(y_1;y_2) \big|\big(\mcR^2_\mu\dot G_\mu\big)(z_1,z_2)\big|\dint z_2\\
        &\quad\times \sup_{z_2\in\Lambda}\int_\Lambda \langle z_2\rangle^{-\kappa_j}\langle z_2;w_2\rangle_\mu^{\varsigma-\kappa_j}\mcb e_\ell(y_2;x_2)\big|\big(K_\mu^{\otimes2}\upxi_\epmu^j\big)(z_2,w_2)\big|\dint w_2\,.
    \end{equs}
    We have used the inequality $\langle z_2\rangle^{\kappa_j}\lesssim\langle z_2;z_1\rangle_\mu^{\kappa_j}\langle z_1;w_1\rangle_\mu^{\kappa_j} \langle w_1\rangle^{\kappa_j}$. It yields, using Lemma \ref{heatkernelestimate} and the Hölder inequality,
    \begin{multline*}
        \mu^{-\deg(n)+1+\kappa_n}\E\big[\|\D_\mu\upxi_\epmu^n\|_{\mu,\kappa_n,\ell,\varsigma}^m\big]^\frac{1}{m}\\\lesssim\sum_{i+j=n-1}\mu^{-\deg(i)+\kappa_i}\E\big[\|\upxi_\epmu^i\|_{\mu,\kappa_i,\ell,\varsigma}^{2m}\big]^\frac{1}{2m}\mu^{-\deg(j)+\kappa_j}\E\big[\|\upxi_\epmu^j\|_{\mu,\kappa_j,\ell,\varsigma}^{2m}\big]^\frac{1}{2m}\,,
    \end{multline*}
    which concludes the proof of \eqref{eq:incrmomentderivative}. The proof of \eqref{eq:incrmoment} follows as in \cite[Lemma 38]{DGR23}. The important point to see this is that the exponential weights do not depend on $\mu$, and this do not interact with the other weights in the proof.
\end{proof}

\begin{lemma}\label{lemma:transfercompactweight}
    There is a kernel $\Tilde {\mcb k}_\mu$ with compact support in the parabolic ball of radius $\mu$ such that
    \begin{equation}
        \|\mcb k_\mu\upxi_\epmu^n\|_{\mu,\kappa,\ell,\varsigma}\lesssim\sup_{w_1\in\Lambda}\int_\Lambda \langle w_1\rangle^{-\kappa_n}\langle w_1;w_2\rangle^{\varsigma-\kappa_n}_\mu\mcb e_\ell(x_1;x_2)\Tilde {\mcb k}_\mu(w_1;w_2)\big|\big(K_\mu^{\otimes2}\upxi_\epmu^n\big)(w_1,w_2)\big|\dint w_2\,.
    \end{equation}
\end{lemma}

\begin{proof}
    We have
    \begin{equs}
        \big(K_\mu^{\otimes2}\mcb k_\mu\upxi_\epmu^n\big)(w_1,w_2)&=\int_{\Lambda^2}K_\mu(w_1-z_1)K_\mu(w_2-z_2)\mcb k_\mu(z_1;z_2)\upxi_\epmu^n(z_1,z_2)\dint z_1\dint z_2\\
        &=\int_{\Lambda^2}\mcR_\mu^2\Big(K_\mu(w_1-z_1)K_\mu(w_2-z_2)\mcb k_\mu(z_1;z_2)\Big)\big(K_\mu^{\otimes2}\upxi_\epmu^n\big)(z_1,z_2)\dint z_1\dint z_2\,,
    \end{equs}
    where the operators $\mcR_\mu$ act on the $z$ variables. One can see that there is a discrete set $\mathfrak{P}$ such that
    \begin{multline*}
        \mcR_\mu^2\Big(K_\mu(w_1-z_1)K_\mu(w_2-z_2)\mcb k_\mu(z_1;z_2)\Big)\\=\sum_{\pi_1,\pi_2,\pi_3\in \mathfrak{P}}(\mcS_\mu K^{\pi_1})(w_1-z_1)(\mcS_\mu K^{\pi_2})(w_2-z_2)\mcb k_\mu^{\pi_3}(z_1;z_2)\,,
    \end{multline*}
    where $K^\pi$ is integrable with exponential decay, and $\mcb k_\mu^{\pi}$ has compact support in the parabolic ball of radius $\mu$. For more details, we refer the reader to the proof of Lemma 61 in \cite{DGR23}. Letting $\Tilde{\mcb k}_\mu=\sum_\pi \mcb k_\mu^{\pi}$, using the same triangle inequalities as in the proof of Lemma \ref{lemma:incrmoment}, and using Young's convolution inequality, one can conclude the proof.
\end{proof}

We recall the following moment-cumulant relation, whose proof can be found in \cite[Proposition 3.2.1]{PT11}.
\begin{lemma}
We have, for $m\in\N$,
\begin{equation}\label{momcumu}
\E[X_1\dots X_m]=\sum_{j=1}^m\sum_{I_1\sqcup\dots\sqcup I_j=[m]}\mcK(X_{I_1})\times\dots\times\mcK(X_{I_j}).
\end{equation}
\end{lemma}
\begin{lemma}\label{LinfLm}
Let $m\in2\N$ and $U\in L^\infty_{w_1,w_2}$. The following estimate holds
\begin{equs}
    \E\Big[\big\|K_{1,\mu}^{\otimes2}U\big\|_{L^\infty_{w_1,w_2}}^m\Big]\lesssim\mu^{-2(d+\alpha)}\E\Big[\|U\|_{L^m_{w_1,w_2}}^m\Big].
\end{equs}
\end{lemma}

\begin{proof}
    See for example Lemma A.2 in \cite{Duc23} for the same statement in an elliptic context.
\end{proof}

\begin{lemma}\label{LinfV}
Let $U\in \mcV^2$. The following estimate holds
\begin{equs}
\big\|\langle \bigcdot\rangle_{\mu,2}^{\alpha-\beta} K_{1,\mu}^{\otimes4}U\big\|_{L^\infty(\Lambda^4)}^2\lesssim\mu^{-2(d+\alpha)-d/p}\big\|\langle\bigcdot\rangle_{\mu,2}^{\alpha-\beta}U\big\|_{\mcV^2}.
\end{equs}
\end{lemma}

\begin{proof}
    Let us write, using the notations introduced at the beginning of Section~\ref{sec:stocest},
    \begin{equs}
        \big\|\langle \bigcdot\rangle_{\mu,2}^{\alpha-\beta} K_{1,\mu}^{\otimes4}U\big\|_{L^\infty(\Lambda^4)}&=\sup_{\w\in\Lambda^4}\langle w_1^1;w_2^1\rangle_\mu^{\alpha-\beta}\langle w_1^2;w_2^2\rangle_\mu^{\alpha-\beta}\big|\big(K_{1,\mu}^{\otimes4}U\big)(\w)\big|\\
        &\leq\sup_{\w\in\Lambda^4}\langle w_1^1;w_2^1\rangle_\mu^{\alpha-\beta}\langle w_1^2;w_2^2\rangle_\mu^{\alpha-\beta}\\
        &\quad\times\int_{\Lambda^4} \prod_{i=1}^2 K_{1,\mu}(w_1^i-z_1^i)K_{1,\mu}(w_2^i-z_2^i)|U(\z)|\dint\z\,.
    \end{equs}
    Using the triangle inequality \eqref{triangleineq} for $i\in\{1,2\}$,
    \begin{equs}
        \langle w_1^i;w_2^i\rangle^{\alpha-\beta}\lesssim\langle w_1^i;z_1^i\rangle^{\alpha-\beta}\langle z_1^i;z_2^i\rangle^{\alpha-\beta}\langle z_2^i;w_2^i\rangle^{\alpha-\beta}\,,
    \end{equs}
    and the estimate \eqref{eq:KmuLp} on $K_{1,\mu}$ coupled with Young's convolution inequality, we get the result.
\end{proof}

\begin{lemma}\label{supinside}
    Let $\uplambda_\mu$ a random field such that, for $s\in\{0,1\}$, and uniformly in $\eps$,
    \begin{equation}
        \sup_{\mu\in(0,1]}\mu^{-\eta+s}\E\big[\|\D_\mu^s\uplambda_\mu\|_{\mcV^1}^m\big]^\frac{1}{m}\lesssim1\,,
    \end{equation}
    then
    \begin{equation}
        \E\Big[\Big(\sup_{\mu\in(0,1]}\mu^{-\eta}\|\uplambda_\mu\|_{\mcV^1}\Big)^m\Big]^\frac{1}{m}\lesssim1\,.
    \end{equation}
\end{lemma}

\begin{proof}
    See for example Exercise A.2 in the lecture notes \cite{Duc23}.
\end{proof}


\section{Characterisation and properties of the operator}

We gather in this section the proofs of all the items of Theorem \ref{Th:ExistenceAndersonOp}, \ie the constructions of all the operators considered in this article.

\subsection{Recollection of functional analysis}

We start with some reminders of the functional analysis tools that we will use in the sequel.

\subsubsection{Sectorial operators}

For $\theta\in(0,\pi]$, we define the open sector $\Sigma_\theta$ as 
$$\Sigma_\theta = \left\{z\in\C^*,\; |\operatorname{Arg} z|<\theta \right\}=\left\{r \e^{\i \tau},\; r>0,\; |\tau|<\theta \right\}\,,$$
where $\operatorname{Arg} z\in(-\pi,\pi]$ is the principal argument of $z$.

\begin{definition}
    Let $A$ be an unbounded operator on a Banach space $E$. We say that $A$ is sectorial of angle $\theta\in(0,\pi)$ (or $\theta$-sectorial) if
    \begin{equation}\label{Def:SectorialOp1}
        \sigma(A)\subset\overline{\Sigma_\theta}
    \end{equation}
    and for all $\omega\in(\theta,\pi)$ it holds
    \begin{equation}\label{Def:SectorialOp2}
        \sup_{\lambda\in\C\backslash\overline{\Sigma_{\omega}}} |\lambda|\big\|(\lambda-A)^{-1}\big\|<\infty.
    \end{equation}
    We say an operator $A$ is quasi-$\theta$-sectorial if $A+\delta$ is $\theta$-sectorial for some $\delta>0$. We say an operator is quasi-sectorial if it is quasi-$\theta$-sectorial for some $\theta$. We denote by $$\theta_S(A)=\inf\left\{\theta\in(0,\pi),\; A \text{ is $\theta$-sectorial}\right\}\in[0,\pi)$$
    the infimum of angles for which $A$ is sectorial.
\end{definition}

It is easily verified that any quasi-sectorial operator is closed. Moreover, for $\theta\in(\theta_S(A),\pi)$ and $f$ holomorphic on a neighbourhood of $\overline{\Sigma_\theta}$ regular at $0$ and infinity, i.e. there exists $\alpha>0$ such that $f(z)=o(z^\alpha)$ near $0$ and $f(z)=o(z^{-\alpha})$ near infinity, one can define the unbounded operator $f(A)$ by the Dunford calculus, \ie using Cauchy's formula,
$$f(A)=\frac{1}{2\i\pi}\oint_\Gamma f(\lambda)(\lambda-A)^{-1}\dint\lambda,$$
where $\Gamma$ is an oriented curve in $\Sigma_\theta\backslash\overline{\Sigma_{\theta_S(A)}}$ with any $\lambda\in\sigma(A)$ having winding number one. If $A$ is invertible, this construction extends to any holomorphic $f$ with $f(z)=o(z^\alpha)$ near infinity ($\alpha\in\R$). Remark in this case no assumption is needed at the origin. See for example Chapter 2 of \cite{Ha06} and of \cite{Ya10} for more details. In particular, for any $z\in\C$ with $\operatorname{Re} z\neq 0$, one can define $A^z$ and it verifies $A^zA^{z'}=A^{z+z'}$ ($\operatorname{Re} z,\operatorname{Re} z',\operatorname{Re}(z+z')\neq 0$), see for example Proposition 3.2.1 in \cite{Ha06}. For a quasi-sectorial operator $A$, these results extends for $f(A+\delta)$ for any $\delta>0$ large enough.\\

For a quasi-sectorial operator on all $L^p_x$ spaces, it is then possible to define the associated Sobolev spaces\footnote{One can also define the associated Besov spaces, which turn out to be the real interpolation spaces between Sobolev spaces.} as 
$$W^{s,p}(A) = \left\{u\in L^p_x,\; (A+\delta)^{s}u\in L^p_x\right\}$$
for $s\geqslant$ and  $W^{-s,p}(A)=\left(W^{s,p'}(A')\right)'$, with $A'$ the adjoint of $A$ and $\frac{1}{p}+\frac{1}{p'}=1$. A criterion for these spaces to be stable by complex interpolation is for $A+\delta$ to have bounded imaginary powers in $L_x^p$ for $\delta$ large enough. We say a sectorial operator $H$ has bounded imaginary powers on $E$ (or shortly $H$ is BIP on $E$) if for any $t\in\R$, $H^{\i t}\in\mathcal{L}(E)$ and there exists $c,C>0$ such that
$$\forall t\in\R,\; \|H^{\i t}\|_{\mathcal{L}(E)}\leqslant C\e^{c|t|}.$$
For example, the fractional laplacian (with or without mass) is BIP on $L^p_x(\R^d)$ for any $1<p<\infty$ by the Hörmander-Mikhlin theorem and the transference principle implies it is also true on $L^p_x(\T^d)$.

\subsubsection{Semigroups}

\begin{definition}
    Let $E$ be a Banach space. A family $\big(S(t)\big)_{t\geqslant 0}$ of bounded linear maps is called a $\mathcal{C}_0$-semigroup if it verifies
    \begin{enumerate}
        \item $S(0)=1$,
        \item for every $t,s>0$, $S(t)S(s)=S(t+s)$ and
        \item for every $x\in E$, $S(t)x-x$ converges to $0$ as $t$ goes to $0$.
    \end{enumerate}
    A $\mathcal{C}_0$-semigroup is called exponentially bounded if 
    $$\omega_E(S) = \inf_{t>0}\frac{1}{t}\log\|S(t)\|_{\mathcal{L}(E)}\in\R,$$
    or equivalently, if there exists $M>0$ and $\omega\in\R$ such that
    $$\forall t>0,\; \|S(t)\|_{\mathcal{L}(E)}\leqslant M\mathrm{e}^{\omega t}.$$
    The real $\omega_E(S)$ is then called the growth bound of $S$ and is the infimum of the $\omega$ for which the previous bound holds for some $M$. 
\end{definition}

\begin{definition}
    For a $\mathcal{C}_0$-semigroup $S$, we define its generator $A$ as the unbounded operator of domain
    $$\mathrm{D}(A)=\left\{x\in E,\; \lim_{t\to0^+}\frac{S(t)x-x}{t} \text{ exists in } E\right\}$$
    and formula $Ax=\lim_{t\to0^+}\frac{S(t)x-x}{t}$. 
\end{definition}

The celebrated Hille-Yosida theorem implies that any exponentially bounded $\mathcal{C}_0$-semigroup is generated by a quasi-sectorial operator.

\begin{theorem}[Hille-Yosida]\label{Prop:Hille-Yosida}
    An unbounded operator $A\colon\mathrm{D}(A)\subset E\to E$ generates a $\mathcal{C}_0$-semigroup $S$ on $E$ verifying 
    $$\forall t>0,\; \|S(t)\|_{\mathcal{L}(E)}\leqslant M\mathrm{e}^{\omega t}\,,$$
    if and only if
    \begin{itemize}
        \item $A$ is closed and densely defined,
        \item for any $\Re\lambda>\omega$ and any $n\in\N_{\geq1}$,
        $$\big\|(\lambda-A)^{-n}\big\|_{\mathcal{L}(E)}\leqslant\frac{M}{(\Re\lambda-\omega)^n}\,.$$
    \end{itemize}
\end{theorem}

For symmetric unbounded semigroups on Hilbert spaces, which are not covered by the previous theory, there exists an equivalent of the Hille-Yosida theorem due to Klein and Landau \cite{KL81}.

\begin{theorem}[Klein-Landau]\label{Prop:KleinLandau}
    Fix $T>0$. Let $\big(S(t)\big)_{t\in[0,T]}$ be a collection of unbounded operators on a Hilbert space $\mcH$ with inner product $\langle\bigcdot,\bigcdot\rangle_\mcH$ such that $S(0)=1$ and $S(t)$ has domain $D_t$. Assume that $\big(S(t)\big)_{t\in[0,T]}$ satisfies the following properties:
    \begin{enumerate}
        \item (Non-increasing domains with dense union) For $s\leq t$, $D_t\subset D_s$ and $\bigcup_{0\leq t\leq T}D_t$ is dense in $\mcH$.
        \item (Semigroup) For $0\leq s\leq t\leq T$, $S(s) D_t\subset D_{t-s}$ and $S(t-s)S(s)=S(t)$ on $D_t$.
        \item (Symmetry) For $f,g\in D_t$, $\langle f,S(t) g\rangle_\mcH=\langle S(t) f,g\rangle_\mcH$.
        \item (Weak continuity) For $0\leq t\leq T$ and $f\in D_t$, the function $s\mapsto\langle f,S(s)f\rangle_\mcH$ is continuous on the interval $[0,t]$.
    \end{enumerate}
    Then, there exists a unique self-adjoint operator $A$ on $\mcH$ such that for every $t\in[0,T]$, $D_t\subset D(\e^{-tA})$ and such that $S(t)$ is the restriction of $\e^{-tA}$ to $D_t$. Moreover, the operator $A$ is essentially self-adjoint on the core $\hat D\eqdef\bigcup_{0\leq t\leq T'}\bigcup_{0<s<t}S(s) D_t$, for any $T'\in(0,T]$.
\end{theorem}

\subsection{The fractional Laplacian in finite volume}\label{sec:4-2}

We now prove we can define the renormalized solution of \eqref{mainSPDE} as an almost surely exponentially bounded $\mathcal{C}_0$-semigroup on Sobolev and Besov spaces of regularity $|s|<\beta$.

\begin{proposition}\label{Prop:ExistenceRenormalizedSemigroupFiniteVolume}
    There exists a deterministic sequence $\eps_k$ converging to $0$ and an event $\overline{\Omega}$ of probability one, on which the following holds. Let $E$ be $H^{s,p}$ or $\mcB^{s}_{p,q}$ ($p\in(1,\infty)$, $q\in[1,\infty]$ and $|s|<\beta$). For any $0<\eps<1$, let $\big(S_\eps(t)\big)_{t\geqslant 0}$ be the $\mathcal{C}_0$-semigroup solving \eqref{renormSPDE} on $E$. Then, 
    \begin{enumerate}
        \item there exists $M>0$ and $c\in\R$ such that,
        \begin{equation}
            \forall t>0,\; \|S_\eps(t)\|_{\mathcal{L}(E)}\leqslant M\mathrm{e}^{ct}\,,
        \end{equation}
        \item $\big(S_{\eps_k}(t)\big)_{t\geqslant 0}$ converges strongly to a $\mathcal{C}_0$-semigroup $\big(S(t)\big)_{t\geqslant 0}$ on E verifying the same bound,
        \item for any $f\in E$ and $g\in E'$, it holds
        \begin{equation}\label{Eq:SymS}
            \forall t>0,\; \langle S(t)f,g\rangle_{E,E'}=\langle f,S(t)g\rangle_{E,E'}.
        \end{equation}
    \end{enumerate}
\end{proposition}

\begin{proof}
    Let $E$ be $H^{s_l,p}$ with $p\in \{2,3,\dotsc \}\cup\{\frac{n}{n-1},\; n\in\N,\; n\geqslant 3\}$ and $s_l=\frac{l\beta}{|l|+1}$. Note that there exists $0<\sigma_l<\beta$ and $0<\gamma_l<\beta-\sigma_l$ such that $-\sigma_l<s_l<\gamma_l$ thus $\mcB^{\gamma_l}_{p,\infty}\subset E \subset \mcB^{-\sigma_l}_{p,\infty}$ densely and compactly. Fix $\kappa$ small enough such that $0<4\kappa<\beta-\gamma_l-\sigma_l$. Let $\Omega_{l,p}$ be the event on which 1. and 2. hold on $E$. We will prove it is of full probability. Then, by taking 
    $$\overline{\Omega} = \bigcap_{p=2}^\infty\bigcap_{l\in\Z}\Omega_{l,p}\cap\Omega_{l,\frac{p}{p-1}}\,,$$
    which is of probability $1$, 1. and 2. for general Sobolev and Besov spaces will follow by interpolation because $H^{s_\theta,p}=[H^{s_0,p},H^{s_1,p}]_\theta$, $H^{s,p_\theta}=[H^{s,p_0},H^{s,p_1}]_\theta$ and $\mcB^{s_\theta}_{p,q}=(H^{s_0,p},H^{s_1,p})_{\theta,q}$,  with $s_\theta=\theta s_1+(1-\theta)s_0$ ($s_0\neq s_1$), $\frac{1}{p_\theta}=\frac{\theta}{p_1}+\frac{1-\theta}{p_0}$ ($p_0\neq p_1$), $[\cdot,\cdot]_\theta$ the complex interpolation functor and $(\cdot,\cdot)_{\theta,q}$ the real interpolation functor. By Lemma \ref{Cor:GlobalisationSeps}, there exists $M>0$ and $c\in\R$, such that
    \begin{equation}\label{Eq:ExponentialBoundSeps}
        \forall t\geqslant 0,\; \|S_\eps(t)f\|_{E}\leqslant M\mathrm{e}^{ct}\|f\|_E.
    \end{equation}
    Let $(f_j)_{j\in\N_{\geq1}}$ be an enumeration of all trigonometric polynomials with rational coefficients. By continuity on the force $F_{\eps}$ of Lemma \ref{lemmafixedpointplanted} and the convergence of $F_{\eps}$ of Proposition \ref{prop:stoestimate}, for any $j$, $\big(S_\eps(t)f_j\big)_{0<t\leqslant T}$ converges in probability with respect to the Schauder norm
    $$\sup_{0<t\leqslant T}\left[t^{\frac{\sigma_l+\gamma_l}{\alpha}}\|u(t)\|_{\mcB^{\gamma_l}_{p,\infty}}+\|u(t)\|_{\mcB^{-\sigma_l}_{p,\infty}}\right].$$
    By a bootstrap argument, using $S_\eps$ is a semigroup, it follows for any $N\in\N_{\geq1}$, $\big(S_\eps(t)f_j\big)_{\frac{T}{N}<t<NT}$ converges in probability in $L^\infty\left(\left(\frac{T}{N},NT\right),\mcB^{\gamma_l}_{p,\infty}\right)$. By a diagonal argument on $\left(\big(S_{\eps}(t)f_j\big)_{\frac{T}{N}<t<NT}\right)_{N,p,j\in\N_{\geq1},l\in\Z}$, there exists a deterministic sequence $\eps_k$ going to $0$, independent of $N$, $j$, $p$ and $l$, such that for any $N\in\N_{\geq1}$ and any $j\in\N$, $\big(S_{\eps_k}(t)f_j\big)_{\frac{T}{N}<t<NT}$ converges almost surely in $L^\infty\left(\left(\frac{T}{N},NT\right),\mcB^{\gamma_l}_{p,\infty}\right)$ to some $\big(S(t)f_j\big)_{\frac{T}{N}<t<NT}$. Remark the limit $S(t)$ is well-defined for any $t>0$ as $S_\eps$ is a semigroup and by uniqueness of the limit in probability. Denote by $\Omega'_{l,p}$ the event of probability one on which this convergence holds for all $f_j$. Then, by density of $(f_j)_{j\in\N_{\geq1}}$ and (\ref{Eq:ExponentialBoundSeps}), it follows that for any $f\in E$ and any $t\in E$, $S_{\eps_k}(t)f$ converges to $S(t)f$ on $\Omega'_{l,p}$ and (\ref{Eq:ExponentialBoundSeps}) holds for $S$ with the same $M$ and $c$. It follows easily, by uniqueness of the limit, that $S$ is a semigroup on $E$. Moreover, by Lemma \ref{lemmafixedpointplanted} and the convergence in $\mcB^{\gamma_l}_{p,\infty}$, we get
    \begin{align*}
        \|S(t)f-f\|_E &\leqslant \|S(t)f-G(t)f\|_E+\|G(t)f-f\|_E\\
        &\leqslant C\|S(t)f-G(t)f\|_{\mcB^{\gamma_l}_{p,\infty}}+\|G(t)f-f\|_E\\
        &\leqslant C\lim_{k\to\infty}\|S_{\eps_k}(t)f-G(t)f\|_{\mcB^{\gamma_l}_{p,\infty}}+\|G(t)f-f\|_E\\
        &\leqslant Ct^\eta\|f\|_{\mcB^{-\sigma_l}_{p,\infty}}+\|G(t)f-f\|_E\\
        &\leqslant C't^\eta\|f\|_{E}+\|G(t)f-f\|_E\,,
    \end{align*}
    for some $\eta>0$, using $\mcB^{\gamma_l}_{p,\infty}\subset E \subset \mcB^{-\sigma_l}_{p,\infty}$. Taking $t$ to $0^+$ shows $S$ is in fact a $\mathcal{C}_0$-semigroup on $E$, on the event $\Omega'_{l,p}$. This achieve the proof that $\Omega'_{l,p}\subset\Omega_{l,p}$, which is thus of full probability. Now, on $\overline{\Omega}$, let $E$ be any Sobolev or Besov space of regularity $|s|<\beta$ and integrability $p\in(1,\infty)$ for any $f\in E$ and $g \in E'$, it holds
    $$\forall t,\eps>0,\; \langle S_\eps(t)f,g\rangle_{E,E'}=\langle f,S_\eps(t)g\rangle_{E,E'}.$$
    In fact, it holds for $f$ and $g$ trigonometric polynomials, which are dense in both $E$ and $E'$, by self adjointness and both bracket are continuous bilinear forms on $E\times E'$. Thus, taking the limit along $\eps_k$, 3. follows.  
\end{proof}

To conclude the proof of Theorem \ref{maintheoremspde}, item 1, it suffices to prove the estimate \eqref{Eq:EstimateS-G}, which follows from Lemma \ref{Cor:GlobalisationSeps} and the convergence along the deterministic sequence $\eps_k$ of Proposition \ref{Prop:ExistenceRenormalizedSemigroupFiniteVolume}. 

\begin{proof}[of Theorem~\ref{Th:ExistenceAndersonOp}.1]

    The main observation to prove item 1.(c) is that one has to see that the first and third items of Theorem~\ref{maintheoremspde} hold for the more general equation
    \begin{equation}\label{mainSPDErotated}
        (\D_t+\e^{\i\phi}\mathrm{D}^{\alpha,m})u^{\alpha,m}=\e^{\i\phi}(\xi-\infty) u^{\alpha,m}\,,\qquad u^{\alpha,m}\!\!\upharpoonright_{t=0}\,=u^{\alpha,m}_0\,
    \end{equation}
    for $|\phi|<\frac{\pi}{2}$, viewed as the limit of 
    \begin{equation}\label{Eq:RenomalizedParabolicRegularizedRotated}
        (\D_t+\e^{\i\phi}\mathrm D^{\alpha,m})u^{\alpha,m}_\eps=\e^{\i\phi}(\xi_\eps-\mfC^{\alpha,m}_\eps) u^{\alpha,m}_\eps\,,\qquad u^{\alpha,m}_\eps\!\!\upharpoonright_{t=0}\,=u^{\alpha,m}_{0,\eps}\,.
    \end{equation}
    This is a direct consequence of the semigroup associated with $\mathrm{D}^{\alpha,m}$ being bounded analytic on $\{\Re z>0\}$ on every ${L}^p_x$. In that case, \eqref{Eq:EstimateS-G} and \eqref{Eq:EstimateS-GOH} may hold with constants depending on $\phi$ (which typically blow up as $|\phi|$ goes to $\frac{\pi}{2}$). The only fact to prove is that the renormalisation of \eqref{mainSPDErotated} has the form given in \eqref{Eq:RenomalizedParabolicRegularizedRotated}. For that purpose, one can show by induction that 
    \begin{equs}
        \upxi_\epmu^{\phi,n}(t,x;s,y)= e^{2\i \phi}\upxi^n_\epmu (\e^{\i \phi}t,x;\e^{\i \phi}s,y)\,.
    \end{equs}
    We have denoted for this proof $\upxi_\epmu^{\phi,n}$ the force coefficients associated with the operator $\D_t+\e^{\i\phi}\mathrm D^{\alpha,m}$. The latter implies that $\mfC^\phi_\eps=e^{\i \phi}\mfC_\eps$. Hence, Proposition~\ref{Prop:ExistenceRenormalizedSemigroupFiniteVolume} holds for the semigroups $S^{\phi}$ and $(S^\phi_\eps)_{0<\eps<1}$ on a event $\overline{\Omega}_\phi$ of probability one. We now work on the full-probability event
    $$\overline{\overline{\Omega}} = \bigcap_{k=0}^{\infty}\overline{\Omega}_{\frac{k\pi}{2k+1}}\cap\overline{\Omega}_{-\frac{k\pi}{2k+1}}.$$

    \begin{itemize}
        \item[(a)] As a consequence of Theorem \ref{Prop:Hille-Yosida} and Proposition \ref{Prop:ExistenceRenormalizedSemigroupFiniteVolume}, one can now construct the operator $\mathrm{A}^{\!\alpha,m}_\xi$
        as the generator of the semigroup $S$ constructed in Proposition \ref{Prop:ExistenceRenormalizedSemigroupFiniteVolume} on $L_x^p(\T^d)$. Then, obviously, \eqref{mainSPDE} can be rewritten as
        $$\partial_t u + \mathrm{A}^{\!\alpha,m}_\xi u = 0.$$
        The quasi-sectoriality of $A$ follows directly from Theorem \ref{Prop:Hille-Yosida}. Let $A_\eps = \rmD^{\alpha,m}+\xi_{\eps}-\mfC_{\eps}$. There exists $c,C>0$ finite a.s. such that
        $$\forall \eps,t>0,\; \|S_\eps(t)\|_{\mathcal{L}(L_x^p)}\leqslant C\e^{ct},$$ thus by Lumer calculus, for $\Re z<-c$, it holds
        \begin{equation}\label{Eq:LumerFormulaResolventA-Aeps}
            (A+z)^{-1}-(A_\eps+z)^{-1} = \int_0^\infty \e^{-zt}\big(S(t)-S_{\eps}(t)\big)\dint t,
        \end{equation}
        thus we see that $(A_{\eps_k}+z)^{-1}$ converges to $(A+z)^{-1}$ in $\mathcal{L}(L_x^p)$ by dominated convergence in the Bochner integral in the right-hand side of \eqref{Eq:LumerFormulaResolventA-Aeps}. 
        
        Likewise, for $0<\gamma<\beta$, one can write
        \begin{equation}\label{Eq:LumerFormulaResolventA-D}
            (A+z)^{-\frac{\gamma}{\alpha}}-(\rmD^{\alpha,m}+z)^{-\frac{\gamma}{\alpha}} = \int_0^\infty t^{\frac{\gamma}{\alpha}} \e^{-zt}\big(S(t)-G^{\alpha,m}(t)\big)\,\frac{\mathrm{d} t}{t}.
        \end{equation}
        Let $\sigma>0$ small enough and $\gamma<\gamma'<\beta-\sigma$. Then, for $f\in L_x^p$, \eqref{Eq:EstimateS-G} implies
        \begin{align*}
            \left\|(A+z)^{-\frac{\gamma}{\alpha}}f-(\rmD^{\alpha,m}+z)^{-\frac{\gamma}{\alpha}}f\right\|_{H^{\gamma,p}} &\leqslant \int_0^\infty t^{\frac{\gamma}{\alpha}} \e^{-\Re zt}\|S(t)f-G^{\alpha,m}(t)f\|_{H^{\gamma,p}}\,\frac{\mathrm{d} t}{t}\\
            &\lesssim \int_0^\infty t^{\frac{\beta+\gamma-\gamma'-\sigma}{\alpha}}\e^{(c-\Re z)t}\|f\|_{\mcB^{-\frac{\sigma}{2}}_{p,\infty}}\,\frac{\mathrm{d} t}{t}\\
            &\lesssim \left(\int_0^\infty t^{\frac{\beta+\gamma-\gamma'-\sigma}{\alpha}}\e^{(c-\Re z)t}\,\frac{\mathrm{d} t}{t}\right)\|f\|_{L_x^p},
        \end{align*}
        and the integral is finite as $\beta+\gamma-\gamma'-\sigma>0$ and $c-\Re z<0$. Thus, as $(\rmD^{\alpha,m}+z)^{-\frac{\gamma}{\alpha}}\in\mathcal{L}(L_x^p,H^{\gamma,p})$, it follows $(A+z)^{-\frac{\gamma}{\alpha}}\in\mathcal{L}(L_x^p,H^{\gamma,p})$. This give a kind of elliptic regularity estimate for $A$.
        \item[(b)] Consider $p=2$. Let $f,g\in\mathrm{D}(A)$, then it holds
        \begin{align*}
            \langle Af,g\rangle_{L_x^2} &= \lim_{t\to 0} \left\langle \frac{S(t)f-f}{t},g\right\rangle_{L_x^2}\\
            &= \lim_{t\to 0} \left\langle f, \frac{S(t)g-g}{t}\right\rangle_{L_x^2}\\
            &= \langle f,Ag\rangle_{L_x^2},
        \end{align*}
        where in second line we used \eqref{Eq:SymS}. Thus, $A$ is symmetric and, by the previous point, $-c-1$ is in the resolvent set of $A$, implying the self-adjointness of $A$. Moreover, $(A+c+1)^{-1}$ is compact on $L_x^2$ by the previous point and the compact embedding of $H^{\gamma,2}$ in $L_x^2$. Thus, by \cite[Theorem 3.6.29.]{Ka95}, there exists a complete orthonormal basis $(\phi_n)_{n\in\N}$ and a non-decreasing sequence
        $$\lambda_0 \leqslant \lambda_1 \leqslant \cdots \leqslant \lambda_n \to\infty$$
        such that $A\phi_n=\lambda_n\phi_n$. By the previous point and Sobolev embedding, it is easily verified that $\phi_n\in H^{\beta-,p}$ for any $p\in(1,\infty)$. As a consequence, $\lambda_n$ are eigenvalues of $A$ on  $L_x^p$. Using Sobolev embeddings and the "elliptic regularity" estimate, it is easily seen that $A+\lambda$ is injective (respectively surjective) on some $L_x^p$ if and only if it is on all $L_x^p$-spaces. 
        
        It only remains to prove Weyl law \eqref{Eq:WeylLawAndersonTheorem} for $A$. By Lemma \ref{Cor:WeylLawEquivTrace}, it suffices to prove a small time asymptotic of $\operatorname{Tr}\big(S(t)\big)$. A direct computation gives
        $$0\leqslant \operatorname{Tr}\left(\e^{-t|\nabla|^\alpha}-G^{\alpha,m}(t)\right)\lesssim t \operatorname{Tr}\left(\e^{-\frac{t}{2}|\nabla|^\alpha}\right)\,,$$
        thus Lemma \ref{Lem:LooseAsymptoticFractionalHeatTrace} implies
        \begin{equation}\label{Eq:SmallTimeAsymptoticTraceMassiveFractionalHeatKernel}
            \operatorname{Tr}\big(G^{\alpha,m}(t)\big)\sim \operatorname{Tr}\left(\e^{-t|\nabla|^\alpha}\right)\,.
        \end{equation}
        Now, as in \cite{BDM22}, using the semigroup property and the self adjointness, one can see
        $$\operatorname{Tr}\big(S(t)\big) = \int_{\T^d} S(t,x,x) \dint x\,,$$
        where $S(t,x,y)\eqdef S(t)\big(\delta_y\big)(x)$. Let $0<\eps<\frac{\beta}{2}$, there exists $1<p<2$ such that $\delta_y \in \mcB^{-\eps}_{p,\infty}$. Using \eqref{Eq:EstimateS-G}, it follows
        $$\sup_{y\in\T^d}\|S(t)\delta_y-G^{\alpha,m}(t)\delta_y\|_{L^1_x} \lesssim \sup_{y\in\T^d}\|S(t)\delta_y-G^{\alpha,m}(t)\delta_y\|_{\mcB^{\eps}_{p,\infty}} \lesssim t^{\frac{\beta-d-2\eps}{\alpha}}\,.$$
        Thus, using Sobolev embeddings,
        \begin{align*}
            \operatorname{Tr}\big(S(t)\big) &= \int_{\T^d} S(t,x,x)\dint x\\
            &= \int_{\T^d} G^{\alpha,m}(t,x,x)\dint x + \int_{\T^d} S(t,x,x)-G^{\alpha,m}(t,x,x) \dint x\\
            &= \operatorname{Tr}\big(G^{\alpha,m}(t)\big) + o\left(t^{-\frac{d}{\alpha}}\right)\\
            &\sim \operatorname{Tr}\left(\e^{-t|\nabla|^\alpha}\right)\,,
        \end{align*}
        using \eqref{Eq:SmallTimeAsymptoticTraceMassiveFractionalHeatKernel} and Lemma \ref{Lem:LooseAsymptoticFractionalHeatTrace}. Then, \eqref{Eq:WeylLawAndersonTheorem} follows from Lemma \ref{Cor:WeylLawEquivTrace}.
       
        \item[(c)] Let $\lambda>-\lambda_0$. On $L^2_x$, self-adjointness and positiveness imply $\lambda+A$ generates a bounded analytic semigroup on $\{\Re z>0\}$. By the Hille-Yoshida theorem, for any $\phi\in\Phi=\{\pm\frac{k\pi}{2k+1},\;k\in\N\}$, $S^\phi$, the semigroup defined by the solution to \eqref{mainSPDErotated}, admits a quasi sectorial generator on any $L^p_x$, which is equal to $\e^{i \phi}A$ on $D_{p,2}=(\lambda+A)^{-1}(L^p_x\cap L^2_x)$, which is dense in both $D_p=(\lambda+A)^{-1}L^p_x$ and $D_2=(\lambda+A)^{-1}L^2_x$. Let 
        $$\omega(p,\phi) = \omega_{L_x^p}(S^\phi) = \lim_{t\to\infty}\frac{1}{t}\log\left\|S^\phi(t)\right\|_{\mathcal{L}(L_x^p)}$$
        the growth bound of $S^\phi$ on $L^p_x$. By the Riesz-Thorin theorem, $\omega(p,\phi)$ is a convex function of $\frac{1}{p}$. Let $1<q<p$ and $0<\eps<1$, it holds
        \begin{align*}
            \left\|S^\phi(t)u\right\|_{L^q_x} &\leqslant \left\|S^\phi(t)u\right\|_{L^p_x}\\
            &\lesssim \left\|S^\phi((1-\eps)t)\right\|_{\mathcal{L}(L_x^p)}\left\|S^\phi(\eps t)u\right\|_{\mcB^{d(\frac{1}{q}-\frac{1}{p})}_{q,\infty}}\\
            &\lesssim_\phi t^{\frac{\beta-2d(\frac{1}{q}-\frac{1}{p})}{\alpha}}\left\|S^\phi((1-\eps)t)\right\|_{\mathcal{L}(L_x^p)}\|u\|_{L^q_x}
        \end{align*}
        for $|\frac{1}{q}-\frac{1}{p}|$ small enough. Thus, $\omega(q,\phi) \leqslant \frac{1}{1-\eps}\omega(p,\phi)$. Sending $\eps$ to $0$, we see $p\mapsto \omega(p,\phi)$ is non-decreasing. Similarly, 
        \begin{align*}
            \left\|S^\phi(t)u\right\|_{L^p_x} &\lesssim \left\|S^\phi(t)u\right\|_{\mcB^{d(\frac{1}{q}-\frac{1}{p})}_{q,\infty}}\\
            &\lesssim_\phi t^{\frac{\beta-2d(\frac{1}{q}-\frac{1}{p})}{\alpha}}\left\|S^\phi((1-\eps)t)\right\|_{\mathcal{L}(L_x^q)}\|u\|_{L^p_x}\,,
        \end{align*}
        implying $p\mapsto \omega(p,\phi)$ is also non-increasing, thus constant. Hence $\omega(p,\phi) = \omega(2,\phi) = -\cos(\phi)\lambda_0$. In particular, $\e^{i\phi}(\lambda+A)$ generates a bounded $\mathcal{C}_0$-semigroup on any $L^p_x$. Hence, it is $\theta$-sectorial for any $\theta\in(\frac{\pi}{2},\pi)$, which implies that $\lambda+A$ is $\theta$-sectorial for any $\theta>0$ on any $L^p_x$.\\

        Now let $\theta \in(0,\pi)$ and $F$ be a bounded holomorphic function on $\Sigma_\theta$. We want to prove that there exists a constant $C$ depending only $\theta,p,\rmD^{\alpha,m}$, such that
        $$\left\|\int_\Gamma F(z) (z-\lambda-A)^{-1}\dint z\right\|_{\mathcal{L}(L_x^p)}\leqslant C\|F\|_{L^\infty(\Sigma_\theta)}.$$
        This result is true for $\rmD^{\alpha,m}$ by the Mikhlin-Hormänder theorem. Thus, we can focus on
        $$\left\|\int_\Gamma F(z) \left((z-\lambda-A)^{-1}-(z+\lambda-\rmD^{\alpha,m})^{-1}\right)\dint z\right\|_{\mathcal{L}(L_x^p)}\leqslant C\|F\|_{L^\infty(\Sigma_\theta)}.$$
        Let $\Gamma = \Gamma_+\cup \Gamma_0\cup \Gamma_-$, with $\Gamma_+ =  e^{i \theta/2}(\infty,\delta)$, $\Gamma_- = e^{-i\theta/2}(\delta,\infty)$ and $\Gamma_0 = \delta \e^{i[\theta/2,-\theta/2]}$, where $\delta>0$ is small enough so that $\left((z-\lambda-A)^{-1}\right)_{|z|=\delta}$ is uniformly bounded by the open mapping theorem. For $z\in\Gamma_\pm$,
        $$(z-\lambda-A)^{-1}-(z-\lambda-\rmD^{\alpha,m})^{-1} = \int_0^{\infty} \e^{\e^{\pm i\phi}t(z-\lambda)}\big(S^{\pm\phi}(t)-G^{\alpha,m}(\e^{\pm i\phi}t)\big)\dint t$$
        for $\phi\in\Phi$ such that $\frac{\pi}{2}<\phi+\frac{\theta}{2}<\pi$. Now, by \eqref{Eq:EstimateS-G} and Remark~\ref{Rem:RemarksTheorem1}, there exists $\gamma>0$ such that for any $\eps>0$ small enough,
        \begin{align*}
            \left\|(z+\lambda+A)^{-1}-(z+\lambda-\rmD^{\alpha,m})^{-1}\right\|_{\mathcal{L}(L_x^p)}&\lesssim_{\phi,\eps} \int_0^{1} t^\gamma \e^{t|z|\cos(\phi+\frac{\theta}{2})}\dint t\\
            &+ \int_1^{\infty}\e^{t\left(|z|\cos(\phi+\frac{\theta}{2})-\cos(\phi)\lambda\right)}\left(1+\e^{t\left(\cos(\phi)\lambda_0+\eps\right)}\right)\dint t\\
            &\lesssim_{\phi,\theta,\eps} |z|^{-1-\gamma}\,,
        \end{align*}
        as $\lambda>-\lambda_0$ and $\cos(\phi+\frac{\theta}{2})<0<\cos(\phi)$, which allows to conclude as $|z|^{-1-\gamma}$ is integrable as $|z|$ goes to infinity.
    \end{itemize}
\end{proof}

\begin{remark}
    This operator is defined up to a constant shift. A convenient choice may be to chose this constant shift in order to fix the lower eigenvalue of $A$. Another way to proceed can be to obtain a fine description of the blow-up of $\mathfrak{C}_\eps^{\alpha,m}$ and keep only the diverging part while discarding any bounded term.
\end{remark}

\subsection{The fractional Laplacian in infinite volume}\label{sec:4-3}

We now prove Theorem \ref{maintheoremspde}, item 2, by using Proposition \ref{Prop:KleinLandau}, as in \cite{HsL26}.

\begin{proof}[of Theorem~\ref{Th:ExistenceAndersonOp}.2]
    Let $T=1$, let $D_t \eqdef L^2_{\mcb e_{t}}$ (we recall that this space has been defined below \eqref{eq:defLp}) and let $S(t) f$ be the unique solution to \eqref{mainSPDE} at time $0\leqslant t\leqslant 1$, starting from $f\in D_t$. For $\eps>0$, let $S_{\eps}(t) = e^{t(\rmD^{\alpha,m}_\xi-\xi_\eps+\mfC^{\alpha,m}_\eps)}$. Then, acting like in the proof of Proposition \ref{Prop:ExistenceRenormalizedSemigroupFiniteVolume}, we can ensure the existence of an event $\Omega$ of probability one and a deterministic sequence $\eps_k$ going to $0$ such that for any $f\in V\eqdef \mathrm{Span}_{\Q}\left(\mathbf{1}_{B(a,r)},\; a\in\Q^d,\; r\in\Q_+\right)$, $\big(S_{\eps_k}(t)f\big)_{0<t\leqslant 1}$ converges to $\big(S(t)f\big)_{0<t\leqslant 1}$ in $\mathcal{E}_{2,2,l,1}^{\gamma,\gamma-\eta}$, for any $\gamma<\beta$ and $\eta>0$, on $\Omega$. We now work on $\Omega$. Let us verify each item of Theorem \ref{Prop:KleinLandau}.
    \begin{enumerate}
        \item Clearly, for $s\leqslant t$, $D_s\subset D_t$ and $\mathcal{C}^\infty_c(\R^d)\subset \bigcup_{0\leqslant t\leqslant T}D_t$, which is thus dense in $L^2_x$.
        \item Thanks to Theorem \ref{maintheoremspde}, and in particular \eqref{eq:estimateS-Ginf}, we are allowed to start the dynamic from $L^2_{\mcb e_{\ell}}$ for any $\ell\in\R$, for $0\leqslant s\leqslant t\leqslant T$, $S(s)D_{t-s}\subset D_s$ and $S(t-s)S(s)=S(t)$ on $D_t$.
        \item Let $f,g\in V$.  Then, $S_{\eps_k}(t)f$ and $S_{\eps_k}(t)g$ converges in $L^2_x$ to $S(t)f$ and $S(t)g$ respectively. Thus, $\langle f,S_{\eps_k}(t) g\rangle_{L^2_x}=\langle S_{\eps_k}(t) f,g\rangle_{L^2_x}$ converges to $\langle f,S(t) g\rangle_{L^2_x}=\langle S(t) f,g\rangle_{L^2_x}$. Now, remark that both $\langle f,S(t) g\rangle_{L^2_x}$ and $\langle S(t) f,g\rangle_{L^2_x}$ are bilinear continuous on $D_t$ and $V$ is dense, so $\langle f,S(t) g\rangle_{L^2_x}=\langle S(t) f,g\rangle_{L^2_x}$ holds for any $f,g\in D_t$.
        \item Let $g\in L^2_x$ and $f\in V$. For $0\leqslant s\leqslant t$, for $\eta>0$ small enough, \eqref{eq:estimateS-Ginf} implies
        \begin{align*}
            \left|\langle g,\big(S(s)-\e^{-s\mathrm{D}^{\alpha,m}}\big) f\rangle_{L^2_x}\right| 
            &\leqslant \|g\|_{L^2_x}\|\big(S(s)-G^{\alpha,m}(s)\big)f\|_{L^2_x}\\
            &\lesssim (s\wedge 1)^{(\beta-2\sigma-\eta)/\alpha}\|\mcb e_{t} f\|_{L^2_x}\|g\|_{L^2_x}\,.
        \end{align*}
        By density of $V$, this holds for any $f\in D_t$. As $\big(G^{\alpha,m}(s)\big)_{s\geq0}$ is a $\mathcal{C}_0$-semigroup, $S(s) f$ converges weakly to $f$ in $L^2_x(\R^d)$ as $s$ goes to $0$. Moreover, as $\|S(s) f\|_{L^2_x}^2=\langle f,S(2s)f\rangle_{L^2_x}$ converges to $\langle f,f\rangle_{L^2_x}=\|f\|_{L^2_x}^2$, $S(s) f$ converges strongly to $f$ in $L^2_x$. Now, for $0< s\leqslant s'\leqslant t$ using $S(s')f=S(s'-s)S(s) f$, we get 
        \begin{align*}
            \langle f,\big(S(s')-S(s)\big)f\rangle_{L^2_x} &= \langle S(s)f,\big(S(s'-s)-1\big)f\rangle_{L^2_x}\\
            &\leqslant \|S(s) f\|_{L^2_x}\|S(s'-s)f-f\|_{L^2_x}\\
            &\lesssim \|\mcb e_{t} f\|_{L^2_x}\|S(s'-s)f-f\|_{L^2_x}\,,
        \end{align*}
        and the continuity follows.
    \end{enumerate}
    Thus, by Theorem \ref{Prop:KleinLandau}, $\big(S(t)\big)_{t>0}$ is the restriction of a unique self-adjoint operator $\mathrm{A}^{\!\alpha,m}$ on $L^2_x$.
\end{proof}

\subsection{The harmonic oscillator}\label{sec:4-4}

In this section, we adapt the argument of the finite volume case to the harmonic oscillator. The proof of the following proposition follows exactly the one of Proposition \ref{Prop:ExistenceRenormalizedSemigroupFiniteVolume}, using Proposition \ref{Prop:HermiteSpacesInterpolation} and Lemma \ref{lemma:constructsolOH}.

\begin{proposition}\label{Prop:ExistenceRenormalizedSemigroupHO}
    There exists a deterministic sequence $\eps_k$ converging to $0$ and an event $\overline{\Omega}$ of probability one, on which the following holds. Let $E$ be $\msH^{s,p}$ or $\msB^{s}_{p,q}$ ($p\in(1,\infty)$, $q\in(1,\infty)$ and $|s|<\beta$). For any $0<\eps<1$, let $\big(S_\eps(t)\big)_{t\geqslant 0}$ be the $\mathcal{C}_0$-semigroup solving \eqref{renormSPDE} on $E$. Then, 
    \begin{enumerate}
        \item there exists $M>0$ and $c\in\R$ such that,
        \begin{equation}
            \forall t>0,\; \|S_\eps(t)\|_{\mathcal{L}(E)}\leqslant M\mathrm{e}^{ct}\,,
        \end{equation}
        \item $\big(S_{\eps_k}(t)\big)_{t\geqslant 0}$ converges strongly to a $\mathcal{C}_0$-semigroup $\big(S(t)\big)_{t\geqslant 0}$ on E verifying the same bound,
        \item for any $f\in E$ and $g\in E'$, it holds
        \begin{equation}\label{Eq:SymSOH}
            \forall t>0,\; \langle S(t)f,g\rangle_{E,E'}=\langle f,S(t)g\rangle_{E,E'}.
        \end{equation}
    \end{enumerate}
\end{proposition}

\begin{proof}[of Theorem~\ref{Th:ExistenceAndersonOp}.3] This implies Theorem \ref{maintheoremspde}, item 3, exactly like Proposition \ref{Prop:ExistenceRenormalizedSemigroupFiniteVolume} implies the first item of Theorem \ref{maintheoremspde}, using Lemma~\ref{Lem:LooseAsymptoticMehlerTrace}, Corollary~\ref{Cor:WeylLawEquivTraceOH} and Proposition~\ref{Prop:HermiteSpacesEmbeddings}.
\end{proof}

\begin{appendix}

\renewcommand*{\thetheorem}{\mbox{\textrm A.\arabic{theorem}}}

\section{Estimates on the kernels}

\subsection{Kernels in the translation-invariant setting}

\begin{lemma}\label{estimatesKmu}
Let $\varsigma>0$ and $p,q\in[1,\infty]$. We have
\begin{equs} \label{eq:normKmu}\big\|\langle\bigcdot\rangle_\mu^\varsigma K_\mu\big\|_{L^1_{t,x}}&\leqslant1   \,,\\
\big\|\langle\bigcdot\rangle_\mu^\varsigma\partial_t^{\mfl_0}\D_x^{\mfl}K_\mu\big\|_{L^1_{t,x}}&\lesssim\mu^{-|\mfl_0,\mfl|_\mfs}\;\mathrm{for}\;\mathrm{all} \;(\mfl_0,\mfl)\in\N^{d+1}\,,\;|\mfl_0,\mfl|_\mfs\leqslant N\,,
\label{eq:spacederivKmu}
\\
\label{eq:KmuLp} \big\|\langle\bigcdot\rangle_\mu^\varsigma K_\mu\|_{L_{t,x}^{p,q}}&\lesssim \mu^{-\alpha\left(1-\frac{1}{p}\right)-d\left(1-\frac{1}{q}\right)} \;\mathrm{if}\;\alpha\left(1-\frac{1}{p}\right)+d\left(1-\frac{1}{q}\right)\leqslant N\,,
\end{equs}
uniformly in $\mu\in(0,1]$. In addition, we have the estimate with exponential weights for any $\ell\in\R$
\begin{equs}\label{Kmuexp}
    \big\|\mcb e_\ell Q_\mu\big\|_{L^1_x}\lesssim1\,.
\end{equs}
and for $\psi\in \mcD(\R^d)$ and $\nu\leqslant\mu$,
\begin{equs}\label{eq:comKmuKnu1}
    \big\| K_\mu\big(1-K_\nu\big)\psi\big\|_{L_x^p}&\lesssim \big(\nu/\mu\big)^{2}\|K_\nu\psi\|_{L_x^p}\,,
\end{equs}
and
\begin{equs}\label{eq:comKmuKnu}
    \big\|Q_\mu\mcb e_{-\ell}\big(1-Q_\nu\big)\psi\big\|_{L_x^p}&\lesssim \big(\nu/\mu\big)^{2}\| \mcb e_{-\ell}K_\nu\psi\|_{L_x^p}\,,
\end{equs}
Moreover, for $\psi\in\mcD(\Lambda)$, we have
\begin{equs} \label{eq:Kmunu} 
    \|\mcb e_{-\ell}K_\mu\psi\|_{L_{t,x}^{p,q}}&\leqslant\bigg(1\vee\Big(2\big(\nu/\mu\big)^{\alpha}-1\Big)^{2N}\bigg)  \|\mcb e_{-\ell}K_\nu\psi\|_{L_{t,x}^{p,q}}\,,
\end{equs}
uniformly in $\mu,\nu\in(0,1]$.
\end{lemma}

\begin{proof}
    Let us prove \eqref{eq:normKmu} and \eqref{eq:spacederivKmu} at once. We start by computing explicitly
    \begin{equs}
        K_{1,\mu}^{*N}(t,x)=\mu^{-N\alpha }\one_{(0,\infty)}(t)\frac{t^{N-1}}{(N-1)!}\e^{-t/\mu^\alpha}Q_\mu^{*N}(x)\,.
    \end{equs}
    We have the heat kernel representation
    \begin{equs}\label{heatkernelrep}
        \D_x^{\mfl}(1-\mu^2\Delta)^{-N}(x)=\frac{1}{(N-1)!}\int_0^\infty s^{N-1}\e^{-s}\D_x^\mfl h(\mu^2 s,x)\dint s\,,
    \end{equs}
    where $h$ is the euclidean heat kernel. We recall the standard estimate
    \begin{equs}\label{eq:classiqhk}
        |\D_x^\mfl h(t,x)|\lesssim t^{-(d+|\mfl|)/2}\left(1+\frac{|x|}{\sqrt{t}}\right)^{|\mfl|}\e^{-\frac{|x|^2}{4t}}\,.
    \end{equs}
    We first get the estimate in space
    \begin{equs}
        ~&\int_{\R^d}\langle x\rangle_\mu^\varsigma|\D_x^{\mfl}(1-\mu^2\Delta)^{-N}(x)|\dint x\\ &\lesssim\mu^{-(d+|\mfl|)}\int_0^\infty s^{-1+(-d+2N-|\mfl|)/2}\e^{-s}\int_{\R^d}\langle x\rangle_\mu^\varsigma\left(1+\frac{|x|}{\sqrt{\mu^2 s}}\right)^{2N-2+|\mfl|}\e^{-\frac{|x|^2}{4\mu^2 s}}\dint x\dint s
    \end{equs}
    We can perform the change of variables $x\to\mu\sqrt{s} x$ to get
    \begin{equs}
        ~&\int_{\R^d}\langle x\rangle_\mu^\varsigma|\D_x^{\mfl}(1-\mu^2\Delta)^{-N}(x)|\dint x\\ 
        &\lesssim\mu^{-|\mfl|}\int_0^\infty s^{-1+(2N-|\mfl|)/2}\e^{-s}\int_{\R^d}\langle x\sqrt{s}\rangle^\varsigma\left(1+|x|\right)^{|\mfl|}\e^{-\frac{|x|^2}{4}}\dint x\dint s
    \end{equs}
    We can see that
    \begin{equs}
        \int_0^\infty s^{-1+(2N-|\mfl|)/2}\e^{-s}\langle x\sqrt{s}\rangle^\varsigma\dint s\lesssim \langle x\rangle^\varsigma\,,
    \end{equs}
    provided that $|\mfl|\leq 2N$, so that the above is bounded by
    \begin{equs}
        \mu^{-|\mfl|}\int_{\R^d}\langle x\rangle^\varsigma\left(1+|x|\right)^{|\mfl|}\e^{-\frac{|x|^2}{4}}\dint x\lesssim\mu^{-|\mfl|}\,.
    \end{equs}
    Using that $\langle t,x\rangle^\varsigma_\mu\lesssim (1+t^{2/\alpha}/\mu^2)^{\varsigma/2}\langle x\rangle_\mu^\varsigma$, we have
    \begin{equs}
        \big\|\langle\bigcdot\rangle_\mu^\varsigma\partial_t^{\mfl_0}\D_x^{\mfl}K_\mu\big\|_{L^1_{t,x}}\lesssim\mu^{-N\alpha}\int_0^\infty t^{N-1-\mfl_0}\e^{-t/\mu^\alpha}\left(1+\frac{t}{\mu^\alpha}\right)^{\mfl_0+\frac{\varsigma}{\alpha}}\dint t\int_{\R^d}\langle x\rangle_\mu^\varsigma|\D_x^{\mfl}(1-\mu^2\Delta)^{-N}(x)|\dint x\,.
    \end{equs}
    We just have to estimate the integral in time. A simple change of variables by scaling in $\mu^\alpha$ yields
    \begin{equs}
        \int_0^\infty t^{N-1-\mfl_0}\e^{-t/\mu^\alpha}\left(1+\frac{t}{\mu^\alpha}\right)^{\mfl_0+\frac{\varsigma}{\alpha}}\dint t=\mu^{\alpha(N-\mfl_0)}\int_0^\infty t^{N-1-\mfl_0}\e^{-t}\left(1+t\right)^{\mfl_0+\frac{\varsigma}{\alpha}}\dint t\lesssim\mu^{\alpha(N-\mfl_0)}\,,
    \end{equs}
    provided that $\mfl_0\leq N$. Putting everything together, we get the result. The proof of \eqref{eq:KmuLp} being in every point similar, we take the freedom not to do it. We now prove \eqref{Kmuexp}. We use again the heat kernel representation \eqref{heatkernelrep} to get
    \begin{equs}
        \big\|\mcb e_\ell Q_\mu\big\|_{L^1_x}&\lesssim\mu^{-d}\int_0^\infty s^{-d/2}\e^{-s}\int_{\R^d}\mcb e_{\ell}(x)\e^{-\frac{|x|^2}{4\mu^2 s}}\dint x\dint s\\
        &=\int_0^\infty \e^{-s}\int_{\R^d}\e^{\ell\langle \mu\sqrt s x\rangle^{1/2}-\frac{|x|^2}{4}}\dint x\dint s
    \end{equs}
    A quick computation shows that
    \begin{equs}
        \e^{-s+\sqrt{\ell\mu|x|} s^{1/4}}\lesssim \e^{-s/2}\e^{c\ell^{4/3}\mu^{2/3}|x|^{2/3}}\,,
    \end{equs}
    where $c$ is a universal constant. So
    \begin{equs}
        \int_0^\infty \e^{-s+\sqrt{\ell\mu|x|} s^{1/4}}\dint s\lesssim \e^{c\ell^{4/3}\mu^{2/3}|x|^{2/3}}\,.
    \end{equs}
    Eventually, we have
    \begin{equs}
        \big\|\mcb e_\ell Q_\mu\big\|_{L^1_x}&\lesssim \int_{\R^d}\e^{c\ell^{4/3}\mu^{2/3}|x|^{2/3}-\frac{|x|^2}{4}}\dint x\lesssim 1\,.
    \end{equs}
    Finally, a proof of \eqref{eq:comKmuKnu} can be found in \cite{Duc22}, with an easy adaptation with the weight as above. The same holds for \eqref{eq:Kmunu} whose proof can be found in \cite[Lemma 4.10]{Duc21}.
\end{proof}

\begin{lemma}\label{fracheatkernel}
    For $\alpha<2$, we have the pointwise bound on the Green's function $G$ of $(m^2-\Delta)^{\alpha/2}$
    \begin{equs}\label{eq:estimate_heatkernel}
        \big|\D_t^{\mfl_0}\D_x^\mfl G(t,x)\big|\lesssim t^{-\frac{d+|\mfl_0,\mfl|_\mfs}{\alpha}} \big(1+{|x|}/{t^{1/\alpha}}\big)^{-(d+\alpha+|\mfl|)}\e^{-c m|x|}\,,
    \end{equs}
    where $c>0$ is a constant depending only on $\alpha$ and $d$. For $\alpha=2$, that is to say that $G$ is the Green's function of $m^2-\Delta$, \eqref{eq:estimate_heatkernel} is of course not optimal and we have
    \begin{equs}\label{eq:estimate_heatkernel_alphais2}
        \big|\D_t^{\mfl_0}\D_x^\mfl G(t,x)\big|\lesssim t^{-\frac{d+|\mfl_0,\mfl|_\mfs}{2}} \e^{-m^2 t}\e^{-c\frac{|x|^2}{4t}}\,,
    \end{equs}
    for any $c\in(0,1)$.
\end{lemma}

\begin{proof}
    We only prove the first statement since the second one is a more standard estimate. We have the heat kernel decomposition
    \begin{equs}\label{heatrep}
        G(t,x)=\int_0^\infty\frac{1}{(4\pi s)^\frac{d}{2}}\e^{-\frac{|x|^2}{4s}}\e^{-m^2 s}\theta_\alpha(t,s)\dint s\,,
    \end{equs}
    where $\theta_\alpha$ is the transition density of an $\alpha/2$-stable subordinator, uniquely characterized by the relation
    \begin{equs}\label{charactheta}
        \int_0^\infty \e^{-\lambda s}\theta_\alpha(t,s)\dint s=\e^{-t\lambda^\frac{\alpha}{2}}\,.
    \end{equs}
    We have the estimate
    \begin{equs}
        |\D_t^{\mfl_0}\theta_\alpha(t,s)|\lesssim t^{1-\mfl_0}s^{-1-\frac{\alpha}{2}}\,.
    \end{equs}
    This yields, reusing the heat kernel estimate \eqref{eq:classiqhk},
    \begin{equs}\label{side2}
        \big|\D_t^{\mfl_0}\D_x^\mfl G(t,x)\big|&\lesssim \int_0^\infty s^{-\frac{d+|\mfl|}{2}}\left(1+\frac{|x|}{\sqrt{s}}\right)^{|\mfl|}\e^{-\frac{|x|^2}{4s}}\e^{-m^2 s}t^{1-\mfl_0}s^{-1-\frac{\alpha}{2}}\dint s\\
        &\lesssim t^{1-\mfl_0}\e^t|x|^{-(d+\alpha+|\mfl|)}\int_0^\infty s^{\frac{d+\alpha+|\mfl|}{2}-1}\big(1+\sqrt{s}\big)^{|\mfl|}\e^{-s/4-m^2|x|^2/s}\dint s\,.
    \end{equs}
    The last integral is bounded by a constant times $\e^{-cm|x|}$, where $c$ only depends on $d$ and $\alpha$. Moreover, the characterisation \eqref{charactheta} gives
    \begin{equs}
        \D_t^{\mfl_0}\theta_\alpha(t,s)=(-1)^{\mfl_0}t^{-\frac{2+\alpha\mfl_0}{\alpha}}\D_t^{\mfl_0}\theta_\alpha(1,st^{-{2/\alpha}})\,.
    \end{equs}
    Taking back \eqref{heatrep}, and performing the change of variables $s\to st^{\frac{2}{\alpha}}$, we get, for $|x|^\alpha\leq t$,
    \begin{equs}
        \big|\D_t^{\mfl_0}\D_x^\mfl G(t,x)\big|&\lesssim t^{-\frac{d+|\mfl|+\alpha\mfl_0}{\alpha}}\int_0^\infty s^{-\frac{d+|\mfl|+\alpha}{2}-1}\left(1+\frac{|x|}{s^{1/2}t^{1/\alpha}}\right)^{|\mfl|}\e^{-\frac{|x|^2}{4st^{2/\alpha}}}\e^{-m^2st^{2/\alpha}}\D_t^{\mfl_0}\theta_\alpha(1,s)\dint s\\
        &\lesssim t^{-\frac{d+|\mfl_0,\mfl|_\mfs}{\alpha}}\,.
    \end{equs}
    This bound, together with \eqref{side2}, gives the desired result.
\end{proof}



\begin{lemma}\label{heatkernelestimate}
    Let $u$ and $\ell\in\R$ such that $\mcb e_{-\ell}u\in\mcB_{r,q}^{-\sigma}$. We have
    \begin{equation}
        \|\mcb e_{-\ell}\mcR_\mu^2G_\mu u\|_{L^{r}_{x}}\lesssim\mu^{-\sigma}\|\mcb e_{-\ell}u\|_{\mcB_{r,q}^{-\sigma}}\,.
    \end{equation}
\end{lemma}

\begin{proof}
    We do the proof in the case $\alpha<2$ since it is the hardest one. We first write, using Lemma \ref{ITexp},
    \begin{equs}
        \mcb e_{-\ell}(x)\big(\mcR_\mu^2G_\mu u\big)(t,x)&=\mcb e_{-\ell}(x)\int_{\R^d}\big(\mcR_\mu^2 G_\mu)(t,x-y)u(y)\dint y\\
        &=\int_{\R^d}\mcb e_{-\ell}(x)\mcb e_\ell(y)\big(\mcR_\mu^2 G_\mu)(t,x-y)\mcb e_{-\ell}(y)u(y)\dint y\,.
    \end{equs}
    Let us take a closer look at the kernel appearing here. First, we can write that there is a finite set $\mathfrak{P}$ such that
    \begin{equs}
        \big(\mcR_\mu^2 G_\mu)(t,x-y)=\sum_{\pi\in\mathfrak{P}}G_\mu^\pi(t,x-y)\,,
    \end{equs}
    where each $G_\mu^\pi(t,x-y)$ is of the form $\chi^{(2N-n)}(t/\mu^\alpha)P_\pi(\mu^2/t) G(t,x-y)$ for some polynomial of degree bounded by $n$ for some $0\leq n\leq 2N$, following Lemma \ref{fracheatkernel}. We get, thanks to Lemma \ref{ITexp},
    \begin{equs}
        \mcb e_{-\ell}(x)\mcb e_\ell(y)|G_\mu^\pi(t,x-y)|&\lesssim\mcb e_\ell(x;y)\chi(t/\mu^\alpha)t^{-\frac{d}{\alpha}} \big(1+{|x-y|}/{t^{1/\alpha}}\big)^{-(d+\alpha)}\e^{-c m|x-y|}\\
        &\lesssim\chi(t/\mu^\alpha)t^{-\frac{d}{\alpha}} \big(1+{|x-y|}/{t^{1/\alpha}}\big)^{-(d+\alpha)}\,.
    \end{equs}
    We can conclude with standard heat kernel estimates and the fact that $\chi$ imposes $t\geq\mu^\alpha$.
\end{proof}

\begin{lemma}\label{estimdotGmu}
The following holds uniformly in $\mu\in(0,1]$, $\ell\geq0$, and $p\in[1,\infty]$,
    \begin{equation}
        \big\|\langle\bigcdot\rangle_\mu^{\varsigma}\mcb e_\ell\mcR_\mu\dot G_\mu\big\|_{L_{t,x}^{1,p}}\lesssim\mu^{-1+\alpha-d(1-1/p)}\,,
    \end{equation}
    where $0\leq\varsigma<\alpha$ if $\alpha<2$ and $\varsigma\geq0$ if $\alpha=2$.
    Moreover, we have the weighted estimate
    \begin{equs}\label{GmuL1weight}
        \left(\int_{\R^d}\left|\int_0^ts^{-\kappa}\big|\big(\mcR^2_\mu\dot G_\mu\big)(s,y)\big|\dint s\mcb e_{\ell}(y)\right|^p\dint y\right)^\frac{1}{p}\lesssim\mu^{-1+\alpha-d(1-1/p)-\kappa}\,.
    \end{equs}
\end{lemma}
\begin{proof}
    Let us prove the first estimate. We prove it for $\alpha<2$ and the other case is easier, as one can benefit from the Gaussian decay of the heat kernel in $x/\sqrt{t}$. Thanks to Lemma \ref{fracheatkernel}, we have the pointwise estimate
    \begin{multline*}
        |\mcR_\mu\dot G_\mu(t,x)|\lesssim \sum_{n=0}^N\mu^{-1-\alpha}\chi^{(N-n+1)}(t/\mu^\alpha)\\\times\sum_{k=0}^N t^{-\frac{d+n\alpha+2kd}{\alpha}+1}\mu^\frac{n\alpha+2kd}{\alpha} \big(1+{|x|}/{t^{1/\alpha}}\big)^{-(d+\alpha+2kd)}\e^{-cm|x|}\,.
    \end{multline*}
    This allows to get
    \begin{multline*}
        \big\|\langle\bigcdot\rangle_\mu^{\varsigma}\mcb e_\ell\mcR_\mu\dot G_\mu\big\|_{L_{t,x}^{1,p}}\lesssim \mu^{-1-\alpha}\sum_{n,k=0}^N\int_{\mu^\alpha}^{2\mu^\alpha}t^{-\frac{d+n\alpha+2kd}{\alpha}+1}\mu^\frac{n\alpha+2kd}{\alpha}\\\times\left(\int_{\R^d}\langle t,x\rangle_\mu^{p\varsigma}\big(1+{|x|}/{t^{1/\alpha}}\big)^{-p(d+\alpha+2kd)}\e^{p\big(\ell\langle x\rangle^\frac{1}{2}-cm|x|\big)}\dint x\right)^\frac{1}{p}\dint t\,.
    \end{multline*}
    We can now perform the changes of variables $t\to\mu^\alpha t$ and $x\to\mu x$. It yields
    \begin{equs}
        &\big\|\langle\bigcdot\rangle_\mu^{\varsigma}\mcR_\mu\dot G_\mu\big\|_{L_{t,x}^{1,p}}\\
        &\lesssim\mu^{-1+\alpha-d(1-1/p)}\sum_{n,k=0}^N\int_1^2 t^{-\frac{d+n\alpha+2kd}{\alpha}+1}\left(\int_{\R^d}\langle t,x\rangle^{p\varsigma}\big(1+{|x|}/{t^{1/\alpha}}\big)^{-p(d+\alpha+2kd)}\dint x\right)^\frac{1}{p}\dint t\\
        &\lesssim \mu^{-1+\alpha-d(1-1/p)}\sum_{k=0}^N\left(\int_{\R^d}\big(1+{|x|}\big)^{-p(d+\alpha-\varsigma+2Nd)}\dint x\right)^\frac{1}{p}\,.
    \end{equs}
    This last integral being finite, we get the result. Note that at this point, it is crucial that the exponent on $\langle\bigcdot\rangle_\mu$ is smaller than $\alpha$ for the integral to be convergent. We now turn to the second estimate. With the same arguments, it can be bounded by
    \begin{multline*}
        \mu^{-1-\alpha-d(1-1/p)-\kappa}\sum_{n,k=0}^N\bigg(\int_{\R^d}\left|\int_1^2 s^{-\frac{d+2n\alpha+2kd}{\alpha}+1-\kappa}\big(1+{|y|}/{s^{1/\alpha}}\big)^{-d+\alpha+4kd}\dint s\right|^p\\
        \times\e^{p\big(\ell|y|^{1/2}-cm\mu|y|\big)}\dint y\bigg)^\frac{1}{p}
    \end{multline*}
    Noticing that the exponential above is bounded, we get the result as above.
\end{proof}

\begin{lemma}\label{intKmuweight}
The following estimate holds uniformly in $\mu\in(0,1]$ and $x\in\R^d$,
\begin{equation}
\int_\R\langle t,x\rangle_\mu^{\alpha-\beta}K_{1,\mu}^{*N}(t,x)\dint t\lesssim \langle x\rangle_\mu^{\alpha-\beta}Q_\mu^{*N}(x).
\end{equation}
\end{lemma}
\begin{proof}
We first compute $K_\mu$ explicitely:
\begin{equs}
    K_{1,\mu}^{*N}(t,x)=\mu^{-N\alpha }\one_{(0,\infty)}(t)\frac{t^{N-1}}{(N-1)!}\e^{-t/\mu^\alpha}Q_\mu^{*N}(x).
\end{equs}
Moreover, we have, for $(t,x)\in\SS$, $\langle t,x\rangle_\mu^{\alpha-\beta}\lesssim(1+t^{2/\alpha}/\mu^2)^\frac{\alpha-\beta}{2}\langle x\rangle_\mu^{\alpha-\beta}$. It yields
\begin{equs}
    \int_\R\langle t,x\rangle_\mu^{\alpha-\beta}K_{1,\mu}^{*N}(t,x)\dint t\lesssim\mu^{-N\alpha}\langle x\rangle_\mu^{\alpha-\beta}Q_\mu^{*N}(x)\int_0^\infty\left(1+\frac{t^{2/\alpha}}{\mu^2}\right)^\frac{\alpha-\beta}{2}t^{N-1}\e^{-t/\mu^\alpha}\dint t.
\end{equs}
Performing the change of variable $t\to\mu^\alpha t$ in the last integral, we get that it equals
\begin{equs}
\mu^{N\alpha}\int_0^\infty(1+t^{2/\alpha})^\frac{\alpha-\beta}{2}t^{N-1}\e^{-t}\dint t\lesssim\mu^{N\alpha}.
\end{equs}
From this, we get the result.
\end{proof}

\subsection{Kernels associated with the harmonic oscillator}

\begin{lemma}
    Let $a\in[0,3/2]$ and $s\leq t$. The following estimates hold for $t\in(0,1]$,
    \begin{equation}\label{eq:estMt}
        \Big|\D_t^{\mfl_0}\D_x^{\mfl_1}\D_y^{\mfl_2}\Big(\e^{as\frac{|x|^2+|y|^2}{8}}M(t,x,y)\Big)\Big|\lesssim t^{-\mfl_0-\frac{|\mfl_1|+|\mfl_2|}{2}}G(8t,x-y)\e^{-3t\frac{|x+y|^2}{32}}\,,
    \end{equation}
    and
    \begin{equs}\begin{aligned}
  \label{eq:estimatediffkernels}
\Big|\D_t^{\mfl_0}\D_x^{\mfl_1}\D_y^{\mfl_2}\Big(\e^{s\frac{|x|^2+|y|^2}{8}}\big(M(t,x,y)-&\e^{-t\frac{|x|^2+|y|^2}{4}}G(t,x-y)\big)\Big)\Big|\\&\lesssim t^{1-\mfl_0-\frac{|\mfl_1|+|\mfl_2|}{2}}\big(|x|^2+|y|^2\big)G(8t,x-y)\,.
           \end{aligned}
    \end{equs}
\end{lemma}

\begin{proof}
    Let us start with the first estimate. We have
    \begin{equs}
         ~&\Big|\D_t^{\mfl_0}\D_x^{\mfl_1}\D_y^{\mfl_2}\Big(\e^{as\frac{|x|^2+|y|^2}{8}}M(t,x,y)\Big)\Big|\lesssim\sum_{i_1=0}^{\mfl_1}\sum_{i_2=0}^{\mfl_2}\big|\D_t^{\mfl_0}\D_x^{i_1}\D_y^{i_2}M(t,x,y)\big|\Big|\D_x^{\mfl_1-i_1}\D_y^{\mfl_2-i_2}\e^{as\frac{|x|^2+|y|^2}{8}}\Big|\\
         &\lesssim \sum_{i_1=0}^{\mfl_1}\sum_{i_2=0}^{\mfl_2}t^{-\mfl_0-\frac{|i_1|+|i_2|}{2}}\bigg(1+\frac{|x-y|}{\sqrt{t}}\bigg)^{|\mfl_0|}\big(1+\sqrt{t}|x|\big)^{|i_1|}\big(1+\sqrt{t}|y|\big)^{|i_2|}\e^{-\frac{|x-y|^2}{4t}}\e^{-t\frac{|x+y|^2}{4}}\\
         &\qquad\times s^{\frac{|\mfl_1-i_1|+|\mfl_2-i_2|}{2}}(1+\sqrt{s}|x|)^{|\mfl_1-i_1|}(1+\sqrt{s}|y|)^{|\mfl_2-i_2|}\e^{as\frac{|x|^2+|y|^2}{8}}\\
         &\lesssim \sum_{i_1=0}^{\mfl_1}\sum_{i_2=0}^{\mfl_2}t^{-\mfl_0-\frac{|i_1|+|i_2|}{2}}\e^{-7\frac{|x-y|^2}{32t}}\e^{-7t\frac{|x+y|^2}{32}}s^{\frac{|\mfl_1-i_1|+|\mfl_2-i_2|}{2}}\e^{3s\frac{|x|^2+|y|^2}{16}}\\
         &\lesssim t^{-\mfl_0-\frac{|\mfl_1|+|\mfl_2|}{2}}\e^{-\frac{|x-y|^2}{32t}}\e^{-3\frac{|x-y|^2}{16t}-7t\frac{|x+y|^2}{32}+3t\frac{|x|^2+|y|^2}{16}}
    \end{equs}
    We rework a bit the term in the second exponential.
    \begin{equs}
        \frac{3t}{16}(|x|^2+|y|^2)-\frac{7t}{32}|x+y|^2-\frac{3}{16t}|x-y|^2\leq \frac{t}{32}\Big(6(|x|^2+|y|^2)-7|x+y|^2-6|x-y|^2\Big)\,.
    \end{equs}
    We then get
    \begin{equs}
       6(|x|^2+|y|^2)-7|x+y|^2-6|x-y|^2&=-6(|x|^2+|y|^2)-|x+y|^2\\
       &\leq -|x+y|^2\,,
    \end{equs}
    which yields the result. The second estimate follows the same lines. One has to note that
    \begin{multline*}
        \big|\D_t^{\mfl_0}\D_x^{\mfl_1}\D_y^{\mfl_2}\big(M(t,x,y)-\e^{-t\frac{|x|^2+|y|^2}{4}}G(t,x-y)\big)\big|\\\lesssim t^{1-\mfl_0-\frac{|\mfl_1|+|\mfl_2|}{2}}\big(|x|^2+|y|^2\big)\bigg(1+\frac{|x-y|}{\sqrt{t}}\bigg)^{|\mfl_0|}\big(1+\sqrt{t}|x|\big)^{|\mfl_1|}\big(1+\sqrt{t}|y|\big)^{|\mfl_2|}\e^{-\frac{|x-y|^2}{4t}}\e^{-t\frac{|x|^2+|y|^2}{4}}\,,
    \end{multline*}
    and proceed in the same lines.
\end{proof}

\begin{lemma}\label{lemma:estimMmu}
    For every $n,m$, $\varsigma>0$, and $\mu\in(0,1]$,
    \begin{equation}
        \big\|\langle\bigcdot\rangle^\varsigma_\mu(\mcT_{\mu,n}\otimes \mcT_{\mu,m})\dot M_\mu\big\|_{\mcL(L_{t,x}^{\infty,p},L_{t,x}^{\infty,r})}\lesssim\mu^{1-d(1/r-1/p)}\,,
    \end{equation}
    and
    \begin{equation}\label{eq:PmuTmuMmu}
        \big\|(\msP_\mu\otimes \mcT_{\mu,m})\dot M_\mu\big\|_{\mcL(L_{t,x}^{\infty,p},L_{t,x}^{\infty,r})}\lesssim\mu^{1-d(1/r-1/p)}\,,
    \end{equation}
    $\mcT_{\mu,n}$ was defined in \eqref{eq:defTmu}.
\end{lemma}
\begin{proof}
    For this proof we let $a=a_n\vee a_m\in[1,2)$. We recall that $\dot M_\mu(t,x,y)=-2\mu^{-3}t\chi'(t/\mu^2)M(t,x,y)$. Then, applying \eqref{eq:estMt} with $s=\mu^2$ (we know that $\mu^2\leq t$ thanks to the support properties of $\chi$), we get
    \begin{equs}
        \big|\big((\mcT_{\mu,n}\otimes \mcT_{\mu,m})\dot M_\mu\big)(t,x,y)\big|&\lesssim\sum_{n=0}^N\mu^{-3}t\chi^{(N-n+1)}(t/\mu^2)t^{-N(1+2d)}\mu^{2N(1+2d)}G(8t/3,x-y)\\
        &\lesssim|\dot G_\mu(8t,x-y)|\,.
    \end{equs}
    We have brutally bounded the Gaussian weight by $1$. Then, by Young's convolution inequality,
    \begin{equs}
        ~&\big\|\langle\bigcdot\rangle^\varsigma_\mu(\mcT_{\mu,n}\otimes \mcT_{\mu,m})\dot M_\mu\big\|_{\mcL(L_{t,x}^{\infty,p},L_{t,x}^{\infty,r})}\lesssim\|\langle\bigcdot\rangle_\mu^\varsigma \dot G_\mu\|_{L_{t,x}^{1,q}}\lesssim\mu^{1-d(1/r-1/p)}\,,
    \end{equs}
    where $\frac{1}{q}=1+\frac{1}{p}-\frac{1}{r}$, and the last step follows Lemma \ref{estimdotGmu}, but with the classical Laplacian ($\alpha=2$). The treatment of the second estimate follows the same lines, observing that
    \begin{equs}
        |(\msP_\mu\otimes\mathrm{Id})M(t,x,y)|=|((1-\mu^{4}\mathrm{H}^2)^N\otimes\mathrm{Id})M(t,x,y)|\lesssim\bigg(1+\frac{\mu^{4}}{t^{2}}\bigg)^N| G(t,x-y)|\e^{-3t\frac{|x|^2+|y|^2}{16}}\,.
    \end{equs}
\end{proof}

\begin{lemma}\label{lemma:diffMG}
    The following holds for $\mu\in(0,1]$
    \begin{equation}
        \big|\mcT_{\mu,0}^{\otimes2}\big(\dot M_\mu-\mcb W_\mu^{-1}\dot G_\mu\big)(t,x,y)\big|\lesssim \mu^2\big(\mu^2+|x|^2+|y|^2\big)|\dot G_\mu(8t,x-y)|\,.
    \end{equation}
\end{lemma}
\begin{proof}
    We first separate the following quantity in two terms
    \begin{equs}
        ~&\e^{\mu^2\frac{|x|^2+|y|^2}{8}}\Big(\dot M_\mu(t,x,y)-\e^{-\mu^2\frac{|x|^2+|y|^2}{8}}\dot G_\mu(t,x-y)\Big)\\
        &=\e^{\mu^2\frac{|x|^2+|y|^2}{8}}\Big(\dot M_\mu(t,x,y)-\e^{-t\frac{|x|^2+|y|^2}{4}}\dot G_\mu(t,x-y)\Big)+\Big(\e^{-(2t-\mu^2)\frac{|x|^2+|y|^2}{8}}-1\Big)\dot G_\mu(t,x-y)\\
        &\eqdef D_\mu^1(t,x,y)+D_\mu^2(t,x,y)\,.
    \end{equs}
    With this notation, we have
    \begin{equs}
        \mcT_{\mu,0}^{\otimes2}(\dot M_\mu-\mcb W_\mu^{-1}\dot G_\mu)=\mcR_\mu^{\otimes2}(D_\mu^1+D_\mu^2)\,.
    \end{equs}
    For the first term, we write
    \begin{multline*}
        \e^{\mu^2\frac{|x|^2+|y|^2}{8}}\Big(\dot M_\mu(t,x,y)-\e^{-t\frac{|x|^2+|y|^2}{4}}\dot G_\mu(t,x-y)\Big)\\
        =\mu^{-3}t\chi'(t/\mu^2)\e^{\mu^2\frac{|x|^2+|y|^2}{8}}\Big(M(t,x,y)-\e^{-t\frac{|x|^2+|y|^2}{4}}G(t,x-y)\Big)\,.
    \end{multline*}
    Then, thanks to \eqref{eq:estimatediffkernels}, we have
    \begin{equs}
        ~&\big|\big(\mcR_\mu^{\otimes2}D_\mu^1\big)(t,x,y)\big|\\
        &\lesssim\sum_{n=0}^N\mu^{-3}t\chi^{(N-n+1)}(t/\mu^2) t^{1-n(1+2d)}\mu^{n(1+2d)}\big(|x|^2+|y|^2\big)G(8t/3,x-y)\\
        &\lesssim\mu^2\big(|x|^2+|y|^2\big)\dot G_\mu(8t,x-y)\,.
    \end{equs}
    For the second term we can write, using that $\mu^2\leq2t-\mu^2\leq3\mu^2$ (again thanks to the support properties of $\chi$),
    \begin{equs}
        ~&|(\mcR_{\mu}^{\otimes2}D_\mu^2)(t,x,y)|\\
        &\lesssim \left(\Big(\e^{-(2t-\mu^2)\frac{|x|^2+|y|^2}{8}}-1\Big)+\sum_{n=1}^N\mu^{4n}\big(1+\mu(|x|+|y|)\big)^{2n}\e^{-(2t-\mu^2)\frac{|x|^2+|y|^2}{8}}\right)\\
        &\quad\times\mu^{-3}t\chi^{(N-n+1)}(t/\mu^2) t^{1-(1+2d)(N-n)}\mu^{(1+2d)(N-n)}G(t,x-y)\\
        &\lesssim \left(\mu^2(|x|^2+|y|^2)+\sum_{n=1}^N\mu^{4n}\big(1+\mu(|x|+|y|)\big)^{2n}\e^{-\mu^2\frac{|x|^2+|y|^2}{8}}\right)\dot G_\mu(t,x-y)\\
        &\lesssim \mu^2(\mu^2+|x|^2+|y|^2)\dot G_\mu(t,x-y)\,,
    \end{equs}
    which concludes the proof.
\end{proof}

\begin{lemma}\label{lemma:heatflowmehler}
    Let $\sigma$, $1<r<\infty$ and $1\leqslant q\leqslant\infty$. For any $u\in\msB_{r,q}^{-\sigma}$, it holds
    \begin{equation}
        \|\mcT_\mu M_\mu u\|_{L_x^r}\lesssim\mu^{-\sigma}\|u\|_{\msB_{r,q}^{-\sigma}}\,.
    \end{equation}
\end{lemma}

\begin{proof}
    First, we factorize $M_\mu(t)=M_\mu((1-\eps)t)M(\eps t)$, for some small enough $\eps>0$ so that $a_n<2-2\eps$, using the semigroup property. Thanks to \eqref{eq:estMt}, we know that $(\mcT_\mu\otimes\mathrm{Id}) M_\mu((1-\eps)t)$ is a sum over a finite set $\mathfrak{P}$ of kernels
    \begin{equs}
        \big((\mcT_\mu\otimes\mathrm{Id}) M_\mu((1-\eps)\bigcdot)\big)(t,x,y)=\sum_{\pi\in \mathfrak{P}} M_\mu^\pi(t,x,y)\,,
    \end{equs}
    where all the kernels $M_\mu^\pi$ are bounded by a quantity of the form
    \begin{equs}
         \chi^{(N-n)}(t/\mu^2)t^{-N(1+d)}\mu^{2N(1+d)}G(8t/3,x-y)e^{-3t\frac{|x+y|^2}{32}}\lesssim G(8t,x-y)e^{-3t\frac{|x+y|^2}{32}}\one_{t\geqslant\mu^2}\,,
    \end{equs}
    for some $0\leq n\leq N$. We then have
    \begin{equs}
        \sup_{x\in\R^d}\int_{\R^d} |M_\mu^\pi(t,x,y)|\dint y + \sup_{y\in\R^d}\int_{\R^d} |M_\mu^\pi(t,x,y)|\dint x\lesssim \one_{t\geqslant\mu^2}\,,
    \end{equs}
    and the Schur test implies
    \begin{equs}
        \|\mcT_\mu M_\mu u\|_{L_x^r}\leq\sum_{\pi\in \mathfrak{P}} \left(\int_{\R^d}\left|\int_{\R^d}M_\mu^\pi(t,x,y)u(y)\dint y\right|^r\dint x\right)^\frac{1}{r} \lesssim \one_{t\geqslant\mu^2} \|M(\eps t)u\|_{L_x^r}\,.
    \end{equs}
    Thus, by Proposition 3.4.3. in \cite{Ha06},
    \begin{equs}\label{Eq:IntermediateSchauderMehlerExponentialWeight}
        \|\mcT_\mu M_\mu u\|_{L_x^r}\lesssim\mu^{-\sigma}\|u\|_{\msH^{-\sigma,r}}\,.
    \end{equs}
    The result follows by real interpolation, using Proposition \ref{Prop:HermiteSpacesInterpolation}.
\end{proof}

\begin{lemma}
    Let $p,r\in[1,\infty]$, we have
    \begin{equation}\label{eq:msQleq1}
        \|\msQ_\mu\|_{\mcL(L_{x}^{\infty})}\leq1\,,
    \end{equation}
    uniformly in $\mu$, and
    \begin{equation}\label{eq:msQmumsQnu}
        \|(\msQ_\mu-\mathrm{Id})\msQ_\nu\|_{\mcL(L_{x}^{\infty})}\lesssim\mu^2\nu^{-2}\,,
    \end{equation}
    uniformly in $\mu\leq\nu$.
\end{lemma}

\begin{proof}
    For the first statement, we use the Mehler kernel representation
    \begin{equs}
        \msQ_\mu(x,y)=\int_0^\infty M(\mu^2 s,x,y)\e^{-s}\dint s\,.
    \end{equs}
    We then just have to use the bound $|M(\mu^2s,x,y)|\leq |G(\mu^2 s,x-y)|$, and conclude as in the proof of \eqref{eq:normKmu}. For the second one, we write
    \begin{equs}
        (1-(1+\mu^2\mathrm{H})^{-1})(1+\nu^2\mathrm{H})^{-1}&=\mu^2\mathrm{H}(1+\mu^2\mathrm{H})^{-1}(1+\nu^2\mathrm{H})^{-1}\\
        &=\mu^2\nu^{-2}(1-(1+\nu^2\mathrm{H})^{-1})(1+\mu^2\mathrm{H})^{-1}\,,
    \end{equs}
    and we conclude with the help of \eqref{eq:msQleq1}.
\end{proof}

\begin{lemma}
    Let $\psi\colon\Lambda\to\R$, $0<\nu\leq\mu\leq1$, $n\leq m$, and $p\in[1,\infty]$, then
    \begin{equs}\label{eq:lessJn}
        \|J_{\mu,m}\psi\|_{L_{t,x}^p}\lesssim\|J_{\mu,n}\psi\|_{L_{t,x}^p}\,,
    \end{equs}
    and
    \begin{equs}\label{eq:lessJmu}
        \|J_{\mu,n}\psi\|_{L_{t,x}^p}\lesssim\|J_{\nu,n}\psi\|_{L_{t,x}^p}\,.
    \end{equs}
\end{lemma}

\begin{proof}
    We start with the proof of the first statement. We have
    \begin{equs}
        \big(J_{\mu,m}\psi\big)(t,x)&=\int_0^\infty\int_{\R^d} K_\mu(t-s,x-y)\e^{-a_m\mu^2\frac{|y|^2}{8}}\psi(s,y)\dint y\dint s\\
        &=\int_0^\infty\int_{\R^d} K_\mu(t-s,x-y)\e^{-a_{n-m}\mu^2\frac{|y|^2}{8}}\e^{-a_n\mu^2\frac{|y|^2}{8}}\psi(s,y)\dint y\dint s\\
        &=\int_0^\infty\int_{\R^d} \mcR_\mu\Big( K_\mu(t-s,x-y)\e^{-a_{n-m}\mu^2\frac{|y|^2}{8}}\Big)\big(J_{\mu,n}\psi\big)(s,y)\dint y\dint s\,.
    \end{equs}
    $\mcR_\mu$ is here applied on $(s,y)$. Thus we get using Young's convolution inequality then \eqref{eq:estimreloue},
    \begin{equs}
        \|J_{\mu,m}\psi\|_{L_{t,x}^p}&\leq\sup_{t,x}\int_0^\infty\int_{\R^d} \Big|\mcR_\mu\Big( K_\mu(t-s,x-y)\e^{-a_{n-m}\mu^2\frac{|y|^2}{8}}\Big)\Big|\dint y\dint s\times\|J_{\mu,n}\psi\|_{L_{t,x}^p}\\
        &\lesssim\sup_{t,x}\int_0^\infty\int_{\R^d}\mcS_\mu \Tilde K(t-s,x-y)\big(1+\mu^2|y|^2\big)^N\e^{-a\mu^2\frac{|y|^2}{8}} \dint y\dint s\times\|J_{\mu,n}\psi\|_{L_{t,x}^p}\\
        &\lesssim\sup_{t,x}\int_0^\infty\int_{\R^d}\mcS_\mu \Tilde K(t-s,x-y) \dint y\dint s\times\|J_{\mu,n}\psi\|_{L_{t,x}^p}\\
        &=\|\Tilde K\|_{L^1_{t,x}}\|J_{\mu,n}\psi\|_{L_{t,x}^p}\lesssim\|J_{\mu,n}\psi\|_{L_{t,x}^p}\,.
    \end{equs}
    The proof of the second statement is somewhat similar. We have
    \begin{equs}
        \big(J_{\mu,n}\psi\big)(t,x)&=\int_0^\infty\int_{\R^d} K_\mu(t-s,x-y)\e^{-a_n\mu^2\frac{|y|^2}{8}}\psi(s,y)\dint y\dint s\\
        &=\int_0^\infty\int_{\R^d} \mcR_\nu K_\mu(t-s,x-y)\e^{-a_n(\mu^2-\nu^2)\frac{|y|^2}{8}}\big(J_{\nu,n}\psi\big)(s,y)\dint y\dint s\,.
    \end{equs}
    Using again Young's convolution inequality, we have
    \begin{equs}
        \|J_{\mu,n}\psi\|_{L_{t,x}^p}\leq\|\mcR_\nu K_\mu\|_{L^1_{t,x}}\|J_{\nu,n}\psi\|_{L_{t,x}^p}\lesssim\|J_{\nu,n}\psi\|_{L_{t,x}^p}\,,
    \end{equs}
    where the last bound follows from \eqref{eq:Kmunu}.
\end{proof}

\begin{lemma}
    Let $a\in[0,2)$, and $\mu\in(0,1]$. There is an integrable kernel $\Tilde K$ with exponential decay such that
    \begin{equation}\label{eq:estimreloue}
        \Big|\mcR_\mu\Big( K_\mu\e^{-a\mu^2\frac{|\bigcdot|^2}{8}}\Big)(t,x)\Big|\lesssim \mcS_\mu\Tilde K(t,x)\big(1+\mu^2|x|^2\big)^N\e^{-a\mu^2\frac{|x|^2}{8}}\,.
    \end{equation}
    There is an integrable kernel $\hat K$ with exponential decay such that
    \begin{equation}\label{eq:estimrelouex2}
        \Big|\mcR_\mu\Big( K_\mu|\bigcdot|^2\e^{-a\mu^2\frac{|\bigcdot|^2}{8}}\Big)(t,x)\Big|\lesssim \mcS_\mu\hat K(t,x)\big(1+|x|^2\big)\big(1+\mu^2|x|^2\big)^N\e^{-a\mu^2\frac{|x|^2}{8}}\,.
    \end{equation}
\end{lemma}

\begin{proof}
    We only provide the main ingredients of the proof. Thanks to the Leibniz rule, the operators in $\mcR_\mu$ are distributed on the two (or three terms) of the expression. When they act on $K_\mu$, they only reduce the convolution exponent, up to creating a Dirac delta. For the other terms, we use that
    \begin{equs}
        \mu^{|k|}\Big|\D_x^k\e^{-a\mu^2\frac{|x|^2}{8}}\Big|\lesssim\mu^{2|k|}\big(1+\mu|x|\big)^{|k|}\e^{-a\mu^2\frac{|x|^2}{8}}\,,
    \end{equs}
    for every multi-index $k$.
\end{proof}

\section{Properties of the weights}

\begin{lemma}
    The following estimate holds for $z_1,z_2,z_3\in\SS$, $\mu\in(0,1]$,
    \begin{equs}\label{triangleineq}
        \langle z_1;z_3\rangle_{\mu}&\lesssim\langle z_1;z_2\rangle_{\mu}\langle z_2;z_3\rangle_{\mu}\,.
    \end{equs}
    Moreover, for $z\in\SS$, and $\nu\leq\mu$,
    \begin{equs}\label{estimweight}
    \frac{\langle z\rangle_{\mu}^{\alpha-\beta}}{\langle z\rangle_{\nu}^{\alpha-\beta}}\big(1-\langle z\rangle_{\mu}^{-2}\big)&\lesssim\mu^{\beta-\alpha}\nu^{\alpha-\beta}\,.
    \end{equs}
\end{lemma}

\begin{proof}
    For the second statement, see Lemma 57 in \cite{DGR23}.
\end{proof}

\begin{lemma}\label{ITexp}
    We have, for $\ell>0$ and $x_1,x_2\in\R^d$,
    \begin{equation}
        \mcb e_{-\ell}(x_1)\mcb e_{\ell}(x_2)\lesssim \mcb e_{\ell}(x_1;x_2)\,.
    \end{equation}
\end{lemma}

\begin{lemma}\label{HaiLabestim}
    The following holds for $s<t$ and $\ell\in\R$.
    \begin{equation}
        \sup_{x\in\R^d}\frac{\langle x\rangle^\kappa\mcb e_{\ell+s}(x)}{\mcb e_{\ell+t}(x)}\lesssim (t-s)^{-\kappa/2}\,.
    \end{equation}
\end{lemma}

\begin{proof}
    This is a simple consequence of the inequality
    \begin{equs}
        \sup_{x\in\R^d}\frac{\langle x\rangle^\kappa\mcb e_{\ell+s}(x)}{\mcb e_{\ell+t}(x)}\leq \left(\frac{\kappa}{2(t-s)}\right)^\frac{\kappa}{2}\e^{-\kappa/2}\,.
    \end{equs}
\end{proof}

\section{Spectral estimates}

\subsection{In finite volume}
We give here some spectral properties of the fractional heat kernel in finite volume and their applications to Weyl law of some perturbations of the fractional laplacian.

\begin{lemma}\label{Lem:LooseAsymptoticFractionalHeatTrace}
    When $t$ goes to $0$,
    $$\operatorname{Tr}\left(\e^{-t|\nabla|^\alpha}\right)\sim \frac{2\pi^{d/2}\Gamma(d/\alpha)}{\alpha\Gamma(d/2)} t^{-\frac{d}{\alpha}}.$$
\end{lemma}

\begin{proof}
    The result follows from Poisson formula
    $$\operatorname{Tr}\left(\e^{-t|\nabla|^\alpha}\right) = \sum_{k\in\Z^d}\e^{-t|k|^\alpha} = t^{-\frac{d}{\alpha}}\sum_{k\in\Z^d}F(2\pi k/t^{1/\alpha})\,,$$
    with $F = \mathcal{F}(\e^{-|x|^\alpha})$, verifying $F(0) = \frac{2\pi^{d/2}\Gamma(d/\alpha)}{\alpha\Gamma(d/2)}$ and $F(\eta)\leqslant \frac{2\pi^{d/2}\Gamma(d/\alpha)}{\alpha\Gamma(d/2)} |\eta|^{-d-\alpha}$.
\end{proof}

Then, Karamata's tauberian theorem imply the Weyl law for the fractional Laplacian and their perturbations.

\begin{corollary}\label{Cor:WeylLawEquivTrace}
    Let $A$ be a selfadjoint operator on $L_x^2(\T^d)$ with discrete spectrum $(\lambda_n)_{n\in\N}$. Assume
    $$\operatorname{Tr}\left(\e^{-t|\nabla|^\alpha}\right)\sim \operatorname{Tr}\left(\e^{-tA}\right)$$
    as $t$ goes to 0. Then, 
    $$\mathcal{N}_A(\lambda) = \#\{n\in\N,\; \lambda_n\leqslant\lambda\}\sim \frac{\pi^{\frac{d}{2}}\lambda^{\frac{d}{\alpha}}}{\Gamma\left(\frac{d}{2}+1\right)}$$
    as $\lambda$ goes to $\infty$.
\end{corollary}

\subsection{The harmonic oscillator}
We give here some spectral properties of the Mehler kernel and their applications to Weyl law of some perturbations of the harmonic oscillator.

\begin{lemma}\label{Lem:LooseAsymptoticMehlerTrace}
    When $t$ goes to $0$,
    $$\operatorname{Tr}\left(\e^{t(\Delta-|x|^2)}\right)\sim (2t)^{-d}.$$
\end{lemma}

\begin{proof}
    The result follows from the direct computation
    $$\operatorname{Tr}\left(\e^{t(\Delta-|x|^2)}\right) = \int_{\R^d} M(t,x,x)\dint x = \left(2\tanh(t)\sinh(2t)\right)^{-\frac{d}{2}}\sim (2t)^{-d}.$$
\end{proof}

Then, Karamata's tauberian theorem imply the Weyl law for the harmonic oscillator and its perturbations.

\begin{corollary}[Weyl law for the harmonic oscillator]\label{Cor:WeylLawEquivTraceOH}
    Let $A$ be a self-adjoint operator on $L_x^2(\R^d)$ with discrete spectrum $(\lambda_n)_{n\in\N}$. Assume
    $$\operatorname{Tr}\left(\e^{t(\Delta-|x|^2)}\right)\sim \operatorname{Tr}\left(\e^{-tA}\right)\,,$$
    as $t$ goes to $0$. Then, 
    $$\mathcal{N}_A(\lambda) = \#\{n\in\N,\; \lambda_n\leqslant\lambda\}\sim \frac{\lambda^{d}}{2^{d}\Gamma\left(d+1\right)}\,,$$
    as $\lambda$ goes to $\infty$.
\end{corollary}

\section{Functional spaces associated with the harmonic oscillator}

The Sobolev and Besov spaces of $\mathrm H=-\Delta+|x|^2$ have been studied by several authors, see for example \cite{BT06} for Sobolev spaces and \cite{BDY12,FI25} for Besov spaces. The following result characterises these spaces as interpolation spaces.

\begin{proposition}\label{Prop:HermiteSpacesInterpolation}
    Let $1<p<\infty$, $1\leqslant q<\infty$ and $\eta\in\R$. Then, the following holds.
    \begin{enumerate}
        \item For any $\eta_0\neq\eta_1\in\R$ and $0<\theta<1$, such that $\eta=(1-\theta)\eta_0+\theta\eta_1$,
        $$\msH^{\eta,p} = \left[\msH^{\eta_0,p},\msH^{\eta_1,p}\right]_\theta\,,$$
        with equivalent norms.
        \item For any $1<p_0\neq p_1<\infty$ and $0<\theta<1$, such that $\frac{1}{p}=\frac{1-\theta}{p_0}+\frac{\theta}{p_1}$,
        $$\msH^{\eta,p} = \left[\msH^{\eta,p_0},\msH^{\eta,p_1}\right]_\theta\,,$$
        with equivalent norms.
        \item If $0<\eta<1$,
        $$\msB^{\eta}_{p,q} = \left(L^p_x,\msH^{1,p}\right)_{\eta,q}= \left(L^p_x,\msH^{2,p}\right)_{\frac{\eta}{2},q}\,,$$
        with equivalent norms.
        \item If $\eta<0$, for any $\eta_0\neq\eta_1<0$ and $0<\theta<1$, such that $\eta=(1-\theta)\eta_0+\theta\eta_1$,
        $$\msB^{\eta}_{p,q} = \left(\msH^{\eta_0,p},\msH^{\eta_1,p}\right)_{\theta,q}\,,$$
        with equivalent norms.
    \end{enumerate}
\end{proposition}

\begin{proof}
    Item 1. is a classical consequence of the bounded imaginary powers of $H$ over $L^p_x(\R^d)$ (see Proposition 3.5.5. in \cite{Ha06} for the interpolation result and \cite{He90} showing $H$ is in fact BIP on $L^p_x(\R^d)$). Item 2. is a direct application of Riesz-Thorin theorem. Item 3. follows by Theorem 6.5.3. and Corollary 6.5.5. in \cite{Ha06} as $\msQ_\mu-\Id = \psi\left(\mu\sqrt{H}\right)$ with $\psi(z)=\left(1+z^2\right)^{-1}-1$ verifying $z^{-\eta}\psi(z)\sim z^{2-\eta}$ when $z$ goes to $0$ and $z^{-\eta}\psi(z)\sim z^{-\eta}$ when $z$ goes to $\infty$. Finally, Item 4. follows from Item 3. by a standard shift of regularity using powers of $H$ and the bounded $\mathcal{H}^\infty$-calculus of $H$ on $L_x^p(\R^d)$ (which is a direct consequence of \cite{He90}).
\end{proof}

Sobolev spaces associated with the harmonic oscillator enjoys continuous embeddings similar to usual Sobolev embeddings.

\begin{proposition}[Hermite-Sobolev embeddings]\label{Prop:HermiteSpacesEmbeddings}
    Let $1<p<q<\infty$ and $\eta,\sigma\in\R$. Then, the following holds.
    \begin{enumerate}
        \item If $\frac{1}{p}-\frac{\eta}{d}\leqslant\frac{1}{q}-\frac{\sigma}{d}$, then $\msH^{\eta,p}\subset\msH^{\sigma,q}$. Moreover, if $\frac{1}{p}-\frac{\eta}{d}<\frac{1}{q}-\frac{\sigma}{d}$, then the inclusion is compact.
        \item If $\eta>\frac{d}{p}$, then $\msH^{\eta,p}\subset L^\infty_x$ compactly.
        \item If $\frac{1}{p}+\frac{\eta}{d}<\frac{1}{q}+\frac{\sigma}{d}$, then $\msH^{\sigma,q}\subset\msH^{\eta,p}$ compactly.
        \item If $\eta>d\left(1-\frac{1}{p}\right)$, then $\msH^{\eta,p}\subset L^1_x$ compactly.
        \item If $\eta<1$, then $\msB^{\eta}_{p,1}\subset\msH^{\eta,p}\subset\msB^{\eta}_{p,\infty}\subset\msB^{\eta-\eps}_{p,1}$ for any $\eps>0$.
    \end{enumerate}
\end{proposition}

\begin{proof}
    Items 1. and 2. follow from Proposition 2.6 in \cite{Ma25a}, 3. and 4. is proved as Proposition 2.8 in \cite{Ma25c} and 5. from Proposition~\ref{Prop:HermiteSpacesInterpolation} and usual properties of interpolation spaces.
\end{proof}

\end{appendix}


\begin{thebibliography}{99}  

\bibitem{AK15}
R.~Allez, K.~Chouk. \newblock{\em The continuous {Anderson} hamiltonian in dimension two}.
\newblock\burlalt{arXiv:1511.02718}{https://arxiv.org/abs/1511.02718}.

\bibitem{BDM22}
I.~{Bailleul}, N. V.~{Dang}, A.~{Mouzard}. 
\newblock {\em Analysis of the Anderson operator}. 
\newblock {Electron. J. Probab. \textbf{30}, Paper No. 144, (2025).}
\burlalt{doi: 10.1214/25-EJP1380}{https://doi.org/10.1214/25-EJP1380}.
  
\bibitem{BB21}
R.~{Bauerschmidt}, T.~{Bodineau}. \newblock {\em Log-Sobolev inequality for the continuum sine-Gordon model}. 
\newblock Comm. Pure Appl. Math., \textbf{74}, no.~10, (2021), 2064--2113.
\burlalt{doi:10.1002/cpa.21926}{https://www.doi.org/10.1002/cpa.21926}. 

\bibitem{BBD24}
R.~{Bauerschmidt}, T.~{Bodineau}, B.~{Dagallier}. \newblock {\em Stochastic dynamics and the Polchinski equation: An introduction}. 
\newblock Probab. Surveys, \textbf{21}, (2024), 200--290.
\burlalt{doi: 10.1214/24-PS27}{https://www.doi.org/10.1214/24-PS27}.

\bibitem{BP57}
N.~N. Bogoliubow, O.~S. Parasiuk.
\newblock { \em Über die Multiplikation der Kausalfunktionen in der
  {Q}uantentheorie der {F}elder.}
\newblock Acta Math. \textbf{97}, (1957), 227--266.
\newblock\burlalt{doi:10.1007/BF02392399}{http://dx.doi.org/10.1007/BF02392399}.

\bibitem{BT06}
B.~{Bongioanni}, J.L.~{Torrea}.\newblock {\em Sobolev spaces associated to the harmonic oscillator}. 
\newblock Proc. Indian Acad. Sci., Math. Sci., \textbf{116}, no.~3 (2006), 337-360.
\burlalt{doi: 10.1007/BF02829750}{https://www.doi.org/10.1007/BF02829750}.

\bibitem{BBF}
M.~Borji, N.~Barashkov, L. Ferdinand. \newblock{\em The full-subcritical dynamical $\Phi^4$ on domains with boundary}. In preparation.

\bibitem{Borji1}
M.~Borji, K.~Kopper. \newblock{\em Perturbative renormalisation of $\Phi_4^4$ theory on the half space $\R^+\times\R^3$ with flow equations}.
\newblock Journal of Mathematical Physics, AIP Publishing, \textbf{63}, (2022), ISSN: 1089-7658.
\burlalt{doi: 10.1063/5.0097164}{http://dx.doi.org/10.1063/5.0097164}

\bibitem{Borji2}
M.~Borji, K.~Kopper. \newblock{\em The Surface counter-terms of the $\Phi_4^4$ theory on the half space $\R^+ \times\R^3$}.
\newblock\burlalt{arXiv:2305.18862}{https://arxiv.org/abs/2305.18862}.

\bibitem{Borji3}
M.~Borji. \newblock{\em Position-Space renormalisation and Half-Space Truncations in $\phi^4_4$}.
\newblock\burlalt{arXiv:2606.23650}{https://arxiv.org/abs/2606.23650}.
  
\bibitem{BCCH}
 { \rm Y. Bruned, A. Chandra, I. Chevyrev,
  M. Hairer}.
\newblock {\em Renormalising SPDEs in regularity structures}.
\newblock J. Eur. Math. Soc. (JEMS), \textbf{23}, no.~3, (2021), 869-947.
\newblock\burlalt{doi:10.4171/JEMS/1025}{http://dx.doi.org/10.4171/JEMS/1025}.

\bibitem{BHZ}
{\rm Y. Bruned, M. Hairer, L. Zambotti}.
\newblock {\em Algebraic renormalisation of regularity structures.}
\newblock Invent. Math. \textbf{215}, no.~3, (2019), 1039--1156.
\newblock\burlalt{doi:10.1007/s00222-018-0841-x}{https://dx.doi.org/10.1007/s00222-018-0841-x}.

\bibitem{BM25}
Y.~Bruned, A.~Minguella. 
\newblock{\em{Renormalisation in the flow approach for singular SPDEs}}
\newblock\burlalt{arXiv:2504.04885}{https://arxiv.org/abs/2504.04885}

\bibitem{BDY12}
H.-Q.~Bui, X.T.~Duong, L.~Yan.\newblock {\em Calder{\'o}n reproducing formulas and new {Besov} spaces associated with operators}. 
\newblock Adv. Math., \textbf{229}, no.~4 (2012), 2449--2502.
\burlalt{doi: 10.1016/j.aim.2012.01.005}{https://www.doi.org/10.1016/j.aim.2012.01.005}.

\bibitem{CH16}
A.~Chandra, M.~Hairer.
\newblock {\em{An analytic {BPHZ} theorem for regularity structures.}}
\newblock \burlalt{arXiv:1612.08138}{http://arxiv.org/abs/1612.08138}.

\bibitem{CF24a}
A.~Chandra, L.~Ferdinand.
\newblock{\em{A flow approach to the generalized KPZ equation}}
\newblock\burlalt{arXiv:2402.03101}{https://arxiv.org/abs/2402.03101}.

\bibitem{CF24b}
A.~Chandra, L.~Ferdinand.
\newblock{\em{Rough differential equations in the flow approach}}
\newblock\burlalt{arXiv:2411.07157}{https://arxiv.org/abs/2411.07157}.


\bibitem{DJZ25}
F.C.~De Vecchi, X.~Ji, I.~Zachhuber.
\newblock {\em Stochastic Hartree NLS in 3d coming from a Many-Body Quantum System with White Noise Potential.}
\newblock \burlalt{arXiv:2505.01157}{http://arxiv.org/abs/2505.01157}.

\bibitem{DL20}
L.~Dumaz, C.~Labb{\'e}. \newblock{\em Localization of the continuous {Anderson} {Hamiltonian} in 1-d}. 
\newblock Probab. Theory Relat. Fields \textbf{176}, no.~1-2, (2024), 353--419
\newblock\burlalt{doi:10.1007/s00440-019-00920-6}{https://dx.doi.org/10.1007/s00440-019-00920-6}.

\bibitem{DL22}
L.~Dumaz, C.~Labb{\'e}. \newblock{\em The stochastic {Airy} operator at large temperature}. 
\newblock Ann. Appl. Probab. \textbf{32}, no.~6, (2022), 4481--4534
\newblock\burlalt{doi:10.1214/22-AAP1793}{https://dx.doi.org/10.1214/22-AAP1793}.

\bibitem{DL23}
L.~Dumaz, C.~Labb{\'e}. \newblock{\em The delocalized phase of the {Anderson} {Hamiltonian} in 1-{D}}. 
\newblock Ann. Probab. \textbf{51}, no.~3, (2023), 805--839
\newblock\burlalt{doi:10.1214/22-AOP1591}{https://dx.doi.org/10.1214/22-AOP1591}.

\bibitem{DL24a}
L.~Dumaz, C.~Labb{\'e}. \newblock{\em Anderson localization for the 1-d {Schr{\"o}dinger} operator with white noise potential}.  
\newblock  journal = {J. Funct. Anal.}, \textbf{286}, no.~1, (2024), 45
\newblock\burlalt{doi:10.1016/j.jfa.2023.110191}{https://dx.doi.org/10.1016/j.jfa.2023.110191}.

\bibitem{DL24c}
L.~Dumaz, C.~Labb{\'e}. \newblock{\em Localization crossover for the continuous {Anderson} {Hamiltonian} in 1-d}.  
\newblock Invent. Math. \textbf{235}, no.~2, (2024), 345--440
\newblock\burlalt{doi:10.1007/s00222-023-01225-1}{https://dx.doi.org/10.1007/s00222-023-01225-1}.


\bibitem{Duc21}
P.~Duch.
\newblock{\em{Flow equation approach to singular stochastic PDEs}}  Probab. and Math. Phys. \textbf{6}, no.~2, (2025), 327–437. T
\newblock
\burlalt{doi:10.2140/pmp.2025.6.327}{https://dx.doi.org/10.2140/pmp.2025.6.327}.

\bibitem{Duc22}
P.~Duch.
\newblock{\em renormalisation of singular elliptic stochastic PDEs using flow equation}. Probab. and Math. Phys. \textbf{6}, no.~1, (2025), 111--138.
\newblock
\burlalt{doi:10.2140/pmp.2025.6.111}{https://dx.doi.org/10.2140/pmp.2025.6.111}.

\bibitem{DGR23}
P.~Duch, M.~Gubinelli, P.~Rinaldi.
\newblock{\em Parabolic stochastic quantisation of the fractional $\Phi^4_3$ model in the full subcritical regime}
\newblock\burlalt{arXiv:2303.18112}{https://arxiv.org/abs/2303.18112}.

\bibitem{Duc23}
P.~Duch.
\newblock{\em Lecture notes on flow equation approach to singular
stochastic PDEs}
\newblock\burlalt{https://pawelduch.github.io/spde_flow_notes.pdf}.

\bibitem{Duc24}
P.~Duch.
\newblock{\em Construction of Gross-Neveu model using Polchinski flow equation}
\newblock\burlalt{arXiv:2403.18562}{https://arxiv.org/abs/2403.18562}.

\bibitem{FI25}
R.~Fukuizumi, T.~Iwabuchi.
\newblock{\em Besov spaces associated with the {Harmonic} oscillator}
\newblock\burlalt{arXiv:2508.20436}{https://arxiv.org/abs/2508.20436}.

\bibitem{FN77}
M.~Fukushima, S.~Nakao.
\newblock {\em On spectra of the Schrödinger operator with a white Gaussian noise potential}.
\newblock Z. Wahrscheinlichkeit-stheor. Verw. Geb \textbf{37}, (1977), 267–-274.
\newblock\burlalt{doi:10.1007/BF00537493}{http://dx.doi.org/10.1007/BF00537493}.


\bibitem{GIP15}
M.~Gubinelli, P.~Imkeller, N.~Perkowski.
\newblock {\em Paracontrolled distributions and singular PDEs}
\newblock Forum of mathematics, Pi, \textbf{3}, (2015), 3:e6.
\newblock\burlalt{doi:10.1017/fmp.2015.2}{http://dx.doi.org/10.1017/fmp.2015.2}.

\bibitem{GM24}
M.~Gubinelli, S-J.~Meyer.
\newblock{\em The FBSDE approach to sine-Gordon up to $6\pi$}
\newblock\burlalt{arXiv:2401.13648}{https://arxiv.org/abs/2401.13648}.

\bibitem{GUZ20}
M.~Gubinelli, B.~Ugurcan, I.~Zachhuber.
\newblock {\em Semilinear evolution equations for the Anderson Hamiltonian in two and three dimensions}
\newblock Stoch. Partial Differ. Equ., Anal. Comput., \textbf{8}, no.~1, (2020), 82--149.
\newblock\burlalt{doi:10.1007/s40072-019-00143-9}{http://dx.doi.org/10.1007/s40072-019-00143-9}.

\bibitem{Ha06}
M.~Haase.
\newblock{\em The functional calculus for sectorial operators}
\newblock {Oper. Theory: Adv. Appl. (2006). ISBN: 978-3-7643-7698-7.}
\newblock \burlalt{doi:10.1007/3-7643-7698-8}{http://dx.doi.org/10.1007/3-7643-7698-8}.

\bibitem{reg}
{\rm M. Hairer}.
\newblock {\em A theory of regularity structures.}
\newblock Invent. Math. \textbf{198}, no.~2, (2014), 269--504.
\newblock\burlalt{doi:10.1007/s00222-014-0505-4}{https://dx.doi.org/10.1007/s00222-014-0505-4}.

\bibitem{HL15}
M.~Hairer, C.~Labbé.
\newblock{\em A simple construction of the continuum parabolic Anderson model on $\R^2$.}
\newblock Electron. Commun. Probab. \textbf{20} 1 - 11, 2015. 
\newblock\burlalt{doi:10.1214/ECP.v20-4038}{https://doi.org/10.1214/ECP.v20-4038} 

\bibitem{HL18}
M.~Hairer, C.~Labbé.
\newblock{\em Multiplicative stochastic heat equations
on the whole space}
\newblock J. Eur. Math. Soc. (JEMS), \textbf{20}, no.~4, (2018), 1005-1054.
\newblock\burlalt{doi:10.4171/JEMS/781}{https://doi.org/10.4171/JEMS/781}. 

\bibitem{He90}
W.~Hebisch.
\newblock{\em A multiplier theorem for Schrödinger operators}
\newblock Colloq. Math. 60/61.2, (1990), 659–-664.
\newblock\burlalt{doi: 10.4064/60-61-2-659-664}{https://doi.org/10.4064/60-61-2-659-664}. 

\bibitem{KH69}
K. ~Hepp.
\newblock {\em On the equivalence of additive and analytic renormalisation}.
\newblock Comm. Math. Phys. \textbf{14}, (1969), 67--69.
\newblock \burlalt{doi:10.1007/
BF01645456}{http://dx.doi.org/10.1007/
BF01645456}.

\bibitem{HsL26}
Y.-S.~Hsu, C.~Labbé
\newblock{\em Construction and spectrum of the {Anderson} {Hamiltonian} with white noise potential on $\R^2$ and $\R^3$.}
\newblock Probab. Math. Phys. \textbf{7}, no.~1, (2026), 1--35.
\newblock\burlalt{doi: 10.2140/pmp.2026.7.1}{https://doi.org/10.2140/pmp.2026.7.1}. 

\bibitem{Ka95}
T.~Kato.
\newblock{\em Perturbation theory for linear operators.}
\newblock {Class. Math. (1995). ISBN: 978-3-540-58661-6}
\newblock \burlalt{doi:10.1007/978-3-642-66282-9}{https://doi.org/10.1007/978-3-642-66282-9}.

\bibitem{KL81}
A.~Klein, L.J.~Landau.
\newblock {\em Construction of a unique self-adjoint generator
for a symmetric local semigroup}. J. Funct. Anal., \textbf{44}, no.~2, (1981), 121--137.
\burlalt{doi:10.1016/0022-1236(81)90007-0}{http://dx.doi.org/10.1016/0022-1236(81)90007-0}.

\bibitem{Kup16}
A.~Kupiainen.
\newblock {\em renormalisation group and stochastic PDEs}. Ann. Henri.
Poincaré, \textbf{17}, no.~3, (2016), 497--535.
\burlalt{doi:10.1007/s00023-015-0408-y}{http://dx.doi.org/10.1007/s00023-015-0408-y}.

\bibitem{KM17}
A.~Kupiainen, M.~Marcozzi.
\newblock {\em renormalisation of generalized KPZ equation}. J. Stat. Phys. \textbf{166}, no.~3, (2017), 876--902.
\burlalt{doi:10.1007/s10955-016-1636-3}{http://dx.doi.org/10.1007/s10955-016-1636-3}.

\bibitem{Lab19}
C.~Labbé.
\newblock{\em The continuous {Anderson} {Hamiltonian} in $d \leq 3$}.
\newblock J. Funct. Anal. \textbf{277}, no.~6, (2019), 3187--3235. 
\newblock\burlalt{doi:10.1016/j.jfa.2019.05.027}{https://doi.org/10.1016/j.jfa.2019.05.027}.

\bibitem{MaPhD}
S.~Mackowiak. \newblock {\em Mathematical Analysis of the Gross-Pitaevskii equation in disordered media}. 
\newblock PhD thesis, Institut polytechnique de Paris, (2025)
\burlalt{doi:10.70675/3ec6de1az3d40z4451zba32z575574591eff}{https://dx.doi.org/10.70675/3ec6de1az3d40z4451zba32z575574591eff}. 

\bibitem{Ma25a}
S.~Mackowiak. \newblock {\em Wellposedness of the cubic Gross–Pitaevskii equation with spatial white noise on $\mathbb{R}^2$}. 
\newblock Nonlinearity \textbf{38}, no.~3, (2025)
\burlalt{doi:10.1088/1361-6544/adb237}{https://dx.doi.org/10.1088/1361-6544/adb237}. 

\bibitem{Ma25b}
S.~Mackowiak. \newblock{\em Local wellposedness of the 2d {Anderson}-{Gross}-{Pitaevskii} equation}.
\newblock\burlalt{arXiv:2511.17063}{https://arxiv.org/abs/2511.17063}.

\bibitem{Ma25c}
S.~Mackowiak.
\newblock{\em{Standing waves of the Anderson-Gross-Pitaevskii equation}}.
\newblock\burlalt{arXiv:2512.22960}{https://arxiv.org/abs/2512.22960}.

\bibitem{Ma26+}
S.~Mackowiak. \newblock{\em Anderson operators as lower order perturbations of sectorial operators}.
In preparation.

\bibitem{Mey26}
S-J.~Meyer.
\newblock{\em An FBSDE Construction of the Sine-Gordon EQFT for $\beta^2<\frac{6}{7}8\pi$ and Perturbative renormalisation in the Full Subcritical Regime}
\newblock\burlalt{arXiv:2607.20632}{https://arxiv.org/abs/2607.20632}.

\bibitem{Mi15}
N.~Minami. \newblock {\em Definition and self-adjointness of the stochastic {Airy} operator}. 
\newblock Markov Process. Relat. Fields \textbf{21}, no.~3, (2015), 695--711
\burlalt{arXiv:1401.0853}{https://arxiv.org/abs/1401.0853}. 

\bibitem{MM22}
L.~Morin, A.~Mouzard. \newblock {\em {2D} random magnetic {Laplacian} with white noise magnetic field}. 
\newblock Stochastic Processes Appl. \textbf{143} (2022) 160--184.
\burlalt{doi:10.1016/j.spa.2021.10.004}{https://www.doi.org/10.1016/j.spa.2021.10.004}.

\bibitem{Mo20}
A.~Mouzard.
\newblock{\em Weyl law for the {Anderson} {Hamiltonian} on a two-dimensional manifold}
\newblock\burlalt{arXiv:2009.03549}{https://arxiv.org/abs/2009.03549}.

\bibitem{MO25}
A.~Mouzard, E.M.~Ouhabaz. \newblock {\em A simple construction of the {Anderson} operator via its quadratic form in dimensions two and three}. 
\newblock C. R., Math., Acad. Sci. Paris \textbf{363} (2025) 183--197.
\burlalt{doi:10.5802/crmath.712}{https://www.doi.org/10.5802/crmath.712}.

\bibitem{MZ24}
A.~Mouzard, I.~Zachhuber. \newblock {\em Strichartz inequalities with white noise potential on compact surfaces}. 
\newblock Anal. PDE \textbf{17}, no.~2, (2024) 421--454.
\burlalt{doi:10.2140/apde.2024.17.421}{https://www.doi.org/10.2140/apde.2024.17.421}. 

\bibitem{MZ26}
A.~Mouzard, I.~Zachhuber. \newblock {\em Nonlinear {Schr{\"o}dinger} equations with spatial white noise potential on full space for {$d\le 3$}}. 
\burlalt{arXiv:2604.15226}{https://arxiv.org/abs/2604.15226}. 

\bibitem{M03}
V. F. Müller. \newblock {\em Perturbative renormalisation by Flow Equation}. 
\newblock Rev. Math. Phys. \textbf{15}, no.~5, (2003) 491--558.
\burlalt{doi:10.1142/S0129055X03001692}{https://www.doi.org/10.1142/S0129055X03001692}. 

\bibitem{PT11}
G.~Peccati, M.S.~Taqqu.
\newblock{\em Wiener Chaos: Moments, Cumulants and Diagrams}. 
\newblock{Bocconi and Springer series. (2011). ISBN: 978-88-470-1678-1.}
\burlalt{doi:10.1007/978-88-470-1679-8}{http://dx.doi.org/10.1007/978-88-470-1679-8}.

\bibitem{Pol84}
J. Polchinski. \newblock {\em renormalisation and effective lagrangians}. 
\newblock Nucl. Phys. B, \textbf{231}, no.~2, (1984) 269--295.
\burlalt{doi:10.1016/0550-3213(84)90287-6}{https://www.doi.org/10.1016/0550-3213(84)90287-6}. 

\bibitem{RRV11}
J.A.~Ramirez, B.~Rider, B.~Vir{\'a}g. \newblock {\em Beta ensembles, stochastic {Airy} spectrum, and a diffusion}. 
\newblock J. Am. Math. Soc., \textbf{24}, no.~4, (2011) 919--944.
\burlalt{doi:10.1090/S0894-0347-2011-00703-0}{https://www.doi.org/10.1090/S0894-0347-2011-00703-0}. 

\bibitem{Ue25}
N. Ueki. \newblock {\em A definition of self-adjoint operators derived from the {Schr{\"o}dinger} operator with the white noise potential on the plane}. 
\newblock Stochastic Processes Appl., \textbf{186}, (1984) 30.
\burlalt{doi:10.1016/j.spa.2025.104642}{https://www.doi.org/10.1016/j.spa.2025.104642}. 

\bibitem{Ya10}
A.~Yagi
\newblock{\em Abstract parabolic evolution equations and their applications}
\newblock {Springer Monogr. Math. (2010). ISBN: 978-3-642-04630-8.}
\newblock \burlalt{doi:10.1007/978-3-642-04631-5}{http://dx.doi.org/10.1007/978-3-642-04631-5}.

\bibitem{WZ69}
W. ~Zimmermann.
\newblock{\em Convergence of Bogoliubov’s method of renormalisation in momentum space.}
\newblock {Comm. Math. Phys. \textbf{15}, (1969), 208--234.}
\newblock \burlalt{doi:10.1007/BF01645676}{http://dx.doi.org/10.1007/BF01645676}.
 
\end{thebibliography}
\end{document}